\documentclass[mnsc,nonblindrev]{informs3}

\SingleSpacedXI

\usepackage{xfrac}  

\usepackage{natbib}
 \bibpunct[, ]{(}{)}{,}{a}{}{,}%
 \def\bibfont{\small}%
 \usepackage[dvipsnames]{xcolor}
   \usepackage{tikz}
\usetikzlibrary{graphs,graphs.standard}
\usetikzlibrary{arrows}
\usetikzlibrary{arrows.meta}
\usetikzlibrary{graphs,graphs.standard,quotes}
\usetikzlibrary{shapes.geometric}
\newcount\mycount
\usetikzlibrary{calc}

 \usepackage[caption=false]{subfig}
 
 \usepackage{float}
\newfloat{algorithm}{t}{lop}
\floatname{algorithm}{Algorithm}

\TheoremsNumberedThrough     
\ECRepeatTheorems

\EquationsNumberedThrough    

\MANUSCRIPTNO{}

\usepackage{afterpage}

\usepackage[normalem]{ulem}
\newcommand\redsout{\bgroup\markoverwith{\textcolor{red}{\rule[0.5ex]{2pt}{0.4pt}}}\ULon}
\newcommand\gsout{\bgroup\markoverwith{\textcolor{green}{\rule[0.5ex]{2pt}{0.4pt}}}\ULon}

\usepackage{float}

\usepackage{array}
\newcommand{\PreserveBackslash}[1]{\let\temp=\\#1\let\\=\temp}
\newcolumntype{C}[1]{>{\PreserveBackslash\centering}p{#1}}
\newcolumntype{R}[1]{>{\PreserveBackslash\raggedleft}p{#1}}
\newcolumntype{L}[1]{>{\PreserveBackslash\raggedright}p{#1}}

\usepackage{mathtools} 
\usepackage{enumitem} 
\usepackage{dsfont} 
\usepackage{mathrsfs}  
\usepackage{verbatim}
\usepackage{romannum}
\AtBeginDocument{\pagenumbering{arabic}}

\newenvironment{itemize*}%
  {\begin{itemize}%
    \setlength{\itemsep}{0.5em}%
    \setlength{\parskip}{0pt}}%
  {\end{itemize}}

\let\R\Real

\usepackage{graphicx}
\usepackage{bbm}
\usepackage{tikz,ifthen}
\usepackage{fixltx2e}    
   \usetikzlibrary{math}
   \usetikzlibrary{shapes.misc}
   \usetikzlibrary{shapes.multipart,positioning,patterns,backgrounds}
\newcommand{\ubar}[1]{\text{\b{$#1$}}}

\usepackage{multirow}
\usepackage{array}

\newcommand{\eg}{{\it{e.g.}}, }

\makeatletter
\let\save@mathaccent\mathaccent
\newcommand*\if@single[3]{%
  \setbox0\hbox{${\mathaccent'0362{#1}}^H$}%
  \setbox2\hbox{${\mathaccent'0362{\kern0pt#1}}^H$}%
  \ifdim\ht0=\ht2 #3\else #2\fi
  }
\newcommand*\rel@kern[1]{\kern#1\dimexpr\macc@kerna}
\newcommand*\widebar[1]{\@ifnextchar^{{\wide@bar{#1}{0}}}{\wide@bar{#1}{1}}}
\newcommand*\wide@bar[2]{\if@single{#1}{\wide@bar@{#1}{#2}{1}}{\wide@bar@{#1}{#2}{2}}}
\newcommand*\wide@bar@[3]{%
  \begingroup
  \def\mathaccent##1##2{%
    \let\mathaccent\save@mathaccent
    \if#32 \let\macc@nucleus\first@char \fi
    \setbox\z@\hbox{$\macc@style{\macc@nucleus}_{}$}%
    \setbox\tw@\hbox{$\macc@style{\macc@nucleus}{}_{}$}%
    \dimen@\wd\tw@
    \advance\dimen@-\wd\z@
    \divide\dimen@ 3
    \@tempdima\wd\tw@
    \advance\@tempdima-\scriptspace
    \divide\@tempdima 10
    \advance\dimen@-\@tempdima
    \ifdim\dimen@>\z@ \dimen@0pt\fi
    \rel@kern{0.6}\kern-\dimen@
    \if#31
      \overline{\rel@kern{-0.6}\kern\dimen@\macc@nucleus\rel@kern{0.4}\kern\dimen@}%
      \advance\dimen@0.4\dimexpr\macc@kerna
      \let\final@kern#2%
      \ifdim\dimen@<\z@ \let\final@kern1\fi
      \if\final@kern1 \kern-\dimen@\fi
    \else
      \overline{\rel@kern{-0.6}\kern\dimen@#1}%
    \fi
  }%
  \macc@depth\@ne
  \let\math@bgroup\@empty \let\math@egroup\macc@set@skewchar
  \mathsurround\z@ \frozen@everymath{\mathgroup\macc@group\relax}%
  \macc@set@skewchar\relax
  \let\mathaccentV\macc@nested@a
  \if#31
    \macc@nested@a\relax111{#1}%
  \else
    \def\gobble@till@marker##1\endmarker{}%
    \futurelet\first@char\gobble@till@marker#1\endmarker
    \ifcat\noexpand\first@char A\else
      \def\first@char{}%
    \fi
    \macc@nested@a\relax111{\first@char}%
  \fi
  \endgroup
}
\makeatother

\begin{document}

 \RUNAUTHOR{Sturt}

\RUNTITLE{The value of robust assortment optimization under ranking-based choice models}

\TITLE{The Value of Robust Assortment Optimization Under Ranking-based Choice Models}

\ARTICLEAUTHORS{%
\AUTHOR{Bradley Sturt}
\AFF{Department of Information and Decision Sciences\\
University of Illinois at Chicago, \EMAIL{bsturt@uic.edu}}
} 

\ABSTRACT{%
We study a class of robust assortment optimization problems that was proposed by Farias, Jagabathula, and Shah (2013). The goal in these problems is to find an assortment that maximizes a firm's worst-case expected revenue under all ranking-based choice models that are consistent with the historical sales data generated by the firm's past assortments. We establish for various settings that these robust optimization problems can either be solved in polynomial-time or can be reformulated as  compact mixed-integer optimization problems.  To establish our results, we prove that optimal assortments  for these robust optimization problems have a  simple structure that is closely related to the structure of revenue-ordered assortments. 
We use our results to show how robust optimization can be used to overcome the risks of estimate-then-optimize and the need for experimentation with ranking-based choice models in the overparameterized regime.
\looseness=-1}%


\HISTORY{First version: Dec.\ 9, 2021. Revisions submitted on Feb.\ 16, 2023 and Nov.\ 6, 2023. Accepted for publication on Feb.\ 21, 2024. }
\maketitle


\defcitealias{farias2013nonparametric}{FJS13}
\vspace{-0.5em}
\begin{center}
``{First, do no harm}" - Hippocratic Oath
\end{center}
\vspace{-0.5em}
\section{Introduction} \label{sec:introduction}
The ranking-based choice model is one of the most fundamental and influential  discrete choice models in revenue management. It is used by firms in industries such as e-commerce and brick-and-mortar retail  to predict the demand for the firm's products as a function of the {subset} of products that the firm offers to their customers.  The popularity of the ranking-based choice model can be attributed to its generality: it can represent {any} random utility maximization model, and thus encompasses many other popular discrete choice models such as the multinomial logit model \citep{blockmarschak,farias2013nonparametric}.  The ranking-based choice model posits that customers who visit the firm  have preferences that are represented by randomly-chosen rankings, and, based on the subset of products offered by the firm, each customer will purchase the product that is most preferred according to their personal ranking.

Unfortunately, accurately estimating a ranking-based choice model from a firm's historical sales data is notoriously challenging. The key issue is that the ranking-based choice model is   comprised of around $n!$ parameters, where $n$ is the number of product alternatives that a firm can elect to offer to their customers. Because the number of past subsets of products $M$ that the firm has previously offered to their customers usually satisfies $M \ll n!$, many selections of these parameters can yield a ranking-based choice model that is consistent (has low or zero prediction error) with the firm's historical sales data. This leaves firms with the challenge of selecting \emph{which} ranking-based choice model, out of all of those   that are consistent with their historical sales data, to use when making operational planning decisions.

The challenge of estimation with ranking-based choice models is particularly acute in the context of \emph{assortment planning}. Here, the typical goal  of a firm is to identify a new subset of products (referred to as an {assortment}) to offer to their customers in order to increase the firm's expected revenue. Because the true relationship between assortments and expected revenue is unknown, firms will typically interpret one selection of a ranking-based choice model as the  ``ground truth" and subsequently solve an optimization problem to find an assortment which maximizes the ``predicted" expected revenue  \citep{aouad2018approximability,bertsimas2019exact,honhon2012optimal,van2015market,van2017expectation,feldman2019assortment,aouad2021assortment,desir2021mallows}. This widely-used technique for identifying a new assortment is referred to in the revenue management literature as \emph{estimate-then-optimize}. The {estimate-then-optimize} technique can be attractive from a computational  standpoint, due to its decoupling of the combinatorial problems related to estimation and optimization.  But the assortment that is optimal under one selection of a ranking-based choice model that is consistent with the historical  sales data may be highly \emph{suboptimal} under another ranking-based choice model that is also consistent with the historical sales data. This raises concerns about whether estimate-then-optimize can be trusted to identify a new assortment for the firm with an expected revenue that, at the very least, is no less than the firm’s highest expected revenue from their past assortments.

To overcome the risks of estimate-then-optimize in the overparameterized regime\footnote{We use the term `overparameterized regime' to refer to assortment planning problems in which many selections of ranking-based choice models are consistent with the historical sales data generated by the firm's past assortments.}, one  potential approach available to  firms is to perform {experimentation}.  That is, a firm can offer various new assortments to their customers for short durations, and then use the sales data obtained by this experimentation to further constrain the set of ranking-based choice models that are consistent with the firm's historical sales data. 
But engaging in experimentation may be undesirable for various practical reasons, ranging from a firm's aversion to short-term losses  incurred by experimentation to the implementation challenges for performing experimentation in settings such as brick-and-mortar retail~\citep{ArielyDan2010CWBD,Ignat_2022}. 
 Moreover, the large number of parameters in ranking-based choice models implies that experimentation with an extensive array of new assortments may not guarantee the identification of a unique ranking-based choice model~\citep{sher2011partial,susan2022active}.  Given that experimentation may be costly and may not lead to the identification of a  unique ranking-based choice model, our paper investigates an alternative approach to experimentation centered around \emph{robust optimization}.

\subsection{A Robust Optimization Approach} \label{sec:intro:robust} 
Our paper investigates a class of robust optimization problems proposed by  \citet*[henceforth abbreviated  as \citetalias{farias2013nonparametric}]{farias2013nonparametric} that aims to circumvent the risks of estimate-then-optimize  with ranking-based choice models in the overparameterized regime. The goal of the class of robust optimization problems, stated succinctly, is to find  an assortment that has high expected revenue under \emph{all} of the  ranking-based choice models that are consistent with the historical  sales data generated by the firm's past assortments. The robust optimization problems achieve this goal by yielding an assortment that maximizes the predicted expected revenue  under the {worst-case} ranking-based choice model that is consistent with  the historical  sales data generated by the firm's past assortments.  Formally, the class of robust optimization problems is given by
\begin{equation} \tag{RO} \label{prob:robust}
\begin{aligned}
 & \underset{S \in \mathcal{S}}{\textnormal{max}} \min_{\lambda \in \mathcal{U}}\mathscr{R}^{\lambda}(S),
\end{aligned}
\end{equation}
where $\mathcal{S}$ is the set of all assortments that can be offered by the firm to its customers,  $\mathcal{U}$ is the set of all ranking-based choice models that are consistent with the  historical sales data generated by the firm's past assortments, and $\mathscr{R}^\lambda(S)$ is the expected revenue for a new assortment $S$ under a ranking-based choice model $\lambda \in \mathcal{U}$. \looseness=-1 

The class of robust  optimization problems~\eqref{prob:robust} can be attractive from a managerial perspective for a variety of reasons. First, in contrast to estimate-then-optimize, the robust optimization problem~\eqref{prob:robust} aims to  find an assortment that can be trusted to improve the firm’s expected revenue in a way that is not exclusive to just one of the many ranking-based choice models that are consistent with the firm’s historical sales data. Second, the robust optimization problem~\eqref{prob:robust} can potentially obviate the need for experimentation---which can be costly or otherwise undesirable to firms---by identifying assortments with desirable expected revenue guarantees  even though a unique ranking-based choice model cannot be identified using the firm's historical sales data. 
 Third, \citetalias{farias2013nonparametric} showed using real-world data that the worst-case expected revenue $\min_{\lambda \in \mathcal{U}}\mathscr{R}^{\lambda}(S)$ can be a significantly more accurate prediction of expected revenue compared to the predictions made by parametric methods. Hence, there is empirical evidence that the assortments obtained from solving \eqref{prob:robust} may be desirable to firms from a revenue-improving perspective. 

Despite the appeal of the  class of robust optimization problems~\eqref{prob:robust}, very limited progress has been made on solving these robust optimization problems until now. As far as we can tell, the only prior work that considers the computational tractability of  \eqref{prob:robust} is that of \citetalias{farias2013nonparametric}, which presented algorithms for computing the worst-case expected revenue of a fixed assortment.  However, no algorithms to date have been developed for solving this class of robust optimization problems (\citet[p.867]{rusmevichientong2012robust}, \citet[p.8]{jagabathula2014assortment}), and “it is not clear how one may formulate the problem of optimizing the worst-case revenue as an efficiently solvable mathematical optimization reformulation.”  \citep[p.118]{mivsic2016data}.  This lack of progress is especially notable, in view of the fact that algorithms and structural results have been developed for other classes of robust assortment optimization problems  \citep{rusmevichientong2012robust,bertsimas2017robust,desir2019nonconvex,wang2020randomized}. 

The lack of practical algorithms and structural results for \eqref{prob:robust} has ultimately prevented a formal understanding of the settings in which the class of robust optimization problems can provide managerial value.  On  one hand, the robust optimization problem~\eqref{prob:robust} promises to help firms circumvent the risks of estimate-then-optimize by finding an assortment that delivers high expected revenue across all of the ranking-based choice models that are consistent with the firm's historical sales data.    But on the other hand, no previous work has established whether assortments $S \in \mathcal{S}$ that have high expected revenue $\mathscr{R}^\lambda(S)$  across all ranking-based choice models  $\lambda \in \mathcal{U}$ can ever \emph{exist}. Because the question of whether \eqref{prob:robust} is too conservative to be practically valuable has remained open, the extent to which the robust optimization approach can actually overcome the risks of estimate-then-optimize or serve as a viable alternative to experimentation  has been unknown.

\subsection{Main Questions}
In this paper, we pursue an agenda of understanding \emph{if} and \emph{when} the class of robust  optimization problems~\eqref{prob:robust} can be computationally tractable and provide  managerial value as an alternative to estimate-then-optimize and experimentation in the overparameterized regime.   In the pursuit of this agenda, we focus in this paper on four open questions.

 Our first question concerns the conservatism of the class of robust optimization problem~\eqref{prob:robust}: 
  \vspace{1em}
\begin{question}
\label{question:1}
\emph{Can the optimal objective value of \eqref{prob:robust} ever be strictly greater than the expected revenue of the firm's best past assortment?} 
\end{question}
We show in \S\ref{sec:setting}  that Question~\ref{question:1}  is equivalent to asking whether there can ever exist an assortment with an expected revenue that is strictly greater than the expected revenue of the firm's best past assortment  across all of the ranking-based choice models that are consistent with the historical sales data generated by the firm's past assortments. An affirmative answer to Question~\ref{question:1} is thus necessary for it to be possible for the robust optimization problem~\eqref{prob:robust} to yield new assortments that can be trusted to improve the firm’s expected revenue in a way that is not exclusive to just a subset of the many ranking-based choice models that are consistent with the firm’s historical sales data.\looseness=-1

Our second question asks about the relationship between the conservatism of \eqref{prob:robust} and the past assortments that the firm has offered to its customers:
\vspace{1em}
\begin{question}
\label{question:2}
\emph{If the answer to Question~\ref{question:1} is yes, then what structural properties of a firm's past assortments are \emph{necessary} for the optimal objective value of \eqref{prob:robust} to be strictly greater than the expected revenue of the firm's best past assortment?}\looseness=-1
\end{question}
Recall that Question~\ref{question:1} is an existential question, in the sense that Question~\ref{question:1} asks whether there exists a problem instance (meaning a collection of past assortments, prices, and historical sales data) for which the optimal objective value of \eqref{prob:robust} is strictly greater than the expected revenue of the firm's best past assortment. But even if the answer to Question~\ref{question:1} is yes, there exist collections of past assortments for which the optimal objective value of \eqref{prob:robust} will never be strictly greater than the expected revenue of the firm's best past assortment.\footnote{For example, we readily observe that the optimal objective value of \eqref{prob:robust} will never be strictly greater than the expected revenue of the firm's best past assortment if every assortment $S \in \mathcal{S}$ is a past assortment (i.e., has already been offered by the firm to its customers). }
In response, Question~\ref{question:2}  asks whether there are structural properties of the firm's collection of past assortments that are necessary  in order for \eqref{prob:robust} to have a chance of not being overly conservative.\looseness=-1

Our third question relates the  computational tractability of the robust optimization problem~\eqref{prob:robust} to the  collection of past assortments that the firm has offered to its customers: 
\vspace{1em}
\begin{question}
 \label{question:3}
\emph{What structural properties of a firm's past assortments are \emph{sufficient} for designing efficient algorithms for solving the robust optimization problem~\eqref{prob:robust}?  }
\end{question}
Even if the optimal objective value of \eqref{prob:robust} is strictly greater than the expected revenue of the firm's best past assortment for a given problem instance, the practical value of the robust optimization problem~\eqref{prob:robust} for that problem instance is limited if an optimal assortment for \eqref{prob:robust} cannot be found in practical computation times. 
Moreover, as discussed in \S\ref{sec:intro:robust}, the class of robust optimization problems~\eqref{prob:robust} is widely believed in the revenue management community to be computationally intractable, and no practical
algorithms for solving these problems for any collection of past assortments have been developed thus far. Against this backdrop, Question~\ref{question:3} asks whether there are structural properties of a firm's past assortments that make it possible to design theoretically- and practically-efficient algorithms for solving \eqref{prob:robust}.  

Our fourth question asks how the robust optimization problem~\eqref{prob:robust} can be deployed in settings in which the robust optimization approach is overly conservative: 
\vspace{1em}
\begin{question} \label{question:4}
\emph{If the optimal objective value of \eqref{prob:robust} is not strictly greater than the expected revenue of the firm's best past assortment for a particular problem instance, then what assortment should the firm offer to its customers?}
\end{question}
If the optimal objective value  of \eqref{prob:robust} is strictly greater than the expected revenue of the firm's best past assortment for a given problem instance, then we recommend that the firm offer the optimal assortment for \eqref{prob:robust} to its customers, as such an assortment can be trusted to improve the firm’s expected revenue in a way that is not exclusive to just one of the many ranking-based choice models that are consistent with the firm’s historical sales data. But if the optimal objective value  of the robust optimization problem~\eqref{prob:robust} is not strictly greater than the expected revenue of the firm's best past assortment for the given problem instance, then can the robust optimization problem~\eqref{prob:robust} provide any guidance as to what should a firm do?

\subsection{Contributions} \label{sec:contributions}
Our main contributions of this paper consist of 
providing the first answers to Questions~\ref{question:1}-\ref{question:4}. 

First, we establish an affirmative answer to Question~\ref{question:1}. Specifically,  we present an example in \S\ref{sec:question1} of an assortment planning problem with  two past assortments in which the optimal objective value of the robust optimization problem~\eqref{prob:robust} is strictly  greater than the expected revenue of the firm's best past assortment.   The example thus shows that the robust optimization problem~\eqref{prob:robust}  can yield practical value to firms, even when a small number of assortments have been offered in the past.\looseness=-1

Second, we  address Question~\ref{question:2} by 
characterizing  the structure of optimal assortments for the class of robust optimization problems~\eqref{prob:robust}   (Theorem~\ref{thm:main} in \S\ref{sec:characterization}). Our characterization reveals that  optimal assortments for these robust  optimization problems have a simple and interpretable structure that is similar to the structure of the widely-studied class of revenue-ordered assortments. We use Theorem~\ref{thm:main} to show that it is necessary for the  firm's  past assortments to deviate from revenue-ordered assortments  in order for it be possible for the optimal objective value of \eqref{prob:robust} to be strictly greater than the expected revenue of the firm's best past assortment (\S\ref{sec:characterization:question}). We further show that Theorem~\ref{thm:main} can be used to drastically reduce the number of candidate assortments that need to be checked in order to solve the robust optimization problem~\eqref{prob:robust} when the number of past assortments is small (\S\ref{sec:characterization:upperbound}-\S\ref{sec:characterization:reverse}).  Our proof of Theorem~\ref{thm:main}  is based on a simple yet intricate analysis of reachability conditions for vertices in a data-driven class of directed acyclic graphs (\S\ref{sec:characterization:proof}).

Third, we provide an affirmative answer to Question~\ref{question:3} by developing the first polynomial-time algorithms and compact mixed-integer optimization reformulations for the class of robust optimization problems~\eqref{prob:robust}.   The existence of polynomial-time algorithms resolves the open question in the literature regarding the theoretical tractability of \eqref{prob:robust} for settings with small numbers of past assortments, and we provide numerical evidence that our algorithms and mixed-integer optimization reformulations are practical. 
In greater detail, we make the following algorithmic developments:\looseness=-1
\begin{enumerate}[label=(\roman*)]

  \item We use our structure of optimal assortments to develop a strongly polynomial-time algorithm for solving the robust optimization problem when the firm has offered two past assortments (Theorem~\ref{thm:two} in \S\ref{sec:algorithm:twoassortments}). 
Our algorithm reduces the robust  optimization problem to a sequence of minimum-cost network flow problems,  and our algorithm has a total running time of $\mathcal{O}(n^5 \log (n r_n))$, where $n$ is the number of available products and $r_n$ is the integral price of the most expensive product. We provide numerical experiments showing that our algorithm can require less than 30 seconds for problem instances with $n=100$ products. 
 
 \item  We generalize our algorithm for two past assortments to show that the robust optimization problem can be solved in weakly polynomial-time for any fixed number of past assortments (Theorem~\ref{thm:poly} in \S\ref{sec:algorithms:fixed_dim}).   The algorithm provides evidence that it can be possible to develop practical algorithms for solving the robust optimization problem~\eqref{prob:robust} for real-world problem instances, where the composition of products in the past assortments do not exhibit any convenient structure and where there may be no ranking-based choice models that have zero prediction error on the historical sales data.

 \item  To address applications in which the number of past assortments is large, we consider a common structural assumption in the revenue management literature in which the past assortments are nested (see \S\ref{sec:algorithms:nested}). When the past assortments are nested, we show that our technical developments from \S\ref{sec:characterization} can be used to develop a  mixed-integer optimization reformulation of the robust optimization problem~\eqref{prob:robust}   in which the number of decision variables and constraints scales polynomially in both the number of products as well as  the number of past assortments (Theorem~\ref{thm:nested} in \S\ref{sec:algorithms:nested}).  We show using numerical experiments that our mixed-integer linear optimization problem can require less than 30 seconds to solve \eqref{prob:robust} to optimality on problem instances with  up to $M= 20$  past assortments.

\end{enumerate}

Finally, we address Question~\ref{question:4} by showing that our algorithms can yield  managerial value even when the optimal objective value of \eqref{prob:robust} is not strictly greater than the expected revenue of the firm's best past assortment.  Specifically, we show that firms can compute an upper bound on the benefits of experimentation by extending our algorithms from \S\ref{sec:algorithms} to solve an optimistic version of the robust optimization  problem~\eqref{prob:robust} (Theorems~\ref{thm:poly:optimistic} and \ref{thm:nested:optimistic} in \S\ref{sec:question4:bestcase}). If the gap between the upper bound and the expected revenue of the firm's best past assortment is small, then we recommend that the firm should not perform experimentation and continue offering their best past assortment.  If the gap is large, then the firm can potentially benefit from experimentation, and we show in this case that the robust and optimistic optimization problems can be combined into a Pareto approach for finding new assortments for experimentation that have potential to increase the firm's best-case expected revenue while maintaining bounds on the worst-case decrease in expected revenue. We illustrate the practical value of the aforementioned techniques using synthetic and real data. 

In summary, the techniques proposed in this paper can provide value to firms such as brick-and-mortar and small e-commerce retailers who  change their assortments infrequently, wish to find new assortments that generate higher expected revenue than their past assortments, and lack the infrastructure or willingness to experiment with new assortments that lead to short-run declines in expected revenue. Indeed, we show using synthetic and real data in Appendix~\ref{appx:numerics} that our proposed algorithms require practical computation times and yield new assortments with desirable worst-case revenue guarantees when the number of past assortments is small or when the past assortments are nested. The numerical experiments thus show that the methodology developed in this paper is well positioned to provide value for the sizes and types of historical sales data faced by real-world firms.\looseness=-1

The rest of this paper is organized as follows.  In \S\ref{sec:setting}, we present our problem setting.  In \S\ref{sec:question1}, we give an affirmative answer to Question~\ref{question:1} by  presenting  an example of an assortment planning problem with two past assortments in which the optimal objective value of the robust optimization problem~\eqref{prob:robust} is strictly greater than the expected revenue  of the firm's best  past assortment. In \S\ref{sec:characterization}, we  develop our key structural result which characterizes the optimal solutions for robust optimization problems under ranking-based choice models. In \S\ref{sec:algorithms}, we develop polynomial-time algorithms and mixed-integer optimization reformulations for the robust optimization problem. In \S\ref{sec:question4}, we show how our algorithms can provide value in problem instances in which the  optimal objective value of the robust optimization problem~\eqref{prob:robust} is not strictly greater than the expected revenue  of the firm's best  past assortment. In \S\ref{sec:conclusion}, we offer concluding thoughts and directions for future research. The numerical experiments from this paper are found in Appendix~\ref{appx:numerics}, and the proofs of the technical results from the paper are found in Appendices~\ref{appx:characterization_proofs}-\ref{appx:optimistic_algorithms}.

\paragraph{Notation and Terminology. }
We use $\R$ to denote the real numbers, $\R_+$ to denote the nonnegative real numbers, and $y^\intercal x$ to denote the inner product of two vectors. We use the phrase `{collection}' to refer to a set of sets. We let the set of all probability distributions which are supported on a finite set $\mathscr{A}$ be denoted by $\Delta_{\mathscr{A}} \triangleq \{ \lambda: \sum_{a \in \mathscr{A}} \lambda_a = 1, \; \lambda_a \ge 0 \; \forall a \in \mathscr{A} \}$.  We assume throughout that a norm $\| \cdot \|$ is either the $\ell_1$-norm or $\ell_\infty$-norm, and so it follows that optimization problems of the form $\min_{x,y} \{ c^\intercal x + d^\intercal y \mid Ax + By \le b, \; \| y \| \le \eta \}$ can be referred to as linear optimization problems.  We let $\mathbb{I} \{\cdot \}$ denote the indicator function, which equals one if $\cdot$  is true and equals zero otherwise.

\paragraph*{Code Availability.}
The code for conducting the numerical experiments in this paper is  freely available and can be accessed at \url{https://github.com/brad-sturt/IdentificationQuestion}.

\section{Problem Setting} \label{sec:setting}
 We adopt the perspective of a firm that must select a subset of products to offer to their customers.  Let the universe of products available to the firm be denoted by $\mathcal{N} \triangleq \{1,\ldots,n\}$, where the no-purchase option is denoted by index $0$ and $\mathcal{N}_0 \triangleq \mathcal{N} \cup \{0\}$.  The revenue generated by selling one unit of product $i \in \mathcal{N}$ is represented by $r_i > 0$, and the revenue associated with the no-purchase option is  $r_0 = 0$. We assume throughout the paper that the revenues of the products are unique and that products have been sorted in ascending order  by revenue, $0 < r_1 < \cdots < r_n$.  
An assortment is defined as any subset  of products $S \subseteq \mathcal{N}_0$, and we let $\mathcal{S} \triangleq \{ S \subseteq \mathcal{N}_0: 0 \in S \}$ denote the collection of all assortments that  include the no-purchase option. 
 
 We study a problem setting in which the underlying relationship between assortment and customer demand is unknown, and our only information on this relationship comes from historical sales data generated by the firm's past assortments. Let the past assortments be denoted by $\mathscr{M} \triangleq \{S_1,\ldots,S_M\} \subseteq \mathcal{N}_0,$ and let the indices of these past assortments be denoted by $\mathcal{M} \triangleq \{1,\ldots,M\}$. Unless stated otherwise, we will  make no assumptions on the mechanism by which the firm selected the assortments to offer in the past. That is, the firm could have chosen the past assortments by drawing products randomly; alternatively,  the past assortments could have been chosen using managerial intuition or some other systematic approach. 
We assume that the firm offered each past assortment $S_m \in \mathscr{M}$ to their customers for  a sufficient duration to obtain an accurate estimate of the purchase frequencies, \emph{i.e.}, the fraction of customers $v_{m,i} \in [0,1]$ that  purchase product $i \in S_m$ when offered assortment $S_m$. This historical sales data is assumed to be normalized such that $\sum_{i \in \mathcal{N}_0}v_{m,i} = 1$, and the purchase frequencies for products that are not in  an assortment are defined equal to zero, that is, $v_{m,i} = 0$ for all $i \notin S_m$ and $m \in \mathcal{M}$. The expected revenue generated by the past assortment $S_m$ is written compactly as $\sum_{i \in \mathcal{N}_0} r_i v_{m,i} = r^\intercal v_m$.

A discrete choice model is a function that predicts purchase frequencies for the firm based on the assortment that the firm offers to their customers. A \emph{ranking-based choice model} is a type of choice model which is parameterized by a probability distribution $\lambda$ over the set of all distinct rankings of the products,  where a ranking refers to a one-to-one mapping of the form $\sigma: \{0,\ldots,n\} \to \{0,\ldots,n\}$. Specifically, a ranking $\sigma$ encodes a preference for product $i$ over product $j$ if and only if $\sigma(i) < \sigma(j)$. Let the set of all distinct rankings over the products be denoted by $\Sigma$, and we readily observe that the number of distinct rankings in this set satisfies $| \Sigma| = (n+1)!$. Given a probability distribution over rankings $\lambda \in \Delta_\Sigma$ and an assortment $S \in \mathcal{S}$, the prediction made by the ranking-based choice model for the purchase frequency of each product $i \in \mathcal{N}_0$ is given by
\begin{align*}
\mathscr{D}^\lambda_i(S) \triangleq  \sum_{\sigma \in \Sigma} \mathbb{I} \left \{ i = \argmin_{j \in S} \sigma(j)\right \}  \lambda_\sigma. 
\end{align*} 
It is straightforward  to see from the above definition that a ranking-based choice model always satisfies the equality $\mathscr{D}^\lambda_i(S) = 0$ for all  products $i$ that are not in the assortment $S$.  
The predicted expected revenue for a firm that offers assortment $S$ under the ranking-based choice model with parameter $\lambda$ is given by 
\begin{align*}
\mathscr{R}^{\lambda}(S) \triangleq \sum_{i \in \mathcal{N}_0} r_i \mathscr{D}^\lambda_i(S) = r^\intercal \mathscr{D}^\lambda(S). 
\end{align*}

We say that a ranking-based choice model is \emph{consistent} with the historical sales data generated by the firm's past assortments if the difference between the predicted purchase frequency $\mathscr{D}^\lambda_i(S_m)$ and the historical sales data $v_{m,i}$ is small for each of the products $i \in S_m$ that were offered in each of the past assortments $m \in \mathcal{M}$. We define the set of all ranking-based choice models that are consistent with the historical sales data as 
\begin{align*}
\mathcal{U} \triangleq \left \{\lambda \in \Delta_\Sigma: \quad \begin{aligned}
&\textnormal{there exists a vector } \epsilon \textnormal{ such that }\| \epsilon \| \le \eta \textnormal{ and }  \\
&\mathscr{D}^\lambda_i(S_m) - v_{m,i}  = \epsilon_{m,i} \textnormal{ for all }  i \in S_m \text{ and } m \in \mathcal{M} \end{aligned} \right \},
\end{align*} 
where the radius $\eta \ge 0$ of the set $\mathcal{U}$ is a parameter that is selected by the firm, and where $\| \cdot \|$ is defined at the end of \S\ref{sec:introduction} as either the $\ell_1$ or $\ell_\infty$ norm.

To develop an understanding for the above set of ranking-based choice models $\mathcal{U}$, let us consider the case in which the radius $\eta$ of the above set is equal to zero. 
In that case,  we observe that the above set contains exactly  the  probability distributions for which the corresponding ranking-based choice models have perfect accuracy on the historical sales data generated by the past assortments. In other words, if $\eta = 0$, then the set $\mathcal{U}$ is comprised of all of the probability distributions $\lambda \in \Delta_\Sigma$ that satisfy $\mathscr{D}^\lambda_i(S_m) = v_{m,i}$ for each of the products $i \in S_m$ that were offered in each of the past assortments $m \in \mathcal{M}$.   From a theoretical perspective, it is known that the set $\mathcal{U}$ with $\eta = 0$ is guaranteed to be nonempty  if the firm's customers' behavior is captured by a random utility maximization model and if the historical sales data has been observed without noise; see \cite{blockmarschak}. 
From a practical perspective, it can be reasonable to expect that the set $\mathcal{U}$ will be nonempty with $\eta = 0$  in problem instances in which the number of past assortments $M$ is much smaller than the number of parameters in the ranking-based choice model, $| \Sigma| = (n+1)!$. Nonetheless, if there are no ranking-based choice models that have perfect accuracy on the historical sales data, then the radius $\eta$ can always be made sufficiently large to ensure that the set of ranking-based choice models $\mathcal{U}$ is nonempty.\footnote{A sufficiently large choice of $\eta$  which ensures that $\mathcal{U}$ is nonempty can easily be found by applying binary search over the choice of $\eta$, provided that one has an algorithm that solves \eqref{prob:robust}.} For simplicity, we make the standing assumption throughout our paper that  the set of ranking-based choice models $\mathcal{U}$ is nonempty for the firm's selection of the radius $\eta$.\looseness=-1

We conclude the present section with a theoretical derivation of the class of robust optimization problems~\eqref{prob:robust}. Indeed, consider a firm that wishes to identify an assortment with an expected revenue that is strictly greater than expected revenues of the firm's  past assortments under all of the ranking-based choice models that are consistent with firm's historical sales data.  In this case, the firm seeks an assortment $S \in \mathcal{S}$ that satisfies
\begin{align}
 \mathscr{R}^{\lambda}(S) &> \max \left \{  \mathscr{R}^{\lambda}(S_1),\ldots,  \mathscr{R}^{\lambda}(S_M)  \right \} \quad \forall \lambda \in \mathcal{U}. \tag{IMP'}\label{line:improvement_old}
\end{align}
If $\eta = 0$, then it follows from algebra\footnote{If $\eta = 0$, then we observe for each past assortment $m \in \mathcal{M}$ that $\mathscr{R}^\lambda(S_m) = r^\intercal \mathscr{D}^\lambda(S_m) = r^\intercal v_m$ for all $ \lambda \in \mathcal{U}$.} that  an assortment $S \in \mathcal{S}$  satisfies \eqref{line:improvement} if and only if 
\begin{align}
 \min_{\lambda \in \mathcal{U}} \mathscr{R}^\lambda(S) > \max_{m \in \mathcal{M}} r^\intercal v_m \tag{IMP}.\label{line:improvement}
\end{align}
The above analysis shows that there exists an assortment that satisfies \eqref{line:improvement_old} if and only if the optimal objective value of the robust optimization problem~\eqref{prob:robust} is strictly greater than $\max_{m \in \mathcal{M}} r^\intercal v_m$, i.e., the expected revenue generated by the firm's best past assortment. Moreover, if there exists an assortment that satisfies \eqref{line:improvement}, then every optimal solution of the robust optimization problem~\eqref{prob:robust} satisfies~\eqref{line:improvement}. Hence, we conclude that the robust optimization problem~\eqref{prob:robust} enjoys theoretical justification as a tool for identifying a new assortment with an expected revenue that is strictly greater than expected revenues of the firm's  past assortments across all of the ranking-based choice models that are consistent with firm's historical sales data.

\section{Answer to Question~\ref{question:1}} \label{sec:question1}

In this section, we present  an example of an assortment planning problem with two past assortments in which the optimal objective value of the robust optimization problem~\eqref{prob:robust} is strictly greater than the expected revenue  of the firm's best  past assortment. The example is insightful for three reasons. 
 First, the example shows that there exist problem instances in which the  assortment that satisfies~\eqref{line:improvement} is not a revenue-ordered assortment (the formal definition of revenue-ordered assortments is found in \S\ref{sec:characterization}).  Second, the example shows that estimate-then-optimize can perform poorly---that is, yield an assortment with an expected revenue that is strictly less than the  expected revenue of the firm's best past assortment---even in problem instances for which there exist assortments that satisfy~\eqref{line:improvement}. Third,  the example illustrates that  assortments that satisfy~\eqref{line:improvement} can have a non-trivial structure, and discovering such assortments appears to be  challenging without the aid of structural results and algorithms like those developed in \S\ref{sec:characterization} and \S\ref{sec:algorithms}.

We begin by defining the assortment planning problem that will be our focus throughout this section. The problem instance is comprised by a universe of four products $\mathcal{N} \equiv \{1,2,3,4\}$, and the two past assortments that the firm has previously offered to its customers are denoted by $S_1 = \{0,2,3,4\}$ and $S_2 = \{0,1,2,4\}$. 
 The revenues corresponding to the four products are $r_1 = \$10$, $r_2 = \$20$, $r_3 = \$30$, and $r_4 = \$100$. The historical sales data generated by the  past assortments are the purchase frequencies $(v_{1,0}, v_{1,2}, v_{1,3}, v_{1,4}) = (0.3,0.3,0.3,0.1)$ and $(v_{2,0}, v_{2,1}, v_{2,2}, v_{2,4}) = (0.3,0.3,0.1,0.3)$.   We analyze the robust optimization problem~\eqref{prob:robust} in which the uncertainty set is constructed as all ranking-based choice models that perfectly fit this historical sales data, i.e., the radius of $\mathcal{U}$ is $\eta  = 0$.\looseness=-1

Let us offer some useful facts about the problem instance. We first observe that the expected revenues generated by the two past assortments  $S_1 = \{0,2,3,4\}$ and $S_2 = \{0,1,2,4\}$ are
\begin{align*}
r^\intercal v_1 &= \$0 \times 0.3 + \$20 \times 0.3 + \$30 \times 0.3 + \$100 \times 0.1 = \$25,\\
r^\intercal v_2 &= \$0 \times 0.3 + \$10 \times 0.3 + \$20 \times 0.1 + \$100 \times 0.3 = \$35. 
\end{align*}
Hence, the expected revenue generated by the firm's best past assortment is $\max \{\$25, \$35\} = \$35$.  We next show that  the set of ranking-based choice models $\mathcal{U}$  is nonempty when  the radius is $\eta = 0$. To show this, we will construct a ranking-based choice model $\lambda$ that perfectly fits perfectly fits the historical sales data. Specifically, consider the following five rankings:
\begin{align*}
\sigma^1(0) &= 0,&\sigma^2(1) &= 0,&\sigma^3(1) &= 0,& \sigma^4(2)&= 0,&\sigma^5(3)&=0,\\
\sigma^1(\phantom{0}) &= 1,&\sigma^2(2)&= 1,&\sigma^3(4) &= 1,&\sigma^4(4) &= 1,&\sigma^5(4) &= 1,\\
\sigma^1(\phantom{0})&=2,&  \sigma^2(4)&=2 ,&\sigma^3(0)&= 2,& \sigma^4(0) &= 2,&\sigma^5(0)&= 2,\\
\sigma^1(\phantom{0})&=3,&\sigma^2(0) &= 3,& \sigma^3(\phantom{0})&=3,& \sigma^4(\phantom{0}) &= 3,&\sigma^5(\phantom{0}) &= 3,\\
\sigma^1(\phantom{0}) &=4,&\sigma^2(\phantom{0}) &= 4,&\sigma^3(\phantom{0}) &= 4,&\sigma^4(\phantom{0})&=4,&\sigma^5(\phantom{0}) &= 4.
\end{align*}
We recall from \S\ref{sec:setting} that a ranking $\sigma \in \Sigma$ prefers product $i \in \mathcal{N}_0$ over product $j \in \mathcal{N}_0$ if $\sigma(i) < \sigma(j)$. For the sake of simplicity, we have chosen without loss of generality to not specify the preference ordering among products that are less preferred than the no-purchase option in each ranking. We observe that the most preferred products from the past assortments $S_1 = \{0,2,3,4\}$ and $S_2 = \{0,1,2,4\}$ under each of the five rankings are
\begin{align*}
\argmin_{j \in S_1} \sigma^1(j) &= 0,&\argmin_{j \in S_1} \sigma^2(j) &= 2,&\argmin_{j \in S_1} \sigma^3(j) &= 4,&\argmin_{j \in S_1} \sigma^4(j) &= 2,&\argmin_{j \in S_1} \sigma^5(j) &= 3,\\
 \argmin_{j \in S_2} \sigma^1(j) &= 0,& \argmin_{j \in S_2} \sigma^2(j) &= 1,& \argmin_{j \in S_2} \sigma^3(j) &= 1,& \argmin_{j \in S_2} \sigma^4(j) &= 2,& \argmin_{j \in S_2} \sigma^5(j) &= 4.
\end{align*} 
It follows from the above observations that the historical sales data $(v_{1,0}, v_{1,2}, v_{1,3}, v_{1,4}) = (0.3,0.3,0.3,0.1)$ and $(v_{2,0}, v_{2,1}, v_{2,2}, v_{2,4}) = (0.3,0.3,0.1,0.3)$ is perfectly fit by the ranking-based choice model  $\hat{\lambda} \in \Delta_\Sigma$ that satisfies the following equalities:
\begin{align*}
\hat{\lambda}_{\sigma^1} &= 0.3,& \hat{\lambda}_{\sigma^2} &= 0.2,& \hat{\lambda}_{\sigma^3} &= 0.1,&\hat{\lambda}_{\sigma^4} &= 0.1,&\hat{\lambda}_{\sigma^5}& = 0.3.
\end{align*}

Equipped with the above useful facts, we now consider the performance of estimate-then-optimize for this specific problem instance.  Specifically, suppose the firm uses  estimate-then-optimize in which the selected ranking-based choice model is the probability distribution $\hat{\lambda}$. In that case,  estimate-then-optimize will output a new assortment $S^{\hat{\lambda}}$ that is an optimal solution to the combinatorial optimization problem $\max_{S \in \mathcal{S}} \mathscr{R}^{\hat{\lambda}}(S)$. 
To determine the optimal solution to this  combinatorial optimization problem, let us inspect the five rankings $\sigma^1,\ldots,\sigma^5$ that have nonzero probability under the estimated probability distribution. We observe that the first of these five rankings satisfies $\argmin_{j \in S} \sigma^1(j) = 0$ for all assortments $S \in \mathcal{S}$. Moreover, we observe that the assortment $\{0,4\} \in \mathcal{S}$ satisfies $\argmin_{j \in  \{0,4\}} \sigma^2(j) = \cdots = \argmin_{j \in  \{0,4\}} \sigma^5(j) = 4$. We conclude that $\{0,4\}$ is an optimal solution for the above combinatorial optimization problem, and the predicted expected revenue for this assortment is $\mathscr{R}^{\hat{\lambda}}(\{0,4\}) = 0.7 \times \$100 = \$70$. 

We will now show that the assortment $S^{\hat{\lambda}} \triangleq  \{0,4\}$ obtained from estimate-then-optimize can  have an expected revenue  that is strictly less than the expected revenue of the firm's best past assortment under a ranking-based choice model that is consistent with the historical sales data. To show this, we consider the following seven rankings:
\begin{align*}
\sigma^6(0) &= 0,&\sigma^7(1) &= 0,&\sigma^8(2) &= 0,& \sigma^9(3)&= 0,&\sigma^{10}(4)&=0,&\sigma^{11}(1)&=0,&\sigma^{12}(3)&=0, \\
\sigma^6(\phantom{0}) &= 1,&\sigma^7(0) &= 1,&\sigma^8(0) &= 1,& \sigma^9(0)&= 1,&\sigma^{10}(0)&=1,&\sigma^{11}(2)&=1,&\sigma^{12}(4)&=1, \\
\sigma^6(\phantom{0}) &= 2,&\sigma^7(\phantom{0}) &= 2,&\sigma^8(\phantom{0}) &= 2,& \sigma^9(\phantom{0})&= 2,&\sigma^{10}(\phantom{0})&=2,&\sigma^{11}(0)&=2,&\sigma^{12}(0)&=2, \\
\sigma^6(\phantom{0}) &= 3,&\sigma^7(\phantom{0}) &= 3,&\sigma^8(\phantom{0}) &= 3,& \sigma^9(\phantom{0})&= 3,&\sigma^{10}(\phantom{0})&=3,&\sigma^{11}(\phantom{0})&=3,&\sigma^{12}(\phantom{0})&=3, \\
\sigma^6(\phantom{0}) &= 4,&\sigma^7(\phantom{0}) &= 4,&\sigma^8(\phantom{0}) &= 4,& \sigma^9(\phantom{0})&= 4,&\sigma^{10}(\phantom{0})&=4,&\sigma^{11}(\phantom{0})&=4,&\sigma^{12}(\phantom{0})&=4.
\end{align*}
We observe that the most preferred products from the first past assortment $S_1 = \{0,2,3,4\}$ are 
\begin{gather*}
\argmin_{j \in S_1} \sigma^6(j) = 0,\quad\argmin_{j \in S_1} \sigma^7(j) = 0,\quad\argmin_{j \in S_1} \sigma^8(j) = 2,\quad \argmin_{j \in S_1} \sigma^9(j) = 3,\quad \argmin_{j \in S_1} \sigma^{10}(j) = 4, \\
\argmin_{j \in S_1} \sigma^{11}(j) = 2, \quad \argmin_{j \in S_1} \sigma^{12}(j) = 3,
\end{gather*}
the most preferred products from the second past assortment $S_2 = \{0,1,2,4\}$ are
\begin{gather*}
\argmin_{j \in S_2} \sigma^6(j) = 0,\quad \argmin_{j \in S_2} \sigma^7(j) = 1,\quad \argmin_{j \in S_2} \sigma^8(j) = 2,\quad \argmin_{j \in S_2} \sigma^9(j) = 0,\quad \argmin_{j \in S_2} \sigma^{10}(j) = 4,  \\
\argmin_{j \in S_2} \sigma^{11}(j) = 1, \quad \argmin_{j \in S_2} \sigma^{12}(j) = 4, 
\end{gather*}
and the most preferred products from the assortment $S^{\hat{\lambda}} = \{0,4\}$  are
\begin{gather*}
 \argmin_{j \in S^{\hat{\lambda}}} \sigma^6(j) = 0,\quad \argmin_{j \in S^{\hat{\lambda}}} \sigma^7(j) = 0,\quad \argmin_{j \in S^{\hat{\lambda}}} \sigma^8(j) = 0,\quad \argmin_{j \in S^{\hat{\lambda}}} \sigma^9(j) = 0,\quad \argmin_{j \in S^{\hat{\lambda}}} \sigma^{10}(j) = 4,\\
\argmin_{j \in S^{\hat{\lambda}}} \sigma^{11}(j) = 0, \quad  \argmin_{j \in S^{\hat{\lambda}}} \sigma^{12}(j) = 4. 
\end{gather*}
It follows from the above observations that the historical sales data $(v_{1,0}, v_{1,2}, v_{1,3}, v_{1,4}) = (0.3,0.3,0.3,0.1)$ and $(v_{2,0}, v_{2,1}, v_{2,2}, v_{2,4}) = (0.3,0.3,0.1,0.3)$ are perfectly fit  by the ranking-based choice model  $\bar{\lambda} \in \Delta_\Sigma$ that satisfies the following equalities:
\begin{align*}
\bar{\lambda}_{\sigma^6} &= 0.2,& \bar{\lambda}_{\sigma^7} &= 0.1,& \bar{\lambda}_{\sigma^8} &= 0.1,&\bar{\lambda}_{\sigma^9} &= 0.1,&\bar{\lambda}_{\sigma^{10}}& = 0.1,&\bar{\lambda}_{\sigma^{11}}& = 0.2,&\bar{\lambda}_{\sigma^{12}}& = 0.2.
\end{align*}
However, we observe that the expected revenue for the assortment $S^{\hat{\lambda}} = \{0,4\}$ under the ranking-based choice model $\bar{\lambda}$ is $\mathscr{R}^{\bar{\lambda}}(S^{\hat{\lambda}}) = 0.3 \times \$100 = \$30$.  Since this is strictly less than the expected revenue generated by the firm's best past  assortment, $\$35$, we have  thus shown that estimate-then-optimize can yield an assortment for the problem instance with an expected revenue that is strictly worse than the expected revenue generated by the firm's best past assortment under a ranking-based choice model that is consistent with the historical sales data.

We conclude  by showing that $\{0,2,4\}$ is the unique optimal solution for the robust optimization problem~\eqref{prob:robust} and is the unique assortment that satisfies~\eqref{line:improvement}. To show this, we present below an exhaustive list of the worst-case expected revenues $\min_{\lambda \in \mathcal{U}} \mathscr{R}^\lambda(S)$  for all assortments $S \in \mathcal{S}$ that satisfy $4 \in S$.\footnote{For each assortment $S \in \mathcal{S}$ that satisfies $4 \notin S$ and for every  ranking-based choice model $\lambda \in \Delta_\Sigma$, we observe that $\mathscr{R}^\lambda(S) \le 1 \times r_3 = \$30$. Because $\$30$ is strictly less than the expected revenue generated by the firm's best past assortment, $\$35$, we can without loss of generality restrict our analysis to assortments that satisfy $4 \in S$.}  The following quantities are computed using the algorithm from Lemma~\ref{lem:two:flow} in \S\ref{sec:algorithm:twoassortments}. 
\begin{align*}
\begin{aligned}
&\min_{\lambda \in \mathcal{U}} \mathscr{R}^\lambda \left(\{0,4\} \right) &=  \$30, \quad&\min_{\lambda \in \mathcal{U}} \mathscr{R}^\lambda \left(\{0,1,4\} \right) &=  \$33,\\
 &\min_{\lambda \in \mathcal{U}} \mathscr{R}^\lambda \left(  \{0,2,4\} \right) &= \$36, \quad  &\min_{\lambda \in \mathcal{U}} \mathscr{R}^\lambda \left(  \{0,3,4\} \right) &= \$19,\\
 &  \min_{\lambda \in \mathcal{U}} \mathscr{R}^\lambda \left(  \{ 0,1,2,4\} \right) &= \$35,  \quad &\min_{\lambda \in \mathcal{U}} \mathscr{R}^\lambda \left(  \{ 0,1,3,4\} \right) &= \$12, \\
&\min_{\lambda \in \mathcal{U}} \mathscr{R}^\lambda \left( \{0,2,3,4\} \right) &=  \$25,\quad&\min_{\lambda \in \mathcal{U}} \mathscr{R}^\lambda \left(\{0,1,2,3,4\} \right) &=  \$14.
\end{aligned}
\end{align*}
We observe that the assortment $\{0,2,4\}$ satisfies \eqref{line:improvement}, since this assortment satisfies $\mathscr{R}^\lambda(\{0,2,4\}) \ge \$36 > \$35$ for all ranking-based choice models $\lambda \in \mathcal{U}$. In particular, we note that $\{0,2,4\}$ is not a revenue-ordered assortment. We remark that the  correctness of the above equalities $\min_{\lambda \in \mathcal{U}} \mathscr{R}^\lambda \left(  \{ 0,1,2,4\} \right) = \$35$ and $\min_{\lambda \in \mathcal{U}} \mathscr{R}^\lambda \left( \{0,2,3,4\} \right) =  \$25$ for the two past assortments follows immediately from the fact that $\eta = 0$, which implies that $\mathscr{D}^\lambda(S_m) = v_m$ for all $m \in \{1,2\}$ and  all $\lambda \in \mathcal{U}$. We thus conclude that that there exist problem instances for which there exists an assortment that satisfies \eqref{line:improvement}, and have therefore provided an affirmative answer to Question~\ref{question:1}.

\section{Characterization of Optimal Assortments for \eqref{prob:robust}} \label{sec:characterization}
In this section, we establish the first characterization of optimal assortments for the robust optimization problem~\eqref{prob:robust}.
In particular, we show that there are optimal assortments for \eqref{prob:robust}  with a simple structure that is closely related  to the structure of revenue-ordered assortments.  Recall the following definition of the collection of revenue-ordered assortments:
\begin{align*}
 \bar{\mathcal{S}} \triangleq \left \{ S \in \mathcal{S}:\; \textnormal{if } i^* \in S \textnormal{ and } r_{i^*} < r_{i}, \textnormal{ then }i \in S\right \}.
\end{align*}
A fundamental result in the theory of assortment optimization is that revenue-ordered assortments are optimal under the multinomial logit choice model \citep{talluri2004revenue,gallego2004managing,rusmevichientong2014assortment}. Revenue-ordered assortments also have attractive approximation guarantees for assortment optimization problems under mixture-of-logits and ranking-based choice models with known parameters \citep{rusmevichientong2014assortment,aouad2018approximability,berbeglia2020assortment}. Due to their simplicity and strong theoretical and empirical performance, revenue-ordered assortments are widely recommended in the revenue management  literature.

In view of the above background, 
we proceed to develop our main result regarding the structure of optimal assortments for the robust optimization problem~\eqref{prob:robust}. To this end, we first define the following set of past assortments for each product $i \in \mathcal{N}_0$: 
\begin{align*}
\mathcal{M}_i \triangleq \left \{ m \in \mathcal{M}: i \in S_m \right \}.
\end{align*} 
The above set contains all of the past assortments in which the firm offered product $i$ to their customers. In particular, we observe from this definition that the statement $\mathcal{M}_i \subseteq \mathcal{M}_j$ holds if, for all of the past assortments in which the firm offered product $i$, the firm also offered product $j$. 
We now introduce the following new collection of assortments:
\begin{align*}
\widehat{\mathcal{S}} &\triangleq \left \{ S \in \mathcal{S}:\; \textnormal{if } i^* \in S, \;  r_{i^*} <  r_{i},  \textnormal{ and } \mathcal{M}_{i^*} \subseteq \mathcal{M}_i, \textnormal{ then } i \in S\right \}.
\end{align*}
The above collection has a natural interpretation as the collection of all assortments which, speaking informally, can be viewed as revenue-ordered \emph{relative to} the firm's past assortments. Indeed,  consider two products $i^*, i \in \mathcal{N}_0$ for which the revenue $r_{i^*}$ from the first product $i^*$  is strictly less than the revenue $r_{i}$ from the second product $i$. Then every assortment $S \in \widehat{\mathcal{S}}$ which offers the first product  must also offer the second product \emph{unless} there is historical sales data from  a past assortment in  which the first product $i^*$ was offered and the second product $i$ was not offered. Our main result is the following:
\begin{theorem} \label{thm:main}
There exists an assortment $S \in \widehat{\mathcal{S}}$ that is optimal for \eqref{prob:robust}. 
\end{theorem}
The remainder of \S\ref{sec:characterization} focuses on the implications and derivation of Theorem~\ref{thm:main}. In \S\ref{sec:characterization:question}, we use Theorem~\ref{thm:main} to show that the past assortments must deviate from revenue-ordered assortments for the optimal objective value of \eqref{prob:robust} to be strictly greater than the expected revenue of the firm's best past assortment. In \S\ref{sec:characterization:upperbound}, we develop upper bounds on the cardinality of $\widehat{\mathcal{S}}$ based on the number of past assortments and number of products. In \S\ref{sec:characterization:reverse}, we prove that $\widehat{\mathcal{S}}$ can be an exact description of the assortments that are optimal for \eqref{prob:robust}.   In \S\ref{sec:characterization:proof}, we provide an overview of the proof of Theorem~\ref{thm:main}. 

\subsection{Answer to Question~\ref{question:2}}  \label{sec:characterization:question}
Due to revenue-ordered assortment's simplicity and desirable theoretical guarantees, a large body of literature has advocated to firms for offering revenue-ordered assortments across numerous application domains.  Equipped with Theorem~\ref{thm:main}, we show in the following Lemma~\ref{lem:impossibility} and Corollary~\ref{cor:impossibility} that a firm which has offered the revenue-ordered assortments has operated under a worst-possible behavior  from the perspective of finding an assortment that satisfies~\eqref{line:improvement}. 

\begin{lemma}\label{lem:impossibility}
If $\mathscr{M} = \bar{\mathcal{S}}$, then $\widehat{\mathcal{S}} = \mathscr{M}$. \end{lemma}
\begin{corollary}\label{cor:impossibility}
 If $\mathscr{M} = \bar{\mathcal{S}}$ and $\eta = 0$, then $\max \limits_{S \in \mathcal{S}} \min \limits_{\lambda \in \mathcal{U}} \mathscr{R}^{\lambda}(S) = \max \limits_{m \in \mathcal{M}} r^\intercal v_m$. 
 \end{corollary}
 
 The above lemma and corollary establish that it is impossible to find an assortment that satisfies \eqref{line:improvement} when the firm's past assortments are the revenue-ordered assortments. Specifically, it follows from Lemma~\ref{lem:impossibility} and Theorem~\ref{thm:main} that if the firm's past assortments are the  revenue-ordered assortments, then there must exist a revenue-ordered assortment that is an optimal solution for the robust optimization problem~\eqref{prob:robust}, regardless of the prices of the products $r$ and the observations of the historical sales data $v_1,\ldots,v_M$.  
It thus follows as an immediate corollary of Lemma~\ref{lem:impossibility} and Theorem~\ref{thm:main} in the case of $\eta = 0$ that the past assortments must deviate from revenue-ordered assortments in order for it to be possible for the optimal objective value of \eqref{prob:robust} to be strictly greater than the expected revenue of the firm's best past assortment.

 Motivated by Lemma~\ref{lem:impossibility} and Corollary~\ref{cor:impossibility}, we perform numerical experiments in Appendix~\ref{appx:riskseto} to assess  the performance of assortments obtained using estimate-then-optimize when the historical sales data is randomly generated from revenue-ordered assortments.  The results of the experiments are striking: in more than 98\% of the problem instances in which  estimate-then-optimize  recommended a new assortment, the worst-case decline in expected revenue from implementing the new assortment  (relative to the expected revenue from the best past assortment)  exceeded the best-case increase in expected revenue. The difference in magnitude between the worst-case and best-case change in expected revenue from implementing the new assortment found from  estimate-then-optimize  is also significant; the average best-case improvement of the new assortment over the best past assortment is 6.56\%, while the  average worst-case improvement of the new assortment over the best past assortment is  -21.71\%. Our results in Appendix~\ref{appx:riskseto} thus demonstrate that estimate-then-optimize with ranking-based choice models can risk leading to significant declines in a firm's expected revenue, even in settings in which the firm has implemented a celebrated and  widely-recommended class of  assortments.

\subsection{Upper Bounds} \label{sec:characterization:upperbound}
In \S\ref{sec:algorithms}, we will design algorithms for solving the class of robust optimization problems~\eqref{prob:robust} that consist of exhaustive search over the assortments in the collection $\widehat{\mathcal{S}}$.   To motivate these algorithms, we show in the present subsection that the cardinality of $\widehat{\mathcal{S}}$ scales as a polynomial in $n$ for any fixed number of past assortments.

We begin by considering the case of two past assortments. In the following Lemma~\ref{lem:two:S}, we develop a closed-form representation of the collection of assortments $\widehat{\mathcal{S}}$. This  representation will be useful in \S\ref{sec:algorithm:twoassortments} because it provides an efficient procedure for iterating over the assortments in $\widehat{\mathcal{S}}$. Moreover, the following Lemma~\ref{lem:two:S} is useful because it immediately implies that the number of assortments in the collection $\widehat{\mathcal{S}}$ scales quadratically in the number of products $n$ when there are two past assortments.

\begin{lemma} \label{lem:two:S}
If $M = 2$, then
\begin{align}
\widehat{\mathcal{S}} = \left \{ S \in \mathcal{S}: \quad \begin{aligned} &\textnormal{there exists }i_1 \in S_1 \setminus S_2 \textnormal{ and }  i_2 \in S_2 \setminus S_1 \textnormal{ such that}\\
&S =  \left(S_1 \cap S_2 \right) \cup \left \{j \in  S_1 \setminus S_2: j \ge i_1 \right \}  \cup \left \{j \in  S_2 \setminus S_1: j \ge i_2 \right \}\end{aligned} \right \}. \label{line:S_two}
\end{align}
\end{lemma}
Lemma~\ref{lem:two:S} shows that the assortments in the collection $\widehat{\mathcal{S}}$ can be parameterized by the pairs of products from the sets $S_1 \setminus S_2$ and $S_2 \setminus S_1$. In particular, Lemma~\ref{lem:two:S} shows for the case of two past assortments that each assortment $S \in \widehat{\mathcal{S}}$ contains all of the products in $S_1 \cap S_2$, contains all of the products in $S_1 \setminus S_2$ with revenue greater than or equal to $r_{i_1}$ for some $i_1 \in S_1 \setminus S_2$, and  contains all of the products in $S_2 \setminus S_1$ with revenue greater than or equal to $r_{i_2}$ for some $i_2 \in S_1 \setminus S_2$. Hence, the number of assortments in the collection $\widehat{\mathcal{S}}$ is at most $|S_1 \setminus S_2| \times |S_2 \setminus S_1|= \mathcal{O}(n^2)$. 

We now develop an upper bound on the cardinality of $\widehat{\mathcal{S}}$ for general numbers of past assortments. To motivate our upper bound,   we recall  that the number of assortments  in the collection $\widehat{\mathcal{S}}$ can indeed be small in special cases of problem instances:  namely, we recall from Lemma~\ref{lem:impossibility} in \S\ref{sec:characterization:question} that $|\widehat{\mathcal{S}}| = \mathcal{O}(n)$ when the past assortments are the revenue-ordered assortments, and we showed in Lemma~\ref{lem:two:S} that $|\widehat{\mathcal{S}}| = \mathcal{O}(n^2)$ when there are two past assortments. In the next intermediary result, denoted by Lemma~\ref{lem:fixed_dim:S}, we develop a more general result  along these lines. Specifically, the following lemma establishes that the number of assortments in the collection $\widehat{\mathcal{S}}$ can be upper bounded by a polynomial of the number of products $n$ for \emph{any} fixed number of past assortments $M$. While the bound in the following lemma is not the tightest possible, the bound will be sufficient for its theoretical purpose in designing polynomial-time algorithms in \S\ref{sec:algorithms:fixed_dim} for solving the robust optimization problem~\eqref{prob:robust} when the number of past assortments $M$  is fixed. 
\begin{lemma}\label{lem:fixed_dim:S}
$| \widehat{\mathcal{S}} | \le (n+2)^{2^M}$.
\end{lemma}

We conclude the present \S\ref{sec:characterization:upperbound} by developing an algorithm for computing the  collection $\widehat{\mathcal{S}}$. 
 The algorithm which achieves the specified running time of the following lemma is based on dynamic programming over a compact graphical representation of the collection of assortments  $\widehat{\mathcal{S}}$. In particular, the running time of the algorithm from Lemma~\ref{lem:fixed_dim:S:time} can be viewed as attractive due to its mild dependence on the number of assortments in the collection $\widehat{\mathcal{S}}$. For example, it follows from Lemma~\ref{lem:fixed_dim:S} that the algorithm from the following lemma has a running time that is polynomial in the number of products $n$ for any fixed number of past assortments $M$. 
\begin{lemma} \label{lem:fixed_dim:S:time}
The collection  of assortments $\widehat{\mathcal{S}}$ can be constructed in $\mathcal{O}(n^{2} (M+| \widehat{\mathcal{S}}|))$ time. 
  \end{lemma}
  
   \subsection{Discussion of Upper Bound} \label{sec:characterization:reverse}
In Lemma~\ref{lem:fixed_dim:S} of the previous subsection, we developed an upper bound on $|\widehat{\mathcal{S}}|$ that grows exponentially with respect to the number of past assortments.  In this subsection, we provide an example that shows that $|\widehat{\mathcal{S}}|$ can indeed grow exponentially with respect to the number of past assortments, and we show for that example that $\widehat{\mathcal{S}}$ is an exact description of the assortments that can be optimal for the robust optimization problem~\eqref{prob:robust}.   Our discussion will focus on the following specific collection of past assortments:
\begin{definition}
We say that the firm's  past assortments are \emph{reverse revenue-ordered} if  
$S_1 \triangleq \{0,n\}$, $S_m \triangleq \left \{ 0,1,\ldots,m-1,n \right \}$ for each $m \in  \{2,\ldots,n\}$, and $\mathscr{M} = \tilde{\mathcal{S}} \triangleq \left \{S_1,\ldots,S_n\right \}$.
\end{definition}
We readily observe from the above definition that the reverse revenue-ordered assortments are nested, in the sense that $S_1 \subset \cdots \subset S_n$. Moreover, we observe that the following lemma follows immediately from the definition of the collection of assortments $\widehat{S}$ from \S\ref{sec:characterization}. 
\begin{lemma} \label{lem:S_hat_for_S_tilde}
If $\mathscr{M} = \tilde{\mathcal{S}}$, then $\widehat{\mathcal{S}} = \left \{ S \subseteq \mathcal{N}_0: \{0,n\} \subseteq S \right \}$.
\end{lemma}
The above lemma shows that if the past assortments $\mathscr{M} $ are the reverse revenue-ordered assortments, then the corresponding collection of assortments $\widehat{\mathcal{S}}$ is the collection of all subsets of products that include the most expensive product $n$ and the no-purchase option $0$. It follows immediately from Lemma~\ref{lem:S_hat_for_S_tilde} and from the fact that $M = n$ that $| \widehat{\mathcal{S}} | = 2^{M-2}$ when the past assortments are the  reversed revenue-ordered assortments. We have thus shown that the cardinality of the collection of assortments $\widehat{\mathcal{S}}$ can grow exponentially in the number of past assortments. 

Our main contribution of \S\ref{sec:characterization:reverse}, which is presented below as Theorem~\ref{thm:not_improveable},  consists of establishing that the characterization of optimal assortments from Theorem~\ref{thm:main} cannot be refined for reverse revenue-ordered assortments. More precisely, Theorem~\ref{thm:not_improveable} shows that $\widehat{\mathcal{S}}$ is exactly the collection of assortments that can be optimal for the robust optimization problem~\eqref{prob:robust} when the past assortments are reverse revenue-ordered.  This theorem is significant because  it proves that the collection  of assortments $\widehat{\mathcal{S}}$ is an accurate description of the range of possible optimal solutions for the robust optimization problem~\eqref{prob:robust}, at least for the case when the past assortments are reverse revenue-ordered. It also is significant because it implies that the exponential growth of the cardinality of the collection of assortments $\widehat{\mathcal{S}}$ that is established in Lemma~\ref{lem:S_hat_for_S_tilde} is not the result of constructing an overly-conservative  collection  of assortments $\widehat{\mathcal{S}}$.

\begin{theorem} \label{thm:not_improveable}
Suppose that $\mathscr{M} = \tilde{\mathcal{S}}$, that $\eta = 0$, and that the revenues $r_1 < \cdots < r_n$ are fixed. Then for every assortment $\bar{S} \in \widehat{\mathcal{S}}$, there exists a realization of the historical data $v \equiv (v_{m,i}: m \in \mathcal{M}, i \in S_m)$ such that $\bar{S}$ is the unique optimal solution of the robust optimization problem~\eqref{prob:robust}.
\end{theorem}

In summary, we have shown in the present \S\ref{sec:characterization:reverse} that there exist examples in which $|\widehat{\mathcal{S}}|$  grows exponentially in the number of past assortments and in which $\widehat{\mathcal{S}}$ is an accurate description of the range of possible optimal solutions for the robust optimization problem~\eqref{prob:robust}.  Therefore, unless the  firm's past assortments have a special structure such as revenue-ordered assortments (see Lemma~\ref{lem:impossibility}), we conclude that solving \eqref{prob:robust} via an 
exhaustive search over the assortments in $\widehat{\mathcal{S}}$ will be computationally prohibitive when the number of past assortments is large. For such settings, we show in \S\ref{sec:algorithms:nested} that our proof techniques from the following \S\ref{sec:characterization:proof} can nonetheless be extended  to develop compact mixed-integer optimization reformulations of \eqref{prob:robust} when the collection of past assortments enjoys a common structure.

\subsection{Overview of Proof of Theorem~\ref{thm:main}} \label{sec:characterization:proof}
Our proof of Theorem~\ref{thm:main} is split into two steps. 
 In the first step of the proof of Theorem~\ref{thm:main}, which is formalized below as Proposition~\ref{prop:reform_wc}, we show that computing the worst-case expected revenue $\min_{\lambda \in \mathcal{U}} \mathscr{R}^\lambda(S)$ corresponding  to any fixed assortment $S \in \mathcal{S}$ can be reformulated as a linear optimization problem with $\mathcal{O}(n^M)$ decision variables and $\mathcal{O}(nM)$ constraints.  Indeed, for each assortment $S \in \mathcal{S}$ and each product in the assortment $i \in S$, let the set of rankings that prefer product $i$ to all other products in the assortment $S$ be defined as follows:
\begin{definition}\label{defn:D}
$\mathcal{D}_i(S) \triangleq \left \{ \sigma \in \Sigma: \;  i = \argmin_{j \in S} \sigma (j) \right \}$.
\end{definition} 
\noindent Given the past assortments $S_1,\ldots,S_M$, we also define the following set of tuples of products:
\begin{definition} \label{defn:L}
$\mathcal{L} \triangleq \left \{(i_1,\ldots,i_M) \in S_1 \times \cdots \times S_M: \; \bigcap_{m \in \mathcal{M}} \mathcal{D}_{i_m}(S_m) \neq \emptyset \right \}$. 
\end{definition}
We readily observe from the above definitions that a tuple of products satisfies $(i_1,\ldots,i_M) \in \mathcal{L}$  if and only if there exists a ranking  $\sigma \in \Sigma$ for which $i_1$ is the most preferred product under that ranking in assortment $S_1$, $i_2$ is the most preferred product under that ranking in assortment $S_2$,  and so on and so forth.
 For convenience, we henceforth say that a ranking $\sigma \in \Sigma$ corresponds to a tuple $(i_1,\ldots,i_M) \in \mathcal{L}$ if and only if $\sigma \in \bigcap_{m \in \mathcal{M}} \mathcal{D}_{i_m}(S_m)$.  
 Finally, for each  assortment $S \in \mathcal{S}$ and each tuple of products $(i_1,\ldots,i_M) \in \mathcal{L}$, we let $\rho_{i_1 \cdots i_M}(S)$ be defined as   the minimum revenue among the products in the assortment $S $ that can be the most preferred product in $S$ under a ranking that corresponds to the tuple $(i_1,\ldots,i_M)$. 
\begin{definition} \label{defn:rho}
$\rho_{i_1 \cdots i_M}(S) \triangleq \min \limits_{i \in S: \; \cap_{m \in \mathcal{M}} \mathcal{D}_{i_m}(S_m) \cap \mathcal{D}_i(S) \neq \emptyset} r_i$. 
\end{definition}
 Equipped with the above notation, the following result shows that the worst-case expected revenue of any fixed assortment  can be reformulated as a linear  optimization problem: 
\begin{proposition} \label{prop:reform_wc}
 For each  $S \in \mathcal{S}$, $\min_{\lambda \in \mathcal{U}} \mathscr{R}^\lambda(S)$ is equal to the optimal objective value of the following linear optimization problem:
\begin{equation} 
\label{prob:robust_simplified}
 \begin{aligned}
& \; \underset{\lambda, \epsilon}{\textnormal{minimize}} && \sum_{(i_1,\ldots,i_M) \in \mathcal{L}} \rho_{i_1 \cdots i_M}(S) \lambda_{i_1\cdots i_M}\\
&\textnormal{subject to}&&  \sum_{(i_1,\ldots,i_M) \in \mathcal{L}: \; i_m = i} \lambda_{i_1 \cdots i_M} - \epsilon_{m,i}= v_{m,i} \quad \forall m \in \mathcal{M}, \; i \in S_m\\
&&& \sum_{(i_1,\ldots,i_M) \in \mathcal{L}} \lambda_{i_1 \cdots i_M} = 1 \\
&&& \| \epsilon \| \le \eta\\
&&& \lambda_{i_1 \cdots i_M} \ge 0 \quad \forall (i_1,\ldots,i_M) \in \mathcal{L}.
\end{aligned}
\end{equation}
\end{proposition}
The decision variables $\lambda_{i_1 \cdots i_M} \ge 0$ in the linear optimization problem~\eqref{prob:robust_simplified} can be interpreted as the likelihood that a randomly-arriving customer will have preferences that follow a ranking that corresponds to the tuple $(i_1,\ldots,i_M) \in \mathcal{L}$. The following two remarks articulate the key properties of \eqref{prob:robust_simplified} that will be important to our developments in the  subsequent discussions.
\begin{remark} \label{remark:rho}
The assortment $S$ does not appear in any of the constraints of \eqref{prob:robust_simplified}. 
\end{remark}
\begin{remark}\label{remark:tractability}
 \eqref{prob:robust_simplified} has $| \mathcal{L}| = \mathcal{O}(n^M)$ decision variables and $\mathcal{O}(nM)$ constraints.
  \end{remark}
  We note that the asymptotic upper bound  in Remark~\ref{remark:tractability}  on the number of decision variables in the linear optimization problem~\eqref{prob:robust_simplified} follows immediately from the fact that the set of tuples of products $\mathcal{L}$ is a subset of $S_1 \times \cdots \times S_M$. 
  
The second step of the proof of Theorem~\ref{thm:main}, which is formalized below as Proposition~\ref{prop:plus_one_inequality_prop}, consists of analyzing the behavior of the functions  $S \mapsto \rho_{i_1 \cdots i_M}(S)$. To motivate this second step, let us reflect on the structure of the linear optimization problem from 
Proposition~\ref{prop:reform_wc}. We recall from Remark~\ref{remark:rho} that the set of feasible solutions of the linear optimization problem~\eqref{prob:robust_simplified} is independent of the assortment $S \in \mathcal{S}$. We thus observe that if two assortments $S, S' \in\mathcal{S}$ satisfy $\rho_{i_1 \cdots i_M}(S) \le \rho_{i_1 \cdots i_M}(S')$ for all tuples of products $(i_1,\ldots,i_M) \in \mathcal{L}$, then Proposition~\ref{prop:reform_wc} implies that the worst-case expected revenue of the first assortment, $\min_{\lambda \in \mathcal{U}} \mathscr{R}^\lambda(S),$ will be less than or equal to the worst-case expected revenue of the second assortment, $\min_{\lambda \in \mathcal{U}} \mathscr{R}^\lambda(S')$. Against this backdrop, the following Proposition~\ref{prop:plus_one_inequality_prop} readily implies that every assortment $S \in \mathcal{S}$  can be exchanged for an assortment $S' \in \widehat{\mathcal{S}}$ with the same or better worst-case expected revenue.
\begin{proposition}\label{prop:plus_one_inequality_prop}
Let $S \in \mathcal{S}$ and $i \notin S$. If there exists  $i^* \in S$ that satisfies $r_{i^*} < r_i$ and $\mathcal{M}_{i^*} \subseteq \mathcal{M}_i$, then $\rho_{i_1 \cdots i_M}(S) \le \rho_{i_1 \cdots i_M}(S \cup \{i\})$ for each $(i_1,\ldots,i_M) \in \mathcal{L}$. 
\end{proposition}

Our proof of Proposition~\ref{prop:plus_one_inequality_prop} consists of developing a graphical interpretation of the quantities $\rho_{i_1 \cdots i_M}(S)$. Specifically, we show in Appendix~\ref{appx:graphical}  that every tuple of products $(i_1,\ldots,i_M) \in \mathcal{L}$ can be represented as a directed acyclic graph in which there is a vertex for each product $i \in \mathcal{N}_0$, and we show that $\rho_{i_1 \cdots i_M}(S)$ can be computed by searching over the products in the directed acyclic graph corresponding to the tuple  $(i_1,\ldots,i_M) \in \mathcal{L}$ that do not have a directed path to any vertices corresponding to the products in the assortment $S$. We then exploit this graphical  interpretation of $\rho_{i_1 \cdots i_M}(S)$ in Appendix~\ref{appx:prop2} to relate the quantities $\rho_{i_1 \cdots i_M}(S)$ and $\rho_{i_1 \cdots i_M}(S \cup \{i\})$ and complete the proof of  Proposition~\ref{prop:plus_one_inequality_prop}.  The proof of Theorem~\ref{thm:main}, which follows readily from Propositions~\ref{prop:reform_wc} and \ref{prop:plus_one_inequality_prop},  is found in Appendix~\ref{appx:proof_main}.

\section{Algorithms} \label{sec:algorithms}
In this section, we use our developments from \S\ref{sec:characterization}  to develop the first polynomial-time algorithms and the first compact mixed-integer  optimization reformulations for the robust optimization problem~\eqref{prob:robust}.

\begin{remark} \label{rem:N_0}
 We assume throughout \S\ref{sec:algorithms} that every product $i \in \mathcal{N}_0$ has been offered by the firm to its customers in at least one of the past assortments, that is,  $\cup_{m \in \mathcal{M}} S_m = \mathcal{N}_0$. This assumption simplifies our notation and can always be satisfied by redefining $\mathcal{N}_0$ to be equal to $\cup_{m \in \mathcal{M}} S_m$. Such a redefining of $\mathcal{N}_0$ does not impact the optimal objective value of the robust optimization problem~\eqref{prob:robust}; see Appendix~\ref{appx:N_0_assumption} for further details.  
 \end{remark}

\subsection{Polynomial-Time Algorithm for Two Past Assortments} \label{sec:algorithm:twoassortments}
We begin by leveraging Theorem~\ref{thm:main} to develop the first algorithm for solving the robust  optimization problem~\eqref{prob:robust} with running time that is strongly polynomial in the number of products $n$ for two past assortments. Our development of an algorithm for the case of $M = 2$ is  important because it marks  the first step towards  the development of efficient \emph{general} algorithms for solving the robust optimization problem~\eqref{prob:robust} in settings with large numbers of products. 
 We will show that the algorithms from this subsection extend to general settings with any fixed number of past assortments in the following \S\ref{sec:algorithms:fixed_dim}.  In Appendix~\ref{sec:twoassortments:numerics}, we provide numerical experiments showing that our algorithm for two past assortments can be practically efficient, running in less than 30 seconds for problem instances with $n=100$ products.

At a high level, our algorithm for solving the robust optimization problem~\eqref{prob:robust} in the case of $M = 2$ consists of reducing \eqref{prob:robust} to solving a sequence of minimum-cost network flow problems. 
Stated formally, the main contribution of \S\ref{sec:algorithm:twoassortments} is the following: 
\begin{theorem} \label{thm:two}
If $M = 2$ and $\eta = 0$, then \eqref{prob:robust} can be solved in $\mathcal{O}(n^5 \log (n r_n))$  time.
\end{theorem}
In the above theorem and throughout the rest of \S\ref{sec:algorithm:twoassortments}, we assume that the revenues $r_1,\ldots,r_n $ are represented as positive integers. We also assume without any loss of generality  
that $n \in S_1 \cap S_2$.\footnote{To see why the assumption that $n \in S_1 \cap S_2$ can be made without loss of generality, suppose for the sake of argument that the product $n$ is not contained in $S_1 \cap S_2$. In that case, we can create a fictitious product with index $n+1$ that is associated with any arbitrary revenue in the range $r_{n+1} = (r_n,\infty)$, and we can augment the historical sales data to be $v_1' \triangleq (v_1,0) \in \R^{n+2}_+$ and $v_2' \triangleq (v_2,0) \in \R^{n+2}_+$. It is straightforward  to see that any feasible solution to \eqref{prob:robust_simplified} with the augmented historical sales data will satisfy $\lambda_{n+1,i_2} = \lambda_{i_1,n+1} = 0$ for all $i_1 \in S_1$ and $i_2 \in S_2$ that satisfy $(n+1,i_2) \in \mathcal{L}$ and $(i_1,n+1) \in \mathcal{L}$. This implies that the optimal objective value of \eqref{prob:robust_simplified} will be unchanged using the augmented historical sales data for all assortments $S \in \mathcal{S}$. Hence, we have shown that we can assume without loss of generality that the past assortments satisfy $n \in S_1 \cap S_2$. } 

To prove Theorem~\ref{thm:two}, and to develop an algorithm with the desired computation time, we utilize three intermediary results. The first intermediary result is Lemma~\ref{lem:two:S} from \S\ref{sec:characterization:upperbound}. In that lemma, we developed a closed-form representation of the collection of assortments $\widehat{\mathcal{S}}$. That  representation is useful in the proof of Theorem~\ref{thm:two} because it provides an efficient procedure for iterating over the assortments in $\widehat{\mathcal{S}}$ and established that $| \widehat{\mathcal{S}}| = \mathcal{O}(n^2)$. In our second intermediary result, denoted by Lemma~\ref{lem:two:L}, we develop a closed-form representation of the set of tuples of products $\mathcal{L}$. The following representation of $\mathcal{L}$ will be useful in the proof of Theorem~\ref{thm:two} because it will allow us to show that the worst-case expected revenue $\min_{\lambda \in \mathcal{U}} \mathscr{R}^{\lambda}(S)$ for each assortment $S \in \mathcal{S}$ can be computed by solving a minimum-cost network flow problem. 
\begin{lemma} \label{lem:two:L}
If $M = 2$, then
$\mathcal{L} = \left((S_1 \setminus S_2) \times S_2 \right) \cup \left( S_1 \times (S_2 \setminus S_1) \right) \cup  \left\{(i,i): i \in S_1 \cap S_2 \right \}$. 
\end{lemma}
 In our third and final intermediary result, denoted by Lemma~\ref{lem:two:flow}, we show that the worst-case expected revenue $ \min_{\lambda \in \mathcal{U}}\mathscr{R}^{\lambda}(S)$  for each assortment $S \in {\mathcal{S}}$ can be computed by solving a minimum-cost network flow problem. 
 \begin{lemma} \label{lem:two:flow}
If $M = 2$ and $\eta = 0$, then $\min_{\lambda \in \mathcal{U}} \mathscr{R}^\lambda(S)$ can be evaluated for each assortment $S \in \mathcal{S}$ by solving a minimum-cost network flow problem in a directed acyclic graph with $ \mathcal{O}(n)$ vertices and $\mathcal{O}(n^2)$ directed edges.

 \end{lemma}
 
Using the above three intermediary results, we now describe our algorithm for solving the robust optimization problem~\eqref{prob:robust}, which  follows a brute-force strategy. Namely, our algorithm  iterates over each of the assortments $S \in \widehat{\mathcal{S}}$, and, for each such assortment, the algorithm computes the corresponding worst-case expected revenue $\min_{\lambda \in \mathcal{U}}\mathscr{R}^{\lambda}(S)$  by solving a minimum-cost network flow problem. The algorithm concludes by returning the maximum value of $\min_{\lambda \in \mathcal{U}}\mathscr{R}^{\lambda}(S)$ across the assortments $S \in \widehat{\mathcal{S}}$. The correctness of this algorithm for solving the robust optimization problem~\eqref{prob:robust} follows immediately from Theorem~\ref{thm:main}. The proof that this algorithm achieves the desired running time of Theorem~\ref{thm:two} can be found in Appendix~\ref{appx:thm:two}.

\subsection{Polynomial-Time Algorithm for Fixed Number of Past Assortments} \label{sec:algorithms:fixed_dim}
We next develop a polynomial-time algorithm for solving the robust optimization problem~\eqref{prob:robust} 
for any fixed number of past assortments $M$. 
The algorithm developed in this subsection is particularly attractive due to its generality: it does not require any assumptions on the composition of products in the past assortments. Moreover, the algorithm allows for the radius $\eta$ in the set of ranking-based choice models to be strictly positive (see \S\ref{sec:setting}); hence, the algorithm can be applied in settings in which there are no ranking-based choice models that perfectly fit  the firm's historical sales data.

Our polynomial-time algorithm for solving the robust optimization problem~\eqref{prob:robust} for any fixed number of past assortments  follows the same strategy as developed in \S\ref{sec:algorithm:twoassortments}. Specifically, our algorithm reduces the  robust optimization problem~\eqref{prob:robust} to computing the worst-case expected revenue $\min_{\lambda \in \mathcal{U}} \mathscr{R}^\lambda(S)$ for each candidate assortment $S \in  \widehat{\mathcal{S}}$ and outputting the assortment that has the maximum worst-case expected revenue. 
The correctness of such a brute-force algorithm for solving the robust optimization problem \eqref{prob:robust}  follows immediately from Theorem~\ref{thm:main}.  In the remainder of \S\ref{sec:algorithms:fixed_dim}, we  prove that there exists an implementation of the aforementioned brute-force algorithm which yields the following theoretical guarantee on its computational tractability: 
\begin{theorem} \label{thm:poly}
\eqref{prob:robust} can be solved in weakly $\mathcal{O}(\textnormal{poly}(n))$  time for every fixed $M$.
\end{theorem}
The above theorem holds for any choice of the radius $\eta \ge 0$ in the set of ranking-based choice models and any composition of products in the past assortments. We note that our algorithm is guaranteed to run in weakly, as opposed to strongly,  polynomial-time due to its reduction to solving linear optimization problems. 

Our proof of Theorem~\ref{thm:poly} makes use of several intermediary results. First, we recall from Lemmas~\ref{lem:fixed_dim:S} and \ref{lem:fixed_dim:S:time} in  \S\ref{sec:characterization:upperbound} that if the number of past assortments is fixed, then  the collection of assortments $\widehat{\mathcal{S}}$ contains $ \mathcal{O}(\text{poly}(n))$ assortments and can be constructed in $\mathcal{O}(\text{poly}(n))$ time.  Therefore, all that remains is showing that the worst-case expected revenue $ \min_{\lambda \in \mathcal{U}}\mathscr{R}^{\lambda}(S)$ can be computed for each assortment $S \in \widehat{\mathcal{S}}$ in weakly polynomial-time for every fixed number of past assortments $M$. The following Lemmas~\ref{lem:fixed_dim:construct_L} and \ref{lem:fixed_dim:construct_rho}, together with Proposition~\ref{prop:reform_wc} from \S\ref{sec:characterization:proof}, show for every fixed $M \ge 2$ and assortment $S \in \widehat{\mathcal{S}}$ that the worst-case expected revenue  $ \min_{\lambda \in \mathcal{U}}\mathscr{R}^{\lambda}(S)$ can be computed by solving a linear optimization problem that can be constructed in polynomial-time and has a polynomial number of decision variables and constraints. 

\begin{lemma} \label{lem:fixed_dim:construct_L}
 $|\mathcal{L}| = \mathcal{O}(n^M)$, and $\mathcal{L}$ can be constructed in $\mathcal{O}(Mn^{M+1})$ time. 
\end{lemma}
\begin{lemma} \label{lem:fixed_dim:construct_rho}
 For each assortment $S \in \widehat{\mathcal{S}}$ and  each tuple of products $(i_1,\ldots,i_M) \in \mathcal{L}$, the quantity $\rho_{i_1 \cdots i_M}(S)$ can be computed in $\mathcal{O}(M^2 n)$  time. 
\end{lemma}
Combining the above intermediary lemmas  with Lemmas~\ref{lem:fixed_dim:S} and \ref{lem:fixed_dim:S:time} from   \S\ref{sec:characterization:upperbound} and Proposition~\ref{prop:reform_wc} from \S\ref{sec:characterization:proof}, we obtain an algorithm that satisfies the running time requirements of Theorem~\ref{thm:poly}. The detailed description of the algorithm and the proof of its running time can be found in Appendix~\ref{appx:thm:poly}.

 \subsection{Mixed-Integer Optimization Reformulation} \label{sec:algorithms:nested}
  In  \S\ref{sec:algorithm:twoassortments} and \S\ref{sec:algorithms:fixed_dim}, we developed algorithms  for solving the robust optimization problem~\eqref{prob:robust} that are theoretically efficient in applications in which the number of products $n$ is large and the number of past assortments $M$ is small. The crux of those algorithms was the observation that solving the robust optimization problem~\eqref{prob:robust} can be reduced to an exhaustive search over the collection of assortments $\widehat{\mathcal{S}}$. However, we showed in Theorem~\ref{thm:not_improveable} from \S\ref{sec:characterization:reverse} that the cardinality of $\widehat{\mathcal{S}}$ can grow exponentially in the number of past assortments, thereby preventing the algorithm from  \S\ref{sec:algorithms:fixed_dim} from being computationally efficient in general in applications with large numbers of past assortments.   In this subsection, we show nonetheless that our technical developments from \S\ref{sec:characterization:proof}  can also be used to develop practical algorithms for solving the robust optimization problem in applications in which the number of past assortments and number of products are both reasonably large.

   Our discussion in the present subsection focuses on applications in which the collection of past assortments has a widely-studied structure called \emph{nested}. Formally, the past assortments are nested if and only if there is an ordering of the past assortments such as that each assortment is a subset of the subsequent assortment: 
  \begin{assumption}[Nested]\label{ass:nested} 
The past assortments $\mathscr{M} \equiv \{S_1,\ldots,S_M\} \subseteq \mathcal{N}_0$ satisfy $S_1 \subset \cdots \subset S_M$. 
\end{assumption}
Applications in which historical sales data is obtained from nested past assortments are ``common in retail and revenue management settings"  \citep[p.2202]{jagabathula2019limit}. For example, the past assortments will be nested if a retailer initially offers a large assortment of products to its customers and does not replenish inventory, in which the case the initial assortment will dwindle to smaller and smaller assortments as products run out of inventory. Nested past assortments also arise as new products become available in the marketplace, as well as in online retail as customers scroll through products across multiple pages of a website \citep{davis2015assortment}. 
Finally,  nested assortments are a generalization of revenue-ordered assortments, and thus arise in any situations in which a firm has  only offered revenue-ordered assortments.

Our main contribution of the present subsection is showing that our technical developments from \S\ref{sec:characterization:proof} can be used to obtain a compact mixed-integer  optimization reformulation of the robust optimization problem~\eqref{prob:robust} when the past assortments are nested. This reformulation can offer significant value compared to the algorithm from \S\ref{sec:algorithms:fixed_dim} when the number of past assortments as large; indeed, the number of decision variables and constraints of the mixed-integer  optimization reformulation scale as a polynomial in the number of past assortments as well as the number of products. As a result, we show through numerical experiments in Appendix~\ref{appx:numerical:nested} that our compact mixed-integer optimization reformulation can solve the robust optimization problem~\eqref{prob:robust} in less than 30 seconds in applications with as many as $M = 20$ past assortments.
 More generally,  this subsection lays a roadmap for combining the technical developments behind Theorem~\ref{thm:main}  in \S\ref{sec:characterization:proof} with  structural properties of a collection of  past assortments to develop practical algorithms for solving  the robust optimization problem~\eqref{prob:robust} when $| \widehat{\mathcal{S}}|$ is large.

To motivate the construction of our mixed-integer optimization reformulation of the robust optimization problem~\eqref{prob:robust}, 
we begin by showing for the case of nested past assortments that the linear optimization problem~\eqref{prob:robust_simplified} has a number of decision variables that grows exponentially in the number of past assortments. Indeed, for each $m \in \mathcal{M}$, let the set of products that appear in the past assortment $S_m$ and do not appear in  any preceding past assortments be denoted by 
\begin{align}
\mathcal{B}_m \triangleq S_{m} \setminus \left( \bigcup_{m'=1}^{m-1} S_{m'} \right) = \begin{cases}
S_1,&\text{if } m = 1,\\
S_m \setminus S_{m-1},&\text{otherwise},
\end{cases} \label{defn:B_m}
\end{align}
where the right-most equality follows from algebra and from the assumption in \S\ref{sec:algorithms:nested} that the past assortments are nested. We observe from the above definition that the sets $\mathcal{B}_1,\ldots,\mathcal{B}_M$ are disjoint.
   In view of this notation,  the following lemma and its corollary   show that the number of decision variables in the linear optimization problem~\eqref{prob:robust_simplified}  grows exponentially with respect to the number of past assortments $M$ when the past assortments are nested. 
\begin{lemma} \label{lem:reform_of_L} 
If Assumption~\ref{ass:nested} holds, then
$$\mathcal{L} = \left \{(i_1,\ldots,i_M) \in S_1 \times \cdots \times S_M:\textnormal{ for each }m \in \{1,\ldots,M-1\},  \;  i_{m+1} \in \{i_m \} \cup \mathcal{B}_{m+1}  \right \}.$$
\end{lemma}
\begin{corollary} \label{cor:reform_of_L} 
If Assumption~\ref{ass:nested} holds, then $|\mathcal{L}|  \ge 2^{M-1}$.
 \end{corollary}

With the above motivation, we now demonstrate 
that our technical developments behind Theorem~\ref{thm:main} in \S\ref{sec:characterization:proof} can be used to reformulate \eqref{prob:robust_simplified} as a linear optimization problem with numbers of decision variables and constraints that scale polynomially in both $n$ and $M$.  To explain our reformulation strategy,  we observe in the case of $\eta = 0$ that the linear optimization problem~\eqref{prob:robust_simplified} can be rewritten equivalently as 
\begin{align} \label{prob:robust_simplified_compact}
  \begin{aligned}
& \; \underset{y \in \mathcal{Y}(S)}{\textnormal{minimize}} && \sum_{i \in \mathcal{N}_0} r_i y_i, 
\end{aligned} 
\end{align}
where 
\begin{align*}
\mathcal{Y}(S) \triangleq \left \{ y \in \R^{\mathcal{N}_0}: \begin{gathered}\textnormal{there exists a vector } \lambda \equiv (\lambda_{i_1 \cdots i_M}: (i_1,\ldots,i_M) \in \mathcal{L}) \ge 0 \textnormal{ such that:}  \\[5pt]
\begin{aligned}
\sum_{(i_1,\ldots,i_M) \in \mathcal{L}: \rho_{i_1 \cdots i_M}(S) = r_i} \lambda_{i_1 \cdots i_M}  &=  y_i && \textnormal{for all } i \in S\\[5pt]
\sum_{(i_1,\ldots,i_M) \in \mathcal{L}: i_m = i} \lambda_{i_1 \cdots i_M} &=  v_{m,i}   && \textnormal{for all } m \in \mathcal{M}, \; i \in S_m\\[5pt]
 \sum_{(i_1,\ldots,i_M) \in \mathcal{L}} \lambda_{i_1 \cdots i_M} &= 1
\end{aligned}
\end{gathered} \right \}.
\end{align*}
The number of decision variables in the linear optimization problem~\eqref{prob:robust_simplified_compact} is equal to $| \mathcal{N}_0| = n+1$. 
Hence, we observe that the computational tractability of the linear optimization problem~\eqref{prob:robust_simplified_compact} depends on our ability to develop a compact representation of the   set $\mathcal{Y}(S)$.\footnote{A linear optimization problem over the feasible set $\mathcal{Y}(S)$ bears resemblance to the \textsc{RankAggregationLP} optimization  problem that is studied by  \cite{jagabathula2019limit}. In contrast, our set $\mathcal{Y}(S)$ requires an aggregation of the decision variables $\lambda_{i_1 \cdots i_M}$ in such a way that preserves the quantity $\rho_{i_1 \cdots i_M}(S)$, where we recall  from \S\ref{sec:characterization:proof} that $\rho_{i_1 \cdots i_M}(S)$ can be interpreted as the worst-case revenue that is earned by the firm when assortment $S$ is offered to a customer whose preferences are represented by a ranking that corresponds to the tuple of products $(i_1,\ldots,i_M) \in \mathcal{L}$. Because of this difference, the reformulations developed by \citet[\S4 and \S5]{jagabathula2019limit} are not readily applicable to the set $\mathcal{Y}(S)$.}
 In view of this observation, we proceed to show that the constraints on the decision variables $y \in \R^{\mathcal{N}_0}$ can be represented compactly when the past assortments are nested.  The key insight that will enable the compact  reformulation of the constraints on the decision variables $y \in \R^{\mathcal{N}_0}$ is the following Lemma~\ref{lem:rho_m}, which shows that the quantities $\rho_{i_1 \cdots i_M}(S \cap S_m)$ can be written \emph{recursively}. 
Similarly to the proof of Theorem~\ref{thm:main}, 
  the proof of the following Lemma~\ref{lem:rho_m} is based on the graphical interpretation of $\rho_{i_1 \cdots i_M}(S)$ that is developed in  Appendix~\ref{appx:graphical}. 

\begin{lemma} \label{lem:rho_m}
If Assumption~\ref{ass:nested} holds and the assortment $S  \subseteq \mathcal{N}_0$ satisfies $S \cap S_1 \neq \emptyset$,\footnote{The condition that $S \cap S_1 \neq \emptyset$ in Lemma~\ref{lem:rho_m} is satisfied, for example,  when $S, S_1 \in \mathcal{S}$.} then the following equality holds for each tuple $(i_1,\ldots,i_M) \in \mathcal{L}$ and $m \in \mathcal{M}$:
\begin{align} \label{defn:rho_m}
\rho_{i_1 \cdots i_M}(S \cap S_m) &= \begin{cases}
r_{i_m},&\textnormal{if } i_m \in S, \\
\min \limits_{i \in \mathcal{B}_m \cap S} r_i, &\textnormal{if }   i_m \notin S \textnormal{ and } m = 1,\\
\min \left \{ \min \limits_{i \in \mathcal{B}_m \cap S} r_i ,\; \rho_{i_1 \cdots i_M} \left(S \cap S_{m-1}  \right) \right \}, &\textnormal{if }  i_m \notin S \textnormal{ and }  m \ge 2. 
\end{cases}
\end{align}
\end{lemma}
   Equipped with the observation that the quantities $\rho_{i_1 \cdots i_M}(S \cap S_m)$ can be written recursively, we now introduce a directed acyclic graph that will serve as a low-dimensional network representation of the decision variables of the linear optimization problem~\eqref{prob:robust_simplified}. 
  Indeed, we define the vertices and directed edges of the graph as follows:
\begin{align*}
\mathfrak{V} &\triangleq \left \{ (m,i,\kappa): \; m \in \mathcal{M},\; i \in S_m, \; \kappa \in \{r_j: j \in S_m\} \right \};\\
\mathfrak{E} &\triangleq \left \{((m,i,\kappa),(m',i',\kappa')) \in \mathfrak{V} \times \mathfrak{V}: \;
m' = m+1,\; i' \in \{i\} \cup \mathcal{B}_{m+1},\;\kappa' \in \{ \kappa\} \cup \left \{ r_j: j \in \mathcal{B}_{m+1}  \right \} \right \}.
\end{align*}
We observe that each of the directed edges in the graph $\mathfrak{G} \equiv (\mathfrak{V},\mathfrak{E})$ is of the form $((m,i,\kappa),(m+1,i',\kappa'))$,  which readily implies that the directed graph $\mathfrak{G} \equiv (\mathfrak{V},\mathfrak{E})$  is acyclic. Moreover, the number of vertices $|\mathfrak{V}|$ and the number of directed edges $|\mathfrak{E}|$ clearly scale as a polynomial with respect to both the number of products $n$ as well as the number of past assortments $M$.  Using this directed acyclic graph, we now state our compact reformulation of the linear optimization problem~\eqref{prob:robust_simplified}. 
\begin{proposition} \label{prop:lp_reform_nested}
Let Assumption~\ref{ass:nested} hold. Then for each  $S \subseteq \mathcal{N}_0$ that satisfies $S \cap S_1 \neq \emptyset$, $\min_{\lambda \in \mathcal{U}} \mathscr{R}^\lambda(S)$ is equal to the optimal objective value of the following linear optimization problem:
{\small
 \begin{equation}  \label{prob:lp_reform_nested}
\begin{aligned}
& \; \underset{f, g,\epsilon}{\textnormal{minimize}} && \sum_{i,\kappa: (M,i,\kappa) \in \mathfrak{V}} \kappa  g_{M,i,\kappa}    \\
&\textnormal{subject to}&& 
\sum_{\kappa: (m,i,\kappa) \in \mathfrak{V}}  g_{m,i,\kappa}  - \epsilon_{m,i} = v_{m,i} && \forall m \in \mathcal{M}, \; i \in S_m \\
&&& \sum_{i',\kappa': ((m,i,\kappa),(m+1,i',\kappa')) \in \mathfrak{E}} f_{m,i,\kappa,i',\kappa'}  = g_{m,i,\kappa} && \forall (m,i,\kappa) \in \mathfrak{V}: m \in \{1,\ldots,M-1\} \\
&&& \sum_{i',\kappa': ((m-1,i',\kappa'),(m,i,\kappa)) \in \mathfrak{E}} f_{m-1,i',\kappa',i,\kappa}  = g_{m,i,\kappa} && \forall (m,i,\kappa) \in \mathfrak{V}: m \in \{2,\ldots,M\} \\
&&& g_{m,i,\kappa} = 0 && \forall (m,i,\kappa)\in  \mathfrak{V}: i \in S,\; i \in \mathcal{B}_m,  \; \kappa \neq r_i  \\
&&& g_{m,i,\kappa} = 0 && \forall (m,i,\kappa)\in  \mathfrak{V}: i \in S,\; i \notin \mathcal{B}_m, \;  \kappa \in \{r_j: j \in \mathcal{B}_m \}  \\
&&& g_{m,i,\kappa} = 0 && \forall (m,i,\kappa)\in  \mathfrak{V}: \kappa \in \{ r_j: j \in  \mathcal{B}_m \setminus S \} \\
&&& \sum_{i,\kappa: (M,i,\kappa) \in \mathfrak{V}} g_{M,i,\kappa} = 1 \\
&&& \| \epsilon \| \le \eta \\
&&& f_{m,i,\kappa,i',\kappa'} \ge 0 && \forall ((m,i,\kappa),(m+1,i',\kappa')) \in \mathfrak{E}  \\
&&& g_{m,i,\kappa} \ge 0 && \forall (m,i,\kappa) \in \mathfrak{V}.
\end{aligned} 
\end{equation}}
\end{proposition}

Let us make several remarks  about the linear optimization problem~\eqref{prob:lp_reform_nested}.  First, we remark that the decision variables $f,g$ can be interpreted as  a network flow through the directed acyclic graph $\mathfrak{G} \equiv (\mathfrak{V},\mathfrak{E})$.   Indeed, the decision variable $g_{m,i,\kappa}$ represents the amount of flow that passes through  each vertex $(m,i,\kappa) \in \mathfrak{V}$, the decision variable $f_{m,i,\kappa,i',\kappa'}$ represents the amount of flow through the directed edge from vertex $(m,i,\kappa) \in \mathfrak{V}$ to vertex $(m+1,i',\kappa') \in \mathfrak{V}$, and the constraints of the form $\sum_{i',\kappa': ((m,i,\kappa),(m+1,i',\kappa')) \in \mathfrak{E}} f_{m,i,\kappa,i',\kappa'}  = g_{m,i,\kappa}$ and  $\sum_{i',\kappa': ((m-1,i',\kappa'),(m,i,\kappa)) \in \mathfrak{E}} f_{m-1,i',\kappa',i,\kappa}  = g_{m,i,\kappa}$ ensure that flow is conserved through each  vertex $(m,i,\kappa) \in \mathfrak{V}$ with $m \in \{2,\ldots,M-1\}$. Second,  we  show in the proof of Proposition~\ref{prop:lp_reform_nested} that the decision variables in the linear optimization problem~\eqref{prob:lp_reform_nested} can be interpreted as an aggregation of the decision variables in the linear optimization problem~\eqref{prob:robust_simplified}. In particular, 
we show in Appendices~\ref{appx:mip:intermediary} and \ref{appx:lp_reform_nested} that each decision variable $\lambda_{i_1 \cdots i_M}$ can be interpreted as flow through the path $(1,i_1,\rho_{i_1 \cdots i_M}(S \cap S_1)),\ldots, (M,i_M,\rho_{i_1 \cdots i_M}(S \cap S_M))$ in the directed acyclic graph $\mathfrak{G} \equiv (\mathfrak{V},\mathfrak{E})$. Hence, any  optimal solution for the linear optimization problem~\eqref{prob:lp_reform_nested} can be transformed back into the decision variables  of the linear optimization problem~\eqref{prob:robust_simplified} by applying the flow decomposition theorem \citep[Theorem 3.5]{ahuja1988network}.

In summary, we have shown in Proposition~\ref{prop:lp_reform_nested} that  the worst-case expected revenue $\min_{\lambda \in \mathcal{U}} \mathscr{R}^\lambda(S)$ in the case of nested past assortments can be evaluated by solving a linear optimization problem with size that is polynomial in the number of products $n$ and the number of past assortments $M$. The key step behind developing the compact reformulation is exploiting the graphical interpretation of the quantities $\rho_{i_1 \cdots i_M}(S)$ that is developed in  Appendix~\ref{appx:graphical} to design a low-dimensional  network representation of the decision variables in the linear optimization problem~\eqref{prob:robust_simplified}.   By applying strong duality to the linear optimization problem from Proposition~\ref{prop:lp_reform_nested} and introducing binary decision variables to represent the assortment $S$, we readily obtain a mixed-integer optimization reformulation of the robust optimization problem  \eqref{prob:robust}: 
\begin{theorem}\label{thm:nested}
Let Assumption~\ref{ass:nested} hold and $S_1,\ldots,S_M \in \mathcal{S}$. Then the robust optimization problem~\eqref{prob:robust} can be reformulated as a mixed-integer linear optimization problem with $\mathcal{O}(n)$ binary decision variables, $\mathcal{O}(\textnormal{poly}(n,M))$ continuous decision variables, and $\mathcal{O}(\textnormal{poly}(n,M))$ constraints. 
\end{theorem}
The description of the mixed-integer optimization reformulation from Theorem~\ref{thm:nested}   is found at the end of Appendix~\ref{appx:mip:lp_reform_nested}. 

\begin{remark}The reformulation in Appendix~\ref{appx:mip:lp_reform_nested} is presented for cases in which the firm does not have any business constraints on the assortments $S \in \mathcal{S}$. In many applications of assortment optimization, firms have additional business constraints (e.g., due to limited shelf space or bundling requirements) that limit the space of possible new assortments; see \citet[p.113]{mivsic2016data}. Because our algorithm from Theorem~\ref{thm:nested} consists of solving a mixed-integer optimization problem with a binary decision variable $x_i \in \{0,1\}$ for each product $i \in \mathcal{N}_0$ that satisfies $x_i = 1$ if and only if product $i$ is included in the assortment $S$, such business   constraints can be easily integrated directly into our mixed-integer optimization problem. In particular, as long as the number of such business constraints is not exponential in $n$ or $M$,  the size of our mixed-integer optimization reformulation of \eqref{prob:robust}  will remain polynomial in the number of products and the number of past assortments. \end{remark}

 \section{Answer to Question~\ref{question:4}} \label{sec:question4}
In the previous sections, we showed that the robust optimization problem~\eqref{prob:robust} can find assortments that outperform the past assortments (see \S\ref{sec:question1}), and we developed algorithms for solving the robust optimization problem~\eqref{prob:robust} that are theoretically and practically efficient (see \S\ref{sec:algorithms}).  Therefore, when confronted with an assortment planning problem, we recommend that firms begin by solving the robust optimization problem~\eqref{prob:robust}. If  the optimal objective value  of the robust optimization problem~\eqref{prob:robust} is strictly greater than the expected revenue of the firm's best past assortment, then the optimal assortment for \eqref{prob:robust} should be offered by the firm to its customers, as such an assortment   can be trusted to improve the firm’s expected revenue in a way that is not exclusive to just one of the many ranking-based choice models that are consistent with the firm’s historical sales data.\looseness=-1
 
  In this section, we provide guidance to firms for what  assortment to offer to customers in cases in which the optimal objective value  of the robust optimization problem~\eqref{prob:robust} is \emph{not} strictly greater than the expected revenue of the firm's best past assortment. We begin in \S\ref{sec:question4:bestcase} by recommending that the firm first solve an optimistic version of the robust optimization problem~\eqref{prob:robust} to  obtain  a data-driven upper bound on the potential increase in expected revenue that a firm can hope to gain by performing experimentation. 
If the  gap between the optimal objective value of the optimistic optimization problem and the expected revenue of the firm's best past assortment is small, then we can conclude that  the potential increase in expected revenue that a firm can hope to gain by exploring with new assortments is small, thereby guiding the firm to avoid performing experimentation.   If the gap is large, then we conclude that the firm can potentially benefit from experimentation. For this case, we show in \S\ref{sec:question4:pareto} that the robust and optimistic optimization problems  can be combined into a Pareto optimization problem for finding new assortments for experimentation that have potential to increase the firm's best-case expected revenue while maintaining bounds on the worst-case decrease in expected revenue. 

 \subsection{Upper Bound on the Value of Experimentation} \label{sec:question4:bestcase}
If the optimal objective value  of the robust optimization problem~\eqref{prob:robust} is not strictly greater than the expected revenue of the firm's best past assortment, then there are several options available to the firm. One option is for the firm to continue offering their  best  past assortment, i.e., the past assortment with the highest expected revenue. This option can be interpreted as  `maintaining the status-quo'. A second option is for the firm to conduct experimentation by offering new assortments to their customers, with the goal of further constraining the set of ranking-based choice models that are consistent with the firm's historical sales data. However, experimentation  can be costly or otherwise undesirable to firms as discussed in  \S\ref{sec:introduction}. As such, a natural question faced by firms is whether the potential benefits of conducting experimentation will outweigh the costs. 

In this subsection, we show that our algorithms from \S\ref{sec:algorithms} for solving the robust optimization problem~\eqref{prob:robust} can be repurposed to  quantify the potential benefits of experimentation. Specifically, we observe that an upper bound on the potential increase in expected revenue that the firm can hope to gain by performing experimentation is given by the optimal objective value of an `optimistic' optimization problem of the form 
 \begin{equation} \tag{OO} \label{prob:optimistic}
\begin{aligned}
 & \underset{S \in \mathcal{S}}{\textnormal{max}} \max_{\lambda \in \mathcal{U}}\mathscr{R}^{\lambda}(S).
\end{aligned}
\end{equation}
The above optimization problem can be viewed as an `best-case' alternative to the robust optimization problem~\eqref{prob:robust}, in the sense that \eqref{prob:optimistic} will yield the assortment that has the maximum increase in the firm's predicted expected revenue on a ranking-based choice model that is consistent with the historical sales data generated by the firm's past assortments. 
 The optimization problem~\eqref{prob:optimistic} is attractive because it delivers an upper bound on the increase in  expected revenue that a firm can hope to gain by performing experimentation. 
If the  gap between the optimal objective values of \eqref{prob:optimistic} and \eqref{prob:robust} is small, then the firm can conclude that  the potential increase in expected revenue that a firm can hope to gain by experimentation with new assortments is small, thereby guiding the firm to avoid performing experimentation. 

The main technical contribution of the present subsection is showing that our polynomial-time algorithms and mixed-integer optimization reformulation  from \S\ref{sec:algorithms} that were designed for the robust optimization problem~\eqref{prob:robust}  can be extended to the optimistic optimization problem~\eqref{prob:optimistic}. These technical contributions are summarized as follows.
\begin{theorem} \label{thm:poly:optimistic}
\eqref{prob:optimistic} can be solved in weakly $\mathcal{O}(\textnormal{poly}(n))$  time for every fixed $M$.
\end{theorem}
\begin{theorem} \label{thm:nested:optimistic}
If Assumption~\ref{ass:nested} holds and $S_1,\ldots,S_M \in \mathcal{S}$, then \eqref{prob:optimistic} can be reformulated as a mixed-integer linear optimization problem with $\mathcal{O}(n)$ binary decision variables, $\mathcal{O}(\textnormal{poly}(n,M))$ continuous decision variables, and $\mathcal{O}(\textnormal{poly}(n,M))$ constraints. 
\end{theorem}
Theorems~\ref{thm:poly:optimistic} and \ref{thm:nested:optimistic}, which are analogues of Theorems~\ref{thm:poly} and \ref{thm:nested} from \S\ref{sec:algorithms}, establish that the optimistic optimization problem~\eqref{prob:optimistic} enjoys similar  tractability as the robust optimization problem~\eqref{prob:robust}. The proofs of Theorems~\ref{thm:poly:optimistic} and \ref{thm:nested:optimistic}, which are found in Appendix~\ref{appx:optimistic_algorithms}, follow using similar arguments as those from Theorems~\ref{thm:poly} and \ref{thm:nested}. In Appendix~\ref{appx:numerical:optimistic}, we show using synthetic data with nested past assortments that the optimistic optimization problem~\eqref{prob:optimistic} can yield practical upper bounds on the potential value of experimentation that can help a firm determine that the benefits of conducting experimentation will not outweigh the costs.

 \subsection{Low-Risk Approaches to Experimentation} \label{sec:question4:pareto}
 Suppose that the gap between the upper bound given by \eqref{prob:optimistic} and the expected revenue of the firm's best past assortment is deemed by the firm to be large enough to warrant experimentation.  The firm, in this case, must decide which new assortments it will offer to its customers. Here, we take the perspective of a firm that wishes to perform experimentation in such a way that mitigates the worst-case short-run declines in expected revenue that come from experimenting with new assortments.  That is, we assume that the firm is interested in experimenting with new assortments $S \in \mathcal{S}$ for offering to their customers that can increase the firm's expected revenue while ensuring the worst-case expected revenue of the new assortment, $\min_{\lambda \in \mathcal{U}} \mathscr{R}^\lambda(S)$,  remains close to the expected revenue of the firm's  best past assortment, $ \max_{m \in \mathcal{M}} r^\intercal v_m$. We offer two approaches below for conducting low-risk experimentation  that utilize our algorithms for solving \eqref{prob:robust} and \eqref{prob:optimistic}.\looseness=-1
 
 \subsubsection{Exact Algorithm for Pareto Optimization.}  \label{sec:question4:pareto:exact}
Our first proposed approach to finding new assortments for experimentation is to solve the following `Pareto' optimization problem:
\begin{equation} \tag{PO} \label{prob:pareto}
 \begin{aligned}
 & \underset{S \in \mathcal{S}}{\textnormal{maximize}} && \max_{\lambda \in \mathcal{U}}\mathscr{R}^{\lambda}(S)\\
 &\textnormal{subject to}&& \min_{\lambda \in \mathcal{U}}\mathscr{R}^{\lambda}(S) \ge \theta.
\end{aligned}
\end{equation}
In contrast to the robust optimization problem~\eqref{prob:robust} and the optimistic optimization problem~\eqref{prob:optimistic}, the Pareto optimization problem~\eqref{prob:pareto} aims to find an assortment that maximizes the {best-case} expected revenue while imposing a lower bound on the worst-case expected revenue. 
By varying the lower bound $\theta$, the above optimization problem thus leads to a \emph{Pareto frontier} of new assortments that a firm can offer to its customers that offer varying degrees of possible improvement to the firm's expected revenue while enjoying guarantees on the worst-case decrease in expected revenue that the firm might incur by experimenting with the new assortment. 
Such a Pareto frontier of new assortments provided by \eqref{prob:pareto} can be attractive to risk-averse firms in situations in which there do not exist any assortments that satisfy \eqref{line:improvement}, as there may exist optimal assortments for the optimization problem~\eqref{prob:pareto} that are different from the firm's past assortments when choosing the parameter  $\theta$ to be slightly less than the expected revenue of the firm's best past assortment.\looseness=-1

In the case of nested past assortments, we can use the same reformulation techniques from Theorems~\ref{thm:nested} and \ref{thm:nested:optimistic} to reformulate \eqref{prob:pareto} as a compact mixed-integer optimization problem: 
\begin{theorem}\label{thm:nested_reform_pareto}
Let Assumption~\ref{ass:nested} hold and $S_1,\ldots,S_M \in \mathcal{S}$. Then the optimization problem~\eqref{prob:pareto} can be reformulated as a mixed-integer linear optimization problem with $\mathcal{O}(n)$ binary decision variables, $\mathcal{O}(\textnormal{poly}(n,M))$ continuous decision variables, and $\mathcal{O}(\textnormal{poly}(n,M))$ constraints. 
\end{theorem}
The description of the mixed-integer optimization reformulation from Theorem~\ref{thm:nested_reform_pareto}   is found at the end of Appendix~\ref{appx:mip:nested_reform_pareto}. In contrast to the polynomial-time algorithms that were developed for \eqref{prob:robust} and \eqref{prob:optimistic},  we do not have an algorithm for solving \eqref{prob:pareto} that runs in polynomial-time for any fixed number of past assortments (see \S\ref{sec:question4:pareto:approximate} for discussion). In Appendix~\ref{appx:numerical:pareto}, we show using synthetic data with various collections of nested past assortments that \eqref{prob:pareto} can find new assortments that offer best-case increases in expected revenue of 20\%-60\% while enjoying worst-case decreases in expected revenues as low as 1\%-10\%.

\subsubsection{Heuristic for Pareto Optimization.} \label{sec:question4:pareto:approximate}
It is straightforward to reformulate \eqref{prob:pareto} as a compact mixed-integer optimization problem in the case of nested past assortments (see Theorem~\ref{thm:nested_reform_pareto}) by combining the compact mixed-integer optimization reformulations for \eqref{prob:robust} and \eqref{prob:optimistic} (see Theorems~\ref{thm:nested} and \ref{thm:nested:optimistic}). The same reasoning, however,  does not extend to solving \eqref{prob:pareto} in polynomial-time for any fixed number of past assortments. The issue stems from the fact that our polynomial-time algorithms for solving  \eqref{prob:robust} and \eqref{prob:optimistic}  (see Theorems~\ref{thm:poly} and \ref{thm:poly:optimistic}) rely on performing an exhaustive search over polynomial-size collections of assortments that are guaranteed to contain optimal assortments for \eqref{prob:robust} and \eqref{prob:optimistic}. In contrast to \eqref{prob:robust} and \eqref{prob:optimistic}, we do not have a characterization of the structural of optimal assortments for \eqref{prob:pareto}, and thus cannot perform a polynomial-time exhaustive search that is guaranteed to deliver an optimal assortment for \eqref{prob:pareto}. 

To get around this, we propose a simple heuristic for finding new assortments for risk-averse experimentation. The heuristic runs in polynomial-time for any fixed number of past assortments and consists of three steps. In the first step, we construct the collection of assortments $\widehat{\mathcal{S}}$ from \S\ref{sec:characterization} that is guaranteed to contain an optimal solution to \eqref{prob:robust}. In the second step,  we compute the worst-case expected revenue $\min_{\lambda \in \mathcal{U}} \mathscr{R}^\lambda(S)$ and the best-case expected revenue $\max_{\lambda \in \mathcal{U}} \mathscr{R}^\lambda(S)$ for each assortment $S \in \widehat{\mathcal{S}}$. In the third step, we compute the following \emph{approximate Pareto frontier}: 
\begin{align}
\left\{S \in \widehat{\mathcal{S}}: \left \{S' \in \widehat{\mathcal{S}}: \begin{gathered} \left[ \min_{\lambda \in \mathcal{U}} \mathscr{R}^\lambda(S') \ge \min_{\lambda \in \mathcal{U}} \mathscr{R}^\lambda(S)  \text{ and } \max_{\lambda \in \mathcal{U}} \mathscr{R}^\lambda(S') > \max_{\lambda \in \mathcal{U}} \mathscr{R}^\lambda(S) \right] \\ \text{or} \\\left[ \min_{\lambda \in \mathcal{U}} \mathscr{R}^\lambda(S') > \min_{\lambda \in \mathcal{U}} \mathscr{R}^\lambda(S)  \text{ and } \max_{\lambda \in \mathcal{U}} \mathscr{R}^\lambda(S') \ge \max_{\lambda \in \mathcal{U}} \mathscr{R}^\lambda(S) \right]  \end{gathered} \right \}  = \emptyset \right \}. \label{line:paretofrontier}
\end{align}
The collection of assortments from line~\eqref{line:paretofrontier} can be interpreted as an approximation of the Pareto frontier that would be obtained by solving \eqref{prob:pareto} for all $\theta \ge 0$. In Appendix~\ref{appx:numerical:pareto:approximate}, we show using real data from a conjoint dataset \citep{toubia2003fast}  that the approximate Pareto frontier from line~\eqref{line:paretofrontier} can contain new assortments that offer increases in expected revenue with respect to the true ranking-based choice model, while simultaneously enjoying guarantees that implementing the new assortments will not lead to meaningful worst-case decreases in the firm's expected revenue.

\section{Conclusion and Future Research} \label{sec:conclusion}

In this work, we investigated a popular class of high-dimensional discrete choice models, known as {ranking-based} choice models, in the context of assortment planning problems. 
Motivated by the fact that many ranking-based choice models can be consistent with a firm's historical sales data, we considered a class of robust optimization problems in which the goal is find an assortment that has high expected revenue across all of the ranking-based choice models that are consistent with the firm's historical sales data. By developing the first structural results and practical algorithms for solving these robust optimization problems, our paper showed in a variety of applications that these robust optimization problems can be solved in reasonable computation times and offer   value to firms in comparison to estimate-then-optimize in the context of data-driven assortment optimization with ranking-based choice models. 

We believe this work opens up a number of promising directions for future research. First, our work showed for the first time that the question of whether it can be possible to identify assortments that satisfy condition~\eqref{line:improvement} can be answered for one popular class of high-dimensional discrete choice models. Yet there are many other high-dimensional discrete choice models  beyond the ranking-based choice model for which this identification question can be asked, 
such as models for capturing irrational customer choice \citep{berbeglia2018generalized,chen2020decision,jena2021estimation}. Second, our work showed that polynomial-time algorithms can be developed for solving the robust optimization problem~\eqref{prob:robust} for ranking-based choice models.  However, 
it may be possible in certain settings that existing algorithms for estimating high-dimensional discrete choice models such as expectation-maximization \citep{talluri2004revenue,van2017expectation,csimcsek2018expectation}, in combination with algorithms for finding assortments that maximize the predicted expected revenue, 
can lead to assortments with provable performance guarantees with respect to the identification question. Establishing such guarantees would provide new assurances to firms for using estimate-then-optimize in high-stakes assortment planning problems. Finally, we believe that the present and related work (\eg \cite{kallus2018confounding,sturt2021nonparametric}) provide a starting point for using robust optimization to develop efficient algorithms that are valuable for operations management problems in which good \emph{average} performance is paramount.  In particular, future work may extend the algorithms developed in the present paper to numerous other revenue management problems, ranging from multi-product pricing to dynamic assortment planning.  

\SingleSpacedXI
\bibliographystyle{informs2014} 
\bibliography{robust_assortment_optimization_eprint}
\OneAndAHalfSpacedXI 
 
\ECSwitch


\AppendixTitle{Technical Proofs and Additional Results}
\SingleSpacedXI


\begin{APPENDICES}

\section{Numerical Experiments} \label{appx:numerics}
This appendix contains the numerical experiments for the paper. \begin{itemize}
\item Appendix~\ref{appx:riskseto} shows that estimate-then-optimize with ranking-based choice models can risk leading to significant declines in a firm's expected revenue when the firm's past assortments are revenue-ordered. The motivation for Appendix~\ref{appx:riskseto} is found at the end of  \S\ref{sec:characterization:question}. 
\item Appendix~\ref{sec:twoassortments:numerics} shows the practical efficiency of the polynomial-time algorithm from \S\ref{sec:algorithm:twoassortments} for solving the robust optimization problem~\eqref{prob:robust} in the case of two past assortments. 
\item Appendix~\ref{appx:numerical:nested} shows the practical tractability of the compact mixed-integer optimization reformulation of the robust optimization problem~\eqref{prob:robust} from \S\ref{sec:algorithms:nested}   for the case of nested past assortments.
\item Appendix~\ref{appx:numerical:optimistic}  shows the managerial value of  using the optimistic optimization problem~\eqref{prob:optimistic} from \S\ref{sec:question4:bestcase} to compute upper bounds on the maximum increase in expected revenue that a firm can hope to gain by experimenting with new assortments.
\item  Appendix~\ref{appx:numerical:pareto}  shows the managerial value of the exact algorithm from \S\ref{sec:question4:pareto:exact} for solving the Pareto optimization problem~\eqref{prob:pareto} for finding new assortments that can be used by firms for  low-risk experimentation.\looseness=-1
\item Appendix~\ref{appx:numerical:pareto:approximate} uses real data to show the managerial value of the heuristic from \S\ref{sec:question4:pareto:approximate} for the Pareto optimization problem~\eqref{prob:pareto} for finding new assortments that can be used by firms for  low-risk experimentation.\looseness=-1
\end{itemize} 

\subsection{Numerical Illustration  of the Risks of Estimate-Then-Optimize} \label{appx:riskseto}
In this appendix, we perform numerical experiments to assess  the performance of assortments obtained using estimate-then-optimize when the historical sales data is randomly generated from revenue-ordered assortments.  
 To motivate our experiment setup, we recall the details of estimate-then-optimize from \S\ref{sec:introduction}. Indeed, suppose that one estimates a ranking-based choice model from the historical sales data generated by the revenue-ordered assortments and then recommends that the firm implement a new assortment 
which maximizes the predicted expected revenue
 under the estimated ranking-based choice model. For this setting,  Corollary~\ref{cor:impossibility} from \S\ref{sec:characterization:question}
 guarantees that this estimate-then-optimize technique will never offer \emph{fidelity} to the firm, in the sense that there  will always exist  a ranking-based choice model which is consistent with the historical sales data for which the expected revenue for the new assortment will be less than or equal to the expected revenue of the best past assortment. Moreover, as we will further see through the following numerical experiment, the expected revenue from the assortment recommended by  estimate-then-optimize  can be {strictly} and significantly {less} than the expected revenue of the best past assortment.

 To perform our numerical experiment, we begin by constructing randomly-generated problem instances. In each problem instance,  the revenues for the products are drawn from the distribution $r_1,\ldots,r_n \sim \textnormal{Uniform}[0,1]$, and a base choice   for the parameters $\lambda^*$ of a ranking-based choice model is drawn uniformly over the $(n+1)!$-dimensional probability simplex.\footnote{If the rankings in $\Sigma$ are indexed by $\{\sigma_1,\ldots,\sigma_{(n+1)!}\}$, then uniform sampling over the probability simplex is obtained by drawing $u_1,\ldots,u_{(n+1)!} \sim \textnormal{Uniform}[0,1]$ and setting $\lambda_{\sigma_k} \leftarrow \log(u_k) /  \sum_{k' =1}^{(n+1)!}\log(u_{k'})$ for each $k=1,\ldots,(n+1)!$. }  Using this base choice for the parameters, we generate historical sales data of the form $v_1,\ldots,v_n$, where each $v_m$ is the historical sales data generated by the revenue-ordered assortment $\{0,m,m+1,\ldots,n-1,n\}$ under the ranking-based choice model with the base parameter $\lambda^*$.\footnote{We sort the products in ascending order by revenue before performing our analysis, which ensures that $\{0,m,m+1,\ldots,n\}$ for each $m  \in \{1,\ldots,n\}$ is a revenue-ordered assortment.}  After we compute the historical sales data, we \emph{forget} the base parameters $\lambda^*$ of the ranking-based choice model and apply the estimate-then-optimize technique to obtain a new assortment.  Specifically, we first estimate the parameters $\hat{\lambda}$ of a ranking-based choice model  using the historical sales data; since many selections of the parameters may be consistent with the historical  sales data, we choose our estimate $\hat{\lambda}$ as the optimal solution to the linear optimization problem $\min_{\lambda \in \mathcal{U}} c^\intercal \lambda$, where the cost vector $c$ is drawn uniformly over $[0,1]^{(n+1)!}$.\footnote{Alternative approaches for estimating the parameters of a ranking-based choice model from historical sales data are provided in \cite{mivsic2016data,van2015market,van2017expectation,desir2021mallows}. Our approach of estimating the parameters as $\hat{\lambda} \in \argmin_{\lambda \in \mathcal{U}} c^\intercal \lambda$ for a randomly-chosen cost vector $c$ is motivated by our desire to decouple any potential biases associated with any particular estimation procedure from an empirical assessment of the estimate-then-optimize technique. In particular, our approach is viewed as a simple way of randomly sampling the parameters from the set of all parameters of ranking-based choice models which are consistent with the historical sales data. }\footnote{An obvious downside of our simple estimation procedure for the parameters of the ranking-based choice model is that it requires solving a linear optimization problem with $\mathcal{O}(n!)$ decision variables, and, thus, our simple estimation procedure does not scale efficiently  to problem instances with many products. Nonetheless, this estimation procedure is sufficiently fast for the purposes of this numerical study, where the aim is simply to assess the performance of assortments obtained by the estimate-then-optimize technique over revenue-ordered assortments.  In particular, we believe it is a reasonable assumption that similar findings  from our numerical experiment with $n=4$ (see Figures~\ref{fig:roa} and \ref{fig:boxplot}) would be found in experiments with larger values of $n$.} We then obtain a new assortment $S'$ as any optimal solution to the combinatorial optimization problem $\max_{S \in \mathcal{S}} \mathscr{R}^{\hat{\lambda}}(S)$ which maximizes the predicted expected revenue under the estimated ranking-based choice model.\footnote{We solve this optimization problem using the mixed-integer optimization reformulation given by \citet[\S3.2]{bertsimas2019exact}, which is implemented using the Julia programming language with JuMP and solved using Gurobi.} Finally, we evaluate the new assortment obtained using estimate-then-optimize by computing the worst-case expected revenue of the new assortment under all ranking-based choice models that are consistent with the historical sales data, $\min_{\lambda \in \mathcal{U}} \mathscr{R}^{{\lambda}}(S')$, the best-case expected revenue of the new assortment under all ranking-based choice models that are consistent with the historical sales data, $\max_{\lambda \in \mathcal{U}} \mathscr{R}^{{\lambda}}(S')$, and the expected revenue of the best past assortment, $\max_{m \in \mathcal{M}} r^\intercal v_m$.

\afterpage{%
\null
\vfill
 \begin{figure}[H]
\centering
\FIGURE{
\includegraphics[width=0.6\linewidth]{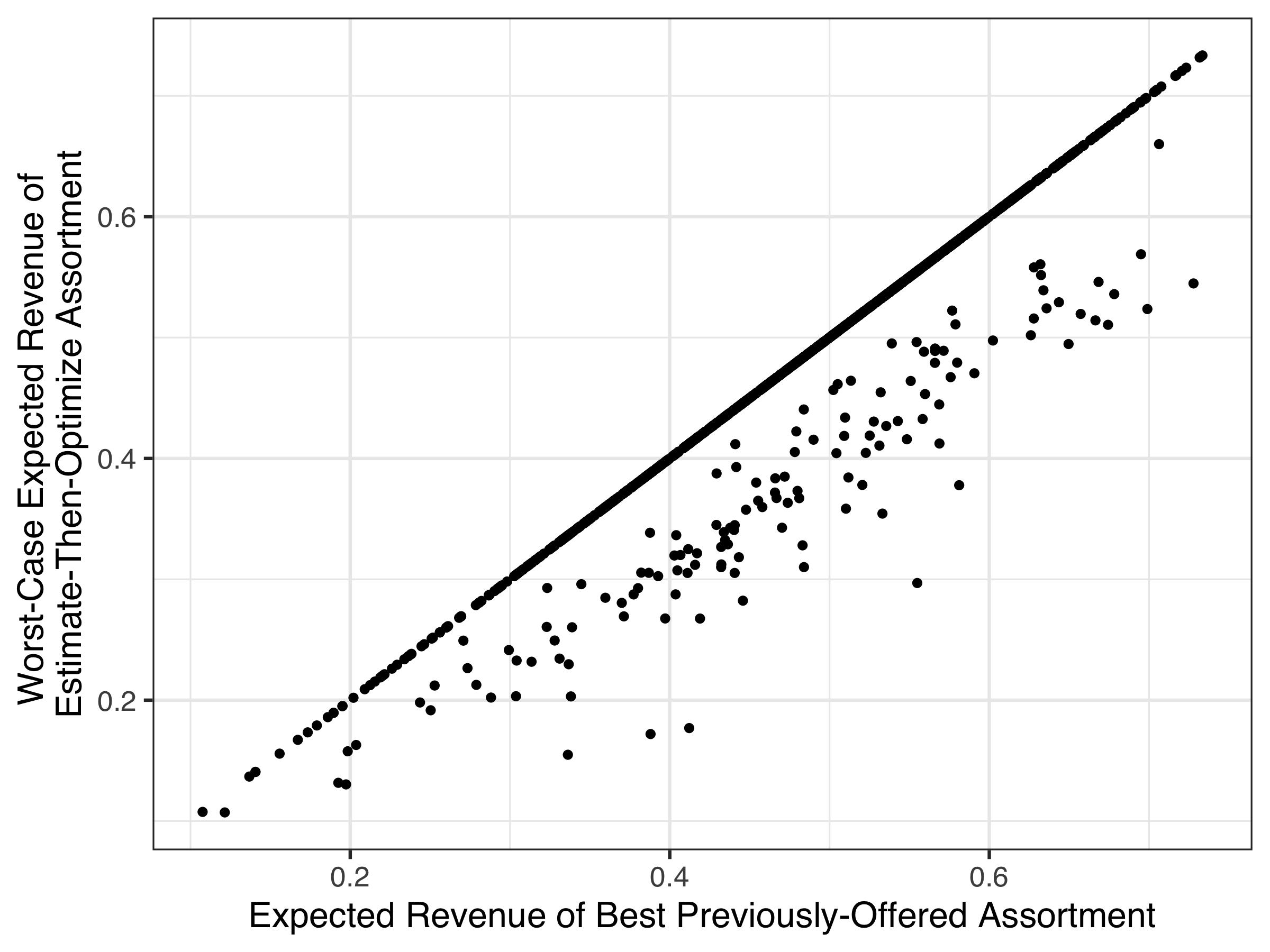}
}
{Performance of new assortments obtained using  estimate-then-optimize  when $\mathscr{M} = \bar{\mathcal{S}}$. \label{fig:roa} }
{\TABLEfootnotesizeIX Each point corresponds to a randomly-generated problem instance. The $x$-axis shows the expected revenue of the best past assortment,  $\max_{m \in \mathcal{M}} r^\intercal v_m$. The $y$-axis shows the worst-case expected revenue $\min_{\lambda \in \mathcal{U}} \mathscr{R}^{\lambda}(S')$ for the new assortment $S'$ obtained using the estimate-then-optimize technique. } 
\end{figure}

\begin{figure}[H]
\centering
\FIGURE{
\includegraphics[width=0.6\linewidth]{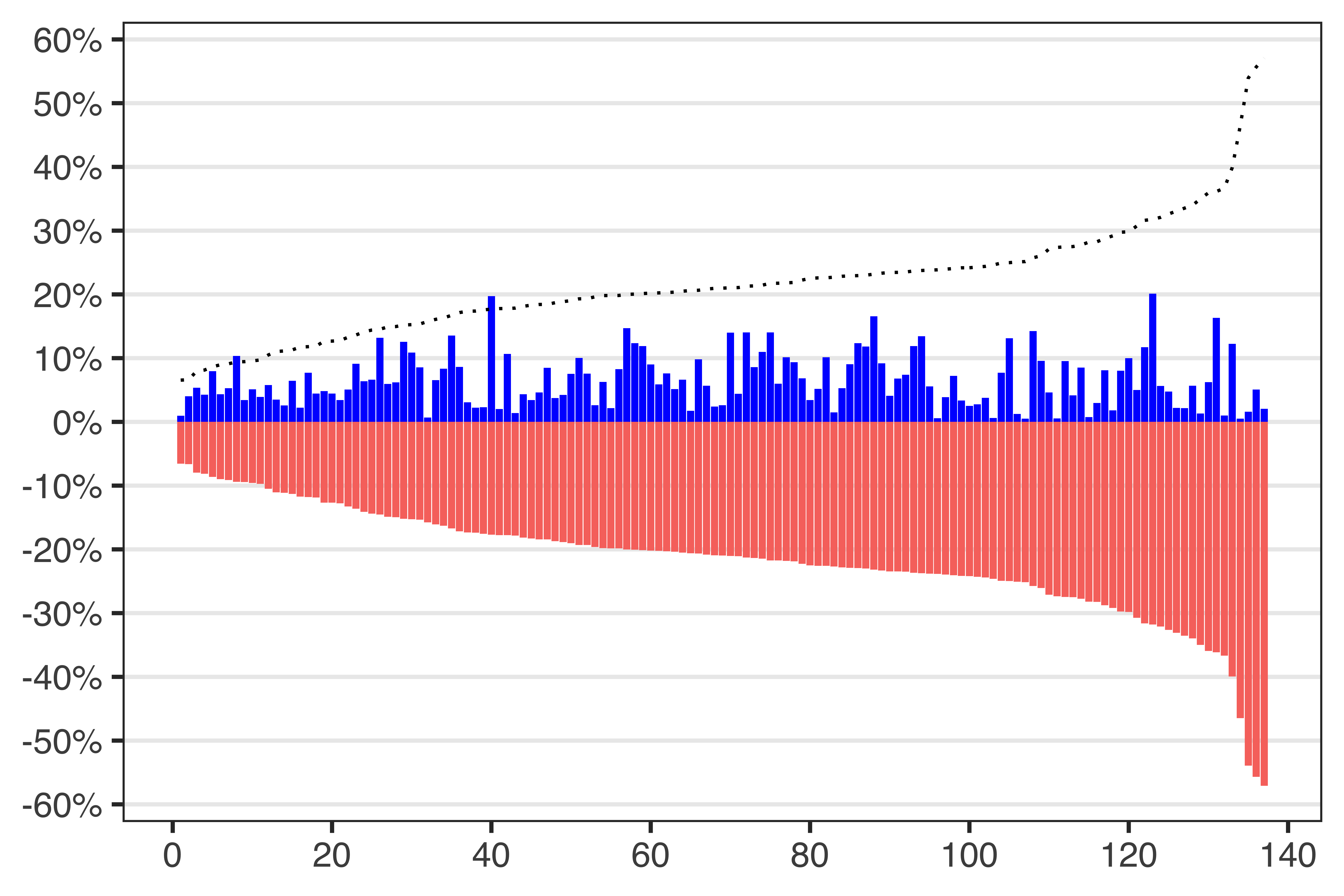}
}
{Relative improvement in expected revenue of new assortments obtained using  estimate-then-optimize  when $\mathscr{M} = \bar{\mathcal{S}}$. \label{fig:boxplot} }
{\TABLEfootnotesizeIX Each of the $x$-values corresponds to one of the 137 randomly-generated problem instances from Figure~\ref{fig:roa} for which the expected revenue of the best past assortment,  $\max_{m \in \mathcal{M}} r^\intercal v_m$, was strictly greater than the worst-case expected revenue $\min_{\lambda \in \mathcal{U}} \mathscr{R}^{\lambda}(S')$ for the new assortment $S'$ obtained using the estimate-then-optimize technique. The corresponding interval of $y$-values formed by the red and blue bars is the interval $[\Pi_\lambda: \lambda \in \mathcal{U}]$,  where $\Pi_\lambda \triangleq 100\% \times (\mathscr{R}^{\lambda}(S') - \max_{m \in \mathcal{M}} r^\intercal v_m) / ( \max_{m \in \mathcal{M}} r^\intercal v_m)$ is the relative percentage improvement of the expected revenue of the new assortment  obtained using estimate-then-optimize over the expected revenue of the firm's best past assortment for a given ranking-based choice model $\lambda \in \mathcal{U}$.  For clarity, the problem instances are sorted along the $x$-axis by the endpoints of the red bars, and the dotted line shows the reflection of the endpoints of red bars over the horizontal line at zero. }
\end{figure}
\null
\vfill
\clearpage
}

In Figures~\ref{fig:roa} and \ref{fig:boxplot}, we present the results of these numerical experiments for the case of $n = 4$ products over 1000 randomly-generated problem instances. In Figure~\ref{fig:roa}, we compare the expected revenue under the best past assortment,   $\max_{m \in \mathcal{M}} r^\intercal v_m$, to the expected revenue of the new assortments obtained using estimate-then-optimize under the worst-case ranking-based choice model that is consistent with the historical sales data, $\min_{\lambda \in \mathcal{U}} \mathscr{R}^{{\lambda}}(S')$. We observe that the results in Figure~\ref{fig:roa} are consistent with the impossibility result from Corollary~\ref{cor:impossibility}; indeed, the worst-case expected revenues of the new assortments obtained using estimate-then-optimize  never exceed the expected revenues of the best past assortments.  Furthermore, we observe for many of the problem instances that there are ranking-based choice models $\lambda \in \mathcal{U}$ that are consistent with the historical sales data for which the resulting expected revenue of the new assortment $\mathscr{R}^{\lambda}(S')$ is {strictly less} than the expected revenue under the best past assortment,   $\max_{m \in \mathcal{M}} r^\intercal v_m$.

To assess whether the above findings are overly conservative from a practical standpoint, we turn to a detailed analysis of the $137$ problem instances for which the expected revenue of the best past assortment,  $\max_{m \in \mathcal{M}} r^\intercal v_m$, is strictly greater than the worst-case expected revenue $\min_{\lambda \in \mathcal{U}} \mathscr{R}^{\lambda}(S')$ for the new assortment $S'$ obtained using  estimate-then-optimize. In Figure~\ref{fig:boxplot} we present, for each of these 137 problem instances, a visualization of the range of relative percentage improvements of the predicted expected revenue of the new assortment, $\mathscr{R}^{\lambda}(S')$, over the expected revenue of the firm's best past assortment,   $\max_{m \in \mathcal{M}} r^\intercal v_m$, which are possible to obtain under a ranking-based choice model which is consistent with the historical sales data, $\lambda \in \mathcal{U}$. Stated more precisely, each of the $x$ values in Figure~\ref{fig:boxplot} corresponds to one of these 137 problem instances, and the corresponding interval of $y$-values formed by the red and blue bars is the interval $[\Pi_\lambda: \lambda \in \mathcal{U}]$,  where $\Pi_\lambda \triangleq 100\% \times (\mathscr{R}^{\lambda}(S') - \max_{m \in \mathcal{M}} r^\intercal v_m) / ( \max_{m \in \mathcal{M}} r^\intercal v_m)$ is the relative percentage improvement of the predicted expected revenue of the new assortment $S'$ over the expected revenue of the firm's best past assortment for a given ranking-based choice model $\lambda \in \mathcal{U}$.  Hence, the red bars are the negative values of $\Pi_\lambda$ that can be attained under the ranking-based choice models $\lambda \in \mathcal{U}$, and the dotted line shows the reflection of the endpoints of red bars over the horizontal line at zero. 
We do not assign any likelihood to the values in each interval $[\Pi_\lambda: \lambda \in \mathcal{U}]$, as there is no information available for inferring which of the ranking-based choice models are more likely to be the `truth' among the ranking-based choice models $\lambda \in \mathcal{U}$ that are  consistent with the historical sales data. 

The results in  Figure~\ref{fig:boxplot} reveal a striking {asymmetry} between the downside and upside of implementing a new assortment found by estimate-then-optimize. 
In all but two of the 137 instances, the worst-case decline in expected  revenue from implementing the new assortment exceeded in magnitude the best-case increase in expected revenue. The difference in magnitude between the downside and upside is also found to be significant: the average best-case improvement of the new assortment over the best past assortment (i.e., the average of the blue endpoints) is 6.56\%, while the  average worst-case improvement of the new assortment over the best past assortment (i.e., the average of the red endpoints) is  -21.71\%. These numerical findings demonstrate that the downside risks to a firm from implementing a new assortment found by estimate-then-optimize can significantly exceed the potential upsides. 

In conclusion, our theoretical and numerical analysis in this subsection  lead to three main takeaways. 
The first takeaway is that there exist collections of  past assortments in which it can be impossible to identify assortments that satisfy~\eqref{line:improvement}. 
The second takeaway is that  estimate-then-optimize can lead to a strictly worse expected revenue than those of the past assortments. In fact, the numerical results in Figures~\ref{fig:roa} and \ref{fig:boxplot} show that this decline in expected revenue can be significant and outweigh the potential upside for implementing the new assortment. The third takeaway is that all of the aforementioned issues  arise  when the past assortments are comprised of one of the most celebrated and widely-used classes of assortments from the literature, namely, the revenue-ordered assortments. All together, these takeaways raise concerns about whether the estimate-then-optimize technique with ranking-based choice models should be trusted to ``first, do no harm" in assortment planning problems. 

\subsection{Numerical Experiments for Two  Past Assortments}\label{sec:twoassortments:numerics}
In this appendix, we perform numerical experiments to show the practical efficiency of our polynomial-time algorithm from \S\ref{sec:algorithm:twoassortments} for solving the robust optimization problem~\eqref{prob:robust} in the case of $M = 2$ past assortments.

To perform our numerical experiments, we begin by constructing randomly-generated problem instances in a manner that is similar to that taken in Appendix~\ref{appx:riskseto}. In each randomly-generated problem instance,  the revenues are drawn from the distribution $r_1,\ldots,r_n \sim \textnormal{Uniform}[0,1]$, and we sort the products such that $r_1 < \cdots < r_n$.  The two past assortments $S_1,S_2 \in \mathcal{S}$ are also constructed randomly, whereby the two assortments satisfy $\{0,n\} \subseteq S_1 \cap S_2$ and, for each of the remaining products $j \in \{1,\ldots,n-1\}$, we randomly assign the product to the assortments with distribution given by $P(j \in S_1 \cap S_2) = \frac{1}{3}$, $P(j \in S_1 \setminus S_2) =  \frac{1}{3}$, and $P(j \in S_2 \setminus S_1) =  \frac{1}{3}$. 
In order to generate historical sales data for these two assortments, we generate a base choice for the parameters $\lambda^*$ of the ranking-based choice model. 
Because we  will be performing numerical experiments on problem instances with larger values of $n$ than were considered in Appendix~\ref{appx:riskseto},  
it will not be viable from a computational tractability standpoint to generate base parameters $\lambda^*$ that have nonzero values for each of the $(n+1)!$ parameters of a  ranking-based choice model.  To get around this, we restrict the numerical experiments to generating base parameters $\lambda^*$ which are \emph{sparse}. Specifically,  we generate the base parameters in each problem instance by  first randomly selecting a subset of rankings $\Sigma' \subseteq \Sigma$ of length $| \Sigma'| = K$;\footnote{We use rejection sampling to ensure that each of the $\binom{(n+1)!}{K}$  subsets of rankings  is selected with equal probability.} we then assign $\lambda_\sigma^* \leftarrow 0$ for each ranking $\sigma \notin \Sigma'$, and we choose the remaining parameters  $\{ \lambda_\sigma: \sigma \in \Sigma'\}$   by drawing uniformly over the $K$-dimensional probability simplex. 
Using this base choice for the parameters, we generate historical sales data of the form $v_1$ and $v_2$ corresponding to the two assortments $S_1$ and $S_2$ under the ranking-based choice model with the base parameters $\lambda^*$.  After we compute the historical sales data, we \emph{forget} the base parameters $\lambda^*$ of the ranking-based choice model as well as the choice of $K$, and we apply the algorithm from \S\ref{sec:algorithm:twoassortments} to obtain a new assortment, denoted by $S'$.

\afterpage{%
\null
\vfill
\begin{figure}[H]
\centering
\FIGURE{
\includegraphics[width=0.6\linewidth]{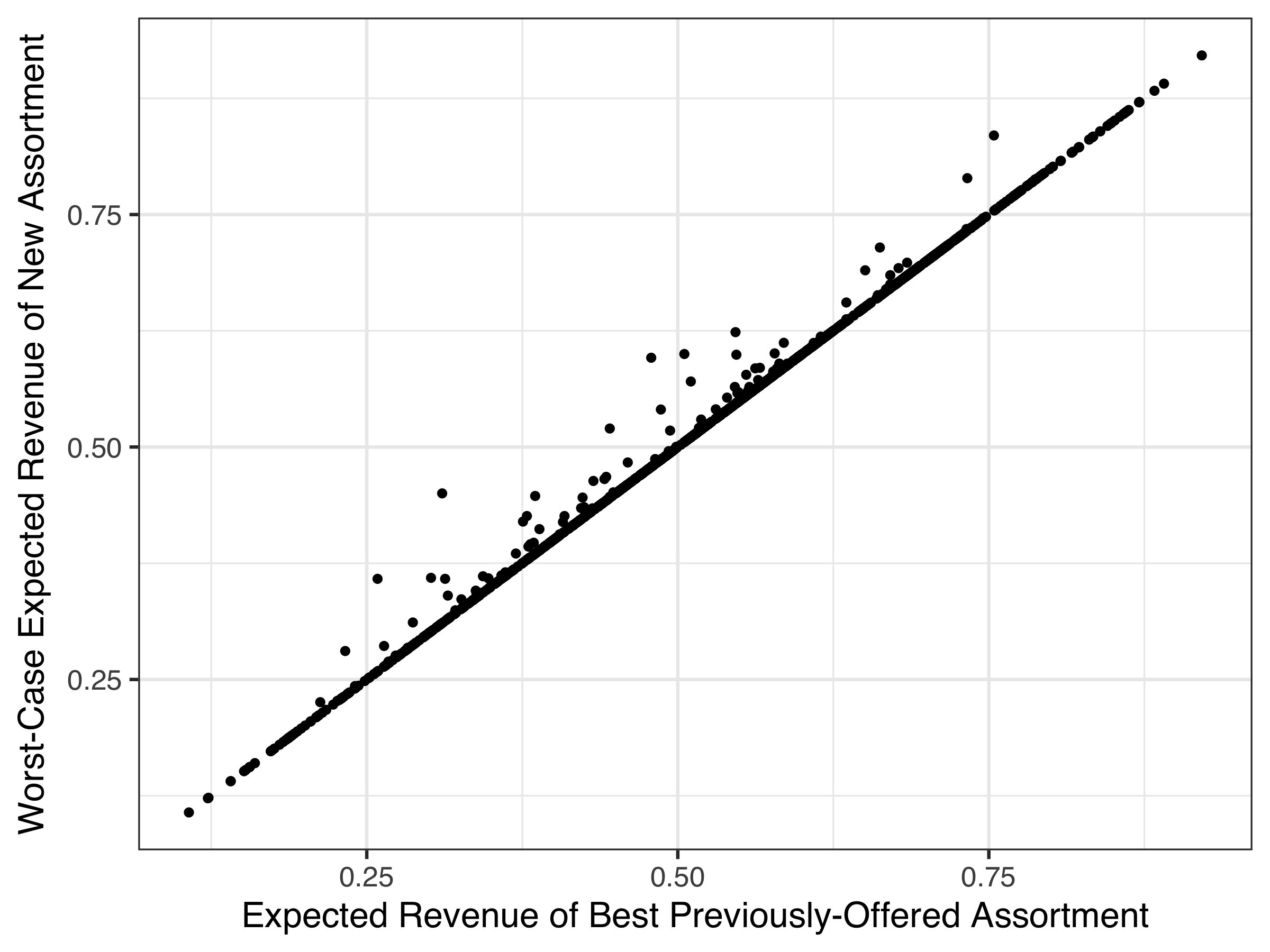}
}
{Performance of assortments obtained using algorithm from  \S\ref{sec:algorithm:twoassortments} when $M=2$, $K = 10$, and $n = 10$.  \label{fig:twoassortments:plots}}{ \TABLEfootnotesizeIX  Each point corresponds to a randomly-generated problem instance. The $x$-axis shows the expected revenue of the best past assortment,  $\max_{m \in \mathcal{M}} r^\intercal v_m$. The $y$-axis shows the worst-case expected revenue $\min_{\lambda \in \mathcal{U}} \mathscr{R}^{\lambda}(S')$ for the new assortment $S'$ obtained by using the algorithm from  \S\ref{sec:algorithm:twoassortments}. }
\end{figure}

\begin{figure}[H]
\centering
\FIGURE{
\includegraphics[width=0.6\linewidth]{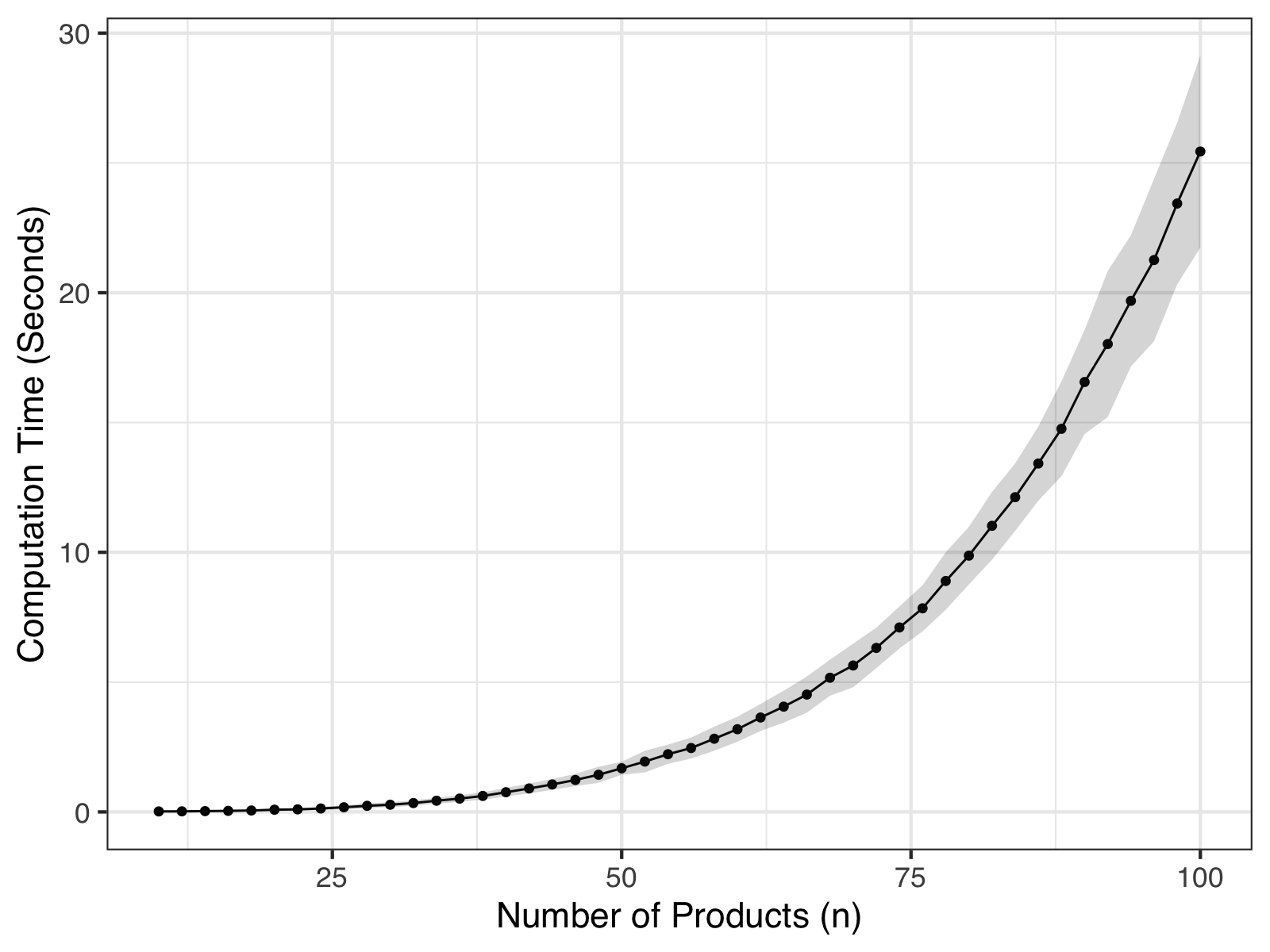}
}
{Computation time for the algorithm from  \S\ref{sec:algorithm:twoassortments} when $M=2$ and $K = 1000$. \label{fig:speed_twoassortments} }{\TABLEfootnotesizeIX Results are averaged over 100 replications for each $n \in \{10,12,14,\ldots,100\}$.}
\end{figure}
\null
\vfill
\clearpage
}

 In Figures~\ref{fig:twoassortments:plots} and \ref{fig:speed_twoassortments}, we present the results of the numerical experiments conducted as described above.  In Figure~\ref{fig:twoassortments:plots}, we present the results of these numerical experiments in the case of $n = 10$ products over 1000 randomly-generated problem instances and for $K  = 10$. This figure compares the expected revenue under the best past assortment,   $\max_{m \in \mathcal{M}} r^\intercal v_m$, to the predicted expected revenue of the new assortment under the worst-case ranking-based choice model that is consistent with the historical sales data, $\min_{\lambda \in \mathcal{U}} \mathscr{R}^{{\lambda}}(S')$. These results, in addition to the example from \S\ref{sec:question1}, establish that there can exist assortments that satisfy line~\eqref{line:improvement}  in the case of $M = 2$; indeed, we observe from Figure~\ref{fig:twoassortments:plots} that there are problem instances for which the predicted expected revenue of the new assortment obtained by our algorithm from \S\ref{sec:algorithm:twoassortments}  is strictly greater than the expected revenue of the best past assortment under all ranking-based choice model that are consistent with the historical  sales data.  In Figure~\ref{fig:speed_twoassortments}, we show the average computation times for our algorithm from \S\ref{sec:algorithm:twoassortments} on problem instances in which the number of products is varied across $n \in \{10,12,14,\ldots,100\}$ and $K = 1000$. Here, we use the larger $K = 1000$ to ensure that the sparsity of the randomly-generated base parameters $\lambda^*$ does not  introduce any biases on the resulting computation times of our algorithm.  The results in Figure~\ref{fig:speed_twoassortments} show that the computation time of the algorithm remains under 30 seconds even when there are one hundred products. This is viewed as promising from a practical perspective, as it shows that  a general algorithm for solving the robust optimization problem~\eqref{prob:robust} in the case of $M = 2$ can scale to problem instances with  numbers of products that are realistic in applications such as brick-and-mortar retail.

\subsection{Numerical Experiments for Nested Past  Assortments}\label{appx:numerical:nested}

In this appendix, we show that the compact mixed-integer optimization problem from Theorem~\ref{thm:nested} can be solved in reasonable computation times for robust optimization problems~\eqref{prob:robust}  with  up to $M=20$ nested past assortments. 

To show this, we conduct numerical experiments on randomly-generated problem instances with nested past assortments. Specifically, in each problem instance, we generate a sparse base parameter vector $\lambda^*$ by  drawing a subset of rankings $\Sigma' \subseteq \Sigma$ of length $| \Sigma'| = 80$  uniformly at random, and then choosing the values of the nonzero base parameters  $\{ \lambda_\sigma^*: \sigma \in \Sigma'\}$   by drawing uniformly over the $80$-dimensional probability simplex. The revenues of the products are drawn from the distribution $r_1,\ldots,r_n \sim \textnormal{Uniform}\{1,\ldots,10000\}$, and we sort the products such that $r_1 < \cdots < r_n$. To construct a random collection of nested past assortments of length $M$, we draw a  permutation of products $\sigma: \mathcal{N} \to \mathcal{N}$ uniformly at random, draw a subset of distinct integers $1 \le q_1 < \cdots < q_{M-1} \le n-1$ uniformly at random,  and define the past assortments as $S_m \triangleq \{0 \} \cup  \{\sigma(i): i \in \{1,\ldots,q_m \}\}$ for each $m \in \{1,\ldots,M-1\}$ and $S_M \triangleq \mathcal{N}_0$.  We generate the historical sales data for each of the past assortments under the ranking-based choice model with the base parameters $\lambda^*$, and  we then use the algorithm from Theorem~\ref{thm:nested} to solve the robust optimization problem~\eqref{prob:robust} in the case of $\eta = 0$.

\begin{figure}[t]
\centering
\FIGURE{
\includegraphics[width=0.6\linewidth]{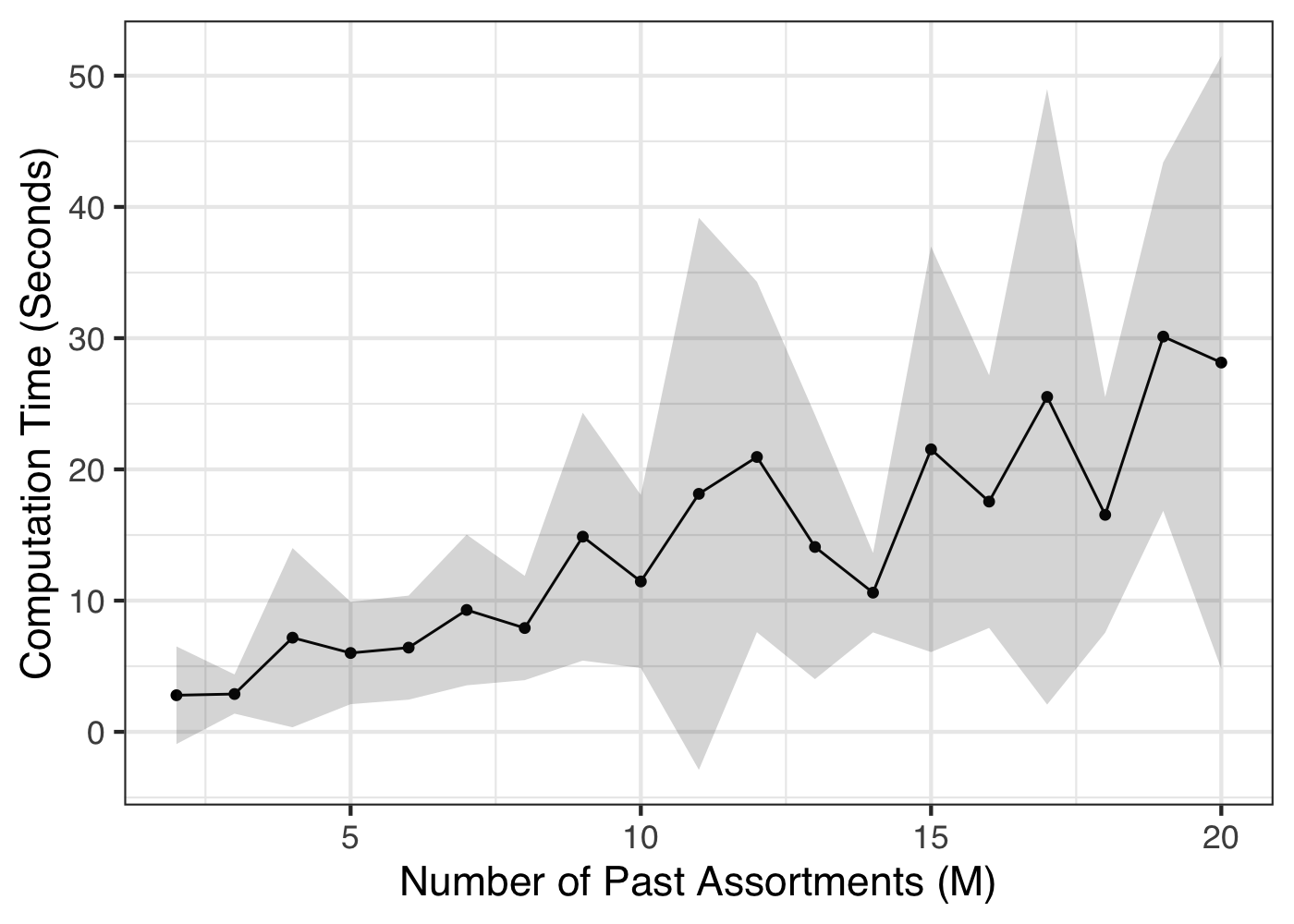}
}%
{Computation time for solving mixed-integer optimization problem from Theorem~\ref{thm:nested} with $n = 20$ products and $M \in \{2,\ldots,20\}$ nested past assortments.  \label{fig:nested:time}}%
{ Results are averaged over 10 replications for each $M \in \{2,\ldots,20\}$.}
\end{figure}

 In Figure~\ref{fig:nested:time}, we present the computation times of our numerical experiments averaged over ten replications for each selection of $M \in \{2,\ldots,20\}$ and $n = 20$.  To appreciate the computation times in Figure~\ref{fig:nested:time}, let us recall the algorithm from \S\ref{sec:algorithms:fixed_dim} for solving the robust optimization problem~\eqref{prob:robust}.  In the case of $M = 20$  past assortments, it follows from Lemma~\ref{lem:S_hat_for_S_tilde} that our algorithm from \S\ref{sec:algorithms:fixed_dim} would require iterating over as many as $2^{20-2} \approx 260000 $ assortments, and each iteration would have required solving a linear optimization problem with more than $2^{20-1} \approx 520000$ decision variables (see Corollary~\ref{cor:reform_of_L} and Proposition~\ref{prop:reform_wc}). In contrast, we observe from Figure~\ref{fig:nested:time} that the computation time of our mixed-integer optimization problem remained below 30 seconds on average. More generally, we conclude from Figure~\ref{fig:nested:time} that our mixed-integer optimization reformulation of the robust optimization problem~\eqref{prob:robust} from Theorem~\ref{thm:nested} can scale to applications with numbers of products and past assortments that are realistic in applications such as brick-and-mortar retail.

 \subsection{Numerical Experiments for Optimistic Optimization} \label{appx:numerical:optimistic}
In this appendix, we present numerical experiments that showcase the managerial value of  using the optimistic optimization problem~\eqref{prob:optimistic} from \S\ref{sec:question4:bestcase} to compute upper bounds on the maximum increase in expected revenue that a firm can hope to gain by experimenting with new assortments.\looseness=-1

To perform our numerical experiments,  we begin by constructing randomly-generated problem instances with nested past assortments. Specifically, in each problem instance, we generate a sparse base parameter vector $\lambda^*$ by randomly selecting a subset of rankings $\Sigma' \subseteq \Sigma$ of length $| \Sigma'| = 100$ and assigning  the  parameters  $\{ \lambda_\sigma: \sigma \in \Sigma'\}$   by drawing uniformly over the $100$-dimensional probability simplex.  Using this base choice for the parameters, we then generate historical sales data for each of the past assortments under the ranking-based choice model with the base parameters $\lambda^*$, after which we {forget} the base parameters $\lambda^*$. In each randomly-generated problem instance,  the revenues are drawn from the distribution $r_1,\ldots,r_n \sim \textnormal{Uniform}\{1,\ldots,10000\}$, and we sort the products such that $r_1 < \cdots < r_n$.  Finally, we use the algorithm from Theorem~\ref{thm:nested:optimistic} to compute the optimal objective value of the optimistic optimization problem~\eqref{prob:optimistic}. 

In Figure~\ref{fig:best_case}, we show the results of our experiments in cases where the collection of past assortments is the revenue-ordered assortments, the reverse revenue-ordered assortments, and neither revenue-ordered nor reverse revenue-ordered. In each of these three experiments, we randomly generate 100 problem instances and compute the optimal objective value of the optimistic optimization problem~\eqref{prob:optimistic}. Figure~\ref{fig:best_case} shows the value of $100\% \times (\max_{S \in \mathcal{S}} \max_{\lambda \in \mathcal{U}} \mathscr{R}^{\lambda}(S) - \max_{m \in \mathcal{M}} r^\intercal v_m) / ( \max_{m \in \mathcal{M}} r^\intercal v_m)$ for each problem instance and each experiment, which represents the upper bound on the potential increase in expected revenue (relative to the firm's best past assortment) that can be gained by experimentation. 

 The results from Figure~\ref{fig:best_case} show that the upper bound provided by \eqref{prob:optimistic} can provide practical value in helping firms decide whether to perform experimentation. For example, each of the three experiments shown in Figure~\ref{fig:best_case}  contain problem instances in which the optimal objective value of \eqref{prob:optimistic} is within 30\% of the expected revenue of the firm's best past assortment. 
 We further observe from Figure~\ref{fig:best_case:rev_ordered} that the optimal objective value of \eqref{prob:optimistic} is within 15\% of the expected revenue of the firm's best past assortment for nearly every problem instance when the past assortments are the revenue-ordered assortments. 
Such upper bounds observed for the problem instances in Figure~\ref{fig:best_case:rev_ordered} and many of the problem instances from Figure~\ref{fig:best_case:reverse_rev_ordered} and \ref{fig:best_case:variety} can be viewed as compelling evidence to firms that the benefits of experimentation can be limited. 
 
At first glance, the upper bounds shown in Figures~\ref{fig:best_case:reverse_rev_ordered} and \ref{fig:best_case:variety} appear to be  worse than the upper bounds shown in Figure~\ref{fig:best_case:rev_ordered}. However, the assessment of the quality of the upper bounds in Figures~\ref{fig:best_case:reverse_rev_ordered} and \ref{fig:best_case:variety} should be calibrated by the quality of the collections of past assortments used in those two subfigures. Indeed, we observe that the collections of past assortments used in Figures~\ref{fig:best_case:reverse_rev_ordered} and \ref{fig:best_case:variety} are not revenue-ordered assortments, and we have no theoretical or intuitive reasons to believe that the collections of past assortments used in Figures~\ref{fig:best_case:reverse_rev_ordered} and \ref{fig:best_case:variety} will typically contain assortments that are near-optimal with respect to the true ranking-based choice models $\lambda^*$ that generated the historical sales data. With these points in mind, the observations that the optimal objective value of \eqref{prob:optimistic} in  Figures~\ref{fig:best_case:reverse_rev_ordered} and \ref{fig:best_case:variety} are  sometimes tight (say, within 50\% or 30\% of the expected revenue of the firm's best past assortment) provides compelling evidence that \eqref{prob:optimistic} is capable of giving actionable bounds even when the firm's past assortments are not chosen strategically.

 \begin{figure}[t]
{\centering
\caption{Best-case improvement for different collections of past assortments  \label{fig:best_case}}
\subfloat[$\mathscr{M} = \left\{ \{0,10\},\{0,9,10\},\{0,8,9,10\},\ldots,\{0,1,\ldots,9,10\}\right\}$]{
\includegraphics[width=1\textwidth]{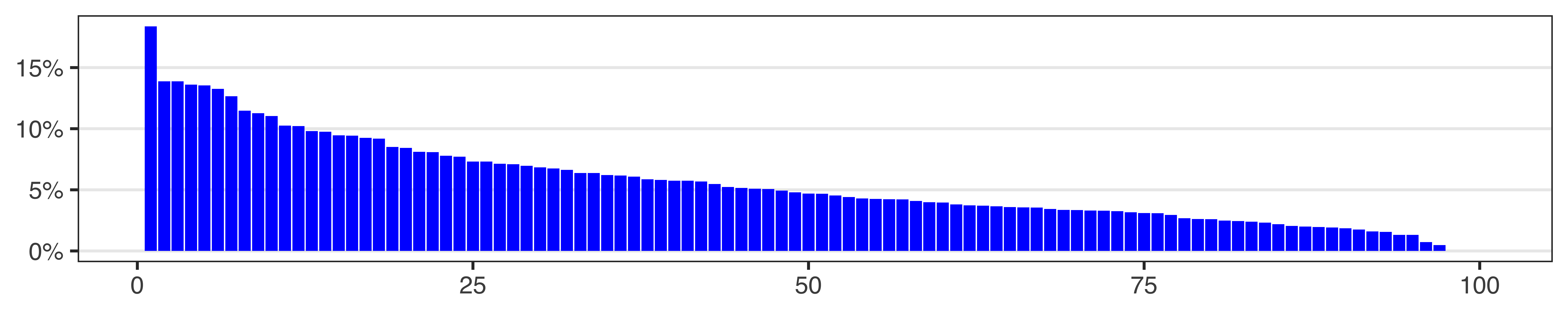}
\label{fig:best_case:rev_ordered}}

\subfloat[$\mathscr{M} = \left\{ \{0,10\},\{0,1,10\},\{0,1,2,10\},\ldots,\{0,1,\ldots,9,10\}\right\}$]{
\includegraphics[width=1\textwidth]{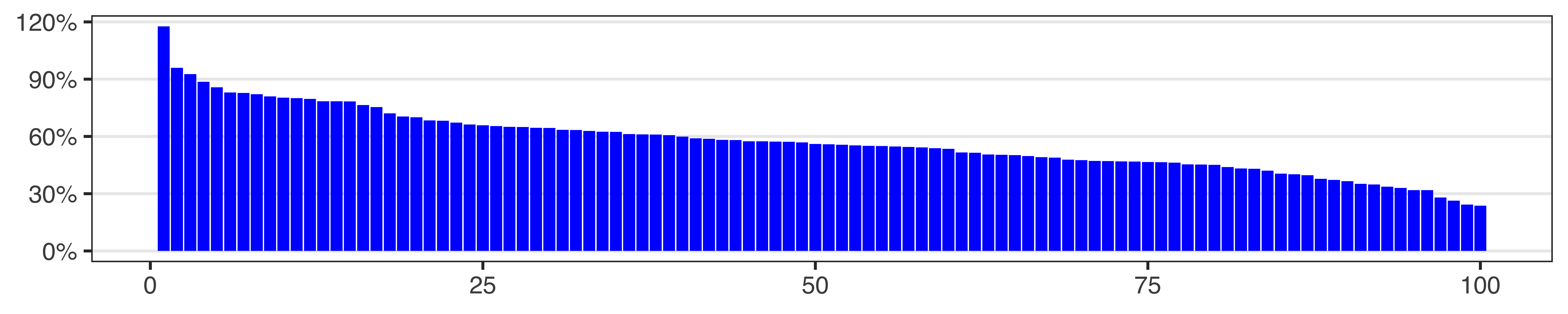}
\label{fig:best_case:reverse_rev_ordered}}

\subfloat[$\mathscr{M} = \{ \{0,3,8,13\},\{0,3,5,8,10,13,15\},\{0,1,3,5,6,8,10,11,13,15\}, \{0,1,3,4,5,6,8,9,10,11,13,14,15 \},\{0,\ldots,15\} \}$]{
\includegraphics[width=\linewidth]{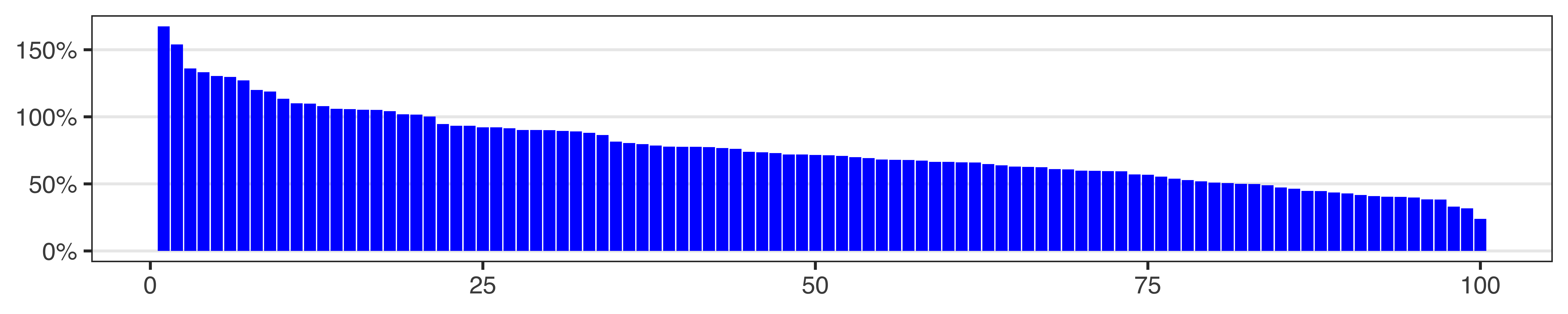}
 \label{fig:best_case:variety}
}}

\medskip

{\par\noindent  \FigureNoteStyle\HD{11}{0}{\it\FigureNoteName}\enskip \TABLEfootnotesizeIX {Each of the $x$-values in each experiment  corresponds to one of the 100 randomly-generated problem instances. For each randomly-generated problem instance, the corresponding $y$-value shown by the  blue bar is $100\% \times (\max_{S \in \mathcal{S}} \max_{\lambda \in \mathcal{U}} \mathscr{R}^{\lambda}(S) - \max_{m \in \mathcal{M}} r^\intercal v_m) / ( \max_{m \in \mathcal{M}} r^\intercal v_m)$, which is the  percentage increase of the optimal objective value of  \eqref{prob:optimistic} compared to the expected revenue of the firm's best past assortment.  The top figure shows problem instances in which the past assortments are the revenue-ordered assortments with $n = 10$ products; the middle figure shows problem instances in which the past assortments are the reverse revenue-ordered assortments with $n = 10$ products; the bottom figure shows problem instances with a collection of nested  past assortments with a complicated and more realistic structure, $n = 15$ products, and $M = 5$ past assortments.} \endgraf}
  
\end{figure}

 \subsection{Numerical Experiments for Pareto Optimization: The Exact Case} \label{appx:numerical:pareto}
In this appendix, we showcase the managerial value of the exact algorithm from \S\ref{sec:question4:pareto:exact} for solving the Pareto optimization problem~\eqref{prob:pareto} to find new assortments that can be used by firms for  low-risk experimentation.

 To perform our numerical experiments,  we begin by constructing randomly-generated problem instances with nested past assortments. Specifically, in each problem instance, we generate a sparse base parameter vector $\lambda^*$ by randomly selecting a subset of rankings $\Sigma' \subseteq \Sigma$ of length $| \Sigma'| = 80$ and assigning  the  parameters  $\{ \lambda_\sigma: \sigma \in \Sigma'\}$   by drawing uniformly over the $80$-dimensional probability simplex.  Using this base choice for the parameters, we then generate historical sales data for each of the past assortments under the ranking-based choice model with the base parameters $\lambda^*$, after which we {forget} the base parameters $\lambda^*$. In each randomly-generated problem instance,  the revenues are drawn from the distribution $r_1,\ldots,r_n \sim \textnormal{Uniform}\{1,\ldots,10000\}$, and we sort the products such that $r_1 < \cdots < r_n$.  Finally, we use the algorithm from Theorem~\ref{thm:nested} to compute the optimal objective value of the robust optimization problem~\eqref{prob:robust}, and we then use the algorithm from Theorem~\ref{thm:nested_reform_pareto} to solve the optimization problem~\eqref{prob:pareto} for each parameter $\theta \in \{q \times \eqref{prob:robust}: q \in \{0, 0.01,\ldots,0.99,1.00\}\}$ to obtain a new assortment, denoted by $S'$.

\begin{figure}[t]
{\centering
\caption{Pareto frontiers of new assortments from different collections of past assortments  \label{fig:nested}}
\subfloat[$\mathscr{M} = \left\{ \{0,10\},\{0,9,10\},\{0,8,9,10\},\ldots,\{0,1,\ldots,9,10\}\right\}$]{
\includegraphics[width=1\textwidth]{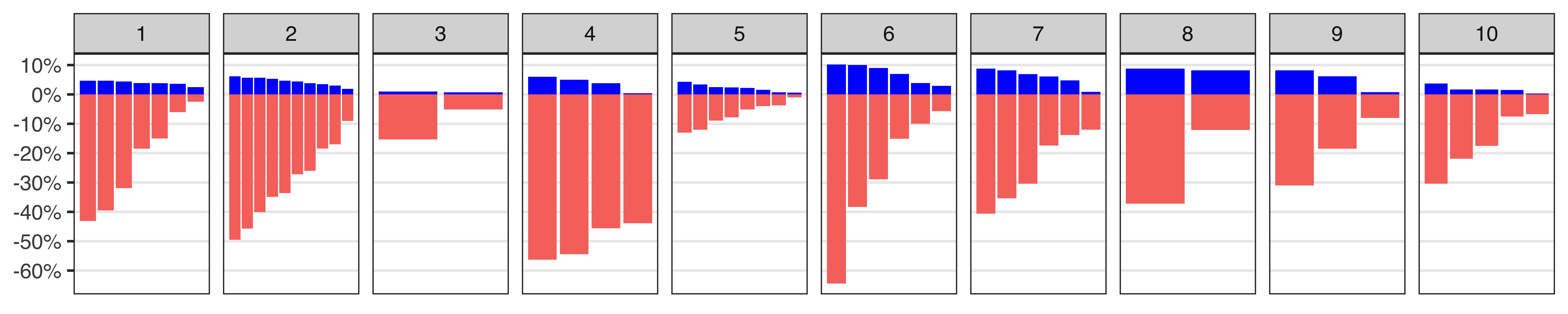}
\label{fig:nested:rev_ordered}}

\subfloat[$\mathscr{M} = \left\{ \{0,10\},\{0,1,10\},\{0,1,2,10\},\ldots,\{0,1,\ldots,9,10\}\right\}$]{
\includegraphics[width=1\textwidth]{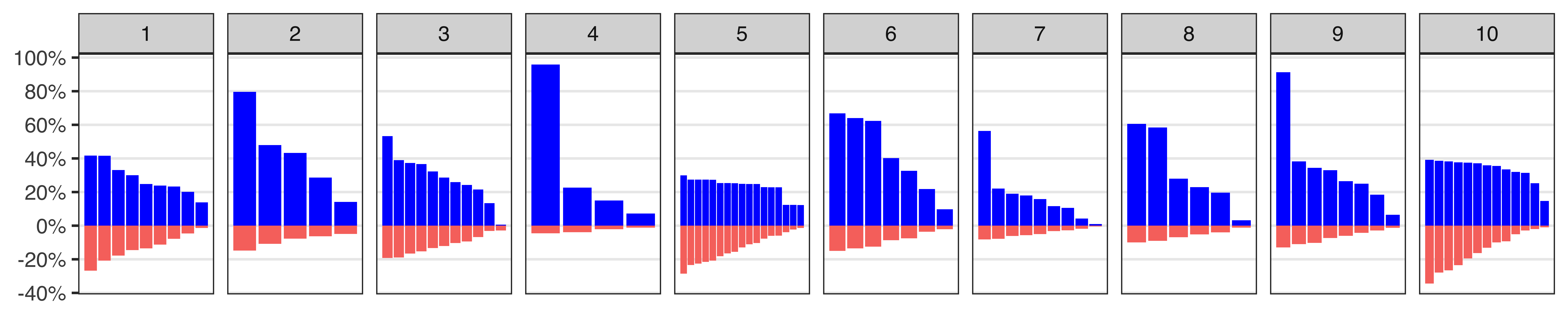}
\label{fig:nested:reverse_rev_ordered}}

\subfloat[$\mathscr{M} = \{ \{0,3,8,13\},\{0,3,5,8,10,13,15\},\{0,1,3,5,6,8,10,11,13,15\}, \{0,1,3,4,5,6,8,9,10,11,13,14,15 \},\{0,\ldots,15\} \}$]{
\includegraphics[width=\linewidth]{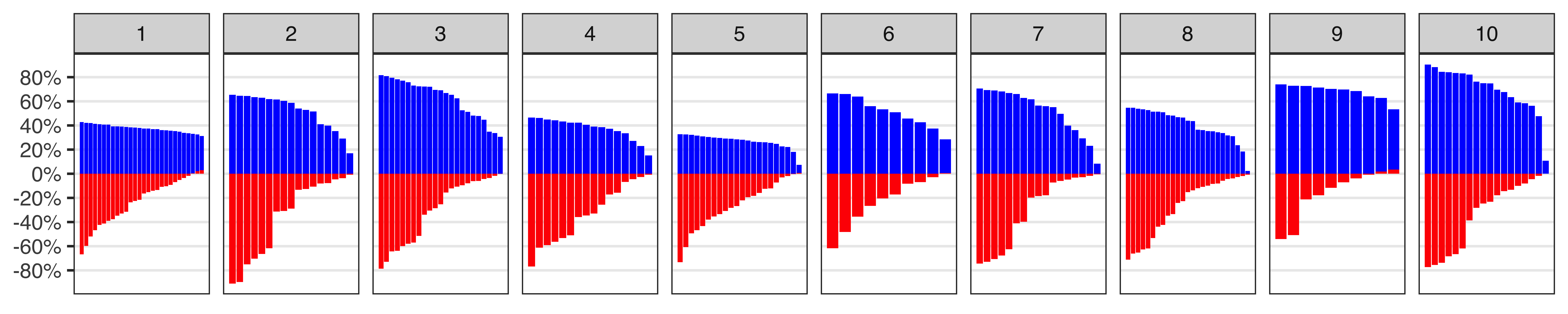}
 \label{fig:nested:variety}
}
}

\medskip

{\par\noindent  \FigureNoteStyle\HD{11}{0}{\it\FigureNoteName}\enskip \TABLEfootnotesizeIX {Each of the $x$-values in each problem instance corresponds to one of the \emph{unique} assortments} obtained from solving the optimization problem~\eqref{prob:pareto} across the choices of parameters $\theta \in \{q \times \eqref{prob:robust}: q \in \{0, 0.01,\ldots,0.99,1.00\}\}$. The corresponding interval of $y$-values formed by the red and blue bars is the interval $[\Pi_\lambda: \lambda \in \mathcal{U}]$,  where $\Pi_\lambda \triangleq 100\% \times (\mathscr{R}^{\lambda}(S') - \max_{m \in \mathcal{M}} r^\intercal v_m) / ( \max_{m \in \mathcal{M}} r^\intercal v_m)$ is the relative percentage improvement of the expected revenue of the new assortment over the expected revenue of the firm's best past assortment for a given ranking-based choice model $\lambda \in \mathcal{U}$.  The top figure shows problem instances in which the past assortments are the revenue-ordered assortments with $n = 10$ products; the middle figure shows problem instances in which the past assortments are the reverse revenue-ordered assortments with $n = 10$ products; the bottom figure shows problem instances with a collection of nested  past assortments with a complicated and more realistic structure, $n = 15$ products, and $M = 5$ past assortments. \endgraf}
  
\end{figure}

In Figure~\ref{fig:nested}, we show the results of our experiments in cases where the collection of past assortments is the revenue-ordered assortments, the reverse revenue-ordered assortments, and neither revenue-ordered nor reverse revenue-ordered. In each of these experiments, we visualize the Pareto frontier of new assortments generated under ten randomly-generated instances of the specified collections of past assortments. Compared to new assortments found using estimate-then-optimize,  the Pareto frontier of new assortments found from solving \eqref{prob:pareto}  enables firms to select a new assortment that maximizes their best-case improvement in expected revenue for any guaranteed level of risk that a firm is willing to incur. This distinction is particularly salient in the case where the past assortments are the revenue-ordered assortments (top of Figure~\ref{fig:nested}). Indeed, we recall from \S\ref{sec:question1} and Appendix~\ref{appx:riskseto} that the assortments found using estimate-then-optimize can have significant declines in expected revenue compared to the best expected revenue from the firm's past assortments.  In contrast, we observe from Figure~\ref{fig:nested}a that nine out of the ten problem instances yielded new assortments that firms can offer to their customers that, under any ranking-based choice model that is consistent with the firm's historical sales data, will cause the firm's expected revenue to decline by  no more than 12\%. Hence, the results from Figure~\ref{fig:nested} show that our algorithms can offer firms a low-risk way to experiment with the assortments that they offer to their customers.

 \subsection{Numerical Experiments for Pareto Optimization: The Approximate Case} \label{appx:numerical:pareto:approximate}
In this appendix, we showcase the managerial value of the heuristic from \S\ref{sec:question4:pareto:approximate} for the Pareto optimization problem~\eqref{prob:pareto} to find new assortments that can be used by firms for  low-risk experimentation.

Our experiments in this appendix are based on a real conjoint dataset involving laptop bags sold by the firm  Timbuk2. This conjoint dataset was first described by \cite{toubia2003fast} and has subsequently been studied in the context of assortment optimization by \cite{belloni2008optimizing} and \cite{bertsimas2019exact}. The dataset from Timbuk2\footnote{We accessed the publicly available dataset from  \url{https://github.com/vvmisic/optimalPLD}.} consists of $3584$ variations of laptop bags, and the goal of the firm in the assortment optimization context to select a subset of the laptop bag variations to that maximizes the firm's expected revenue. The conjoint study performed by  \cite{toubia2003fast} included $330$ respondents, each of whom conveyed their preferences for different variations of the products. Based on their responses, \cite{toubia2003fast} estimated the utilities for each of the variations of laptop bags for each respondent, which was then converted into rankings at various price points by \cite{belloni2008optimizing}. The resulting problem faced in this example is an assortment optimization problem in which the number of products is $n = 3594$, the ranking-based choice model satisfies $\|\lambda^*\|_0 = 330$, and the weights of the rankings corresponding to each of the respondents satisfies $\lambda^*_\sigma = \frac{1}{330}$.

We take several steps to convert the real data from the conjoint study into the setting studied in this paper.  We begin by assuming that  true ranking-based choice model $\lambda^*$ is unknown, and that the only information available to the firm about the true choice model comes from historical sales data generated by $M \in \{3,4,5\}$ past assortments. We also restrict the universe of products from 3594 products to $15$ products by drawing the subset uniformly at random and rejecting any subset that contains products with the identical revenues. We then randomly generate the  past assortments $S_1,\ldots,S_M \in \mathcal{S}$, whereby the assortments satisfy $\{0,n\} \subseteq S_1 \cap  \cdots \cap  S_M$ and, for each of the remaining products $j \in \{1,\ldots,n-1\}$, we randomly assign the product to the past assortments $I \subseteq \mathcal{M}$ with distribution given by 
\begin{align*}
P \left(j \in \bigcap_{m \in I}  S_m \setminus \bigcup_{m \notin I} S_m \right) = \begin{cases}
\frac{1}{2^M-1},&\text{if } | I | \ge 1,\\
0,&\text{otherwise}. 
\end{cases}
\end{align*}
The above approach to constructing past assortments is a generalization of the approach used in Appendix~\ref{sec:twoassortments:numerics}, and our approach for constructing past assortments ensures that each product appears in at least one past assortment. Finally, we use the heuristic from \S\ref{sec:question4:pareto:approximate} to compute  the approximate Pareto frontier of new assortments defined on line~\eqref{line:paretofrontier}. We repeat this entire procedure ten times for each $M \in \{3,4,5\}$.

In Figure~\ref{fig:conjoint}, we show the results of our experiments with $M \in \{3,4,5\}$ past assortments.  
 The results of Figure~\ref{fig:conjoint} show that the approximate Pareto frontier often contains new assortments that have high best-case expected revenue and high true expected revenue  relative to the  expected revenue of the firm's best past assortment. Moreover, the results from Figure~\ref{fig:conjoint} indicate that the assortments from the approximate Pareto frontier with the highest true expected revenue (i.e., the assortments with the highest black dots) are often the assortments from the approximate Pareto frontier with worst-case expected revenue that is rather close (e.g., within 10\%) of the expected revenue of the firm's best past assortment. These results suggest that the approximate Pareto frontier obtained using the heuristic from \S\ref{sec:question4:pareto:approximate} can contain new assortments that lead to  increases in the firm's expected revenue with respect to a true but unknown ranking-based choice model $\lambda^*$, while simultaneously enjoying guarantees that implementing the new assortments will not lead to meaningful worst-case decreases in the firm's expected revenue.

\begin{figure}[t]
{\centering
\caption{Approximate Pareto frontiers for Conjoint Dataset \label{fig:conjoint}}
\includegraphics[width=1\textwidth]{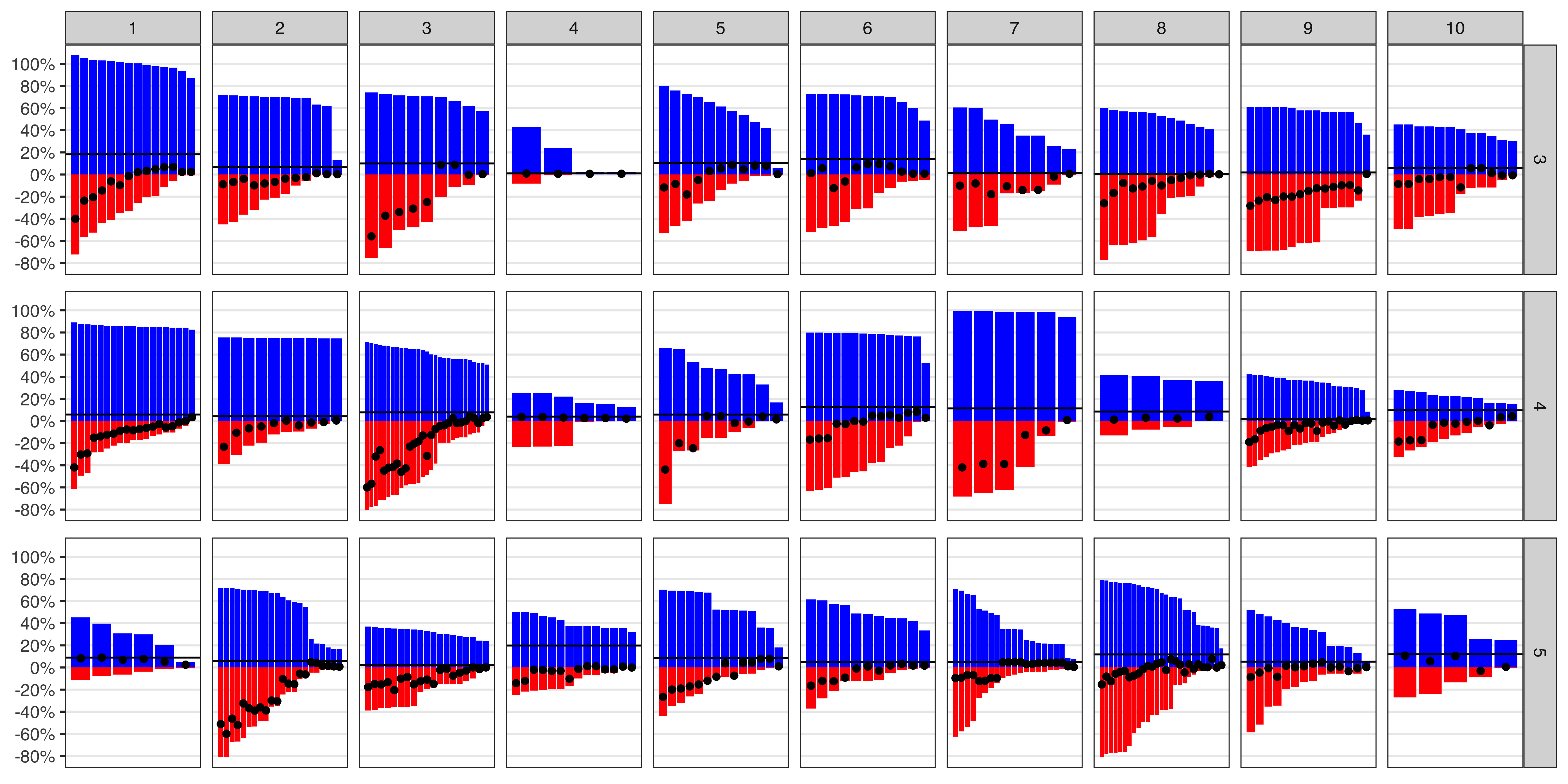}
}

\medskip

{\par\noindent  \FigureNoteStyle\HD{11}{0}{\it\FigureNoteName}\enskip \TABLEfootnotesizeIX Each row corresponds to the number of past assortments, $M \in \{3,4,5\}$. For each number of past assortments $M$, we randomly generate ten problem instances (where each problem instance consists of a randomly-generated collection of past assortments and a random selection of the $n=15$ products).  For each problem instance, the distinct $x$-values have a one-to-one correspondence with the approximate Pareto frontier of assortments defined on line~\eqref{line:paretofrontier} in \S\ref{sec:question4:pareto:approximate}. For each assortment $S'$ in the approximate Pareto frontier, the corresponding interval of $y$-values formed by the red and blue bars is the interval $[\Pi_\lambda: \lambda \in \mathcal{U}]$,  where $\Pi_\lambda \triangleq 100\% \times (\mathscr{R}^{\lambda}(S') - \max_{m \in \mathcal{M}} r^\intercal v_m) / ( \max_{m \in \mathcal{M}} r^\intercal v_m)$ is the relative percentage improvement of the expected revenue of the new assortment over the expected revenue of the firm's best past assortment for a given ranking-based choice model $\lambda \in \mathcal{U}$.  The black dot for each assortment in the approximate Pareto frontier shows $100\% \times (\mathscr{R}^{\lambda^*}(S') - \max_{m \in \mathcal{M}} r^\intercal v_m) / ( \max_{m \in \mathcal{M}} r^\intercal v_m$, i.e., the percentage improvement of the true expected revenue of assortment $S'$ with respect to the true ranking-based choice model $\lambda^*$ from the conjoint dataset over the expected revenue of the firm's best past assortment. The solid horizontal line in each problem instance shows $100\% \times (\max_{S \in \mathcal{S}} \mathscr{R}^{\lambda^*}(S) - \max_{m \in \mathcal{M}} r^\intercal v_m) / ( \max_{m \in \mathcal{M}} r^\intercal v_m$, i.e., the percentage improvement of the expected revenue of the optimal assortment with respect to the true ranking-based choice model $\lambda^*$ from the conjoint dataset over the expected revenue of the firm's best past assortment.    \endgraf}
  
\end{figure}

\section{Proofs of Theorem~\ref{thm:main} from \S\ref{sec:characterization} and Technical Results from \S\ref{sec:characterization:proof}} \label{appx:characterization_proofs}
This appendix contains the proof of Theorem~\ref{thm:main} by following the proof outline described in \S\ref{sec:characterization:proof}. 
We begin in Appendix~\ref{appx:characterization_proofs:reformulation} by  presenting the proof of Proposition~\ref{prop:reform_wc} from \S\ref{sec:characterization:proof}. In Appendix~\ref{appx:graphical}, we present the graphical interpretation of Definition~\ref{defn:rho} from \S\ref{sec:characterization:proof}.  In Appendix~\ref{appx:prop2}, we use this graphical interpretation  to prove   Proposition~\ref{prop:plus_one_inequality_prop} from \S\ref{sec:characterization:proof}. We conclude in Appendix~\ref{appx:proof_main} by using Propositions~\ref{prop:reform_wc} and \ref{prop:plus_one_inequality_prop}  to prove Theorem~\ref{thm:main}.

\subsection{Proof of Proposition~\ref{prop:reform_wc}} \label{appx:characterization_proofs:reformulation}
Following the notation from \S\ref{sec:setting}, we  readily observe that the worst-case expected revenue  for the fixed assortment $S \in \mathcal{S}$ is equal to the optimal objective value of the following linear optimization 
 problem:
\begin{align}  \tag{WC-$S$} \label{prob:wc}
 &\begin{aligned}
& \; \underset{\lambda, \epsilon}{\textnormal{minimize}} && \sum_{\sigma \in \Sigma}  \sum_{i \in S} r_i  \mathbb{I} \left \{ i = \argmin_{j \in S} \sigma(j)\right \} \lambda_\sigma   \\
&\textnormal{subject to}&&   \sum_{\sigma \in \Sigma}\mathbb{I} \left \{ i = \argmin_{j \in S_m} \sigma(j)\right \} \lambda_\sigma  - \epsilon_{m,i} =  v_{m,i} \quad  \forall m \in \mathcal{M} \text{ and } i \in S_m\\
&&& \sum_{\sigma \in \Sigma} \lambda_\sigma = 1\\
&&& \| \epsilon \| \le \eta \\
&&& \lambda_{\sigma} \ge 0 \quad \forall \sigma \in \Sigma.
\end{aligned}
 \end{align}
For every assortment $S' \in \mathcal{S}$, product in the assortment $i \in S'$, and ranking $\sigma \in \Sigma$, it follows immediately from Definition~\ref{defn:D} that the equality $\mathbb{I} \left \{ i = \argmin_{j \in S'} \sigma(j) \right \} = 1$ is satisfied if and only if $\sigma \in \mathcal{D}_i(S')$. Therefore, for every assortment $S \in \mathcal{S}$, we observe that the linear optimization problem~\eqref{prob:wc} is equivalent to
\begin{align}  
 &\begin{aligned}
& \; \underset{\lambda, \epsilon}{\textnormal{minimize}} &&   \sum_{i \in S} \left(\sum_{\sigma \in \mathcal{D}_i(S)}  \lambda_\sigma \right) r_i    \\
&\textnormal{subject to}&&   \sum_{\sigma \in \mathcal{D}_i(S_m)} \lambda_\sigma  - \epsilon_{m,i} = v_{m,i}  &&  \forall m \in \mathcal{M} \text{ and } i \in S_m\\
&&&  \sum_{\sigma \in \Sigma} \lambda_\sigma = 1\\
&&& \| \epsilon \| \le \eta \\
&&& \lambda_{\sigma} \ge 0 \quad \forall \sigma \in \Sigma. 
\end{aligned}\tag{WC-$S$-1} \label{prob:wc:1}
 \end{align}
 Moreover,  we recall from Definitions~\ref{defn:D} and \ref{defn:L}  that there is a unique tuple of products $(i_1,\ldots,i_M) \in \mathcal{L}$ corresponding to each ranking $\sigma \in \Sigma$ that satisfies  $\sigma \in \cap_{m \in \mathcal{M}} \mathcal{D}_{i_m}(S_m)$. Therefore,   it follows immediately from Definitions~\ref{defn:D} and \ref{defn:L}  that the linear optimization problem~\eqref{prob:wc:1} is equivalent to 
 \begin{align}  
 &\begin{aligned}
& \; \underset{\lambda, \epsilon}{\textnormal{minimize}} &&   \sum_{(i_1,\ldots,i_M) \in \mathcal{L}} \sum_{i \in S} r_i \left( \sum_{\sigma \in  \cap_{m \in \mathcal{M}} \mathcal{D}_{i_m}(S_m) \cap \mathcal{D}_i(S)  }  \lambda_\sigma \right)    \\
&\textnormal{subject to}&&  \sum_{(i_1,\ldots,i_M) \in \mathcal{L}: \; i_m = i} \left(  \sum_{\sigma \in  \cap_{m \in \mathcal{M}} \mathcal{D}_{i_m}(S_m) } \lambda_\sigma \right) - \epsilon_{m,i} = v_{m,i}  &&  \forall m \in \mathcal{M} \text{ and } i \in S_m\\
&&& \sum_{(i_1,\ldots,i_M) \in \mathcal{L}}\left(  \sum_{\sigma \in \cap_{m \in \mathcal{M}} \mathcal{D}_{i_m}(S_m)} \lambda_\sigma \right) = 1\\
&&& \| \epsilon \| \le \eta \\
&&& \lambda_{\sigma} \ge 0 \quad \forall \sigma \in \Sigma.
\end{aligned}\tag{WC-$S$-2} \label{prob:wc:2}
 \end{align}
We now simplify the linear optimization problem~\eqref{prob:wc:2} by performing a transformation on its decision variables. Specifically, we transform \eqref{prob:wc:2} by creating the following new decision variables for each tuple of products $(i_1,\ldots,i_M) \in \mathcal{L}$ and each product $i \in S$: 
\begin{align*}
\lambda_{i_1 \cdots i_M} &\leftarrow \sum_{\sigma \in \cap_{m \in \mathcal{M}} \mathcal{D}_{i_m}(S_m)} \lambda_\sigma; &  \omega_{i,i_1 \cdots i_M} &\leftarrow \sum_{\sigma \in \cap_{m \in \mathcal{M}} \mathcal{D}_{i_m}(S_m) \cap \mathcal{D}_i(S)} \lambda_\sigma. 
 \end{align*}
Let us make three observations about the new decision variables defined above. First, we observe immediately from the definition of the new decision variables $\omega_{i,i_1 \cdots i_M}$ that the equality $\omega_{i,i_1 \cdots i_M}  = 0$ must hold for each tuple of products $(i_1,\ldots,i_M) \in \mathcal{L}$ and product $i \in S$ that satisfy $\cap_{m \in \mathcal{M}} \mathcal{D}_{i_m}(S_m) \cap \mathcal{D}_i(S) = \emptyset$. Second, we observe that the nonnegativity constraints $\lambda_\sigma \ge 0$ in the linear optimization problem~\eqref{prob:wc:2} imply that the new  decision variables must also satisfy the nonnegativity constraints $\lambda_{i_1 \cdots i_M} \ge 0$ and $\omega_{i, i_1 \cdots i_M} \ge 0$. Third, we observe from the  definitions of the new decision variables that the equality  $\sum_{i \in S} \omega_{i,i_1 \cdots i_M}  = \lambda_{i_1 \cdots i_M}$ must hold for every tuple of products $(i_1,\ldots,i_M) \in \mathcal{L}$.  Based on the definitions of the new decision variables and by applying the aforementioned three observations, we  have shown that the linear optimization problem~\eqref{prob:wc:2} is equivalent to
   \begin{align}  &\begin{aligned}
& \; \underset{\lambda, \epsilon, \omega}{\textnormal{minimize}} &&   \sum_{(i_1,\ldots,i_M) \in \mathcal{L}}  \sum_{i \in S} r_i \omega_{i,i_1 \cdots i_M} \\
&\textnormal{subject to}&&
\begin{aligned}[t]
& \sum_{i \in S} \omega_{i,i_1 \cdots i_M}  = \lambda_{i_1 \cdots i_M} && \forall (i_1,\ldots,i_M) \in \mathcal{L}\\
 & \sum_{(i_1,\ldots,i_M) \in \mathcal{L}: \; i_m = i}  \lambda_{i_1 \cdots i_M}  - \epsilon_{m,i} = v_{m,i}  && \forall m \in \mathcal{M} \text{ and } i \in S_m\\
& \sum_{(i_1,\ldots,i_M) \in \mathcal{L}} \lambda_{i_1 \cdots i_M} = 1\\
\end{aligned}\\
&&&\begin{aligned}[t]
&\| \epsilon \| \le \eta \\
& \lambda_{i_1 \cdots i_M} \ge 0  && \forall (i_1,\ldots,i_M) \in \mathcal{L}\\
 & \omega_{i,i_1 \cdots i_M} \ge 0 &&  \forall i \in S \text{ and } (i_1,\ldots,i_M) \in \mathcal{L}\\
 & \omega_{i,i_1 \cdots i_M} = 0 &&  \forall i \in S   \text{ and } (i_1,\ldots,i_M) \in \mathcal{L} \textnormal{ such that } \cap_{m \in \mathcal{M}} \mathcal{D}_{i_m}(S_m) \cap \mathcal{D}_i(S) = \emptyset.
\end{aligned}
\end{aligned}\tag{WC-$S$-3} \label{prob:wc:3}
 \end{align}
After rearranging terms, we observe that the linear optimization problem~\eqref{prob:wc:3} is equivalent to
   \begin{align} 
 &\begin{aligned}
& \; \underset{\lambda, \epsilon}{\textnormal{minimize}} &&   \sum_{(i_1,\ldots,i_M) \in \mathcal{L}} \left[  \begin{aligned}
&\underset{\omega 
}{\textnormal{minimize}}&& \sum_{i \in S} r_i \omega_{i,i_1 \cdots i_M}\\
&\textnormal{subject to}&& \sum_{i \in S} \omega_{i,i_1 \cdots i_M}  = \lambda_{i_1 \cdots i_M}\\
&&& \omega_{i,i_1 \cdots i_M} \ge 0 \quad \forall i \in S\\
 &&& \omega_{i,i_1 \cdots i_M} = 0  \quad \forall i \in S   \textnormal{ such that}  \cap_{m \in \mathcal{M}} \mathcal{D}_{i_m}(S_m) \cap \mathcal{D}_i(S) = \emptyset
\end{aligned}\right]   \\
&\textnormal{subject to}&& \begin{aligned}[t]
& \sum_{(i_1,\ldots,i_M) \in \mathcal{L}: \; i_m = i}  \lambda_{i_1 \cdots i_M}  - \epsilon_{m,i} = v_{m,i}  \quad  \forall m \in \mathcal{M} \text{ and } i \in S_m\\
& \sum_{(i_1,\ldots,i_M) \in \mathcal{L}} \lambda_{i_1 \cdots i_M} = 1\\
& \| \epsilon \| \le \eta \\
& \lambda_{i_1 \cdots i_M} \ge 0 \quad \forall (i_1,\ldots,i_M) \in \mathcal{L}.
\end{aligned}
\end{aligned}\tag{WC-$S$-4} \label{prob:wc:4}
 \end{align}
Finally, we readily observe from Definition~\ref{defn:rho} that the following equality holds for every $(i_1,\ldots,i_M) \in \mathcal{L}$:
\begin{align*}
 \left[  \begin{aligned}
&\underset{\omega 
}{\textnormal{minimize}}&& \sum_{i \in S} r_i \omega_{i,i_1 \cdots i_M}\\
&\textnormal{subject to}&& \sum_{i \in S} \omega_{i,i_1 \cdots i_M}  = \lambda_{i_1 \cdots i_M}\\
&&& \omega_{i,i_1 \cdots i_M} \ge 0 \quad \forall i \in S\\
 &&& \omega_{i,i_1 \cdots i_M} = 0  \quad \forall i \in S   \textnormal{ such that}  \cap_{m \in \mathcal{M}} \mathcal{D}_{i_m}(S_m) \cap \mathcal{D}_i(S) = \emptyset
\end{aligned}\right]  =  \rho_{i_1 \cdots i_M}(S) \lambda_{i_1\cdots i_M}.
\end{align*}
It follows from the above equality that \eqref{prob:wc:4} is equivalent to the linear optimization problem~\eqref{prob:robust_simplified}, which concludes our proof of Proposition~\ref{prop:reform_wc}. \halmos 

\subsection{A Graphical Interpretation of Definition~\ref{defn:rho}} 
\label{appx:graphical}
In view of our overarching strategy for the proof of Proposition~\ref{prop:plus_one_inequality_prop} that is described at the end of \S\ref{sec:characterization}, 
we now proceed to analyze the behavior of the functions  $S \mapsto \rho_{i_1 \cdots i_M}(S)$.  In particular, we will show in the rest of Appendix~\ref{appx:graphical} that 
$\rho_{i_1 \cdots i_M}(S)$ can be computed by analyzing  the {reachability} of vertices in a directed acyclic graph. In the subsequent  Appendix~\ref{appx:prop2}, we will use this graphical interpretation of Definition~\ref{defn:rho} to show for every arbitrary assortment $S \in \mathcal{S}$ that we can construct  an assortment ${S}' \in \widehat{\mathcal{S}}$ that   satisfies $\rho_{i_1 \cdots i_M}(S) \le \rho_{i_1 \cdots i_M}(S')$ for all tuples of products $(i_1,\ldots,i_M) \in \mathcal{L}$, thereby completing the proof of Theorem~\ref{thm:main}.

  To develop our alternative representation of $\rho_{i_1 \cdots i_M}(S)$, we begin  by showing that the set of tuples of products $\mathcal{L}$ can be interpreted as a set of {directed acyclic graphs}. 
Indeed, consider any selection of products from each of the past assortments,  $(i_1,\ldots,i_M) \in S_1 \times \cdots \times S_M$. From this tuple of products, we will construct a directed graph, denoted by $\mathcal{G}_{i_1 \cdots i_M}$, in which the set of vertices in the graph is equal to $\mathcal{N}_0$, and the graph has a directed edge $(i,i_m)$ from vertex $i$ to vertex $i_m$  for each past assortment $m \in \mathcal{M}$ and each product $i \in S_m \setminus \{i_m\}$. 
In Figure~\ref{fig:exhaustive}, we present visualizations of the directed graphs generated by this construction procedure. In the first intermediary result of this subsection, presented below as Lemma~\ref{lem:dag}, we show that the tuple of products $(i_1,\ldots,i_M)$  is an element of $\mathcal{L}$ if and only if the directed graph $\mathcal{G}_{i_1 \cdots i_M}$ is acyclic.

\begin{lemma} \label{lem:dag}
$(i_1,\ldots,i_M) \in \mathcal{L}$ if and only if $\mathcal{G}_{i_1 \cdots i_M}$ is acyclic.
\end{lemma}
\begin{proof}{Proof.}
Consider any tuple of products $(i_1,\ldots,i_M) \in S_1 \times \cdots \times S_M$. 
First, we recall the fundamental result in graph theory that a directed graph is {acyclic} if and only if there exists a ranking $\sigma \in \Sigma$ that satisfies $\sigma(i) < \sigma(j)$ for every directed edge $(j,i)$ in the graph \cite[p. 77]{ahuja1988network}. 
Second, we recall from our construction of $\mathcal{G}_{i_1 \cdots i_M}$ that the directed edges in $\mathcal{G}_{i_1 \cdots i_M}$ are  exactly those of the form $(j,i_m)$ for each past assortment $m \in \mathcal{M}$ and product $j \in S_m \setminus \{i_m\}$. Combining these two recollections, we see that the directed graph $\mathcal{G}_{i_1 \cdots i_M}$ is acyclic if and only if there exists a ranking $\sigma \in \Sigma$ that satisfies $\sigma(i_m) < \sigma(j)$ for each $m \in \mathcal{M}$ and $j \in S_m \setminus \{i_m\}$. Since
\begin{align*}
\left[ \sigma(i_m) < \sigma(j) \textnormal{ for all }m \in \mathcal{M} \textnormal{ and }j \in S_m \setminus \{i_m\}\right] &\iff \left[i_m = \argmin_{j \in S_m} \sigma(j) \textnormal{ for all }m \in \mathcal{M} \right]\\
&\iff \left[ \sigma \in \cap_{m \in \mathcal{M}} \mathcal{D}_{i_m}(S_m) \right],
\end{align*}
we conclude that $\mathcal{G}_{i_1 \cdots i_M}$ is acyclic if and only if $(i_1,\ldots,i_M) \in \mathcal{L}$. 
\halmos  \end{proof}

 \begin{figure}[t]
\centering
\FIGURE{
\begin{minipage}{\linewidth}
\centering
\vspace{0.5em}
\subfloat[$(i_1,i_2,i_3) = (0,0,0)$]{%
\begin{tikzpicture}
\tikzset{vertex/.style = {shape=circle,draw,minimum size=1.5em}}
\tikzset{edge/.style = {->,> = latex'}}
\node[vertex] (b) at  (0,0) {1};
\node[vertex] (a) at  (1.4,2.8) {0};
\node[vertex] (c) at  (2.8,0) {2};
\draw [-{Stealth[scale=1.25]}] (b) -- (a); 
\draw [-{Stealth[scale=1.25]}] (c) -- (a); 
\end{tikzpicture}%
}\qquad\quad
\subfloat[$(i_1,i_2,i_3) = (1,1,0)$]{%
\begin{tikzpicture}
\tikzset{vertex/.style = {shape=circle,draw,minimum size=1.5em}}
\tikzset{edge/.style = {->,> = latex'}}
\node[vertex] (b) at  (0,0) {1};
\node[vertex] (a) at  (1.4,2.8) {0};
\node[vertex] (c) at  (2.8,0) {2};
\draw [-{Stealth[scale=1.25]}] (a) -- (b); 
\draw [-{Stealth[scale=1.25]}] (c) -- (b); 
\draw [-{Stealth[scale=1.25]}] (c) -- (a); 
\end{tikzpicture}
}\qquad \quad
\subfloat[$(i_1,i_2,i_3) = (1,1,2)$]{%
\begin{tikzpicture}
\tikzset{vertex/.style = {shape=circle,draw,minimum size=1.5em}}
\tikzset{edge/.style = {->,> = latex'}}
\node[vertex] (b) at  (0,0) {1};
\node[vertex] (a) at  (1.4,2.8) {0};
\node[vertex] (c) at  (2.8,0) {2};
\draw [-{Stealth[scale=1.25]}] (a) -- (b); 
\draw [-{Stealth[scale=1.25]}] (c) -- (b); 
\draw [-{Stealth[scale=1.25]}] (a) -- (c); 
\end{tikzpicture}
}\\
\subfloat[$(i_1,i_2,i_3) = (2,0,2)$]{%
\begin{tikzpicture}
\tikzset{vertex/.style = {shape=circle,draw,minimum size=1.5em}}
\tikzset{edge/.style = {->,> = latex'}}
\node[vertex] (b) at  (0,0) {1};
\node[vertex] (a) at  (1.4,2.8) {0};
\node[vertex] (c) at  (2.8,0) {2};
\draw [-{Stealth[scale=1.25]}] (a) -- (c); 
\draw [-{Stealth[scale=1.25]}] (b) -- (c); 
\draw [-{Stealth[scale=1.25]}] (b) -- (a); 
\end{tikzpicture}
}\qquad  \quad
\subfloat[$(i_1,i_2,i_3) = (2,1,2)$]{%
\begin{tikzpicture}
\tikzset{vertex/.style = {shape=circle,draw,minimum size=1.5em}}
\tikzset{edge/.style = {->,> = latex'}}
\node[vertex] (b) at  (0,0) {1};
\node[vertex] (a) at  (1.4,2.8) {0};
\node[vertex] (c) at  (2.8,0) {2};
\draw [-{Stealth[scale=1.25]}] (a) -- (c); 
\draw [-{Stealth[scale=1.25]}] (b) -- (c); 
\draw [-{Stealth[scale=1.25]}] (a) -- (b); 
\end{tikzpicture}
}
\vspace{0.5em}
\end{minipage}}
{Visualizations of directed graphs $\mathcal{G}_{i_1 \cdots i_M}$ corresponding to tuples of products $(i_1,\ldots,i_M) \in \mathcal{L}$.\label{fig:exhaustive}}
{\TABLEfootnotesizeIX Each of the five figures presents a visualization of the directed graph $\mathcal{G}_{i_1 i_2 i_3}$ corresponding to a tuple of products $(i_1,i_2,i_3) \in \mathcal{L}$ in the case where there are $M=3$ past assortments of the form $S_1 = \{0,1,2\}$, $S_2 = \{0,1\}$, and $S_3 = \{0,2\}$. We observe that there exists an incoming edge to vertex $i$ if and only if there exists a past assortment $m \in \{1,2,3\}$ that satisfies $i = i_m$.} 
\end{figure}

We next introduce the definition of \emph{reachability} for vertices in the directed graph $\mathcal{G}_{i_1 \cdots i_M}$. 
 In particular, the following definition is standard in the study of directed graphs, and we will make extensive use of Definition~\ref{def:reachable} throughout the  appendices of this paper. 
\begin{definition}  \label{def:reachable}
Let $(i_1,\ldots,i_M) \in S_1 \times \cdots \times S_M$ and $i, j \in \mathcal{N}_0$. We say that  vertex $j$ is \emph{reachable} from vertex $i$, denoted by $j \prec_{i_1 \cdots i_M} i$, if there exists a directed path in $\mathcal{G}_{i_1 \cdots i_M}$ from $i$ to $j$.
\end{definition}
 We adopt the  convention throughout this paper that a directed path must contain at least one directed edge. Hence, we observe that if the  directed graph $\mathcal{G}_{i_1 \cdots i_M}$ is acyclic, then it must be the case that a vertex may never be reachable from itself, that is, $i \nprec_{i_1 \cdots i_M} i$  for  all $i \in \mathcal{N}_0$.  

We now use Definition~\ref{def:reachable} to develop several intermediary results regarding the structure of the directed acyclic graph $\mathcal{G}_{i_1 \cdots i_M}$ for each tuple of products  $(i_1,\ldots,i_M) \in \mathcal{L}$. We begin with a simple intermediary result, denoted below by Lemma~\ref{lem:reachable_im}, in which we characterize the vertices in the directed acyclic graph  $\mathcal{G}_{i_1 \cdots i_M}$ that can be reachable from other vertices. 
\begin{lemma}\label{lem:reachable_im}
Let $(i_1,\ldots,i_M) \in \mathcal{L}$ and $i,j \in \mathcal{N}_0$. If $j \prec_{i_1 \cdots i_M} i$, then there exists a past assortment $m \in \mathcal{M}$ that satisfies $j = i_m$.
\end{lemma}
\begin{proof}{Proof.}
Consider any tuple of products $(i_1,\ldots,i_M) \in \mathcal{L}$. 
We recall from our construction of $\mathcal{G}_{i_1 \cdots i_M}$ that  $i_1,\ldots,i_M$ are the only vertices in $\mathcal{G}_{i_1 \cdots i_M}$ that have incoming edges.  Since Lemma~\ref{lem:dag} implies that $\mathcal{G}_{i_1 \cdots i_M}$ is acyclic, we conclude that $i_1,\ldots,i_M$  are the only vertices in $\mathcal{G}_{i_1 \cdots i_M}$  that can be reachable from other vertices.  \halmos 
\end{proof}
In our next intermediary result, denoted by Lemma~\ref{lem:reachable}, we relate the reachability of vertices in a directed acyclic graph $\mathcal{G}_{i_1 \cdots i_M}$ to the set of rankings that correspond to $(i_1,\ldots,i_M) \in \mathcal{L}$. 
\begin{lemma} \label{lem:reachable}
Let $(i_1,\ldots,i_M) \in \mathcal{L}$ and $i,j \in \mathcal{N}_0$. Then, there exists a ranking  $\sigma \in \cap_{m \in \mathcal{M}} \mathcal{D}_{i_m}(S_m)$ that satisfies $\sigma(i) < \sigma(j)$ if and only if $j \nprec_{i_1 \cdots i_M} i$.
\end{lemma}
\begin{proof}{Proof.}
Let $i,j \in \mathcal{N}_0$ and $(i_1,\ldots,i_M) \in \mathcal{L}$, in which case it follows from Lemma~\ref{lem:dag} that the directed graph ${\mathcal{G}}_{i_1 \cdots i_M}$ is acyclic.

 To show the first direction of the desired result, let us suppose that $j \nprec_{i_1 \cdots i_M} i$. In this case, we observe that the directed edge $(j,i)$ can be added to $\mathcal{G}_{i_1 \cdots i_M}$ without inducing any cycles. %
 For notational convenience, let this augmented graph with the directed edge $(j,i)$ be denoted by $\tilde{\mathcal{G}}_{i_1 \cdots i_M}$. %
  Since the augmented graph  $\tilde{\mathcal{G}}_{i_1 \cdots i_M}$ is acyclic, it follows from  \citet[p. 77]{ahuja1988network} that there exists a ranking $\sigma \in \Sigma$ that satisfies $\sigma(j') < \sigma(i')$ for all directed edges $(i',j')$ in the augmented graph. %
  Moreover, since the set of directed edges in the augmented graph $\tilde{\mathcal{G}}_{i_1 \cdots i_M}$ is a superset of the set of directed edges in the original graph $\mathcal{G}_{i_1 \cdots i_M}$, it follows from the construction of the original graph that  this ranking satisfies $\sigma(i_m) < \sigma(i')$ for all $m \in \mathcal{M}$ and all $i' \in S_m \setminus \{i_m\}$, which implies that  $\sigma \in \cap_{m \in \mathcal{M}} \mathcal{D}_{i_m}(S_m)$. Since the augmented graph has the directed edge $(j,i)$, we observe that this ranking satisfies $\sigma(i) < \sigma(j)$. In summary, we have shown that if $j \nprec_{i_1 \cdots i_M} i$, then there exists a ranking $\sigma \in \cap_{m \in \mathcal{M}} \mathcal{D}_{i_m}(S_m)$ that satisfies $\sigma(i) < \sigma(j)$. This concludes our proof of the first direction of Lemma~\ref{lem:reachable}. 

To show the other direction, let us suppose that  $j \prec_{i_1 \cdots i_M} i$. In this case, consider any  arbitrary ranking $\sigma \in \Sigma$  that satisfies $\sigma \in \cap_{m \in \mathcal{M}} \mathcal{D}_{i_m}(S_m)$. 
It follows from the fact that $\sigma \in \cap_{m \in \mathcal{M}} \mathcal{D}_{i_m}(S_m)$ and from the construction of $\mathcal{G}_{i_1 \cdots i_M}$ that the inequality $\sigma(j') < \sigma(i')$ holds for each directed edge $(i',j')$ in this graph. %
 Moreover, without loss of generality, let the directed path from vertex $i$ to vertex $j$  be denoted  by the sequence of vertices $i,j_1,\ldots,j_\nu,j$ which satisfies the property that $(i,j_1),(j_1,j_2),\ldots,(j_{\nu-1},j_\nu),(j_\nu,j)$ are directed edges in $\mathcal{G}_{i_1 \cdots i_M}$. It then follows from our earlier reasoning that $\sigma(i) > \sigma(j_1) > \cdots > \sigma(j_\nu) > \sigma(j)$. Since the ranking $\sigma \in \cap_{m \in \mathcal{M}} \mathcal{D}_{i_m}(S_m)$ was chosen arbitrarily, our proof of the other direction is complete. 
\halmos \end{proof}

Intuitively, Lemma~\ref{lem:reachable} shows that the reachability of vertices in a directed acyclic graph $\mathcal{G}_{i_1 \cdots i_M}$ provides an encoding of the rankings that correspond to a tuple of products $(i_1,\ldots,i_M) \in \mathcal{L}$.  Said another way,  Lemma~\ref{lem:reachable} implies that  if vertex $i$ has a directed path to vertex $j$ in a directed acyclic graph $\mathcal{G}_{i_1 \cdots i_M}$, then product $j$ is always preferred to product $i$ under all rankings that correspond to the tuple of products $(i_1,\ldots,i_M) \in \mathcal{L}$.  
In our final intermediary result in Appendix~\ref{appx:graphical}, denoted below by Lemma~\ref{lem:reachable_v2}, we develop a generalization of Lemma~\ref{lem:reachable} that relates the reachability of vertices in a directed acyclic graph $\mathcal{G}_{i_1 \cdots i_M}$  to the most preferred products in an assortment $S \in \mathcal{S}$.  
\begin{lemma} \label{lem:reachable_v2}
Let $(i_1,\ldots,i_M) \in \mathcal{L}$, $S \in \mathcal{S}$, and $i \in S$. Then, there exists a ranking $\sigma \in \cap_{m \in \mathcal{M}} \mathcal{D}_{i_m}(S_m)$  that satisfies  $i = \argmin_{j \in S} \sigma(j)$ 
if and only if $j \nprec_{i_1 \cdots i_M} i$ for all $j \in S$.
\end{lemma}
\begin{proof}{Proof.}
Consider any tuple of products $(i_1,\ldots,i_M) \in \mathcal{L}$, assortment $S \in \mathcal{S}$, and product $i \in S$. 
 
To prove the first direction of Lemma~\ref{lem:reachable_v2},  suppose   there exists a product $j \in S$ that satisfies $j \prec_{i_1 \cdots i_M} i$. In this case,  it follows immediately from  Lemma~\ref{lem:reachable_im} that there exists a past assortment $m \in \mathcal{M}$ that satisfies $i_m \in S$ and $i_m \prec_{i_1 \cdots i_M} i$. Therefore, for each ranking $\sigma \in \cap_{m \in \mathcal{M}} \mathcal{D}_{i_m}(S_m)$, it follows from Lemma~\ref{lem:reachable} that $\sigma(i_m) < \sigma(i)$, which combined with the fact that $i_m \in S$ implies that $i \neq \argmin_{j \in S} \sigma(j)$. 
That concludes our proof of the first direction. 

To prove the other direction, suppose that $i \neq \argmin_{j \in S} \sigma(j)$ for all rankings $\sigma \in \cap_{m \in \mathcal{M}} \mathcal{D}_{i_m}(S_m)$. In this case,  it follows immediately from  Definition~\ref{defn:D} that  $\cap_{m \in \mathcal{M}} \mathcal{D}_{i_m}(S_m) \cap \mathcal{D}_i(S) = \emptyset$. Moreover, it follows from the fact that $(i_1,\ldots,i_M) \in \mathcal{L}$ and  from Lemma~\ref{lem:dag}  
 that the directed graph $\mathcal{G}_{i_1 \cdots i_M}$ is acyclic. Therefore, we observe from the fact that $\cap_{m \in \mathcal{M}} \mathcal{D}_{i_m}(S_m) \cap \mathcal{D}_i(S) = \emptyset$ and from Lemma~\ref{lem:dag} that 
 adding  the directed edges $\{(j, i): j \in S \setminus \{i\} \}$  to the directed acyclic graph $\mathcal{G}_{i_1 \cdots i_M}$ would result in a  directed cycle. Since a cycle visits every vertex at most once, there must exist a single vertex $j \in S \setminus \{i\}$ such that adding only the directed edge $(j,i)$ to $\mathcal{G}_{i_1 \cdots i_M}$ would result in a directed cycle, which implies that $j \prec_{i_1 \cdots i_M} i$. 
This concludes our proof of the other direction. \halmos \end{proof}
Lemma~\ref{lem:reachable_v2} establishes that a product $i \in S$ is the most preferred product from assortment $S \in \mathcal{S}$ under a ranking that corresponds to the tuple of products $(i_1,\ldots,i_M) \in \mathcal{L}$ if and only if there does not exist a directed path in the directed acyclic graph $\mathcal{G}_{i_1 \cdots i_M}$ from vertex $i$ to any vertex $j$ that satisfies $j \in S$. In other words, Lemma~\ref{lem:reachable_v2} implies that the set $\cap_{m \in \mathcal{M}} \mathcal{D}_{i_m}(S_m) \cap \mathcal{D}_i(S)$ is nonempty if and only if there is no vertex $j \in S$ which is reachable from vertex $i$.

In view of the above, we are now ready to develop our graphical interpretation of $\rho_{i_1 \cdots i_M}(S)$. This interpretation of $\rho_{i_1 \cdots i_M}(S)$, which is presented below in  Proposition~\ref{prop:cost_reform},  will be instrumental to our analysis in the subsequent Appendix~\ref{appx:prop2}, where we will use this interpretation to analyze the behavior of the functions $S \mapsto \rho_{i_1 \cdots i_M}(S)$ for each tuple of products $(i_1,\ldots,i_M) \in \mathcal{L}$. Our graphical representation of $\rho_{i_1 \cdots i_M}(S)$  requires the  following definition of the set $\mathcal{I}_{i_1 \cdots i_M}(S)$, which can be interpreted as the set of all vertices $i \in \mathcal{N}_0$ in the directed graph $\mathcal{G}_{i_1 \cdots i_M}$ that do not have a directed path to any of the vertices $i_1,\ldots,i_M$ that are elements of the assortment $S$. 
\begin{definition} \label{defn:I}
$\mathcal{I}_{i_1 \cdots i_M}(S) \triangleq \left \{ i \in \mathcal{N}_0:\; \textnormal{for all } m  \in \mathcal{M}, \; \textnormal{if } i_m \in S, \text{ then } i_m \nprec_{i_1 \cdots i_M} i \right \}$. 
\end{definition}
To make sense of Definition~\ref{defn:I}, we recall from Lemma~\ref{lem:reachable_im} that a vertex $j$ in a directed acyclic graph $\mathcal{G}_{i_1 \cdots i_M}$ can be  reachable from another vertex only if $j = i_m$ for some  past assortment $m \in \mathcal{M}$. Therefore, it follows immediately from Lemma~\ref{lem:reachable_v2} 
 that $S \cap \mathcal{I}_{i_1 \cdots i_M}(S)$  is the set of products $i$ for which the set of rankings $\cap_{m \in \mathcal{M}} \mathcal{D}_{i_m}(S_m) \cap \mathcal{D}_i(S)$ is nonempty. Combining this with Definition~\ref{defn:rho}, we have concluded the proof of the following Proposition~\ref{prop:cost_reform}, which establishes our   graphical interpretation of $\rho_{i_1 \cdots i_M}(S)$.  
\begin{proposition} \label{prop:cost_reform}
For all $S \in \mathcal{S}$ and $(i_1,\ldots,i_M) \in \mathcal{L}$, 
$\rho_{i_1 \cdots i_M}(S) = \min_{i \in S \cap \mathcal{I}_{i_1 \cdots i_M}(S)} r_i.$
\end{proposition}

\subsection{Proof of Proposition~\ref{prop:plus_one_inequality_prop}} \label{appx:prop2}

Equipped with Proposition~\ref{prop:cost_reform} from Appendix~\ref{appx:graphical}, we are now ready to present our proof of Proposition~\ref{prop:plus_one_inequality_prop}. 
To show this, we begin by developing an intermediary result, denoted below by Claim~\ref{claim:add_one}, that will allow us to compare the values of $\rho_{i_1 \cdots i_M}(S)$ and $\rho_{i_1 \cdots i_M}(S \cup \{i\})$ for every assortment $S \in \mathcal{S}$ and every product $i$ which is not in the assortment. 
\begin{claim} \label{claim:add_one}
For all $S \in \mathcal{S}$, $(i_1,\ldots,i_M) \in \mathcal{L}$, and $i \notin S$,
\begin{align*}
&\rho_{i_1 \cdots i_M}(S \cup \{i\}) = \begin{cases}
\rho_{i_1 \cdots i_M}(S),&\textnormal{if } i \notin \mathcal{I}_{i_1 \cdots i_M}(S),\\
\min \left \{ \min \limits_{j \in S \cap \mathcal{I}_{i_1 \cdots i_M}(S) \cap \left\{ j' \in \mathcal{N}_0: i \nprec_{i_1 \cdots i_M} j' \right\}} r_j, r_i \right \}, &\textnormal{if } i \in \mathcal{I}_{i_1 \cdots i_M}(S).\\
\end{cases}
\end{align*}
\end{claim}
\begin{proof}{Proof of Claim~\ref{claim:add_one}.}
Consider any  assortment $S \in \mathcal{S}$, tuple of products $(i_1,\ldots,i_M) \in \mathcal{L}$, and product $i \notin S$.  We first observe that 
\begin{align}
\mathcal{I}_{i_1 \cdots i_M}(S \cup \{i\}) &=  \left \{ j  \in \mathcal{N}_0: \textnormal{for all } m  \in \mathcal{M}, \; \textnormal{if } i_m \in S \cup \{i\}, \text{ then } i_m \nprec_{i_1 \cdots i_M} j \right \} \notag \\
&=  \left \{ j  \in \mathcal{N}_0: \textnormal{for all } m  \in \mathcal{M}, \; \textnormal{if } i_m \in S, \text{ then } i_m \nprec_{i_1 \cdots i_M} j \right \}  \notag \\
&\quad \cap  \left \{ j  \in \mathcal{N}_0:  \textnormal{if there exists } m \in \mathcal{M} \textnormal{ such that } i = i_m, \textnormal{ then }i \nprec_{i_1 \cdots i_M} j \right \} \notag \\
&= \mathcal{I}_{i_1 \cdots i_M}(S) \cap   \left \{ j  \in \mathcal{N}_0:  \textnormal{if there exists } m \in \mathcal{M} \textnormal{ such that } i = i_m, \textnormal{ then }i \nprec_{i_1 \cdots i_M} j \right \} \notag \\
&=\begin{cases}
\mathcal{I}_{i_1 \cdots i_M}(S),&\text{if }i \notin \{i_1,\ldots,i_M\},\\
\mathcal{I}_{i_1 \cdots i_M}(S) \cap \left \{ j \in \mathcal{N}_0: i \nprec_{i_1 \cdots i_M} j \right \},&\text{if }i \in \{i_1,\ldots,i_M\}
\end{cases} \notag \\ 
&= \mathcal{I}_{i_1 \cdots i_M}(S) \cap \left \{ j \in \mathcal{N}_0: i \nprec_{i_1 \cdots i_M} j \right \}. \label{line:adding_one_step_two} 
\end{align}
Indeed, the first and third equalities follow from Definition~\ref{defn:I}. The second and fourth equalities follow from algebra. The final equality follows from the fact that if $i \notin \{i_1,\ldots,i_M\}$, then it follows from the construction of $\mathcal{G}_{i_1 \cdots i_M}$ that there are no incoming edges to vertex $i$, which implies that $\left \{ j \in \mathcal{N}_0: i \nprec_{i_1 \cdots i_M} j \right \} = \mathcal{N}_0$.  

Therefore, it follows from line~\eqref{line:adding_one_step_two} and Proposition~\ref{prop:cost_reform} that 
\begin{align}
\rho_{i_1 \cdots i_M}(S \cup \{i \}) &= \min_{j \in (S \cup \{i\}) \cap \mathcal{I}_{i_1 \cdots i_M}(S \cup \{i\})} r_j \notag \\
&=  \min_{j \in (S \cup \{i\}) \cap \mathcal{I}_{i_1 \cdots i_M}(S) \cap \{j' \in \mathcal{N}_0: \; i \nprec_{i_1 \cdots i_M} j' \}} r_j \notag \\
&= \begin{cases}
\min \limits_{j \in S \cap \mathcal{I}_{i_1 \cdots i_M}(S) \cap \{ j' \in \mathcal{N}_0: i \nprec_{i_1 \cdots i_M} j' \}} r_j,&\text{if } i \notin \mathcal{I}_{i_1 \cdots i_M}(S),\\
\min \left \{ \min \limits_{j \in S \cap \mathcal{I}_{i_1 \cdots i_M}(S) \cap \{ j' \in \mathcal{N}_0: i \nprec_{i_1 \cdots i_M} j' \}} r_j, r_i \right \}, &\text{if } i \in \mathcal{I}_{i_1 \cdots i_M}(S),
\end{cases} \label{line:adding_one_step_three}
\end{align}
where the first equality follows from Proposition~\ref{prop:cost_reform}, the second equality follows from line~\eqref{line:adding_one_step_two}, and the last equality follows from algebra and the fact that $i \nprec_{i_1 \cdots i_M} i$. 

We conclude the proof of Claim~\ref{claim:add_one} by considering the case where $i \notin \mathcal{I}_{i_1 \cdots i_M}(S)$. 
 For this case, suppose for the sake of developing a contradiction that there exists $j \in S \cap \mathcal{I}_{i_1 \cdots i_M}(S)$ which satisfies $i \prec_{i_1 \cdots i_M} j$. Under that supposition, it would follow from the fact that $i \notin \mathcal{I}_{i_1 \cdots i_M}(S)$ and Definition~\ref{defn:I} that there would exist a past assortment $m \in \mathcal{M}$ that satisfies $i_m \in S$ and $i_m \prec_{i_1 \cdots i_M} i$. Thus, by the transitive property, we have $i_m \prec_{i_1 \cdots i_M} i \prec_{i_1 \cdots i_M} j$, which contradicts the supposition that $j \in \mathcal{I}_{i_1 \cdots i_M}(S)$. Because we have a contradiction, we have shown that $i \nprec_{i_1 \cdots i_M} j$ for all $j \in S \cap \mathcal{I}_{i_1 \cdots i_M}(S)$, and so the desired result follows  immediately from line~\eqref{line:adding_one_step_three} and Definition~\ref{defn:rho}.  
\halmos 
\end{proof}

Using the above intermediary result, we now complete the proof of Proposition~\ref{prop:plus_one_inequality_prop}. Indeed, consider any assortment $S\in \mathcal{S}$ and any product $i \notin S$.  Suppose that there exists a product $i^* \in S$ which satisfies $r_{i^*} < r_i$ and $\mathcal{M}_{i^*} \subseteq \mathcal{M}_i$.  For each tuple of products $(i_1,\ldots,i_M) \in \mathcal{L}$, we have two cases to consider:

\begin{itemize}
\item \underline{Case 1:} Suppose that $i \notin \mathcal{I}_{i_1 \cdots i_M}(S)$. In this case, it follows  immediately from Claim~\ref{claim:add_one} that $$\rho_{i_1 \cdots i_M}(S \cup \{i\}) = \rho_{i_1 \cdots i_M}(S),$$ 
and so the inequality $\rho_{i_1 \cdots i_M}(S) \le \rho_{i_1 \cdots i_M}(S \cup \{i\})$ holds when  $i \notin \mathcal{I}_{i_1 \cdots i_M}(S)$. 

\vspace{1em}

\item \underline{Case 2:} Suppose that $i \in \mathcal{I}_{i_1 \cdots i_M}(S)$.

  In this case, we begin by showing that $i_m \notin S$ for each $m \in \mathcal{M}_i$. Indeed, consider any past assortment $m \in \mathcal{M}_i$. On one hand, if $i = i_m$, then it follows immediately from the fact that $i \notin S$ that $i_m \notin S$.  On the other hand, if $i \neq i_m$, then it also must be the case that $i_m \notin S$, else we would have a contradiction with the fact that  $i \in \mathcal{I}_{i_1 \cdots i_M}(S)$ and the fact that there is, by our construction of $\mathcal{G}_{i_1 \cdots i_M}$, a directed edge from vertex $i$ to vertex $i_m$. We have thus shown that $i_m \notin S$ for all $m \in \mathcal{M}_i$. 

We next show that $i_m \in \mathcal{I}_{i_1 \cdots i_M}(S)$ for all $m \in \mathcal{M}_i$. Indeed, consider any arbitrary $m \in \mathcal{M}_i$. On one hand, if $i_m = i$, then the  statement  $i_m \in \mathcal{I}_{i_1 \cdots i_M}(S)$ follows immediately from the fact that $i \in \mathcal{I}_{i_1 \cdots i_M}(S)$. On the other hand, if $i_m \neq i$, then it follows from the fact that $i \in \mathcal{I}_{i_1 \cdots i_M}(S)$ and the fact that there is a directed edge from vertex $i$ to vertex $i_m$ that there must not be a directed path from vertex $i_m$ to a vertex  $i_{m'}$ that satisfies $i_{m'} \in S$ for any $m' \in \mathcal{M}$. We have thus shown that   $i_m \in \mathcal{I}_{i_1 \cdots i_M}(S)$ for all $m \in \mathcal{M}_i$.

Using the above results, we now prove that $i^* \in \mathcal{I}_{i_1 \cdots i_M}(S)$. Indeed, we have shown in the above results that $i_m \notin S$ and $i_m \in \mathcal{I}_{i_1 \cdots i_M}(S)$ for all $m \in \mathcal{M}_{i}$. Therefore, it follows from the supposition that $\mathcal{M}_{i^*} \subseteq \mathcal{M}_i$ that $i_m \notin S$ and $i_m \in \mathcal{I}_{i_1 \cdots i_M}(S)$ for all $m \in \mathcal{M}_{i^*}$.  Since we have supposed that $i^* \in S$, it follows from the fact that $i_m \notin S$  for all $m \in \mathcal{M}_{i^*}$ that $i^* \neq i_m$ for all $m \in \mathcal{M}_{i^*}$. Moreover, it follows from the construction of $\mathcal{G}_{i_1 \cdots i_M}$ that all of the outgoing edges from vertex $i^*$ are incoming edges to  vertices $i_m$ for $m \in \mathcal{M}_{i^*}$. Therefore, it follows from the fact that $i_m \in \mathcal{I}_{i_1 \cdots i_M}(S)$ for all $m \in \mathcal{M}_{i^*}$ that  $i^* \in \mathcal{I}_{i_1 \cdots i_M}(S)$. 

We now prove the desired result for Case 2. First, we observe that
\begin{align}
\rho_{i_1 \cdots i_M}(S) = \min_{j \in S \cap \mathcal{I}_{i_1 \cdots i_M}(S)} r_j \le r_{i^*}, \label{line:rho_woah}
\end{align}
where the equality follows from Proposition~\ref{prop:cost_reform}  and the inequality follows from our supposition that $i^* \in S$ and because we have shown that  $i^* \in \mathcal{I}_{i_1 \cdots i_M}(S)$. 
Therefore, we have that 
\begin{align*}
\rho_{i_1 \cdots i_M}(S \cup \{i\}) &= \min \left \{ \min \limits_{j \in S \cap \mathcal{I}_{i_1 \cdots i_M}(S) \cap \{ j' \in \mathcal{N}_0: i \nprec_{i_1 \cdots i_M} j' \}} r_j, r_i \right \} \\
&\ge \min \left \{   \min \limits_{j \in S \cap \mathcal{I}_{i_1 \cdots i_M}(S)} r_j, r_i \right \} \\
&= \min \left \{  \rho_{i_1 \cdots i_M}(S), r_i \right \} \\
&= \rho_{i_1 \cdots i_M}(S),
\end{align*}
where the first equality follows from Claim~\ref{claim:add_one} and the supposition of Case 2 that $i \in \mathcal{I}_{i_1 \cdots i_M}(S)$, the inequality follows algebra, the second equality follows from  Proposition~\ref{prop:cost_reform},  and the final equality holds because of our supposition that $r_{i^*} < r_{i}$ and because of line~\eqref{line:rho_woah}, which showed that $ \rho_{i_1 \cdots i_M}(S) \le r_{i^*}$. This concludes the proof of Case 2. 
\end{itemize}
In both of the above two cases, we showed that $\rho_{i_1 \cdots i_M}(S) \le \rho_{i_1 \cdots i_M}(S \cup \{i \})$, and so our proof of Proposition~\ref{prop:plus_one_inequality_prop} is complete. 
\halmos

\subsection{Proof of Theorem~\ref{thm:main}} \label{appx:proof_main}
Consider any arbitrary assortment $S \in \mathcal{S}$. For this assortment, we define a new assortment as $$S' \triangleq S \cup \left \{i \in \mathcal{N}_0: \text{there exists } i^* \in S \textnormal{ such that }  \mathcal{M}_{i^*} \subseteq \mathcal{M}_i \textnormal{ and } r_{i^*} < r_i  \right \}.$$
It follows immediately from the definition of the collection $\widehat{\mathcal{S}}$ that this new assortment satisfies $S' \in \widehat{\mathcal{S}}$. Moreover, let $\{ j_1,\ldots,j_\nu \} \triangleq S' \setminus S$ denote the new products that have been added into the assortment. Then we observe for each tuple of products $(i_1,\ldots,i_M) \in \mathcal{L}$ that
\begin{align*}
\rho_{i_1 \cdots i_M}(S') &=  \rho_{i_1 \cdots i_M}(S)  + \sum_{\iota = 1}^{\nu} \left(  \rho_{i_1 \cdots i_M}(S \cup \{j_1,\ldots,j_\iota \}) - \rho_{i_1 \cdots i_M}(S \cup \{j_1,\ldots,j_{\iota-1} \}) \right) \ge  \rho_{i_1 \cdots i_M}(S).
\end{align*} 
Indeed, the equality follows from algebra.  The  inequality  follows from Proposition~\ref{prop:plus_one_inequality_prop}, which implies that $ \rho_{i_1 \cdots i_M}(S \cup \{j_1,\ldots,j_\iota \}) \ge \rho_{i_1 \cdots i_M}(S \cup \{j_1,\ldots,j_{\iota-1} \})$ for each $\iota \in \{1,\ldots,\nu\}$.  Since the assortment $S \in \mathcal{S}$ was chosen arbitrarily, our proof of Theorem~\ref{thm:main} follows from Proposition~\ref{prop:reform_wc}. 
\halmos 

\section{Proofs of Technical Results from \S\ref{sec:characterization:question}}
\label{appx:impossibility}

\subsection{Proof of Lemma~\ref{lem:impossibility}}
Let $\mathscr{M} = \bar{\mathcal{S}}$, and let the past assortments be indexed by  $\mathscr{M} = \{\bar{S}_1,\ldots,\bar{S}_n \}$, whereby the $i$-th past assortment  is $\bar{S}_i \triangleq \{0,i,i+1,\ldots,n-1,n\}$.  Equipped with the above notation, we first prove that the equality 
 $\bar{\mathcal{S}} = \widehat{\mathcal{S}}$ holds. 
Indeed,  choose any arbitrary assortment $S \in \widehat{\mathcal{S}}$, and let $i^* \triangleq \argmin_{j \in S: r_j > 0} r_j$ denote the product in the chosen assortment that has the smallest nonzero revenue.  
It readily follows from the facts that $\mathscr{M} = \bar{\mathcal{S}}$ and $r_1 < \cdots < r_n$ that the equalities $\mathcal{M}_{i^*} = \{ m \in \mathcal{M}: m \le i^*\}$ and $\mathcal{M}_i =  \{ m \in \mathcal{M}: m \le i\}$ hold for each $i \in \{i^*+1,\ldots,n \}$. Therefore, for each $i \in \{i^*+1,\ldots,n\}$, it follows from the definition of the collection $\widehat{\mathcal{S}}$,  from the fact that $r_{i^*} < r_i$, and from the fact that $\mathcal{M}_{i^*} \subseteq \mathcal{M}_i$ that $i \in S$. We have thus shown that $S = \{0, i^*,i^*+1,\ldots,n-1,n\} = \bar{S}_{i^*}$, which implies that $S \in \bar{\mathcal{S}}$. Since the assortment $S \in \widehat{\mathcal{S}}$ was chosen arbitrarily, we have shown that $\widehat{\mathcal{S}} \subseteq \bar{\mathcal{S}}$. The other direction of the proof that  $\bar{\mathcal{S}} = \widehat{\mathcal{S}}$ follows from the fact that  the inclusion $\mathscr{M} \subseteq \widehat{\mathcal{S}}$ always holds\footnote{To see why the inclusion $\mathscr{M} \subseteq \widehat{\mathcal{S}}$ always holds, consider any past assortment $S \in \mathscr{M}$. For each product $i^* \in S$, suppose that there exists another product $i$ which satisfies $r_{i^*} < r_{i}$ and $\mathcal{M}_{i^*} \subseteq \mathcal{M}_{i}$. Since $S \in \mathcal{M}_{i^*} \subseteq \mathcal{M}_i$, we conclude that $i \in S$ must hold, which proves that $S \in \widehat{\mathcal{S}}$. } and from the fact that $\mathscr{M} = \bar{\mathcal{S}}$.  Our proof that $\bar{\mathcal{S}} = \widehat{\mathcal{S}}$ is thus complete. \halmos

\subsection{Proof of Corollary~\ref{cor:impossibility}}

Using Lemma~\ref{lem:impossibility}, we have 
\begin{align*}
\max_{S \in \mathcal{S}} \min_{\lambda \in \mathcal{U}} \mathscr{R}^{\lambda}(S)  = \max_{S \in \widehat{\mathcal{S}}} \min_{\lambda \in \mathcal{U}} \mathscr{R}^{\lambda}(S)  =  \max_{S \in \bar{\mathcal{S}}} \min_{\lambda \in \mathcal{U}} \mathscr{R}^{\lambda}(S) =  \max_{S \in \mathscr{M}} \min_{\lambda \in \mathcal{U}} \mathscr{R}^{\lambda}(S) =   \max_{m \in \mathcal{M}} r^\intercal v_m,
\end{align*}
where the first equality follows from Theorem~\ref{thm:main}, the second equality holds because  $\widehat{\mathcal{S}} = \bar{\mathcal{S}}$, the third equality holds because  $\mathscr{M} = \bar{\mathcal{S}}$, and the final equality follows from the construction of the set of ranking-based choice models $\mathcal{U}$ (see \S\ref{sec:setting}) and from the fact that $\eta = 0$. 
\halmos 

\section{Proofs of Technical Results from \S\ref{sec:characterization:upperbound}}

\subsection{Proof of Lemma~\ref{lem:two:S}}\label{appx:two:S}
For notational convenience, let the collection on the right side of line~\eqref{line:S_two} in Lemma~\ref{lem:two:S} be denoted by $\widehat{\mathcal{S}}'$. It  follows immediately from the definition of $\widehat{\mathcal{S}}$ and from the fact that $r_0 < r_1 < \cdots < r_n$ that each assortment $S \in \widehat{\mathcal{S}}'$ is also an element of $\widehat{\mathcal{S}}$. This proves that $ \widehat{\mathcal{S}} \supseteq \widehat{\mathcal{S}}'$.   To show the other direction, consider any arbitrary assortment $S \in \widehat{\mathcal{S}}$. We first show that $S_1 \cap S_2 \subseteq S$. Indeed,  it follows from the fact that $S,S_1,S_2 \in \mathcal{S}$ that $0 \in S \cap S_1 \cap S_2$.  Moreover, for each product $j \in (S_1 \cap S_2)\setminus \{0\}$, we readily observe that the inequality $r_j > r_0$ and the equality $\mathcal{M}_j = \mathcal{M}_0$ both hold. Therefore, it follows from the definition of $\widehat{\mathcal{S}}$ and the fact that $S \in \widehat{\mathcal{S}}$ that each product $j \in S_1 \cap S_2$ is also an element of $S$. We have thus shown that $S_1 \cap S_2 \subseteq S$ for all $S \in \widehat{\mathcal{S}}$. 
Next, we define the following integers: 
\begin{align*}
i_1 &\triangleq \min \left \{n, \min_{j \in  S_1 \setminus S_2} j \right \}, & i_2 &\triangleq \min \left \{ n, \min_{j \in S_2 \setminus S_1} j \right \},
\end{align*}
where any minimization problem over an empty feasible set is defined equal to $\infty$.
It follows from the fact that $S \in \widehat{\mathcal{S}}$, from the definition of $\widehat{\mathcal{S}}$, and from the assumption of $r_0 < r_1 < \cdots < r_n$ that 
\begin{align*}
\left \{j \in S_1 \setminus S_2: j \ge i_1 \right \}  \subseteq S \text{ and } \left \{j \in  S_2 \setminus S_1: j \ge i_2 \right \} \subseteq S. 
\end{align*}
Hence, it follows from the assumption that $S_1 \cup S_2 = \mathcal{N}_0$ that
\begin{align*}
S = \left(S_1 \cap S_2 \right) \cup \left \{j \in  S_1 \setminus S_2: j \ge i_1 \right \}  \cup \left \{j \in  S_2 \setminus S_1: j \ge i_2 \right \}.
\end{align*}
We have thus shown that the assortment $S$ is an element of the collection of assortments $\widehat{\mathcal{S}}'$. Since $S \in \widehat{\mathcal{S}}$ was chosen arbitrarily, we have shown that $\widehat{\mathcal{S}} \subseteq \widehat{\mathcal{S}}'$, which concludes our proof of Lemma~\ref{lem:two:S}. 
 \halmos

\subsection{Proof of Lemma~\ref{lem:fixed_dim:S}} \label{appx:fixed_dim:S}
We readily observe that the collection of assortments $\widehat{\mathcal{S}}$ is a subset of
\begin{align*}
\widetilde{\mathcal{S}} \triangleq \left \{ S \in \mathcal{S}:\; \textnormal{if } i^* \in S,\; r_{i^*} < r_{i}, \textnormal{ and } \mathcal{M}_{i^*} = \mathcal{M}_i, \textnormal{ then } i \in S \right \},
\end{align*}
where we recall from the beginning of \S\ref{sec:characterization} that $\mathcal{M}_i$ is defined as the subset of the past assortments $\mathcal{M} \equiv \{1,\ldots,M\}$ that offered product $i$. 
For each subset of  past assortments $\mathcal{C} \subseteq \mathcal{M}$, let the products that are offered {only} in the assortments in $\mathcal{C}$ be denoted by
\begin{align*}
\mathcal{N}_0(\mathcal{C}) &\triangleq 
 \left \{ i \in \mathcal{N}_0: \mathcal{M}_i = \mathcal{C}  \right \}.
\end{align*}
Equipped with the above definitions, we observe that $\{\mathcal{N}_0(\mathcal{C}): \mathcal{C} \subseteq \mathcal{M}\}$ is the collection of subsets of products that always appear together in the past assortments, and we readily observe that  $|\{\mathcal{N}_0(\mathcal{C}): \mathcal{C} \subseteq \mathcal{M}\}| = 2^M$. 
Therefore, 
\begin{align}
| \widehat{\mathcal{S}}  |  &\le | \widetilde{\mathcal{S}}  | \le \prod_{\mathcal{C} \subseteq \mathcal{M}} \left( \left| \mathcal{N}_0(\mathcal{C}) \right| + 1 \right)  \le (n+2)^{2^M}.  \label{line:badname}
\end{align}
Indeed, the first inequality on line~\eqref{line:badname} holds because the collection of assortments $\widehat{\mathcal{S}}$ is a subset of the collection of assortments $\widetilde{\mathcal{S}}$. To see why the second inequality on line~\eqref{line:badname} holds, consider any arbitrary subset of past assortments $\mathcal{C} \subseteq \mathcal{M}$, and let the products in $\mathcal{N}_0(\mathcal{C})$ be indexed in ascending order by revenue; that is, let the products that comprise $\mathcal{N}_0(\mathcal{C})$ be denoted by $i_1^{\mathcal{C}},\ldots,i_{|\mathcal{N}_0(\mathcal{C})|}^{\mathcal{C}}$, where $r_{i_1^{\mathcal{C}}} < \cdots < r_{i_{|\mathcal{N}_0(\mathcal{C})|}^{\mathcal{C}}}$.  Then, we observe from the definition of the collection $\widetilde{\mathcal{S}}$ that every assortment $S \in \widetilde{\mathcal{S}}$ must satisfy the condition [$\{i_{j+1}^{\mathcal{C}},\ldots,i_{| \mathcal{N}_0(\mathcal{C})|}^{\mathcal{C}} \} \subseteq S$ and $i_1^{\mathcal{C}},\ldots,i_{j}^{\mathcal{C}} \notin S$] for some $j \in \{0,\ldots,| \mathcal{N}_0(\mathcal{C})|\}$.  Since $\mathcal{C} \subseteq \mathcal{M}$ was chosen arbitrarily, we have shown that 
$$\widetilde{\mathcal{S}} \subseteq\left \{ \bigcup_{\mathcal{C}\subseteq \mathcal{M}} \mathcal{F}^{\mathcal{C}}: \; \mathcal{F}^{\mathcal{C}} \in \left \{  \emptyset, \left\{i_{| \mathcal{N}_0(\mathcal{C})|}^{\mathcal{C}} \right \}, \left\{i_{| \mathcal{N}_0(\mathcal{C})|-1}^{\mathcal{C}},i_{| \mathcal{N}_0(\mathcal{C})|}^{\mathcal{C}} \right \},\ldots,   \left\{i_1^{\mathcal{C}},\ldots,i_{| \mathcal{N}_0(\mathcal{C})|}^{\mathcal{C}} \right \}  \right \}  \right \},  $$
%
which proves that the second inequality on line~\eqref{line:badname} holds. The third inequality on line~\eqref{line:badname} follows from the fact that  $|\{\mathcal{N}_0(\mathcal{C}): \mathcal{C} \subseteq \mathcal{M}\}| = 2^M$ and from the fact that $|\mathcal{N}_0(\mathcal{C})| \le  n+1$ for every $\mathcal{C} \subseteq \mathcal{M}$. We have thus proven that  $| \widehat{\mathcal{S}} |$ is at most $(n+2)^{2^M}$, which concludes our proof of Lemma~\ref{lem:fixed_dim:S}. \halmos

\subsection{Proof of Lemma~\ref{lem:fixed_dim:S:time}}\label{appx:fixed_dim:S:time}
  As the first step in our proof of Lemma~\ref{lem:fixed_dim:S:time}, we develop an algorithm for constructing a directed graph, denoted by $\mathscr{G} \equiv (\mathscr{V},\mathscr{E})$, in which the set of vertices in the directed graph is defined as $\mathscr{V} \triangleq \mathcal{N}_0$ and the set of directed edges is defined as $\mathscr{E} \triangleq \{ (i^*,i) \in \mathcal{N}_0 \times \mathcal{N}_0: \; r_{i^*} < r_{i} \textnormal{ and } \mathcal{M}_{i^*} \subseteq \mathcal{M}_{i} \}$. 
This directed graph  has a natural correspondence with the collection of assortments $\widehat{\mathcal{S}}$, as we readily observe that an assortment satisfies $S \in \widehat{\mathcal{S}}$ if and only if [$0 \in S$] and [we have that $i \in S$ whenever there exists a product $i^* \in S$ and a directed edge $(i^*,i) \in \mathscr{E}$]. It is easy to see that this directed graph is acyclic and has the property that a vertex $j \in \mathscr{V}$  is reachable from a vertex $i \in \mathscr{V}$ if and only if there is a directed edge $(i,j) \in \mathscr{E}$ from vertex $i$ to vertex $j$. 
A directed acyclic graph which has the aforementioned property for each pair of vertices is referred to as a \emph{transitive closure} \citep[p.90]{ahuja1988network}.

Our algorithm for constructing the directed acyclic graph $\mathscr{G}$ that is a transitive closure  is presented in Algorithm~\ref{alg:construct_G}. In the algorithm, we first iterate over each product $i \in \mathcal{N}_0$ and construct the corresponding set $\mathcal{M}_i$ of past assortments which offered that product. It is easy to see that each of the sets $\mathcal{M}_i$ can be constructed in $\mathcal{O}(M)$ computation time by checking whether the product satisfies $i \in S_m$ for each past assortment $m \in \mathcal{M}$; hence, we observe that all of the sets $\mathcal{M}_0,\ldots,\mathcal{M}_n$ can be constructed in a total of $\mathcal{O}(Mn)$ computation time. Assume that we  store the sets $\mathcal{M}_0,\ldots,\mathcal{M}_n$ as  unsorted arrays as well as hash tables.   Given these data structures for $\mathcal{M}_0,\ldots,\mathcal{M}_n$, we then iterate over each pair of products $(i,i^*) \in \mathcal{N}_0 \times \mathcal{N}_0$ and check in $\mathcal{O}(M)$ computation time whether $r_{i^*} < r_{i}$ and $\mathcal{M}_{i^*} \subseteq \mathcal{M}_i$.  Since there are $\mathcal{O}(n^2)$ pairs of products in $\mathcal{N}_0 \times \mathcal{N}_0$, we conclude that the set of directed edges $\mathscr{E}$ can be constructed in  a total of $\mathcal{O}(n^2 M)$ computation time. Combining all of the steps, and since it takes $\mathcal{O}(n)$ computation time to construct the set of vertices $\mathscr{V}$,  we have shown that the total computation time for Algorithm~\ref{alg:construct_G}  is $\mathcal{O}(n + Mn + n^2 M) = \mathcal{O}(n^2 M)$.

\begin{algorithm}[t]
\begin{center}
\fbox{\begin{minipage}{\linewidth}
\begin{center}
\textsc{\underline{Construct-${\mathscr{G}}(\mathscr{M},r)$}}\\
\end{center}
\vspace{1em}
\textbf{Inputs}: 
\begin{itemize}
\item The collection of past assortments, $\mathscr{M}  \equiv \{S_1,\ldots,S_M\}$. 
\item The revenues of the products, $r \equiv (r_0,r_1,\ldots,r_n)$. 
\end{itemize}
\vspace{1em}
\textbf{Output}:
\begin{itemize}
\item  $\mathscr{G} \equiv (\mathscr{V},\mathscr{E})$, where $\mathscr{V} \equiv \mathcal{N}_0$ and $\mathscr{E} \equiv \{ (i^*,i) \in \mathcal{N}_0 \times \mathcal{N}_0: \; r_{i^*} < r_{i} \textnormal{ and } \mathcal{M}_{i^*} \subseteq \mathcal{M}_{i} \}$. 
\end{itemize} 
\vspace{1em}
\textbf{Procedure}: 
\begin{enumerate}
\item Initialize the vertex set $\mathscr{V} \leftarrow \emptyset$ and edge set $\mathscr{E} \leftarrow \emptyset$. 
\item For each product $i \in \mathcal{N}_0$: \label{step:iterative_of_i}
\begin{enumerate}
\item Update $\mathscr{V} \leftarrow \mathscr{V} \cup \{i \}$. 
\item Construct the set $\mathcal{M}_i$ of past assortments which offered product $i$. 
\end{enumerate}
\item For each pair of products $(i^*,i) \in \mathcal{N}_0 \times \mathcal{N}_0$:
\begin{enumerate}
\item If $r_{i^*} < r_{i}$ and $\mathcal{M}_{i^*} \subseteq \mathcal{M}_i$:
\begin{enumerate}
\item Update $\mathscr{E} \leftarrow \mathscr{E} \cup \{(i^*,i) \}$
\end{enumerate}
\end{enumerate}
\item Output $\mathscr{G} \equiv (\mathscr{V},\mathscr{E})$ and terminate. 
\end{enumerate}
\end{minipage}}
\end{center}
\caption{A procedure for constructing the directed acyclic graph $\mathscr{G}$.} \label{alg:construct_G}
\end{algorithm}

We next describe our algorithm for constructing the collection of assortments $\widehat{\mathcal{S}}$ from the directed graph $\mathscr{G}$. This algorithm is  denoted by \textsc{Construct-$\widehat{\mathcal{S}}(\mathscr{M},r)$} and is found in Algorithm~\ref{alg:construct_S}.  In this algorithm,  we first use Algorithm~\ref{alg:construct_G} to construct the directed acyclic graph $\mathscr{G} \equiv (\mathscr{V},\mathscr{E})$ which is a transitive closure, and then we invoke a recursive subroutine denoted by \textsc{RecursiveStep$(\mathscr{G})$} in Algorithm~\ref{alg:recursive}.   The goal of the recursive subroutine is to take as an input a generic directed acyclic graph $\mathscr{G} \equiv (\mathscr{V}, \mathscr{E})$ which is a transitive closure, and for that graph, output the collection of subsets of vertices $\mathscr{A} \equiv \{S \subseteq \mathscr{V}:  \text{ if $i \in S$ and $(i,j) \in \mathscr{E}$, then $j \in S$}\}$. Algorithm~\ref{alg:construct_S} concludes by removing the subsets of vertices from $\mathscr{A}$ which do not include the no-purchase option $0$, and then outputs the remaining subsets of vertices from $\mathscr{A}$.  The correctness of Algorithm~\ref{alg:construct_S} follows immediately from our earlier observation that an assortment satisfies $S \in \widehat{\mathcal{S}}$ if and only if [$0 \in S$] and [we have that $i \in S$ whenever there exists a product $i^* \in S$ and a directed edge $(i^*,i) \in \mathscr{E}$].

   At a high level, the recursive subroutine in Algorithm~\ref{alg:recursive} is comprised of two cases. The base case of the subroutine is when the graph has no vertices, in which case it is clear that $\mathscr{A} =\{ \emptyset \}$. If we are not in the base case, then the aim of the recursive subroutine is to construct the collections $ \left \{ S \in \mathscr{A}: i \notin S \right \} $ and $\left \{ S \in \mathscr{A}: i \in S \right \}$ for a chosen vertex $i \in \mathscr{V}$ and output the union of these two collections. The construction of the collection $ \left \{ S \in \mathscr{A}: i \notin S \right \} $ takes place on lines~\eqref{step:construct_A1:1}-\eqref{step:construct_A1:2} of Algorithm~\ref{alg:recursive}, and the construction of the collection $ \left \{ S \in \mathscr{A}: i \in S \right \} $ takes place on lines~\eqref{step:construct_A2:1}-\eqref{step:construct_A2:3} of Algorithm~\ref{alg:recursive}. 

Up to this point, we have established that Algorithm~\ref{alg:construct_G} is correct (that is,  it delivers the desired output for any valid input), and we have established that Algorithm~\ref{alg:construct_S} is correct under the assumption that Algorithm~\ref{alg:recursive} is correct. Therefore, it remains for us to prove that  the recursive subroutine in Algorithm~\ref{alg:recursive} is correct. To prove the correctness of the recursive subroutine, we will make use of four intermediary claims, which are denoted below by  Claims~\ref{claim:appx:idk_intro}-\ref{claim:appx:idk_3}. The purpose of the first two intermediary claims, denoted by Claims~\ref{claim:appx:idk_intro} and \ref{claim:appx:idk}, is to show that the graphs $\mathscr{G}' \equiv (\mathscr{V}', \mathscr{E}')$ and $\mathscr{G}'' \equiv (\mathscr{V}'', \mathscr{E}'')$  constructed on lines~\eqref{step:construct_A1:1} and \eqref{step:construct_A2:1} of Algorithm~\ref{alg:recursive} are directed acyclic graphs that are transitive closures, which implies that $\mathscr{G}' $ and $\mathscr{G}''$ are valid inputs on lines~\eqref{step:construct_A1:2} and \eqref{step:construct_A2:2} of  Algorithm~\ref{alg:recursive}. The purpose of the second two intermediary claims, denoted by Claims~\ref{claim:appx:idk_2} and \ref{claim:appx:idk_3}, is to show that the union of the two collections $ \mathscr{A}' $ and $ \mathscr{A}'''$ constructed on lines~\eqref{step:construct_A1:2} and \eqref{step:construct_A2:3} of Algorithm~\ref{alg:recursive} provides the desired output on line~\eqref{step:construct_A} of Algorithm~\ref{alg:recursive}.   
\begin{algorithm}[t] 
\begin{center}
\fbox{\begin{minipage}{\linewidth}
\begin{center}
\textsc{\underline{Construct-$\widehat{\mathcal{S}}(\mathscr{M},r)$}}\\
\end{center}
\vspace{1em}
\textbf{Inputs}: 
\begin{itemize}
\item The collection of past assortments, $\mathscr{M}  \equiv \{S_1,\ldots,S_M\}$. 
\item The revenues of the products, $r \equiv (r_0,r_1,\ldots,r_n)$. 
\end{itemize}
\vspace{1em}
\textbf{Output}:
\begin{itemize}
\item  The collection of assortments $\widehat{\mathcal{S}}$ corresponding to the collection of past assortments $\mathscr{M}$ and the revenues $r$. 
\end{itemize} 
\vspace{1em}
\textbf{Procedure}: 
\begin{enumerate}
\item Construct the directed acyclic graph $\mathscr{G} \leftarrow \textsc{Construct-${\mathscr{G}}(\mathscr{M},r)$}$.  \label{step:construct_S:1}
\item Compute the collection of assortments $\widehat{\mathcal{S}} \leftarrow \textsc{RecursiveStep$(\mathscr{G})$}$.  \label{step:construct_S:2}
\item For each $S \in \widehat{\mathcal{S}}$:  \label{step:construct_S:3}
\begin{enumerate}
\item If $0 \notin S$:
\begin{enumerate}
\item Update $\widehat{\mathcal{S}} \leftarrow \widehat{\mathcal{S}} \setminus \{ S \}$. 
\end{enumerate}
\end{enumerate}
\item Output $\widehat{\mathcal{S}}$ and terminate. 
\end{enumerate}
\vspace{1em}
\end{minipage}}
\end{center}
\caption{A procedure for constructing $\widehat{\mathcal{S}}$.} \label{alg:construct_S}
\end{algorithm}

\begin{algorithm}[t]
\begin{center}
\fbox{\begin{minipage}{\linewidth}
\begin{center}
\textsc{\underline{RecursiveStep$(\mathscr{G})$}}\\
\end{center}
\vspace{1em}
\textbf{Inputs}: 
\begin{itemize}
\item A directed acyclic graph $\mathscr{G} \equiv (\mathscr{V}, \mathscr{E})$ which is a transitive closure. 
\end{itemize}
\vspace{1em}
\textbf{Output}:
\begin{itemize}
\item The collection $\mathscr{A} \equiv \{S \subseteq \mathscr{V}:  \text{ if $i \in S$ and $(i,j) \in \mathscr{E}$, then $j \in S$}\}$. 
\end{itemize} 
\vspace{1em}
\textbf{Procedure}: 
\begin{enumerate}
\item If $\mathscr{V} = \emptyset$:
\begin{enumerate}
\item Output the collection $\mathscr{A} \equiv\{ \emptyset \}$ and terminate. 
\end{enumerate}
\item Otherwise:  
\begin{enumerate}
\item Choose any vertex  $i \in \mathscr{V}$.
\label{step:choose_vertex} 

\item  
Create a copy of $\mathscr{G} \equiv (\mathscr{V},\mathscr{E})$ in which the vertices $\{i \} \cup \{\ell: (\ell,i) \in \mathscr{E} \}$ and the incoming and outgoing edges of these vertices are removed. Denote this new graph by $\mathscr{G}' \equiv (\mathscr{V}', \mathscr{E}')$.\label{step:construct_A1:1}
\item Compute the collection $\mathscr{A}' \leftarrow \textsc{RecursiveStep$(\mathscr{G}')$}$. \label{step:construct_A1:2}

\item 
Create a copy of $\mathscr{G} \equiv (\mathscr{V},\mathscr{E})$ in which the vertices $\{i \} \cup  \{\ell: (i,\ell) \in \mathscr{E} \}$ and the incoming and outgoing edges of these vertices are removed. Denote this new graph by $\mathscr{G}'' \equiv (\mathscr{V}'', \mathscr{E}'')$.  \label{step:construct_A2:1}
\item Compute the collection $\mathscr{A}'' \leftarrow \textsc{RecursiveStep$(\mathscr{G}'')$}$.\label{step:construct_A2:2} 
\item Compute the collection $\mathscr{A}''' \leftarrow \{S \cup \{i \} \cup \{\ell: (i,\ell) \in \mathscr{E} \} :  S \in \mathscr{A}'' \}$.  \label{step:construct_A2:3}

\item Output the collection $\mathscr{A} \equiv\mathscr{A}' \cup \mathscr{A}'''$ and terminate.  \label{step:construct_A}
\end{enumerate}
\end{enumerate}
\vspace{1em}
\end{minipage}}
\end{center}
\caption{A recursive subroutine which is used in Algorithm~\ref{alg:construct_S}. }\label{alg:recursive}
\end{algorithm}%

\begin{claim} \label{claim:appx:idk_intro}
Let  $\mathscr{G} \equiv (\mathscr{V},\mathscr{E})$ be a directed acyclic graph that is a transitive closure. For notational convenience, let  $\mathscr{E}_i$ denote the set of incoming and outgoing edges from each vertex $i \in \mathscr{V}$. Then for each vertex $i \in \mathscr{V}$, we have that  $\tilde{\mathscr{G}} \equiv (\mathscr{V} \setminus  \{i \}, \mathscr{E} \setminus  \mathscr{E}_i)$ is a directed acyclic graph that is a transitive closure.  
\end{claim}
\begin{proof}{Proof of Claim~\ref{claim:appx:idk_intro}.} Let  $\mathscr{G} \equiv (\mathscr{V},\mathscr{E})$ be a directed acyclic graph that is a transitive closure, and let $i \in \mathscr{V}$ be any chosen vertex. It is clear that removing a vertex and its associated incoming and outgoing edges from a directed acyclic graph will not induce any cycles. Therefore, it follows from the fact that the original graph $\mathscr{G} \equiv (\mathscr{V},\mathscr{E})$ is a directed acyclic graph  that the new graph $\tilde{\mathscr{G}} \equiv (\mathscr{V} \setminus \{i\}, \mathscr{E} \setminus \mathscr{E}_i)$ is also a directed acyclic graph. Moreover, consider any two arbitrary vertices $j,k \in \mathscr{V} \setminus \{i\}$ which satisfy the property that vertex $k$ is reachable from vertex $j$  in the new  graph  $\tilde{\mathscr{G}} \equiv (\mathscr{V} \setminus \{i\}, \mathscr{E} \setminus \mathscr{E}_i)$. Then it follows immediately from the construction of the new graph that vertex $k$ is reachable from vertex $j$  in the original graph $\mathscr{G} \equiv (\mathscr{V},\mathscr{E})$. Since the original graph $\mathscr{G} \equiv (\mathscr{V},\mathscr{E})$ is a transitive closure, there  must exist a directed edge from vertex $j$ to vertex $k$ in the original graph. Since neither $j$ nor $k$ are equal to $i$, we have thus shown that there is a directed edge $(j,k) \in \mathscr{E} \setminus \mathscr{E}_i$. Since the two vertices $j,k \in \mathscr{V} \setminus \{i\}$ were chosen arbitrarily, we have shown that the new graph is also a transitive closure. Our proof of Claim~\ref{claim:appx:idk_intro} is thus complete. 
 \halmos
\end{proof}

\begin{claim} \label{claim:appx:idk}
Let  $\mathscr{G} \equiv (\mathscr{V},\mathscr{E})$ be a directed acyclic graph that is a transitive closure. For notational convenience, let  $\mathscr{E}_i$ denote the set of incoming and outgoing edges from each vertex $i \in \mathscr{V}$. Then for each subset of vertices $\mathscr{B} \subseteq \mathscr{V}$, we have that  $\tilde{\mathscr{G}} \equiv (\mathscr{V} \setminus \mathscr{B}, \mathscr{E} \setminus (\cup_{i \in \mathscr{B}} \mathscr{E}_i))$ is a directed acyclic graph that is a transitive closure.  
\end{claim}

\begin{proof}{Proof of Claim~\ref{claim:appx:idk}.} Let  $\mathscr{G} \equiv (\mathscr{V},\mathscr{E})$ be a directed acyclic graph that is a transitive closure, let  $\mathscr{B} \subseteq \mathscr{V}$ be a subset of vertices,  and define the directed graph $\tilde{\mathscr{G}} \equiv (\mathscr{V} \setminus \mathscr{B}, \mathscr{E} \setminus (\cup_{i \in \mathscr{B}} \mathscr{E}_i))$. The rest of the proof follows a straightforward induction argument. Indeed, let the vertices of $\mathscr{B}$ be indexed by $\mathscr{B} \equiv \{i^{\mathscr{B}}_1,\ldots,i^{\mathscr{B}}_{|\mathscr{B}|} \}$.  For each $j \in \{0,1,\ldots,| \mathscr{B}|\}$, we define the following directed graph:
\begin{align*}
{\mathscr{G}}^{\mathscr{B}}_j \equiv  
(\underbrace{\mathscr{V} \setminus \{i^{\mathscr{B}}_1,\ldots, i^{\mathscr{B}}_{j} \}}_{\mathscr{V}^{\mathscr{B}}_j}, \;  \underbrace{\mathscr{E} \setminus \left \{ (k,\ell) \in \mathscr{E}: k \in \{i^{\mathscr{B}}_1,\ldots, i^{\mathscr{B}}_{j} \} \textnormal{ or } \ell \in \{i^{\mathscr{B}}_1,\ldots, i^{\mathscr{B}}_{j} \} \right \}}_{\mathscr{E}^{\mathscr{B}}_j} ).
\end{align*}
We readily observe from the above definition that ${\mathscr{G}}^{\mathscr{B}}_0 = \mathscr{G}$ and that ${\mathscr{G}}^{\mathscr{B}}_{| \mathscr{B}|} = \tilde{\mathscr{G}}$. In the remainder, we will prove by induction that each $\tilde{\mathscr{G}}^{\mathscr{B}}_0,\ldots,\tilde{\mathscr{G}}^{\mathscr{B}}_{| \mathscr{B}|}$ is a directed acyclic graph that is a transitive closure. Our induction proof proceeds as follows. In the base case, it follows from the equality ${\mathscr{G}}^{\mathscr{B}}_0 = \mathscr{G}$ that ${\mathscr{G}}^{\mathscr{B}}_0$ is a directed acyclic graph that is a transitive closure. Next, assume by induction that $\tilde{\mathscr{G}}^{\mathscr{B}}_0,\ldots,\tilde{\mathscr{G}}^{\mathscr{B}}_{j-1}$ are directed acyclic graphs that are transitive closures for any $j \in \{1,\ldots,| \mathscr{B}| \}$. Then we readily observe from the definition of the directed graph $\mathscr{G}^{\mathscr{B}}_j \equiv (\mathscr{V}^{\mathscr{B}}_j, \mathscr{E}^{\mathscr{B}}_j)$ that the following equalities hold:
\begin{align*}
\mathscr{V}^{\mathscr{B}}_j &= \mathscr{V}^{\mathscr{B}}_{j-1} \setminus \{i^{\mathscr{B}}_j\}; & \mathscr{E}^{\mathscr{B}}_j &= \mathscr{E}^{\mathscr{B}}_{j-1} \setminus \left \{(k,\ell) \in \mathscr{E}^{\mathscr{B}}_{j-1}: \; k = i^{\mathscr{B}}_j \textnormal{ or } \ell = i^{\mathscr{B}}_j  \right \}. 
\end{align*}
Using the above equalities, it follows immediately from Claim~\ref{claim:appx:idk_intro} and the induction hypothesis that $\mathscr{G}^{\mathscr{B}}_j $ is a directed acyclic graph that is a transitive closure. This  concludes our induction proof, and since the equality ${\mathscr{G}}^{\mathscr{B}}_{| \mathscr{B}|} = \tilde{\mathscr{G}}$  holds, our proof of Claim~\ref{claim:appx:idk} is complete. \halmos
\end{proof}

\begin{claim}\label{claim:appx:idk_2}
Let  $\mathscr{G} \equiv (\mathscr{V},\mathscr{E})$ be a directed acyclic graph that is a transitive closure. For notational convenience, let  $\mathscr{E}_i$ denote the set of incoming and outgoing edges from each vertex $i \in \mathscr{V}$. For each vertex $i \in \mathscr{V}$, 
\begin{align}
  &\left \{S \subseteq \mathscr{V} \setminus \{i \}:  \textnormal{ if $k \in S$ and $(k,j) \in \mathscr{E}$, then $j \in S$} \right\} \notag \\
        & = \left \{ S \subseteq \mathscr{V} \setminus (\{ i \} \cup \{\ell: (\ell,i) \in \mathscr{E} \}): \textnormal{if } k \in  S  \textnormal{ and } (k,j) \in \mathscr{E} \setminus \left( \mathscr{E}_i \cup \bigcup_{\ell: (\ell,i) \in \mathscr{E}} \mathscr{E}_\ell \right), \textnormal{ then } j \in S \right \}. \label{line:remove_i_from_set}
\end{align}
\end{claim}

\begin{proof}{Proof of Claim~\ref{claim:appx:idk_2}.}
 Let  $\mathscr{G} \equiv (\mathscr{V},\mathscr{E})$ be a directed acyclic graph that is a transitive closure, and let $i \in \mathscr{V}$ be any chosen vertex from this graph. We first observe that 
 \begin{align}
&\left \{S \subseteq \mathscr{V} \setminus \{i \}:  \text{ if $k \in S$ and $(k,j) \in \mathscr{E}$, then $j \in S$} \right\} \label{line:remove_i_from_set_yuck}\\
&= \left \{ S \subseteq \mathscr{V} \setminus (\{ i \} \cup \{\ell: (\ell,i) \in \mathscr{E} \}): \textnormal{if } k \in  S  \textnormal{ and } (k,j) \in \mathscr{E}, \textnormal{ then } j \in S \right \},\label{line:remove_i_from_set_mini}
  \end{align}
where the above 
equality follows from the fact that 
any subset of vertices $S$ from the collection on line~\eqref{line:remove_i_from_set_yuck} must not contain any vertices in the graph that have an outgoing edge to vertex $i$. For notational convenience, let the collection of subsets of vertices on line~\eqref{line:remove_i_from_set} be denoted by  $\tilde{\mathscr{A}}$, and let the collection of subsets of vertices on line~\eqref{line:remove_i_from_set_mini} be denoted by  $\tilde{\mathscr{A}}'$.  In other words, let
\begin{align*}
\tilde{\mathscr{A}}' &\triangleq  \left \{ S \subseteq \mathscr{V} \setminus (\{ i \} \cup \{\ell: (\ell,i) \in \mathscr{E} \}): \textnormal{if } k \in  S  \textnormal{ and } (k,j) \in \mathscr{E}, \textnormal{ then } j \in S \right \}\\
\tilde{\mathscr{A}} &\triangleq  \left \{ S \subseteq \mathscr{V} \setminus (\{ i \} \cup \{\ell: (\ell,i) \in \mathscr{E} \}): \textnormal{if } k \in  S  \textnormal{ and } (k,j) \in \mathscr{E} \setminus \left( \mathscr{E}_i \cup \bigcup_{\ell: (\ell,i) \in \mathscr{E}} \mathscr{E}_\ell \right), \textnormal{ then } j \in S \right \}.
\end{align*}
In the remainder of the proof, we will show that $\tilde{\mathscr{A}} = \tilde{\mathscr{A}}'$. 

To show the first direction, consider any arbitrary subset of vertices $S \in \tilde{\mathscr{A}}'$. For this subset of vertices,  consider any two arbitrary vertices $k, j \in  \mathscr{V} \setminus (\{ i \} \cup \{\ell: (\ell,i) \in \mathscr{E} \})$ which satisfy the conditions [$k \in S$] and [$(k,j) \in \mathscr{E} \setminus ( \mathscr{E}_i \cup \bigcup_{\ell: (\ell,i) \in \mathscr{E}} \mathscr{E}_\ell )$]. In this case, it follows immediately from the fact that $\mathscr{E} \setminus ( \mathscr{E}_i \cup \bigcup_{\ell: (\ell,i) \in \mathscr{E}} \mathscr{E}_\ell ) \subseteq \mathscr{E}$ that the condition [$(k,j) \in \mathscr{E} $] holds. Therefore, it follows immediately from the definition of $\tilde{\mathscr{A}}'$ that  $j \in S$.  Since the two vertices $k, j \in  \mathscr{V} \setminus (\{ i \} \cup \{\ell: (\ell,i) \in \mathscr{E} \})$  which satisfy the conditions [$k \in S$] and [$(k,j) \in \mathscr{E} \setminus ( \mathscr{E}_i \cup \bigcup_{\ell: (\ell,i) \in \mathscr{E}} \mathscr{E}_\ell )$] were chosen arbitrarily, we have shown that $S \in \tilde{\mathscr{A}}$, and since the subset of vertices $S \in \tilde{\mathscr{A}}'$ was chosen arbitrarily, we have shown that  $\tilde{\mathscr{A}}'  \subseteq \tilde{\mathscr{A}}$. Our proof of the first direction is thus complete. 

 To show the other direction, consider any arbitrary subset of vertices $S \in  \tilde{\mathscr{A}}$. For this subset of vertices, consider any two arbitrary vertices $k, j \in  \mathscr{V} \setminus (\{ i \} \cup \{\ell: (\ell,i) \in \mathscr{E} \})$ which satisfy the conditions [$k \in S$] and [$(k,j) \in \mathscr{E}$]. We now suppose, for the sake of developing a contradiction,  that the directed edge from vertex $k$ to vertex $j$ satisfies $ (k,j) \in \mathscr{E}_i \cup \bigcup_{\ell: (\ell,i) \in \mathscr{E}} \mathscr{E}_\ell$. 
Under this supposition, we have two cases to consider. In the first case, the directed edge satisfies   $(k,j) \in \mathscr{E}_i$. However,  that would imply that either  $j = i$ or $k = i$, and those equalities contradict the facts that  $k \in S$ and $S \subseteq \mathscr{V} \setminus (\{ i \} \cup \{\ell: (\ell,i) \in \mathscr{E} \})$. In the second case,  the directed edge satisfies $(k,j) \in \bigcup_{\ell: (\ell,i) \in \mathscr{E}} \mathscr{E}_\ell$, which implies that there are directed edges $(k,j) \in \mathscr{E}$ and $(j,i) \in \mathscr{E}$ in the graph $\mathscr{G}$. Since the graph is a transitive closure, the existence of directed edges $(k,j) \in \mathscr{E}$ and $(j,i) \in \mathscr{E}$ implies that there must exist a directed edge $(k,i) \in \mathscr{E}$. However, the existence of a   directed edge $(k,i) \in \mathscr{E}$ contradicts the facts that  $k \in S$ and $S \subseteq \mathscr{V} \setminus (\{ i \} \cup \{\ell: (\ell,i) \in \mathscr{E} \})$.  In all cases, we have proved by contradiction that the vertices $k$ and $j$ must satisfy the condition [$(k,j) \in \mathscr{E} \setminus ( \mathscr{E}_i \cup \bigcup_{\ell: (\ell,i) \in \mathscr{E}} \mathscr{E}_\ell )$].  Since the two vertices $k, j \in  \mathscr{V} \setminus (\{ i \} \cup \{\ell: (\ell,i) \in \mathscr{E} \})$ which satisfy the conditions [$k \in S$] and [$(k,j) \in \mathscr{E}$] were chosen arbitrarily, we have shown that $S \in \tilde{\mathscr{A}}'$, and since the subset of vertices $S \in \tilde{\mathscr{A}}$ was chosen arbitrarily, we have shown that  $\tilde{\mathscr{A}}  \subseteq \tilde{\mathscr{A}}'$. Our proof of the second direction is thus complete. 

 We have thus shown that $\tilde{\mathscr{A}} = \tilde{\mathscr{A}}'$, which concludes our proof of Claim~\ref{claim:appx:idk_2}. \halmos
\end{proof}

\begin{claim}\label{claim:appx:idk_3}
Let  $\mathscr{G} \equiv (\mathscr{V},\mathscr{E})$ be a directed acyclic graph that is a transitive closure. For notational convenience, let  $\mathscr{E}_i$ denote the set of incoming and outgoing edges from each vertex $i \in \mathscr{V}$. For each vertex $i \in \mathscr{V}$, 
\begin{align}
  &\left \{S \subseteq \mathscr{V}: \left[ i \in S \right] \textnormal{ and } \left[  \textnormal{if $k \in S$ and $(k,j) \in \mathscr{E}$, then $j \in S$} \right] \right\} \notag \\
        & = \left \{ \begin{aligned}
        &S' \cup \{i \} \cup \left\{ \ell: (i,\ell) \in \mathscr{E} \right \} :\\
        &\quad S' \in \left \{ S  \subseteq \mathscr{V} \setminus \left( \{i \} \cup \left\{ \ell: (i,\ell) \in \mathscr{E} \right \}  \right):  \textnormal{if } k \in  S  \textnormal{ and } (k,j) \in \mathscr{E} \setminus \left( \mathscr{E}_i \cup \bigcup_{\ell: (i,\ell) \in \mathscr{E}} \mathscr{E}_\ell \right), \textnormal{ then } j \in S \right \} \end{aligned} \right \}.\label{line:add_i_to_set_goofy}
\end{align}
\end{claim}

\begin{proof}{Proof of Claim~\ref{claim:appx:idk_3}.}
 Let  $\mathscr{G} \equiv (\mathscr{V},\mathscr{E})$ be a directed acyclic graph that is a transitive closure, and let $i \in \mathscr{V}$ be any chosen vertex from this graph. Then we observe that 
 \begin{align}
  &\left \{S \subseteq \mathscr{V}: \left[ i \in S \right] \textnormal{ and } \left[  \textnormal{if $k \in S$ and $(k,j) \in \mathscr{E}$, then $j \in S$} \right] \right\} \label{line:add_i_to_set_yuck} \\
&= \left \{S \subseteq \mathscr{V}: \left[ \{i \} \cup \{\ell: (i,\ell) \in \mathscr{E} \} \subseteq S \right] \textnormal{ and } \left[  \textnormal{if $k \in S$ and $(k,j) \in \mathscr{E}$, then $j \in S$} \right] \right\} \label{line:add_i_to_set_mini}\\
        & = \left \{ S \subseteq \mathscr{V}: \left[ \{i \} \cup \left\{ \ell: (i,\ell) \in \mathscr{E} \right \}   \subseteq S\right] \textnormal{ and } \left[ \textnormal{if } k \in  S  \textnormal{ and } (k,j) \in \mathscr{E} \setminus \left( \mathscr{E}_i \cup \bigcup_{\ell: (i,\ell) \in \mathscr{E}} \mathscr{E}_\ell \right), \textnormal{ then } j \in S \right] \right \} \label{line:add_i_to_set}.
     \end{align}
 Indeed, the  equality of the collections of subsets of vertices on lines~\eqref{line:add_i_to_set_yuck} and \eqref{line:add_i_to_set_mini}   follows from the fact that  
any subset of vertices $S$ in the collection on line~\eqref{line:add_i_to_set_yuck} satisfies $i \in S$, and so it must contain all of the vertices in the graph which have an incoming  edge from vertex $i$.
For notational convenience, let the collections of subsets of vertices on lines~\eqref{line:add_i_to_set_mini}   and  \eqref{line:add_i_to_set} be denoted by $\bar{\mathscr{A}}'$ and $\bar{\mathscr{A}}$, respectively.  To show that  line~\eqref{line:add_i_to_set_mini} is equal to line~\eqref{line:add_i_to_set}, we first observe that the inclusion $\bar{\mathscr{A}}'  \subseteq \bar{\mathscr{A}}$ follows  immediately from the definitions of the collections $\bar{\mathscr{A}}'$  and $\bar{\mathscr{A}} $. To show the other direction, consider any subset of vertices $S \in  \bar{\mathscr{A}}$ and consider any vertex $k \in S$. If there exists a directed edge $(k,j) \in \mathscr{E}$ which satisfies $ (k,j) \in \mathscr{E}_i \cup \bigcup_{\ell: (i,\ell) \in \mathscr{E}} \mathscr{E}_\ell$, then we observe that the vertex $j$ must satisfy $j \in \{i \} \cup \{ \ell: (i,\ell) \in \mathscr{E} \}$. Since it follows from $S \in \bar{\mathscr{A}}$ that the inclusion $\{i\} \cup \{\ell: (i,\ell) \in \mathscr{E} \} \subseteq S$ holds, we conclude that $S \in \bar{\mathscr{A}}'$. We have thus shown that line~\eqref{line:add_i_to_set_mini} is equal to line~\eqref{line:add_i_to_set}. Since the equivalence of lines~\eqref{line:add_i_to_set} and \eqref{line:add_i_to_set_goofy} follows readily from algebra, our proof of Claim~\ref{claim:appx:idk_3} is complete. \halmos
\end{proof}

Using the above intermediary claims, we now prove the correctness of the recursive subroutine in Algorithm~\ref{alg:recursive}. Indeed,  it is clear that the recursive subroutine yields the correct output in the base case where $\mathscr{V} = \emptyset$.  Next, let us assume by induction that the recursive subroutine yields the correct output for all valid input graphs with up to $p-1$ vertices, and consider any valid input $\mathscr{G} \equiv (\mathscr{V},\mathscr{E})$ for which the number of vertices is $|\mathscr{V}| = p$.  Let $i \in \mathscr{V}$ denote the vertex from this graph which is chosen in line~\eqref{step:choose_vertex} of Algorithm~\ref{alg:recursive}, where the existence of such a vertex follows from the fact that we are in the case where $\mathscr{V} \neq \emptyset$. It follows immediately from Claim~\ref{claim:appx:idk}  that the graphs $\mathscr{G}' \equiv (\mathscr{V}', \mathscr{E}')$ and $\mathscr{G}'' \equiv (\mathscr{V}'', \mathscr{E}'')$  constructed on lines~\eqref{step:construct_A1:1} and \eqref{step:construct_A2:1} of Algorithm~\ref{alg:recursive} are directed acyclic graphs which are transitive closures, which implies that these graphs are valid inputs to Algorithm~\ref{alg:recursive} in lines~\eqref{step:construct_A1:2} and \eqref{step:construct_A2:2}. Therefore, it follows from the induction hypothesis and lines~\eqref{step:construct_A1:2}, \eqref{step:construct_A2:2}, and \eqref{step:construct_A2:3} of Algorithm~\ref{alg:recursive} that
\begin{align*}
\mathscr{A}' &= \left \{ S \subseteq \mathscr{V} \setminus (\{ i \} \cup \{\ell: (\ell,i) \in \mathscr{E} \}): \textnormal{if } k \in  S  \textnormal{ and } (k,j) \in \mathscr{E} \setminus \left( \mathscr{E}_i \cup \bigcup_{\ell: (\ell,i) \in \mathscr{E}} \mathscr{E}_\ell \right), \textnormal{ then } j \in S \right \}\\
\mathscr{A}'' &= \left \{ S \subseteq \mathscr{V} \setminus (\{ i \} \cup \{\ell: (i,\ell) \in \mathscr{E} \}): \textnormal{if } k \in  S  \textnormal{ and } (k,j) \in \mathscr{E} \setminus \left( \mathscr{E}_i \cup \bigcup_{\ell: (i,\ell) \in \mathscr{E}} \mathscr{E}_\ell \right), \textnormal{ then } j \in S \right \}\\
\mathscr{A}''' &= \left \{ S \cup \{i \} \cup \left\{ \ell: (i,\ell) \in \mathscr{E} \right \} : S \in \mathscr{A}'' \right \},
\end{align*}
where the induction hypothesis can be applied because  $\mathscr{G}' \equiv (\mathscr{V}', \mathscr{E}')$  and $\mathscr{G}'' \equiv (\mathscr{V}'', \mathscr{E}'')$ are valid inputs to Algorithm~\ref{alg:recursive} and because $| \mathscr{V}'| \le p-1$ and $| \mathscr{V}''| \le p-1$.  Therefore, it follows from Claims~\ref{claim:appx:idk_2} and \ref{claim:appx:idk_3} that
 \begin{align*}
 \mathscr{A}' &=  \left \{S \subseteq \mathscr{V} \setminus \{i \}:  \textnormal{ if $k \in S$ and $(k,j) \in \mathscr{E}$, then $j \in S$} \right\} \\
\mathscr{A}''' &= \left \{S \subseteq \mathscr{V}: \left[ i \in S \right] \textnormal{ and } \left[  \textnormal{if $k \in S$ and $(k,j) \in \mathscr{E}$, then $j \in S$} \right] \right\},
 \end{align*}
 which proves that the output of Algorithm~\ref{alg:recursive} is
 \begin{align*}
 \mathscr{A}' \cup  \mathscr{A}''' =  \left \{S \subseteq \mathscr{V}:  \textnormal{ if $k \in S$ and $(k,j) \in \mathscr{E}$, then $j \in S$} \right\}.
 \end{align*}
 This completes our proof of the correctness of Algorithm~\ref{alg:recursive}.
 
 To conclude our proof of Lemma~\ref{lem:fixed_dim:S:time}, we analyze the computation time of Algorithm~\ref{alg:construct_S}. Indeed, we recall that the computation time required for line~\eqref{step:construct_S:1} in Algorithm~\ref{alg:construct_S} is  $\mathcal{O}(n^2M)$. In what follows, we  assume that all directed graphs are stored as {adjacency lists}.  Under this assumption, our analysis of the computation time for line~\eqref{step:construct_S:2} in Algorithm~\ref{alg:recursive} is split into the following two intermediary claims, denoted by Claim~\ref{claim:mainwork} and \ref{claim:recursivework}. In our first intermediary claim, presented below as Claim~\ref{claim:mainwork}, we establish the computation time required for  lines~\eqref{step:construct_A1:1}, \eqref{step:construct_A2:1}, \eqref{step:construct_A2:3}, and \eqref{step:construct_A} of Algorithm~\ref{alg:recursive}. 
 \begin{claim} \label{claim:mainwork}
 If  $\mathscr{G} \equiv (\mathscr{V},\mathscr{E})$ is a directed acyclic graph that is a transitive closure with $| \mathscr{V}| \ge 1$, then lines~\eqref{step:construct_A1:1}, \eqref{step:construct_A2:1}, \eqref{step:construct_A2:3}, and \eqref{step:construct_A} of Algorithm~\ref{alg:recursive} can be performed in $\mathcal{O}\left(| \mathscr{V}| \times  \left|\textsc{RecursiveStep}(\mathscr{G}) \right| \right)$ time. 
 \end{claim}
 \begin{proof}{Proof.}
 We observe that the set of vertices $\{i \} \cup \{ \ell: (\ell, i) \in \mathscr{E} \}$ in line~\eqref{step:construct_A1:1} of Algorithm~\ref{alg:recursive} can be queried and stored as a hash table in $\mathcal{O}(|\{i \} \cup \{ \ell: (\ell, i) \in \mathscr{E} \}|) = \mathcal{O}(| \mathscr{V}|)$ computation time. Therefore, the directed graph $\mathscr{G}' \equiv (\mathscr{V}',\mathscr{E}')$ in line~\eqref{step:construct_A1:1} of Algorithm~\ref{alg:recursive} can be constructed from scratch in $\mathcal{O}(| \mathscr{V}| + | \mathscr{E}|)$ computation time. By identical reasoning, we observe that the directed graph $\mathscr{G}'' \equiv (\mathscr{V}'',\mathscr{E}'')$ in line~\eqref{step:construct_A2:1} of Algorithm~\ref{alg:recursive} can be constructed  in $\mathcal{O}(| \mathscr{V}| + | \mathscr{E}|)$ computation time. It is easy to see  for each $S \in \mathscr{A}''$ that $S \cup \{i \} \cup \{ \ell: (i,\ell) \in \mathscr{E} \}$ is the union of three disjoint sets, which implies that the collection on line~\eqref{step:construct_A2:3}  of Algorithm~\ref{alg:recursive} can be constructed in a total of $\mathcal{O}\left(\sum_{S \in \mathscr{A}''} ( | S| + | \{ \ell: (i,\ell) \in \mathscr{E} \}|) \right) = \mathcal{O} \left(\sum_{S \in \mathscr{A}'''} |S| \right)$ computation time. It is similarly easy to see that  $\mathscr{A} \equiv \mathscr{A}' \cup \mathscr{A}'''$  is the union of two disjoint collections, which implies that line~\eqref{step:construct_A} of Algorithm~\ref{alg:recursive} can be performed in $\mathcal{O}\left( \sum_{S \in \mathscr{A}'} |S| + \sum_{S \in \mathscr{A}'''} |S|  \right) = \mathcal{O} \left( \sum_{S \in \mathscr{A}} | S| \right)$ computation time. All combined, we have shown that the computation time required for lines~\eqref{step:construct_A1:1}, \eqref{step:construct_A2:1}, \eqref{step:construct_A2:3}, and \eqref{step:construct_A} of Algorithm~\ref{alg:recursive} is
 \begin{align*}
\underbrace{ \mathcal{O}(| \mathscr{V}| + |\mathscr{E}|)}_{\eqref{step:construct_A1:1}} + \underbrace{\mathcal{O}(| \mathscr{V}| + | \mathscr{E}|) }_{\eqref{step:construct_A2:1}}+  \underbrace{\mathcal{O} \left(\sum_{S \in \mathscr{A}'''} |S| \right)}_{\eqref{step:construct_A2:3}} + \underbrace{ \mathcal{O} \left( \sum_{S \in \mathscr{A}} | S| \right)}_{\eqref{step:construct_A}} = \mathcal{O}\left(| \mathscr{V}| + |\mathscr{E}| + \sum_{S \in \mathscr{A}} | S|\right).
 \end{align*}
Using the fact that graphs always satisfy the inequality $| \mathscr{E} | \le | \mathscr{V}|^2$ and the fact that  $|S| \le | \mathscr{V}|$ for all $S \in \mathscr{A}$, the above computation time simplifies to
$ \mathcal{O}\left(| \mathscr{V}|^2 +  \left| \mathscr{V} \right| \times \left| \mathscr{A} \right| \right).$ 
Since $\mathscr{A} \equiv \textsc{RecursiveStep}(\mathscr{G})$ is the output of Algorithm~\ref{alg:recursive}, and since it is easy to see that $| \mathscr{A}| \ge | \mathscr{V}|$, our proof of Claim~\ref{claim:mainwork} is complete. 
\halmos  \end{proof}

In our second intermediary claim, presented below as Claim~\ref{claim:recursivework}, we use Claim~\ref{claim:mainwork} to establish the computation time for Algorithm~\ref{alg:recursive} for any valid input. 
\begin{claim} \label{claim:recursivework}
 If  $\mathscr{G} \equiv (\mathscr{V},\mathscr{E})$ is a directed acyclic graph that is a transitive closure with $| \mathscr{V}| \ge 1$, then the computation time for Algorithm~\ref{alg:recursive} is $\mathcal{O}(|\mathscr{V}|^2 \times \left|\textsc{RecursiveStep}(\mathscr{G}) \right| )$. 
 \end{claim}
 \begin{proof}{Proof.}
 Let $T(\mathscr{G})$ denote the computation time required to run Algorithm~\ref{alg:recursive} for any valid input $\mathscr{G} \equiv (\mathscr{V},\mathscr{E})$. It follows from Claim~\ref{claim:mainwork} that $T(\mathscr{G})$ can be represented by an asymptotic recurrence of the form
  \begin{align}
 T(\mathscr{G}) &= \begin{cases}
| \mathscr{V}| \times  \left|\textsc{RecursiveStep}(\mathscr{G}) \right| + T(\mathscr{G}') + T(\mathscr{G}''),&\text{if } \mathscr{V} \neq \emptyset,\\
  1,&\text{if } \mathscr{V} = \emptyset,
  \end{cases} \label{line:asympt}
 \end{align}
 where $\mathscr{G}'$ and $\mathscr{G}''$ are the subgraphs constructed in Algorithm~\ref{alg:recursive} on lines~\eqref{step:construct_A1:1} and \eqref{step:construct_A2:1}.  We will now prove that the above recurrence satisfies the following inequality:
 \begin{align*}
 T(\mathscr{G})  \le  (|\mathscr{V}|+1)^2  \times \left|\textsc{RecursiveStep}(\mathscr{G}) \right|. 
 \end{align*}
  Indeed, the above inequality clearly holds in the base case. Now assume by induction that the above inequality holds for all valid inputs graphs to Algorithm~\ref{alg:recursive}  with up to $p-1$ vertices, and consider any valid input $\mathscr{G} \equiv (\mathscr{V},\mathscr{E})$ to Algorithm~\ref{alg:recursive}  for which the number of vertices is $|\mathscr{V}| = p$. For this input, we have
 \begin{align*}
 T(\mathscr{G}) &= | \mathscr{V}| \times  \left|\textsc{RecursiveStep}(\mathscr{G}) \right| + T(\mathscr{G}') + T(\mathscr{G}'')\\
 &\le | \mathscr{V}| \times  \left|\textsc{RecursiveStep}(\mathscr{G}) \right| +  (|\mathscr{V}'|+1)^2  \times \left|\textsc{RecursiveStep}(\mathscr{G}') \right|  + (|\mathscr{V}''|+1)^2  \times \left|\textsc{RecursiveStep}(\mathscr{G}'') \right| \\
  &\le | \mathscr{V}| \times  \left|\textsc{RecursiveStep}(\mathscr{G}) \right|  + | \mathscr{V}|^2  \times \left|\textsc{RecursiveStep}(\mathscr{G}') \right| + | \mathscr{V}|^2  \times \left|\textsc{RecursiveStep}(\mathscr{G}'') \right| \\
  &= | \mathscr{V}| \times  \left|\textsc{RecursiveStep}(\mathscr{G}) \right|  + | \mathscr{V}|^2  \times \left|\textsc{RecursiveStep}(\mathscr{G}) \right| \\
    &\le (| \mathscr{V}|+1)^2 \times  \left|\textsc{RecursiveStep}(\mathscr{G}) \right|, 
   \end{align*}
 where the first line follows from \eqref{line:asympt} and from the fact that $\mathscr{V} \neq \emptyset$, the second line follows from the induction hypothesis, the third line follows from the facts that $| \mathscr{V}'| < | \mathscr{V}|$ and $| \mathscr{V}''| < | \mathscr{V}|$, the fourth line follows from the fact that $ \left|\textsc{RecursiveStep}(\mathscr{G}) \right| =  \left|\textsc{RecursiveStep}(\mathscr{G}') \right| +  \left|\textsc{RecursiveStep}(\mathscr{G}'') \right| $, and the final line follows from algebra.  This concludes our proof of Claim~\ref{claim:recursivework}.
 \halmos \end{proof}
 
 Our analysis of the computation time for line~\eqref{step:construct_S:2} of Algorithm~\ref{alg:construct_S} follows readily from Claim~\ref{claim:recursivework}. Indeed,  we observe that the graph $\mathscr{G} \equiv (\mathscr{V}, \mathscr{E})$ that was constructed on line~\eqref{step:construct_S:1} of Algorithm~\ref{alg:construct_S} satisfies $|\mathscr{V}| =  n+1$ and $| \mathscr{E}| \le (n+1)^2$. Therefore, it follows from Claim~\ref{claim:recursivework} that line~\eqref{step:construct_S:2} of Algorithm~\ref{alg:construct_S} requires $\mathcal{O}((n+1)^2 \times \left|\textsc{RecursiveStep}(\mathscr{G}) \right|) = \mathcal{O}(n^2 | \widehat{\mathcal{S}}|)$ computation time. Since the loop on line~\eqref{step:construct_S:3} of Algorithm~\ref{alg:construct_S} can be performed in $\mathcal{O}(\sum_{S \in \widehat{\mathcal{S}}} |S|) = \mathcal{O}(n | \widehat{\mathcal{S}}|)$ computation time, we have shown that the total computation time for Algorithm~\ref{alg:construct_S} is
 \begin{align*}
 \underbrace{\mathcal{O}(n^2M)}_{\eqref{step:construct_S:1}} +  \underbrace{\mathcal{O}(n^2 |\widehat{\mathcal{S}}|)}_{\eqref{step:construct_S:2}} +  \underbrace{\mathcal{O}(n |\widehat{\mathcal{S}}|)}_{\eqref{step:construct_S:3}} = \mathcal{O}\left(n^2 (M +  |\widehat{\mathcal{S}}|) \right). 
 \end{align*}
This concludes our proof of Lemma~\ref{lem:fixed_dim:S:time}. \halmos

\section{Proof of Technical Results from \S\ref{sec:characterization:reverse}} \label{appx:not_improveable}

There are two technical results in \S\ref{sec:characterization:reverse}:  Lemma~\ref{lem:S_hat_for_S_tilde} and Theorem~\ref{thm:not_improveable}. Lemma~\ref{lem:S_hat_for_S_tilde} follows immediately from the definition of the collection of assortments $\widehat{S}$ from \S\ref{sec:characterization}, and so the proof of Lemma~\ref{lem:S_hat_for_S_tilde} is omitted. The remainder of this appendix thus contains the proof of  Theorem~\ref{thm:not_improveable}, which is split into several steps. We will begin by choosing any arbitrary $\bar{S} \in \widehat{\mathcal{S}}$ and, for that assortment,  we will define a particular realization of the historical data $v \equiv (v_{m,i}: m \in \mathcal{M}, i \in S_m)$. We will then prove in Lemma~\ref{lem:singleton} in Appendix~\ref{appx:not_improveable:singleton} for that realization of the historical data that the corresponding set of feasible solutions for the linear optimization problem~\eqref{prob:robust_simplified} is a singleton.  By showing that the realization of the historical data generates a unique feasible solution for the linear optimization problem~\eqref{prob:robust_simplified}, we conclude in Appendix~\ref{appx:not_improveable:finale} by showing that $\bar{S}$ is the unique optimal solution of the robust optimization problem~\eqref{prob:robust}.  

\subsection{Assumptions and Notation} \label{appx:not_improveable:prelim}
Following the statement of Theorem~\ref{thm:not_improveable}, we will assume throughout Appendix~\ref{appx:not_improveable} that  $\mathscr{M} = \tilde{\mathcal{S}}$, that $\eta = 0$, and that the revenues $r_1 < \cdots < r_n$ are fixed.  We observe that Theorem~\ref{thm:not_improveable} trivially holds in the case where $n = 1$; therefore, we assume throughout the rest of  Appendix~\ref{appx:not_improveable} that the number of products satisfies $n \ge 2$.\footnote{Suppose the number of products satisfies $n = 1$.  In that case,  we observe that $\tilde{\mathcal{S}} = \{\{0,1\}\}$. Since $\tilde{\mathcal{S}}$ is a singleton, Theorem~\ref{thm:not_improveable} follows immediately from Theorem~\ref{thm:main}.}  

We will make use of the following notation and preliminary results throughout Appendix~\ref{appx:not_improveable}.  First, it follows from the fact that $\mathscr{M} = \tilde{\mathcal{S}}$  and from  Definition~\ref{defn:L} that the set of tuples $\mathcal{L}$ can be represented compactly as\footnote{A formal proof of the equality on line~\eqref{line:defn_L_for_not_improveable} can be found as Lemma~\ref{lem:reform_of_L} in \S\ref{sec:algorithms:nested}.}
\begin{align}
\mathcal{L} = \left \{ (i_1,\ldots,i_n): i_1 \in \{0,n\} \text{ and } i_{j+1} \in \{j, i_j \} \text{ for all } j \in \{1,\ldots,n-1\} \right \}. \label{line:defn_L_for_not_improveable}
\end{align}
Second, for each $j \in \{1,\ldots,n-1\}$, we define the  indices
\begin{align*}
\ubar{i}^j_1 = \cdots = \ubar{i}^j_j = 0, \quad \ubar{i}^j_{j+1} = \cdots = \ubar{i}^j_n = j, \\
\bar{i}^j_1 = \cdots = \bar{i}^j_j = n, \quad \bar{i}^j_{j+1} = \cdots = \bar{i}^j_n = j,
\end{align*}
and we denote the tuples constructed from those indices by
\begin{align*}
 \left(\ubar{i}^j_1,\ldots,\ubar{i}^j_n \right)  \triangleq (0,\ldots,0,j,\ldots,j) \textnormal{ and }  \left(\bar{i}^j_1,\ldots,\bar{i}^j_n \right) \triangleq (n,\ldots,n,j,\ldots,j).
\end{align*}
Finally, it follows from the fact that $\eta = 0$ that the constraints of the linear  optimization problem~\eqref{prob:robust_simplified} can be written compactly as 
\begin{equation}  \label{line:system_feas}
 \begin{aligned}
&&&\sum_{(i_1,\ldots,i_n) \in \mathcal{L}: i_m = i} \lambda_{i_1 \cdots i_n} = v_{m,i}&& \forall m \in \mathcal{M}, \; i \in S_m\\
&&& \sum_{(i_1,\ldots,i_n) \in \mathcal{L}} \lambda_{i_1 \cdots i_n} = 1 \\
&&& \lambda_{i_1 \cdots i_n} \ge 0&& \forall (i_1,\ldots,i_n) \in \mathcal{L},
\end{aligned}
\end{equation}
where we restate for the sake of convenience  that $S_m \equiv \{0,1,\ldots,m-1,n\}$ and $\mathcal{M} \equiv \{1,\ldots,n\}$. 
\subsection{Construction of historical data for any given $\bar{S} \in \widehat{\mathcal{S}}$,  and proof that the corresponding system~\eqref{line:system_feas} has a unique solution} \label{appx:not_improveable:singleton}

In view of the strategy outlined at the beginning of Appendix~\ref{appx:not_improveable}, we now present the details of the proof of Theorem~\ref{thm:not_improveable}. Indeed, consider any arbitrary assortment $\bar{S} \in \widehat{\mathcal{S}}$.  For that assortment, we define the realization of the historical data $v_{m,i}$ for each past assortment $m \in \mathcal{M} \equiv \{1,\ldots,n\}$ and each product $i \in S_m$ as follows:
\begin{align} \label{line:defn_of_v}
v_{m,i} \triangleq \begin{cases}
\frac{1}{n}\left(1 + \left| \left \{ m,\ldots,n-1 \right \} \setminus \bar{S} \right|\right),&\text{if } i = n,\\
\frac{1}{n} \left| \left \{ m,\ldots,n-1 \right \} \cap \bar{S} \right| ,&\text{if } i = 0,\\
\frac{1}{n},&\text{otherwise}.
\end{cases}
\end{align}
With the  realization of the historical data from line~\eqref{line:defn_of_v} corresponding to the assortment $\bar{S}$, the remainder of Appendix~\ref{appx:not_improveable:singleton} is dedicated to proving the following lemma:
\begin{lemma} \label{lem:singleton}
The system~\eqref{line:system_feas} is nonempty and has a unique solution $\lambda$ which satisfies the following equality for each $(i_1,\ldots,i_n) \in \mathcal{L}$:
\begin{align}
\lambda_{i_1 \cdots i_n} &= \begin{cases}
\frac{1}{n},&\textnormal{if there exists }  j \in \bar{S}\setminus \{0,n\} \textnormal{ such that } (i_1,\ldots,i_n) =  \left(\ubar{i}^j_1,\ldots,\ubar{i}^j_n \right),\\
\frac{1}{n},&\textnormal{if there exists } j \in \mathcal{N} \setminus \bar{S} \textnormal{ such that }(i_1,\ldots,i_n) =  \left(\bar{i}^j_1,\ldots,\bar{i}^j_n \right) ,\\
\frac{1}{n},&\textnormal{if }  (i_{1}, \ldots, i_n) = (n,\ldots,n),\\
0,&\textnormal{otherwise}.
\end{cases} \label{line:singleton:defn_lambda}
\end{align}
\end{lemma}
Stated in words, the above Lemma~\ref{lem:singleton} establishes that the vector $\lambda$ defined by line~\eqref{line:singleton:defn_lambda} is the unique feasible solution for the linear optimization problem~\eqref{prob:robust_simplified} under the assumptions that the past assortments are $\mathscr{M} = \tilde{\mathcal{S}}$,  that $\eta = 0$, and that the realization of the historical data satisfies \eqref{line:defn_of_v}. 

Let us provide an interpretation of the vector $\lambda$ defined by line~\eqref{line:singleton:defn_lambda}. Speaking informally, the vector $\lambda$ is comprised of three types of customers. The first type of customers are denoted by the tuples $ \left(\ubar{i}^j_1,\ldots,\ubar{i}^j_n \right)$ for each  $ j \in \bar{S}\setminus \{0,n\} $, and these are  customers that  prefer product $j$ and, if that product is unavailable, will not make a purchase. The second type of customers are denoted by the  tuples $ \left(\bar{i}^j_1,\ldots,\bar{i}^j_n \right)$ for each  $j \in \mathcal{N} \setminus \bar{S}$, and these are  customers whose most  preferred product is $j$, whose second-most preferred product is $n$, and will not purchase a product if neither $j$ nor $n$ are available. The third type of customers are denoted by the tuple $(n,\ldots,n)$, and these are customers that will purchase product $n$ if it is available, and will otherwise not make a purchase.  In Appendix~\ref{appx:not_improveable:finale}, we will use Lemma~\ref{lem:singleton} to prove Theorem~\ref{thm:not_improveable}.


Our proof of Lemma~\ref{lem:singleton} makes use of two intermediary claims,  denoted below as Claims~\ref{claim:singleton_1} and \ref{claim:singleton_2}. These intermediary claims establish that any feasible solution of the system~\eqref{line:system_feas} must satisfy certain structural properties. We first present these two intermediary claims, followed by the proof of   Lemma~\ref{lem:singleton}. 
\begin{claim} \label{claim:singleton_1}
Suppose that $\lambda$ is a feasible solution of the system \eqref{line:system_feas}, and let $(i_1,\ldots,i_n) \in \mathcal{L}$. If $\lambda_{i_1 \cdots i_n} > 0$, then either $(i_1,\ldots,i_n) = (n,\ldots,n)$ or there exists $j \in \{1,\ldots,n-1\}$ such that $i_{j+1} = j$. 
\end{claim}
\begin{proof}{Proof of Claim~\ref{claim:singleton_1}.}
Suppose that $\lambda$ is a solution of the system \eqref{line:system_feas}, and let $(i_1,\ldots,i_n) \in \mathcal{L}$. 

We begin by showing that if $\lambda_{i_1 \cdots i_n} > 0$, then either $(i_1,\ldots,i_n) = (n,\ldots,n)$ or there exists $m \in \{1,\ldots,n\}$ such that $i_m \notin \{0,n\}$.  We will show this claim by proving its contraposition.  Indeed, 
suppose that $i_m \in \{0,n\}$ for all $m \in \{1,\ldots,n\}$ and that there exists $\bar{m} \in \{1,\ldots,n\}$ such that $i_{\bar{m}} = 0$. In this case,  it follows from  \eqref{line:defn_L_for_not_improveable}  and from the fact that $i_m \in \{0,n\}$ for all $m \in \{1,\ldots,n\}$  that $i_1 = \cdots = i_n = 0$.  If $i_1 = \cdots = i_n = 0$, then  
\begin{align*}
\lambda_{i_1 \cdots i_n} = \sum_{(i_1',\ldots,i_n') \in \mathcal{L}: i_{n}' = 0} \lambda_{i_1' \cdots i_n'} = v_{n,0} = 0,
\end{align*}
 where the first equality follows from  \eqref{line:defn_L_for_not_improveable}, the second equality follows from the fact that $\lambda$ is a feasible solution of the system \eqref{line:system_feas}, and the third equality follows from \eqref{line:defn_of_v}. We have thus shown that if $i_m \in \{0,n\}$ for all $m \in \{1,\ldots,n\}$ and if there exists $m \in \{1,\ldots,n\}$ such that $i_m = 0$, then  $\lambda_{i_1 \cdots i_n} = 0$. Note that the contrapositive of this statement is that if $\lambda_{i_1 \cdots i_n} > 0$, then  either $(i_1,\ldots,i_n) = (n,\ldots,n)$ or there exists $m \in \{1,\ldots,n\}$ such that $i_m \notin \{0,n\}$.

We conclude the proof of Claim~\ref{claim:singleton_1} by showing that if there exists $m \in \{1,\ldots,n\}$ such that $i_m \notin \{0,n\}$, then there exists $j \in \{1,\ldots,n-1\}$ such that $i_{j+1} = j$. Indeed, suppose there exists $m \in \{1,\ldots,n\}$ such that $i_m \notin \{0,n\}$. Since the inclusion $i_1 \in \{0,n\}$ always holds, the existence of $m \in \{1,\ldots,n\}$ satisfying $i_m \notin \{0,n\}$ implies that there must  exist $j \in \{1,\ldots,n-1\}$ such that $i_j \neq i_{j+1}$.  Since we have assumed that the inclusion $(i_1,\ldots,i_n) \in \mathcal{L}$ holds, it follows from  \eqref{line:defn_L_for_not_improveable} and from the fact that $i_j \neq i_{j+1}$ that $i_{j+1}$ must be equal to $j$. Our proof of  Claim~\ref{claim:singleton_1} is thus complete.  \halmos \end{proof}


\begin{claim} \label{claim:singleton_2}
Suppose that $\lambda$ is a feasible  solution of the system \eqref{line:system_feas}, and let $(i_1,\ldots,i_n) \in \mathcal{L}$.  If $\lambda_{i_1 \cdots i_n} > 0$ and if $(i_1,\ldots,i_n) \neq (n,\ldots,n)$, then  there exists $j \in \{1,\ldots,n-1\}$ and $\dagger \in \{0,n\}$ such that $i_1 = \cdots = i_j = \dagger$ and $i_{j+1} =  \cdots = i_n = j$. 
\end{claim}
\begin{proof}{Proof of Claim~\ref{claim:singleton_2}.}

Suppose that $\lambda$ is a feasible solution of the system \eqref{line:system_feas}, and suppose that the inequality $\lambda_{i_1 \cdots i_n} > 0$  holds for a given tuple $(i_1,\ldots,i_n) \in \mathcal{L}$ with $(i_1,\ldots,i_n) \neq (n,\ldots,n)$. Then it follows from Claim~\ref{claim:singleton_1} that there exists a product $j' \in \{1,\ldots,n-1\}$ such that $i_{j'+1} = j'$. From this point onward, we let  the  product $j' \in \{1,\ldots,n-1\}$ with the smallest index that satisfies $i_{j'+1} = j'$  be denoted by
\begin{align}
j \triangleq \min \left\{j' \in \{1,\ldots,n-1\}: i_{j' + 1} = j' \right \} = \min \left\{m \in \{1,\ldots,n-1\}: i_{m+1} \notin \{0,n\} \right \}, \label{line:defn_j_in_singleton}
\end{align}
where the last equality follows from  \eqref{line:defn_L_for_not_improveable}. 

For the product $j$ defined in \eqref{line:defn_j_in_singleton}, it follows from the fact that $\lambda$ is a feasible solution of the system~\eqref{line:system_feas} that the following two equalities must hold:
\begin{align}
\sum_{(i_1',\ldots,i_n') \in \mathcal{L}: i_{j+1}' = j} \lambda_{i_1' \cdots i_n'} = v_{j+1,j}; \quad   \sum_{(i_1',\ldots,i_n') \in \mathcal{L}: i_n' = j} \lambda_{i_1' \cdots i_n'} = v_{n,j}. \label{line:dabadee}
\end{align}
We observe from \eqref{line:defn_of_v} and from the inclusion  $j \in \{1,\ldots,n-1\}$ that the equality $ v_{j+1,j} =  v_{n,j}$ also holds. Therefore,  \eqref{line:dabadee} implies that 
\begin{align}
\sum_{(i_1',\ldots,i_n') \in \mathcal{L}: i_{j+1}' = j} \lambda_{i_1' \cdots i_n'} -   \sum_{(i_1',\ldots,i_n') \in \mathcal{L}: i_n' = j} \lambda_{i_1' \cdots i_n'} = 0.\label{line:dabadee_2}
\end{align}
Moreover,  we observe from \eqref{line:defn_L_for_not_improveable} and from the inclusion  $j \in \{1,\ldots,n-1\}$ that a tuple $(i_1',\ldots,i_n') \in \mathcal{L}$ satisfies the equality $i_n' = j$  if and only if $i_{j+1}' = \cdots = i_n' = j$. Therefore, \eqref{line:dabadee_2} implies that
\begin{align}
\sum_{(i_1',\ldots,i_n') \in \mathcal{L}: \; \substack{ i_{j+1}' = j \text{ and } i_{k}' \neq i_{k+1}' \text{ for }\\\text{some } k \in \{j+1,\ldots,n-1\} } } \lambda_{i_1' \cdots i_n'} = 0. \label{line:dabadee_3}
\end{align}
Since $\lambda$ is a feasible solution of the system~\eqref{line:system_feas}, and since \eqref{line:system_feas} implies that the inequality $\lambda_{i_1' \cdots i_n'} \ge 0$ holds for all $(i_1',\ldots,i_n') \in \mathcal{L}$, it follows from \eqref{line:dabadee_3} that
\begin{align}
\lambda_{i_1' \cdots i_n'} = 0 \quad \forall (i_1',\ldots,i_n') \in \mathcal{L} \textnormal{ such that } i_{j+1}' = j \text{ and } i_{k}' \neq i_{k+1}'   \text{ for some } k \in \{j+1,\ldots,n-1\}.  \label{line:dabadee_4}
\end{align}

Since we supposed that $\lambda_{i_1 \cdots i_n} > 0$, and since the tuple $(i_1,\ldots,i_n) \in \mathcal{L}$ satisfied  the equality $i_{j+1} = j$, \eqref{line:dabadee_4} proves that the equalities $i_{j+1} = \cdots = i_n = j$ must hold. Moreover, it follows from the right-most equality in \eqref{line:defn_j_in_singleton} that $i_1,\ldots,i_j \in \{0,n\}$, which together with \eqref{line:defn_L_for_not_improveable} implies that $i_1 = \cdots = i_j$. Since \eqref{line:defn_L_for_not_improveable} implies that $i_1 \in \{0,n\}$,  our proof of Claim~\ref{claim:singleton_2} is complete. 
\halmos \end{proof}

In view of Claims~\ref{claim:singleton_1} and \ref{claim:singleton_2}, we now present the proof of Lemma~\ref{lem:singleton}.

\begin{proof}{Proof of Lemma~\ref{lem:singleton}.}

We begin by showing that if $\lambda$ is a feasible solution of the system~\eqref{line:system_feas}, then $\lambda$ must satisfy the equality on line~\eqref{line:singleton:defn_lambda} for all $(i_1,\ldots,i_n) \in \mathcal{L}$. 
Indeed, suppose that $\lambda$ is a feasible  solution of the system \eqref{line:system_feas}. In Claim~\ref{claim:singleton_2}, we showed that if $\lambda$ is a feasible  solution of the system \eqref{line:system_feas} and if $\lambda_{i_1 \cdots i_n} > 0$ for some tuple $(i_1,\ldots,i_n) \in \mathcal{L}$ where $(i_1,\ldots,i_n) \neq (n,\ldots,n)$, then the tuple $(i_1,\ldots,i_n)$ must be contained in 
\begin{align*}
{\mathcal{L}}^* &\triangleq \left \{(i_1',\ldots,i_n'): 
 \textnormal{ there exists } j \in \{1,\ldots,n-1\} \textnormal{ such that  }  i_1' = \cdots = i_{j}' = 0 \textnormal{ and } i_{j+1}' = \cdots = i_{n}' = j \right \} \\
 &\quad \cup \left \{(i_1',\ldots,i_n'):  \textnormal{ there exists } j \in \{1,\ldots,n-1\} \textnormal{ such that  }  i_1' = \cdots = i_{j}' = n \textnormal{ and } i_{j+1}' = \cdots = i_{n}' = j \right \}\\
&=  \left \{ \left(\ubar{i}^1_1,\ldots,\ubar{i}^1_n \right), \ldots, \left(\ubar{i}^{n-1}_1,\ldots,\ubar{i}^{n-1}_1\right),\left(\bar{i}^1_1,\ldots,\bar{i}^1_n \right), \ldots, \left(\bar{i}^{n-1}_1,\ldots,\bar{i}^{n-1}_1\right) \right \},
\end{align*} 
where the notation $\left(\ubar{i}^j_1,\ldots,\ubar{i}^j_n \right) \triangleq (0,\ldots,0,j,\ldots,j) $ and $\left(\bar{i}^j_1,\ldots,\bar{i}^j_n \right) \triangleq (n,\ldots,n,j,\ldots,j)$ for each $j \in \{1,\ldots,n-1\}$ is defined in Appendix~\ref{appx:not_improveable:prelim}. 
In view of the above notation, we now consider any arbitrary $j \in \{1,\ldots,n-1\}$ and $\dagger \in \{0,n\}$.  We first observe that 
\begin{align}
v_{j,\dagger} = \sum_{(i_1',\ldots,i_n') \in \mathcal{L}: i_{j}' = \dagger} \lambda_{i_1' \cdots i_n'}  =  \sum_{(i_1',\ldots,i_n') \in \mathcal{L}^*: i_{j}' = \dagger} \lambda_{ {i}_1' \cdots {i}_n'}= \begin{cases}
  \sum_{j'=j}^{n-1} \lambda_{ \ubar{i}_1^{j'} \cdots \ubar{i}_n^{j'}}, &\text{if } \dagger = 0,\\
 \lambda_{n \cdots n } +   \sum_{j'=j}^{n-1} \lambda_{ \bar{i}_1^{j'} \cdots \bar{i}_n^{j'}}, &\text{if } \dagger = n,
    \end{cases} \label{line:singleton_3:j}
\end{align}
where the first equality follows from the fact that $\lambda$ is a feasible  solution of the system \eqref{line:system_feas}, the second equality holds because $\lambda_{i_1'\cdots i_n'} > 0$ only if $(i_1',\ldots,i_n') \in \mathcal{L}^* \cup \{(n,\ldots,n)\}$, and the third equality holds because a tuple $(i_1',\ldots,i_n') \in \mathcal{L}^* \cup \{(n,\ldots,n)\}$ satisfies the equality $i_j' = \dagger$ if and only if $[(i_1',\ldots,i_n') = (n,\ldots,n) \text{ and } \dagger = n]$ or $[ \dagger = 0 \textnormal{ and there exists  $j' \in \{j,\ldots,n\}$ such that }(i_1',\ldots,i_n') = (\ubar{i}_1^{j'},\ldots,\ubar{i}_n^{j'})  ]$ or $[ \dagger = n \textnormal{ and there exists  $j' \in \{j,\ldots,n\}$ such that }(i_1',\ldots,i_n') = (\bar{i}_1^{j'},\ldots,\bar{i}_n^{j'})  ]$. Using identical reasoning as above, we also observe that  
\begin{align}
v_{j+1,\dagger} = \sum_{(i_1',\ldots,i_n') \in \mathcal{L}: i_{j+1}' = \dagger} \lambda_{i_1' \cdots i_n'}  =  \sum_{(i_1',\ldots,i_n') \in \mathcal{L}^*: i_{j+1}' = \dagger} \lambda_{ {i}_1' \cdots {i}_n'}=   \begin{cases}
  \sum_{j'=j+1}^{n-1} \lambda_{ \ubar{i}_1^{j'} \cdots \ubar{i}_n^{j'}}, &\text{if } \dagger = 0,\\
    \lambda_{n \cdots n } +  \sum_{j'=j+1}^{n-1} \lambda_{ \bar{i}_1^{j'} \cdots \bar{i}_n^{j'}}, &\text{if } \dagger = n.
    \end{cases} \label{line:singleton_3:j_plus_one}
\end{align}
Combining lines~\eqref{line:singleton_3:j} and \eqref{line:singleton_3:j_plus_one}, we have shown that
\begin{align}
v_{j,\dagger}-v_{j+1,\dagger} = \begin{cases}
\lambda_{ \ubar{i}_1^{j} \cdots \ubar{i}_n^{j}}, &\text{if } \dagger = 0,\\
 \lambda_{ \bar{i}_1^{j} \cdots \bar{i}_n^{j}}, &\text{if } \dagger = n.
    \end{cases} \label{line:singleton_3:v_minus_v_1}
\end{align}
We also recall from \eqref{line:defn_of_v} that
\begin{align}
v_{j,\dagger}-v_{j+1,\dagger} &= \begin{cases}
\frac{1}{n},&\text{if } \dagger = 0 \text{ and } j \in \bar{S},\\
\frac{1}{n},&\text{if } \dagger = n \text{ and } j \notin \bar{S},\\
0,&\text{otherwise}. 
\end{cases} \label{line:singleton_3:v_minus_v_2}
\end{align}
Therefore, combining lines~\eqref{line:singleton_3:v_minus_v_1} and \eqref{line:singleton_3:v_minus_v_2}, we have shown that the following equalities hold for all $j \in \{1,\ldots,n-1\}$:
\begin{align}
 \lambda_{ \ubar{i}_1^{j} \cdots \ubar{i}_n^{j}} &= \begin{cases}
 \frac{1}{n},&\text{if } j \in \bar{S},\\
 0,&\text{otherwise};
 \end{cases} &  
 \lambda_{ \bar{i}_1^{j} \cdots \bar{i}_n^{j}} &= \begin{cases}
 \frac{1}{n},&\text{if } j \notin \bar{S},\\
 0,&\text{otherwise};
 \end{cases} \label{line:dabadee_infinity}
\end{align}
Moreover, it follows from the fact that $\lambda$ is a feasible solution of the system~\eqref{line:system_feas} that
\begin{align*}
\lambda_{n \cdots n} &= 1 - \sum_{(i_1',\ldots,i_n') \in \mathcal{L}:  (i_1',\ldots,i_n') \neq (n,\ldots,n)} \lambda_{i_1' \cdots i_n'}\\
&= 1 - \sum_{(i_1',\ldots,i_n') \in \mathcal{L}^*} \lambda_{i_1' \cdots i_n'}\\
&= 1 - \sum_{j' \in \bar{S} \setminus \{0,n\}} \lambda_{\ubar{i}_1^{j'} \cdots \ubar{i}_n^{j'}} - \sum_{j' \in \mathcal{N} \setminus \bar{S}} \lambda_{\ubar{i}_1^{j'} \cdots \ubar{i}_n^{j'}}\\
&= 1 - \sum_{j' \in \bar{S} \setminus \{0,n\}}\frac{1}{n} - \sum_{j' \in \mathcal{N} \setminus \bar{S}} \frac{1}{n}\\
&= 1 - \frac{n-1}{n} \\
&= \frac{1}{n},
\end{align*}
where the first equality follows from the fact that  $\lambda$ is a feasible solution of the system~\eqref{line:system_feas}, the second equality follows from the fact that $\lambda_{i_1' \cdots i_n'} > 0$ for a tuple $(i_1',\ldots,i_n') \in \mathcal{L} \setminus \{ (n,\ldots,n)\}$ if and only if $(i_1',\ldots,i_n') \in \mathcal{L}^*$,  the third and fourth equalities follow from \eqref{line:dabadee_infinity}, and the fourth and fifth equalities follow from algebra. 
Combining the above analysis, we conclude that any feasible solution $\lambda$ of the system~\eqref{line:system_feas} must satisfy the following equality must hold for all $(i_1,\ldots,i_n) \in \mathcal{L}$:
\begin{align*}
\lambda_{i_1 \cdots i_n} &= \begin{cases}
\frac{1}{n},&\textnormal{if there exists }  j \in \bar{S}\setminus \{0,n\} \textnormal{ such that } (i_1,\ldots,i_n) =  \left(\ubar{i}^j_1,\ldots,\ubar{i}^j_n \right),\\
\frac{1}{n},&\textnormal{if there exists } j \in \mathcal{N} \setminus \bar{S} \textnormal{ such that }(i_1,\ldots,i_n) =  \left(\bar{i}^j_1,\ldots,\bar{i}^j_n \right) ,\\
\frac{1}{n},&\textnormal{if }  (i_{1}, \ldots, i_n) = (n,\ldots,n),\\
0,&\textnormal{otherwise}.
\end{cases} \tag{\ref{line:singleton:defn_lambda}}
\end{align*}

We  conclude the proof of Lemma~\ref{lem:singleton} by showing, for the sake of completeness, that if $\lambda$  satisfies the equality on line~\eqref{line:singleton:defn_lambda} for all $(i_1,\ldots,i_n) \in \mathcal{L}$, then $\lambda$ is a feasible solution of the system~\eqref{line:system_feas}. Indeed, let $\lambda$ satisfy the equality on line~\eqref{line:singleton:defn_lambda} for all $(i_1,\ldots,i_n) \in \mathcal{L}$. For each assortment $m \in \mathcal{M}$ and product $i \in S_m$, we observe that
\begin{align*}
\sum_{(i_1,\ldots,i_n) \in \mathcal{L}: i_m = i} \lambda_{i_1 \cdots i_n} 
&= \begin{cases}
\frac{1}{n}, &\text{if } i \in \mathcal{N} \setminus \bar{S},\\
\frac{1}{n}, &\text{if } i \in \bar{S} \setminus \{0,n\},\\
 \frac{1}{n} + \sum_{j \in \mathcal{N} \setminus \bar{S}: j \ge m} \frac{1}{n} ,&\text{if } i = n,\\
\sum_{j \in  \bar{S} \setminus \{0,n\}: j \ge m} \frac{1}{n},&\text{if } i = 0
\end{cases}\\
&= \begin{cases}
\frac{1}{n}, &\text{if } i \in \mathcal{N} \setminus \{0,n\},\\
\frac{1}{n}\left(1 + \left| \left \{ m,\ldots,n-1 \right \} \setminus \bar{S} \right|\right),&\text{if } i = n,\\
\frac{1}{n} \left| \left \{ m,\ldots,n-1 \right \} \cap \bar{S} \right| ,&\text{if } i = 0,\\
\end{cases}\\
&= v_{m,i},
\end{align*}
where the first equality follows from the fact that $\lambda$ satisfies the equality on line~\eqref{line:singleton:defn_lambda} for all $(i_1,\ldots,i_n) \in \mathcal{L}$, the second equality follows from algebra, and the third equality follows from line~\eqref{line:defn_of_v}. Furthermore, we observe that 
\begin{align*}
\sum_{(i_1,\ldots,i_n) \in \mathcal{L}} \lambda_{i_1 \cdots i_n} &= \sum_{j \in \bar{S} \setminus \{0,n\}} \frac{1}{n} 
 + \sum_{j \in \mathcal{N} \setminus  \bar{S}} \frac{1}{n} + \frac{1}{n}  = \frac{| \mathcal{N} \setminus \{0,n\}|}{n} + \frac{1}{n}= 1,\end{align*}
 where the first equality follows from  the fact that $\lambda$ satisfies the equality on line~\eqref{line:singleton:defn_lambda} for all $(i_1,\ldots,i_n) \in \mathcal{L}$, and the second and third equalities follows from algebra. Finally, we observe from  line~\eqref{line:singleton:defn_lambda}  that $\lambda_{i_1 \cdots i_n} \ge 0$ for all $(i_1,\ldots,i_n) \in \mathcal{L}$. We have thus proven that $\lambda$ is a feasible solution of the  system~\eqref{line:system_feas}, which concludes our proof of Lemma~\ref{lem:singleton}. 
\halmos \end{proof}

\subsection{Proof of Theorem~\ref{thm:not_improveable}} \label{appx:not_improveable:finale}

We now use the above Lemma~\ref{lem:singleton} to conclude our proof of Theorem~\ref{thm:not_improveable}. We begin by using Lemma~\ref{lem:singleton}  to reformulate the robust optimization problem~\eqref{prob:robust}. Indeed,
\begin{align}
 \eqref{prob:robust} &= \max_{S \in \mathcal{S}} \left \{  \begin{aligned}
& \; \underset{\lambda}{\textnormal{minimize}} && \sum_{(i_1,\ldots,i_M) \in \mathcal{L}} \rho_{i_1 \cdots i_M}(S) \lambda_{i_1\cdots i_M}\\
&\textnormal{subject to}&& \sum_{(i_1,\ldots,i_n) \in \mathcal{L}: i_m = i} \lambda_{i_1 \cdots i_n} = v_{m,i}&& \forall m \in \mathcal{M}, \; i \in S_m\\
&&& \sum_{(i_1,\ldots,i_n) \in \mathcal{L}} \lambda_{i_1 \cdots i_n} = 1 \\
&&& \lambda_{i_1 \cdots i_n} \ge 0&& \forall (i_1,\ldots,i_n) \in \mathcal{L},
\end{aligned} \right \} \notag \\
&= \max_{S \in \mathcal{S}} \left \{\frac{1}{n} \sum_{j \in \bar{S} \setminus \{0,n\}} \rho_{\ubar{i}^j_1,\ldots,\ubar{i}^j_n}(S) + \frac{1}{n} \sum_{j \in\mathcal{N} \setminus \bar{S}} \rho_{\bar{i}^j_1,\ldots,\bar{i}^j_n}(S)  + \frac{1}{n} \rho_{n \cdots n}(S) \right \},\label{prob:not_improveable:intermediary_prob}
\end{align}
where the first equality follows from  Proposition~\ref{prop:reform_wc}  and the fact that  $\eta = 0$, and the second equality follows from Lemma~\ref{lem:singleton}. 

We next analyze the objective function in \eqref{prob:not_improveable:intermediary_prob}. Indeed, it follows readily from Proposition~\ref{prop:cost_reform} from Appendix~\ref{appx:graphical} that the following equalities hold for all assortments $S \in \mathcal{S}$:
\begin{align*}
\begin{aligned}
 \rho_{\ubar{i}^j_1 \cdots \ubar{i}^j_n}(S)  &=  r_j \mathbb{I} \left \{j \in S \right \} && \forall j \in \bar{S} \setminus \{0,n\}, \\
  \rho_{\bar{i}^j_1 \cdots \bar{i}^j_n}(S)  &=  r_j \mathbb{I} \left \{j \in S \right \} + r_n \mathbb{I} \left \{j \notin S  \textnormal{ and } n \in S\right \} && \forall j \in \mathcal{N} \setminus \bar{S}, \\
  \rho_{n \cdots n}(S) &= r_n \mathbb{I} \left \{ n \in S\right \}.
  \end{aligned}
\end{align*}
Using the above equalities, we rewrite the optimization problem on line~\eqref{prob:not_improveable:intermediary_prob} as 
\begin{align}
\max_{S \in \mathcal{S}} \left \{\underbrace{\frac{1}{n} \sum_{j \in \bar{S} \setminus \{0,n\}}r_j \mathbb{I} \left \{j \in S \right \}  + \frac{1}{n} \sum_{j \in\mathcal{N} \setminus \bar{S}} \left(  r_j \mathbb{I} \left \{j \in S \right \} + r_n \mathbb{I} \left \{j \notin S  \textnormal{ and } n \in S \right \}  \right)  + \frac{1}{n} r_n \mathbb{I} \left \{ n \in S \right \}}_{\mathscr{R}'(S)} \right \}. \label{prob:not_improveable:intermediary_prob_2}
\end{align}

We next show that every optimal solution $S$ of \eqref{prob:not_improveable:intermediary_prob_2}  must contain the product $n$. Indeed,  for every  arbitrary assortment $S \in \mathcal{S}$ that  satisfies $n \notin S$, we observe from algebra that
\begin{align*}
\mathscr{R}'(S \cup \{n\}) - \mathscr{R}'(S)  &= \frac{1}{n} \sum_{j \in \mathcal{N} \setminus \bar{S}} r_n \left( \mathbb{I} \left \{j \notin S \right \}  - 0 \right) + \frac{1}{n} r_n > 0.
\end{align*}
Since the assortment $S \in \mathcal{S}$ with $n \notin S$ was chosen arbitrarily, we have shown that  every optimal solution $S$ of \eqref{prob:not_improveable:intermediary_prob_2} must satisfy $n \in S$. We can thus rewrite   \eqref{prob:not_improveable:intermediary_prob_2} as
\begin{align}
\max_{S \in \mathcal{S}: n \in S} \left \{\underbrace{\frac{1}{n} \sum_{j \in \bar{S} \setminus \{0,n\}}r_j \mathbb{I} \left \{j \in S \right \}  + \frac{1}{n} \sum_{j \in\mathcal{N} \setminus \bar{S}} \left(  r_j \mathbb{I} \left \{j \in S \right \} + r_n \mathbb{I} \left \{j \notin S \right \}  \right)  + \frac{1}{n} r_n }_{\mathscr{R}''(S)} \right \}. \label{prob:not_improveable:intermediary_prob_3}
\end{align}

We next show that every optimal solution $S$ of the optimization problem~\eqref{prob:not_improveable:intermediary_prob_3}  must satisfy $j' \notin S$ for all $j' \in \mathcal{N} \setminus \bar{S}$. Indeed,  for every  arbitrary assortment $S \in \mathcal{S}$ that  satisfies $n \in S$ and satisfies $j' \in S$ for some product  $j' \in \mathcal{N} \setminus \bar{S}$, we observe from algebra that
\begin{align*}
\mathscr{R}''(S \setminus \{j'\}) - \mathscr{R}''(S)  &= \frac{1}{n} \left(0 + r_n \right) - \frac{1}{n}  \left(r_{j'} + 0 \right)  > 0.
\end{align*}
Since the assortment $S \in \mathcal{S}$ with $n \in S$ and $j' \in S$ for some $j' \in \mathcal{N} \setminus \bar{S}$ was chosen arbitrarily, we have shown that  every optimal solution $S$ of \eqref{prob:not_improveable:intermediary_prob_3} must satisfy $j' \notin S$ for all $j' \in \mathcal{N} \setminus \bar{S}$. We can thus rewrite   \eqref{prob:not_improveable:intermediary_prob_3} as
\begin{align}
\max_{S \in \mathcal{S}: n \in S, \; j' \notin S \forall j' \in \mathcal{N} \setminus \bar{S}} \left \{\underbrace{\frac{1}{n} \sum_{j \in \bar{S} \setminus \{0,n\}}r_j \mathbb{I} \left \{j \in S \right \}  + \frac{1}{n} \sum_{j \in\mathcal{N} \setminus \bar{S}}  r_n + \frac{1}{n} r_n }_{\mathscr{R}''(S)} \right \}. \label{prob:not_improveable:intermediary_prob_4}
\end{align}

Finally, we conclude by inspection that every optimal solution $S$ of the optimization problem \eqref{prob:not_improveable:intermediary_prob_4} must satsify $j' \in S$ for all $j' \in S \setminus \{0,n\}$. We have thus shown that the only optimal solution of the robust optimization problem~\eqref{prob:robust} is $\bar{S}$, which concludes our proof of Theorem~\ref{thm:not_improveable}. \halmos

\section{Assumption that $\mathcal{N}_0 = \cup_{m \in \mathcal{M}} S_m$}\label{appx:N_0_assumption}
In Remark~\ref{rem:N_0} of \S\ref{sec:algorithms}, we state an assumption that the universe of products $\mathcal{N}_0$ is equal to $\cup_{m \in \mathcal{M}} S_m$, where $S_1,\ldots,S_M \subseteq \mathcal{N}_0$ are the assortments that the firm has previously offered to its customers. In that remark, we claim that this assumption can always be satisfied by redefining $\mathcal{N}_0$ to be equal to $\cup_{m \in \mathcal{M}} S_m$.  In this appendix, we argue that this assumption is without loss of generality from the perspective of the robust optimization problem~\eqref{prob:robust}. Specifically, we formalize Remark~\ref{rem:N_0} via the following proposition:
\begin{proposition} \label{prop:N_0_assumption}
 $\max \limits_{S \in \mathcal{S}} \min \limits_{\lambda \in \mathcal{U}}\mathscr{R}^{\lambda}(S) = \max \limits_{S \in \mathcal{S}} \min \limits_{\lambda \in \mathcal{U}}\mathscr{R}^{\lambda}\left(S \cap \left( \bigcup_{m \in \mathcal{M}} S_m  \right) \right). $
\end{proposition}  
The above proposition says that if $S^* \in \mathcal{S}$ is an optimal solution of the robust optimization problem~\eqref{prob:robust}, then $S^* \cap \left( \bigcup_{m \in \mathcal{M}} S_m  \right)$ is also an optimal solution of the robust optimization problem~\eqref{prob:robust}. Hence, the above proposition implies that we can assume without loss of generality that the robust optimization problem~\eqref{prob:robust} will only optimize over assortments $S$ in the set $\mathcal{S} \cap \left( \bigcup_{m \in \mathcal{M}} S_m  \right)$.  The following proof of Proposition~\ref{prop:N_0_assumption} relies on the graphical interpretation of $\rho_{i_1 \cdots i_M}(S)$ that is developed in Appendix~\ref{appx:graphical}. 
\begin{proof}{Proof of Proposition~\ref{prop:N_0_assumption}. }
We first observe that 
\begin{align*}
 \max \limits_{S \in \mathcal{S}} \min \limits_{\lambda \in \mathcal{U}}\mathscr{R}^{\lambda}\left(S \cap \left( \bigcup_{m \in \mathcal{M}} S_m  \right) \right) &=  \max \limits_{S \in \mathcal{S} \cap \left( \bigcup_{m \in \mathcal{M}} S_m  \right)  } \min \limits_{\lambda \in \mathcal{U}}\mathscr{R}^{\lambda}\left(S \right) \le \max \limits_{S \in \mathcal{S}} \min \limits_{\lambda \in \mathcal{U}}\mathscr{R}^{\lambda}(S), 
\end{align*}
where the equality follows from the fact that $\mathcal{S} \triangleq \{S \subseteq \mathcal{N}_0: 0 \in S \}$, and the inequality follows from algebra. To show the other direction, 
consider any arbitrary assortment $S \in \mathcal{S}$. We recall from Proposition~\ref{prop:reform_wc} that the worst-case expected revenue   $ \min_{\lambda \in \mathcal{U}}\mathscr{R}^{\lambda}\left(S \cap \left( \bigcup_{m \in \mathcal{M}} S_m  \right) \right)$ is equal to the optimal objective value of the following linear optimization problem:
\begin{equation}\label{prob:robust_simplified_N_0_assumption}
 \begin{aligned}
& \; \underset{\lambda, \epsilon}{\textnormal{minimize}} && \sum_{(i_1,\ldots,i_M) \in \mathcal{L}} \rho_{i_1 \cdots i_M}\left(S \cap \left( \bigcup_{m \in \mathcal{M}} S_m  \right) \right) \lambda_{i_1\cdots i_M}\\
&\textnormal{subject to}&&  \textnormal{same constraints as \eqref{prob:robust_simplified}}
\end{aligned}
\end{equation}
Moreover, we recall from Proposition~\ref{prop:cost_reform} in Appendix~\ref{appx:graphical} that the following equality holds for each $(i_1,\ldots,i_M) \in \mathcal{L}$: 
\begin{align}
 \rho_{i_1 \cdots i_M}\left(S \cap \left( \bigcup_{m \in \mathcal{M}} S_m  \right) \right)  = \min_{i \in  S \cap \left( \bigcup_{m \in \mathcal{M}} S_m  \right)  \cap \mathcal{I}_{i_1 \cdots i_M}\left( S \cap \left( \bigcup_{m \in \mathcal{M}} S_m  \right) \right)} r_i. \label{line:rho_for_union}
\end{align} 
In particular, we observe  for each $(i_1,\ldots,i_M) \in \mathcal{L}$ that 
\begin{align}
\mathcal{I}_{i_1 \cdots i_M}\left( S \cap \left( \bigcup_{m \in \mathcal{M}} S_m  \right) \right) &=  \left \{ i \in \mathcal{N}_0:\; \textnormal{for all } m  \in \mathcal{M}, \; \textnormal{if } i_m \in S \cap \left( \bigcup_{m \in \mathcal{M}} S_m  \right), \text{ then } i_m \nprec_{i_1 \cdots i_M} i \right \} \notag \\
 &=  \left \{ i \in \mathcal{N}_0:\; \textnormal{for all } m  \in \mathcal{M}, \; \textnormal{if } i_m \in S, \text{ then } i_m \nprec_{i_1 \cdots i_M} i \right \} \notag  \\
&= \mathcal{I}_{i_1 \cdots i_M}\left( S  \right),\label{line:I_for_union}
\end{align}
where the first equality follows from Definition~\ref{defn:I}, the second equality follows from the fact that $i_1,\ldots,i_M \in \cup_{m \in \mathcal{M}} S_m$, and the third equality follows from Definition~\ref{defn:I}. Therefore, for every feasible solution $(\lambda,\epsilon)$ for the linear optimization problem~\eqref{prob:robust_simplified}, we observe that 
\begin{align*}
\sum_{(i_1,\ldots,i_M) \in \mathcal{L}} \rho_{i_1 \cdots i_M}\left(S \cap \left( \bigcup_{m \in \mathcal{M}} S_m  \right) \right) \lambda_{i_1\cdots i_M} &= \sum_{(i_1,\ldots,i_M) \in \mathcal{L}} \left( \min_{i \in  S \cap \left( \bigcup_{m \in \mathcal{M}} S_m  \right)  \cap \mathcal{I}_{i_1 \cdots i_M}\left( S \cap \left( \bigcup_{m \in \mathcal{M}} S_m  \right) \right)} r_i \right)  \lambda_{i_1\cdots i_M} \\
&=  \sum_{(i_1,\ldots,i_M) \in \mathcal{L}} \left( \min_{i \in  S \cap \left( \bigcup_{m \in \mathcal{M}} S_m  \right)  \cap \mathcal{I}_{i_1 \cdots i_M}\left( S \right)} r_i  \right) \lambda_{i_1\cdots i_M}\\
&\ge  \sum_{(i_1,\ldots,i_M) \in \mathcal{L}} \left(  \min_{i \in  S   \cap \mathcal{I}_{i_1 \cdots i_M}\left( S \right)} r_i  \right) \lambda_{i_1\cdots i_M}\\
&=  \sum_{(i_1,\ldots,i_M) \in \mathcal{L}}\rho_{i_1 \cdots i_M}\left(S \right)  \lambda_{i_1\cdots i_M},
\end{align*}
where the first equality follows from \eqref{line:rho_for_union}, the second equality follows from \eqref{line:I_for_union}, the inequality follows from algebra, and the third equality follows from Proposition~\ref{prop:cost_reform}. Since the above analysis holds for every feasible solution $(\lambda,\epsilon)$ for the linear optimization problem~\eqref{prob:robust_simplified}, and since the assortment $S \in \mathcal{S}$ was chosen arbitrarily, we conclude that
\begin{align*}
 \max \limits_{S \in \mathcal{S}} \min \limits_{\lambda \in \mathcal{U}}\mathscr{R}^{\lambda}\left(S \cap \left( \bigcup_{m \in \mathcal{M}} S_m  \right) \right) &\ge \max \limits_{S \in \mathcal{S}} \min \limits_{\lambda \in \mathcal{U}}\mathscr{R}^{\lambda}(S). 
\end{align*}
The proof of  Proposition~\ref{prop:N_0_assumption} is thus complete.  \halmos \end{proof}

\section{Proofs of Technical Results from \S\ref{sec:algorithm:twoassortments}}\label{appx:two}

\subsection{Proof of Lemma~\ref{lem:two:L}} \label{appx:two:L}
It follows from Definition~\ref{defn:L} and the fact that $M = 2$ that
\begin{align*}
\mathcal{L} = \left \{(i_1,i_2): \mathcal{D}_{i_1}(S_1) \cap \mathcal{D}_{i_2}(S_2) \neq \emptyset  \right \}. 
\end{align*}  The rest of the proof of Lemma~\ref{lem:two:L} is split into three cases.

 In the first case, consider any arbitrary pair of products $(i_1,i_2) \in (S_1 \setminus S_2) \times S_2$. For this pair of products,  consider any ranking $\sigma \in \Sigma$ which satisfies the equalities $\sigma(i_1) =0$ and $\sigma(i_2) = 1$. It follows from the fact that $i_1 \in S_1$ that this ranking satisfies the equality $\argmin_{j \in S_1} \sigma(j) = 1$. Moreover, it follows from the facts that $i_1 \notin S_2$ and $i_2 \in S_2$ that this ranking also satisfies the equality $\argmin_{j \in S_2} \sigma(j) = 2$. We have thus shown that $\sigma \in \mathcal{D}_1(S_1) \cap \mathcal{D}_2(S_2)$, which proves that $ \mathcal{D}_1(S_1) \cap \mathcal{D}_2(S_2) \neq \emptyset$. Since the pair of products $(i_1,i_2) \in (S_1 \setminus S_2) \times S_2$ was chosen arbitrarily, we have shown that $(S_1 \setminus S_2) \times S_2 \subseteq  \mathcal{L}$.

 In the second case, consider any arbitrary pair of products $(i_1,i_2) \in S_1 \times (S_2 \setminus S_1)$. Using identical reasoning as the first case, we observe that  $ \mathcal{D}_1(S_1) \cap \mathcal{D}_2(S_2) \neq \emptyset$, which shows that $S_1 \times (S_2 \setminus S_1) \subseteq  \mathcal{L}$. 
 
 In the third case,  consider any arbitrary pair of products $(i_1,i_2) \in (S_1 \cap S_2) \times (S_1 \cap S_2)$.  For this pair of products, it follows from Definition~\ref{defn:D} that the inequality $\sigma(i_1) < \sigma(j)$ must hold for all rankings $\sigma \in \mathcal{D}_{i_1}(S_1)$ and all products $j \in (S_1 \cap S_2) \setminus \{i_1 \}$. It also follows from Definition~\ref{defn:D} that the inequality $\sigma(i_2) < \sigma(j)$ must hold for all rankings $\sigma \in \mathcal{D}_{i_2}(S_2)$ and all products $j \in (S_1 \cap S_2) \setminus \{i_2 \}$. It follows immediately from these inequalities that there exists a ranking that satisfies $\sigma \in \mathcal{D}_{i_1}(S_1) \cap \mathcal{D}_{i_2}(S_2)$ if and only if $i_1 = i_2$, and so it follows from Definition~\ref{defn:L} that $(i_1,i_2) \in \mathcal{L}$ if and only if $i_1 = i_2$. 
 
Since the three cases we have considered  are exhaustive, our proof of Lemma~\ref{lem:two:L} is complete.   \halmos

\subsection{Proof of Lemma~\ref{lem:two:flow}}\label{appx:two:flow}

Let $M = 2$ and $\eta = 0$. In the remainder of the proof of Lemma~\ref{lem:two:flow}, we will show that the following equality holds for all assortments $S \in \mathcal{S}$:
\begin{align} 
 \min_{\lambda \in \mathcal{U}}\mathscr{R}^{\lambda}(S) &= \sum_{i \in S_1\cap S_2} \rho_{ii}(S) v_{1,i} \;+ \notag \\
  &\quad \quad \quad \left[\begin{aligned}
& \; \underset{\lambda}{\textnormal{minimize}} && \sum_{i_1 \in S_1 \setminus S_2} \sum_{i_2 \in S_1 \cap S_2}  \rho_{i_1 i_2}(S) \lambda_{i_1i_2} \\
&&&+ \sum_{i_1 \in S_1 \setminus S_2} \sum_{i_2 \in S_2 \setminus  S_1}  \rho_{i_1 i_2}(S) \lambda_{i_1i_2} \\
&&&+ \sum_{i_1 \in S_1 \cap S_2} \sum_{i_2 \in S_2 \setminus S_1} ( \rho_{i_1 i_2}(S) - \rho_{i_1 i_1}(S)) \lambda_{i_1i_2} \\
&\textnormal{subject to}&& \begin{aligned}[t]
& \sum_{i_2 \in S_2} \lambda_{i_1  i_2} = v_{1,i_1} && \forall  i_1 \in S_1 \setminus S_2\\
& \sum_{i_1 \in S_1} \lambda_{i_1  i_2} = v_{2,i_2} &&\forall   i_2 \in S_2 \setminus S_1\\
& \sum_{i_2 \in S_2 \setminus S_1} \lambda_{i i_2}   - \sum_{i_1 \in S_1 \setminus S_2} \lambda_{i_1  i}  =  v_{1,i} - v_{2,i}  &&\forall   i \in S_1 \cap S_2\\
& \sum_{i_1 \in S_1 \setminus S_2} \lambda_{i_1 i} \le v_{2,i}  && \forall i \in S_1 \cap S_2\\
& \lambda_{i_1 i_2} \ge 0 && \forall (i_1,i_2) \in \mathcal{L}
\end{aligned}
\end{aligned}
\right].  \label{prob:flow}
\end{align}
We readily observe that the linear optimization problem  on line~\eqref{prob:flow} is a minimum-cost network flow problem, where each decision variable $\lambda_{i_1 i_2}$ corresponds to the flow on a directed edge from vertex $i_1$ to vertex $i_2$  \citep[p. 296]{ahuja1988network}. In particular, we observe that the minimum-cost network flow problem on line~\eqref{prob:flow} takes place on a tripartite directed acyclic graph with $|S_1 \setminus S_2| + |S_2 \setminus S_1| + |S_1 \cap S_2| = n+1$ vertices and $ \mathcal{O}(n^2)$ directed edges. In Figure~\ref{fig:network_flow},  we present a visualization of the network corresponding to the minimum-cost network flow problem from line~\eqref{prob:flow}. 
\begin{remark}The inequality of the form $\sum_{i_1 \in S_1 \setminus S_2} \lambda_{i_1  i} \le v_{2,i}$ for each vertex $i \in S_1 \cap S_2$ can be interpreted as an upper bound on the incoming flow entering into vertex $i$. These inequalities can be converted into edge capacities by splitting each vertex $i \in S_1 \cap S_2$ into two vertices, $i^{\text{in}}$ and $i^{\text{out}}$, where the incoming edges to vertex $i^{\text{in}}$ are from the vertices in $S_1 \setminus S_2$, the outgoing edges from $i^{\text{out}}$ are to the  vertices in $S_2 \setminus S_1$, the only outgoing edge from vertex $i^{\text{in}}$ is the only incoming edge to vertex $i^{\text{out}}$, the flow through the edge from vertex $i^{\text{in}}$ to vertex $i^{\text{out}}$ must satisfy the constraints $0 \le g_{i^{\text{in}} i^{\text{out}}} \le v_{2,i}$, and the supplies on the vertices $i^{\text{in}}$ and $i^{\text{out}}$ are $0$ and $v_{1,i} - v_{2,i}$, respectively. 
\end{remark}

\begin{figure}[t]
\centering
\FIGURE{%
\begin{tikzpicture}
\tikzset{vertex/.style = {shape=circle,draw,minimum size=1.5em}}
\tikzset{edge/.style = {->,> = latex'}}
\node[vertex] (1) at  (-0.25*360/6: 3cm) {1};
\node[vertex] (0) at  (1.25*360/6: 3cm) {0};
\node[vertex] (5) at  (1.75*360/6: 3cm) {5};
\node[vertex] (4) at  (3.25*360/6: 3cm) {4};
\node[vertex] (3) at  (3.75*360/6: 3cm) {3};
\node[vertex] (2) at  (5.25*360/6: 3cm) {2};
  \node[ellipse,rotate=0*360/6,line width=0.5mm,draw = red,dotted,minimum width = 3.5cm,  minimum height = 2cm] (e) at (1.5*360/6: 2.897cm) {};
 \node[ellipse,rotate=360/6,draw = ForestGreen,line width=0.5mm,dotted,minimum width = 3.5cm,  minimum height = 2cm] (e) at (-0.5*360/6: 2.897cm) {};
  \node[ellipse,rotate=2*360/6,draw = blue,line width=0.5mm,dotted,minimum width = 3.5cm,  minimum height = 2cm] (e) at (3.5*360/6: 2.897cm) {};
 \node[text width=3cm] at (1.5*360/6: 4.25cm)  {};
\draw [-{Stealth[scale=1.25]}] (0) -- (4);
\draw [-{Stealth[scale=1.25]}] (0) -- (3);
\draw [-{Stealth[scale=1.25]}] (5) -- (3); 
\draw [-{Stealth[scale=1.25]}] (5) -- (4); 
\draw [-{Stealth[scale=1.25]}] (1) -- (0); 
\draw [-{Stealth[scale=1.25]}] (1) -- (3); 
\draw [-{Stealth[scale=1.25]}] (1) -- (4); 
\draw [-{Stealth[scale=1.25]}] (1) -- (5); 
\draw [-{Stealth[scale=1.25]}] (2) -- (0); 
\draw [-{Stealth[scale=1.25]}] (2) -- (3); 
\draw [-{Stealth[scale=1.25]}] (2) -- (4); 
\draw [-{Stealth[scale=1.25]}] (2) -- (5); 
\end{tikzpicture}
}
{Visualization of minimum-cost network flow problem from line~\eqref{prob:flow}.\label{fig:network_flow}}
{\TABLEfootnotesizeIX The figure shows a visualization of the minimum-cost network flow problem corresponding to the linear optimization problem on line~\eqref{prob:flow} for the case where the past assortments are $S_1 = \{0,1,2,5\}$ and $S_2 = \{0,3,4,5\}$. We see that there is a vertex in the graph for each product $i  \in \mathcal{N}_0 \equiv \{0,1,2,3,4,5\}$. The graph is a  tripartite directed graph, where the three partitions of vertices are denoted by the dotted ellipses and correspond to  $S_1 \setminus S_2 = \{1,2\}$, $S_1 \cap S_2 = \{0,5\}$, and $S_2 \setminus S_1 = \{3,4\}$. The flow demands at each of the vertices and the flow cost for each of the directed edges can be found on  line~\eqref{prob:flow}.    }
\end{figure}

We now proceed to prove that line~\eqref{prob:flow} holds. Indeed, consider any arbitrary assortment $S \in {\mathcal{S}}$. It follows from Proposition~\ref{prop:reform_wc}  and our assumptions of $M=2$ and $\eta = 0$ that 
\begin{align*}
\min_{\lambda \in \mathcal{U}}\mathscr{R}^{\lambda}(S) = \left[ \begin{aligned}
& \; \underset{\lambda}{\textnormal{minimize}} && \sum_{(i_1,i_2) \in \mathcal{L}} \rho_{i_1 i_2}(S) \lambda_{i_1 i_2}\\
&\textnormal{subject to}&&  \sum_{(i_1,i_2) \in \mathcal{L}: \; i_1 = i} \lambda_{i_1  i_2}  = v_{1,i} \quad \forall  i \in S_1\\
&&&  \sum_{(i_1,i_2) \in \mathcal{L}: \; i_2 = i} \lambda_{i_1  i_2}  = v_{2,i} \quad \forall  i \in S_2\\
&&& \sum_{(i_1,i_2) \in \mathcal{L}} \lambda_{i_1 i_2} = 1 \\
&&& \lambda_{i_1 i_2} \ge 0 \quad \forall (i_1,i_2) \in \mathcal{L}
\end{aligned}\right]. 
\end{align*}
After applying Lemma~\ref{lem:two:L} to the above optimization problem and removing the redundant constraint $\sum_{(i_1,i_2) \in \mathcal{L}} \lambda_{i_1 i_2} = 1$, we observe that
\begin{align*}
\min_{\lambda \in \mathcal{U}}\mathscr{R}^{\lambda}(S) = \left[ \begin{aligned}
& \; \underset{\lambda}{\textnormal{minimize}} && \sum_{i_1 \in S_1 \setminus S_2} \sum_{i_2 \in S_1 \cap S_2}  \rho_{i_1 i_2}(S) \lambda_{i_1i_2} + \sum_{i_1 \in S_1 \setminus S_2} \sum_{i_2 \in S_2 \setminus  S_1}  \rho_{i_1 i_2}(S) \lambda_{i_1i_2} \\
&&&+ \sum_{i_1 \in S_1 \cap S_2} \sum_{i_2 \in S_2 \setminus S_1}  \rho_{i_1 i_2}(S) \lambda_{i_1i_2} + \sum_{i \in S_1 \cap S_2}  \rho_{i i}(S) \lambda_{i i}\\
&\textnormal{subject to}&&\begin{aligned}[t]
&  \sum_{i_2 \in S_2} \lambda_{i_1  i_2} = v_{1,i_1} && \forall  i_1 \in S_1 \setminus S_2\\
&  \sum_{i_1 \in S_1} \lambda_{i_1  i_2} = v_{2,i_2} &&\forall   i_2 \in S_2 \setminus S_1\\
& \lambda_{i i} +  \sum_{i_2 \in S_2 \setminus S_1} \lambda_{i  i_2} = v_{1,i} &&\forall   i \in S_1 \cap S_2\\
& \lambda_{i i} +  \sum_{i_1 \in S_1 \setminus S_2} \lambda_{i_1  i} = v_{2,i} &&\forall   i \in S_1 \cap S_2\\
& \lambda_{i_1 i_2} \ge 0 && \forall (i_1,i_2) \in \mathcal{L}
\end{aligned}
\end{aligned}\right]. 
\end{align*}
Finally,  for each $i \in S_1 \cap S_2$, we eliminate the decision variable $\lambda_{ii}$  from the above optimization problem by using the equality $\lambda_{ii} = v_{2,i} - \sum_{i_1 \in S_1 \setminus S_2} \lambda_{i_1 i}$. With this elimination, we obtain the desired result.  \halmos

\subsection{Proof of Theorem~\ref{thm:two}}\label{appx:thm:two}

We first describe our algorithm for solving the robust optimization problem~\eqref{prob:robust}, which  follows a brute-force strategy. Namely, our algorithm  iterates over each of the assortments $S \in \widehat{\mathcal{S}}$, and, for each such assortment, the algorithm computes the corresponding worst-case expected revenue $\min_{\lambda \in \mathcal{U}}\mathscr{R}^{\lambda}(S)$. The algorithm concludes by returning the maximum value of $\min_{\lambda \in \mathcal{U}}\mathscr{R}^{\lambda}(S)$ across the assortments $S \in \widehat{\mathcal{S}}$. The correctness of this algorithm for solving the robust optimization problem~\eqref{prob:robust} follows immediately from Theorem~\ref{thm:main}.

 We now analyze the running time of our algorithm by using our three intermediary results. 
 We assume that the two assortments $S_1$ and $S_2$ are given as sorted arrays. Under this assumption, 
  it is straightforward  to see that the sets $S_1 \cap S_2$, $S_1 \setminus S_2$, and $S_2 \setminus S_1$ can be computed and stored as sorted arrays in $\mathcal{O}(n)$ computation time. We also require $\mathcal{O}(n)$ computation time to store  copies of the sets $S_1,S_2,S_1 \cap S_2$, $S_1 \setminus S_2$, and $S_2 \setminus S_1$ in hash tables, which ensures that querying  whether a given product is an element of any of these sets can be performed in $\mathcal{O}(1)$ time.   
 
 We next analyze the computation times for constructing the collection of assortments $\widehat{\mathcal{S}}$, constructing the set of pairs of products $\mathcal{L}$, and computing the quantities $\rho_{i_1 i_2}(S)$ for each  assortment $S \in \widehat{\mathcal{S}}$ and each pair of products $(i_1,i_2) \in \mathcal{L}$. Indeed, using the aforementioned data structures, it follows readily from 
Lemma~\ref{lem:two:S} that we can construct the collection of assortments $\widehat{\mathcal{S}}$ in  $\mathcal{O}(n^3)$ computation time.\footnote{It follows from Lemma~\ref{lem:two:S} that we can efficiently iterate over the assortments in $\widehat{\mathcal{S}}$ by iterating over the pairs of products in $S_1 \setminus S_2$ and $S_2 \setminus S_1$. Constructing the collection $\widehat{\mathcal{S}}$ thus requires iterating over the $|S_1 \setminus S_2| \times |S_2 \setminus S_1| = \mathcal{O}(n^2)$ assortments, and each of the assortments is comprised of at most $\mathcal{O}(n)$ products.  } 
Moreover,   it follows from 
Lemma~\ref{lem:two:L} that the set of pairs of products $\mathcal{L}$ can be computed in $\mathcal{O}(n^2)$ time.  Finally, we analyze the computation times for computing the quantities $\rho_{i_1 i_2}(S)$ for each  assortment $S \in \widehat{\mathcal{S}}$ and each pair of products $(i_1,i_2) \in \mathcal{L}$. Indeed, we recall from Lemma~\ref{lem:two:S} that $| \widehat{\mathcal{S}}| = \mathcal{O}(n^2)$, and we recall from Lemma~\ref{lem:two:L} that $|\mathcal{L}| = \mathcal{O}(n^2)$. Therefore,  there are $| \widehat{\mathcal{S}}| \times |\mathcal{L}| = \mathcal{O}(n^4)$ different ways of choosing an  assortment $S \in \widehat{\mathcal{S}}$ and a pair of products $(i_1,i_2) \in \mathcal{L}$. For each assortment $S \in {\mathcal{S}}$ and pair of products $(i_1,i_2) \in \mathcal{L}$,   it follows readily from Definition~\ref{defn:rho} and Lemmas~\ref{lem:two:S} and \ref{lem:two:L} that 
   \begin{align*}
\rho_{i_1 i_2}(S) &= \begin{cases}
r_{i_1},&\text{if } {\color{black}i_1, i_2 \in S_1 \cap S_2},  \; i_1 = i_2, \textnormal{ and } i_1 \in S, \\
0,&\text{if } {\color{black}i_1, i_2 \in S_1 \cap S_2},  \; i_1 = i_2, \textnormal{ and } i_1 \notin S, \\
\\[-0.5em]
r_{i_2},&\text{if } {\color{black} i_1 \in S_1 \cap S_2,\; i_2 \in S_2 \setminus S_1}, \textnormal{ and } i_2 \in S,\\
\min \left \{ r_{i_1}, \min_{j \in S \cap S_2 \setminus S_1}r_{j} \right \},&\text{if }  {\color{black} i_1 \in S_1 \cap S_2,\; i_2 \in S_2 \setminus S_1},\; i_2 \notin S, \textnormal{ and } i_1 \in S, \\
0,&\text{if }  {\color{black} i_1 \in S_1 \cap S_2,\; i_2 \in S_2 \setminus S_1},\; i_2 \notin S, \textnormal{ and } i_1 \notin S, \\ 
\\[-0.5em]
r_{i_1},&\text{if }{\color{black} i_1 \in S_1 \setminus S_2,\; i_2 \in S_1 \cap S_2}, \textnormal{ and } i_1 \in S,\\
\min \left \{ r_{i_2}, \min_{j \in S \cap S_1 \setminus S_2}r_{j} \right \},&\text{if }{\color{black} i_1 \in S_1 \setminus S_2,\; i_2 \in S_1 \cap S_2},\; i_1 \notin S, \textnormal{ and } i_2 \in S, \\
0,&\text{if }{\color{black} i_1 \in S_1 \setminus S_2,\; i_2 \in S_1 \cap S_2},\; i_1 \notin S, \textnormal{ and } i_2 \notin S, \\
\\[-0.5em]
\min \left \{ r_{i_1},r_{i_2} \right \}, &\text{if }{\color{black}i_1 \in S_1 \setminus S_2, \; i_2 \in S_2 \setminus S_1},\; i_1 \in S, \text{ and } i_2 \in S,\\
\min \left \{ r_{i_1}, \min_{j \in S \cap S_2 \setminus S_1}r_{j} \right \},&\text{if }{\color{black}i_1 \in S_1 \setminus S_2, \; i_2 \in S_2 \setminus S_1},\; i_1 \in S, \text{ and } i_2 \notin S,\\
\min \left \{ r_{i_2}, \min_{j \in S \cap S_1 \setminus S_2}r_{j} \right \},&\text{if }{\color{black}i_1 \in S_1 \setminus S_2, \; i_2 \in S_2 \setminus S_1},\; i_1 \notin S, \text{ and } i_2 \in S,\\
0,&\text{if }{\color{black}i_1 \in S_1 \setminus S_2, \; i_2 \in S_2 \setminus S_1},\; i_1 \notin S, \text{ and } i_2 \notin S.
\end{cases}
\end{align*}
We observe that the quantities $\min_{j \in S \cap S_1 \setminus S_2}r_{j}$ and  $\min_{j \in S \cap S_2 \setminus S_1}r_{j} $ appear in many of the above cases, and we see that these quantities can be precomputed for each of the assortments $S \in \widehat{\mathcal{S}}$ in a total of $| \widehat{\mathcal{S}}|\times \mathcal{O}(n) = \mathcal{O}(n^3)$ computation time. Given that we have precomputed these quantities, and given the fact that our data structures allow us to query whether any product is an element of the sets $S_1 \setminus S_2$, $S_2 \setminus S_1$, and $S_1 \cap S_2$  in $\mathcal{O}(1)$ time, we conclude that all of the $\rho_{i_1 i_2}(S)$ can be computed in a total of $\mathcal{O}(n^4 + n^3) = \mathcal{O}(n^4)$ time. 
In summary, we have established that constructing the collection of assortments $\widehat{\mathcal{S}}$, constructing the set of pairs of products $\mathcal{L}$, and computing the quantities $\rho_{i_1 i_2}(S)$ for each  assortment $S \in \widehat{\mathcal{S}}$ and each pair of products $(i_1,i_2) \in \mathcal{L}$ can be performed  in a total of $\mathcal{O}(n^4)$  computation time. 

We conclude our proof  of Theorem~\ref{thm:two} by establishing the total computation time of our brute-force algorithm using  the information computed above.  In each iteration of our algorithm, we select an assortment $S \in \widehat{\mathcal{S}}$ and compute the worst-case expected revenue $\min_{\lambda \in \mathcal{U}} \mathscr{R}^{\lambda}(S)$. As shown in Lemma~\ref{lem:two:flow}, we can compute  $\min_{\lambda \in \mathcal{U}} \mathscr{R}^{\lambda}(S)$ by solving a minimum-cost network flow problem over a graph with $n+1$ vertices and $\mathcal{O}(n^2)$ edges. Using the minimum-cost network flow algorithm of \cite{orlin1997polynomial} and \cite{tarjan1997dynamic},  we observe that Problem~\eqref{prob:flow} can be solved in $\mathcal{O}(n^3 \log (n r_n))$ computation time.\footnote{The algorithm of \cite{orlin1997polynomial} and \cite{tarjan1997dynamic} computes the minimum-cost network flow on a directed graph in $\mathcal{O}( (VE \log V) \min \left \{\log( V C), E \log V \right \})$ running time, where $V$ is the number of vertices, $E$ is the number of directed edges, and $C$ is the  maximum absolute value of any edge cost. The algorithm requires that $C$ is integral; for more details, see \citet[\S3]{tarjan1997dynamic}. In our case, Problem~\eqref{prob:flow} is a minimum-cost network flow problem in a  directed graph where $V = n$, $E = \mathcal{O}(n^2)$, and $C = \max_{S \in \widehat{\mathcal{S}},(i_1,i_2) \in \mathcal{L}} \rho_{i_1 i_2}(S) = r_n$. } Since the worst-case expected revenue $\min_{\lambda \in \mathcal{U}} \mathscr{R}^{\lambda}(S)$ must be computed for each assortment $S \in \widehat{\mathcal{S}}$, and since it follows readily from  Lemma~\ref{lem:two:S} that $| \widehat{\mathcal{S}}| = \mathcal{O}(n^2)$, we conclude that our algorithm  requires a total  of  $\mathcal{O}(n^5 \log (n r_n))$ computation time.
 \halmos

\section{Proofs of  Technical Results from \S\ref{sec:algorithms:fixed_dim}}\label{appx:fixed_dim}

\subsection{Proof of Lemma~\ref{lem:fixed_dim:construct_L}}
To construct the set of tuples of products $\mathcal{L}$, we iterate over each tuple of products $(i_1,\ldots,i_M) \in S_1 \times \cdots \times S_M$. For each such tuple of products, we can construct the corresponding directed graph $\mathcal{G}_{i_1 \cdots i_M}$ that is described in the beginning of Appendix~\ref{appx:graphical}. According to Lemma~\ref{lem:dag} in Appendix~\ref{appx:graphical}, the tuple of products $(i_1,\ldots,i_M)$ satisfies $(i_1,\ldots,i_M) \in \mathcal{L}$ if and only if the directed graph $\mathcal{G}_{i_1 \cdots i_M}$ is acyclic. We can check whether a directed graph is acyclic by using the well-known topological sorting algorithm \citep[p.79]{ahuja1988network}. Therefore, our algorithm for constructing the set of tuples of products $\mathcal{L}$ is to  iterate over each tuple of products $(i_1,\ldots,i_M) \in S_1 \times \cdots \times S_M$ and, for each such tuple of products, to check whether the tuple of products satisfies $(i_1,\ldots,i_M) \in \mathcal{L}$ by performing the topological sorting algorithm on the directed graph $\mathcal{G}_{i_1 \cdots i_M}$.

The computation time of the above algorithm for constructing $\mathcal{L}$ can be analyzed as follows.  The number of iterations in the algorithm is equal to $|S_1 \times \cdots \times S_M| = \mathcal{O}(n^M)$, where we observe that it is trivial from a computation time analysis to enumerate and iterate over the tuples of products in $S_1 \times  \cdots \times S_M$. For each iteration, we must construct a directed graph $\mathcal{G}_{i_1 \cdots i_M}$ that has $n+1 = \mathcal{O}(n)$ vertices and $\sum_{m\in\mathcal{M}} (| S_m| + 1) = \mathcal{O}(Mn)$ directed edges. Constructing this directed graph requires examining each product in each of the assortments $S_1,\ldots,S_M$, and so we can construct $\mathcal{G}_{i_1 \cdots i_M}$ in $\mathcal{O}(Mn)$ computation time.  The computation time for the topological sort algorithm on a directed graph is equal to the number of vertices plus the number of directed edges in the directed graph, and so we can check whether $\mathcal{G}_{i_1 \cdots i_M}$ is acyclic in $\mathcal{O}(n + Mn) = \mathcal{O}(Mn)$ computation time. Therefore, we have shown that our algorithm for constructing the set of tuples of products $\mathcal{L}$ requires $\mathcal{O}\left(n^M \times \left(Mn + Mn \right)\right) = \mathcal{O}(M n^{M+1})$ computation time. We also observe from the inequality $| \mathcal{L}| \le |S_1 \times \cdots \times S_M|$ that $| \mathcal{L}| = \mathcal{O}(n^M)$, which concludes our proof of Lemma~\ref{lem:fixed_dim:construct_L}. 
\halmos

\subsection{Proof of Lemma~\ref{lem:fixed_dim:construct_rho}}
Consider any assortment $S \in \widehat{\mathcal{S}}$ and any tuple of products $(i_1,\ldots,i_M) \in \mathcal{L}$. To motivate our algorithm for computing $\rho_{i_1 \cdots i_M}(S)$, we begin by recalling the directed graph $\mathcal{G}_{i_1 \cdots i_M}$ corresponding to the tuple of products $(i_1,\ldots,i_M)$ that is described in the beginning of Appendix~\ref{appx:graphical}. According to Lemma~\ref{lem:dag} in Appendix~\ref{appx:graphical}, it follows from the fact that $(i_1,\ldots,i_M) \in \mathcal{L}$ that $\mathcal{G}_{i_1 \cdots i_M}$ is a directed acyclic graph. Moreover, we recall that
\begin{align}
\rho_{i_1 \cdots i_M}(S) &= \min_{i \in S \cap \mathcal{I}_{i_1 \cdots i_M}(S)} r_i = \min_{i \in S: \; i_m \nprec_{i_1 \cdots i_M} i  \textnormal{ for all } i_m \in S} r_i, \label{line:duhhh}
\end{align}
where the first equality follows from Proposition~\ref{prop:cost_reform} in  Appendix~\ref{appx:graphical} and the second equality follows from Definition~\ref{defn:I} in  Appendix~\ref{appx:graphical}. We recall from Definition~\ref{def:reachable} in Appendix~\ref{appx:graphical} that the notation $i \nprec_{i_1 \cdots i_M} j$ means that  there is no directed path from  vertex $j$ to vertex $i$ in the graph $\mathcal{G}_{i_1 \cdots i_M}$. Stated in words, line~\eqref{line:duhhh} shows that $\rho_{i_1 \cdots i_M}(S)$ is equal to the minimum  revenue $r_i$ among all of the vertices $i$ in the directed acyclic graph $\mathcal{G}_{i_1 \cdots i_M}$ which do not have a directed path to a vertex $i_m$ for any past assortment $m \in \mathcal{M}$ which satisfies $i_m \in S$. 

Based on the above observations, we arrive at the following straightforward algorithm (presented as Algorithm~\ref{alg:construct_rho}) for computing $\rho_{i_1 \cdots i_M}(S)$. The algorithm begins on line~\eqref{step:construct_rho:1} of Algorithm~\ref{alg:construct_rho} by constructing  the directed acyclic graph $\mathcal{G}_{i_1 \cdots i_M}$.  In particular, it follows from identical reasoning as in the proof of Lemma~\ref{lem:fixed_dim:construct_L} that $\mathcal{G}_{i_1 \cdots i_M}$ can be constructed in $\mathcal{O}(Mn)$ computation time and that this graph  is comprised of  $\mathcal{O}(n)$ vertices and $\mathcal{O}(Mn)$ directed edges. In line~\eqref{step:construct_rho:2} of Algorithm~\ref{alg:construct_rho}, we iterate over each of the previously offered assortments $m \in \mathcal{M}$ which satisfy $i_m \in S$. For each such past assortment $m$, we mark all of the vertices in the graph which have a directed path to vertex $i_m$. Since the graph has  $\mathcal{O}(Mn)$ directed edges, and assuming henceforth that $\mathcal{G}_{i_1 \cdots i_M}$ is stored in memory as an adjacency list,  it is easy to see using a standard graph traversal algorithm like depth-first search that the loop in  line~\eqref{step:construct_rho:2} of Algorithm~\ref{alg:construct_rho} can be performed in a total of $\mathcal{O}(M \times Mn) = \mathcal{O}(M^2 n)$ computation time. Finally, line~\eqref{step:construct_rho:3} of Algorithm~\ref{alg:construct_rho} iterates over all of the vertices in at most $\mathcal{O}(Mn)$ computation time and outputs the minimum $r_i$ among all of the vertices $i$ which are unmarked. The correctness of Algorithm~\ref{alg:construct_rho} follows immediately from our earlier reasoning on line~\eqref{line:duhhh}, and the total computation time required for Algorithm~\ref{alg:construct_rho} is
\begin{align*}
\underbrace{\mathcal{O}\left( Mn \right)}_{\eqref{step:construct_rho:1}} + \underbrace{\mathcal{O}\left( M^2n \right)}_{\eqref{step:construct_rho:2}} + \underbrace{\mathcal{O}\left( n \right)}_{\eqref{step:construct_rho:3}} = \mathcal{O}(M^2 n). 
\end{align*}
This concludes our proof of Lemma~\ref{lem:fixed_dim:construct_rho}. \halmos 
\begin{algorithm}[t] 
\begin{center}
\fbox{\begin{minipage}{\linewidth}
\begin{center}
\textsc{\underline{Construct-${{\rho}}(S, (i_1,\ldots,i_M), \mathscr{M}, r)$}}\\
\end{center}
\vspace{1em}
\textbf{Inputs}: 
\begin{itemize}
\item An assortment, $S \in \mathcal{S}$. 
\item A tuple of products, $(i_1,\ldots,i_M) \in \mathcal{L}$. 
\item The collection of past assortments,  $\mathscr{M}  \equiv \{S_1,\ldots,S_M\}$. 
\item The revenues of the products, $r \equiv (r_0,r_1,\ldots,r_n)$. 
\end{itemize}
\vspace{1em}
\textbf{Output}:
\begin{itemize}
\item  The quantity $\rho_{i_1 \cdots i_M}(S)$. 
\end{itemize} 
\vspace{1em}
\textbf{Procedure}: 
\begin{enumerate}
\item Construct the directed acyclic graph $\mathcal{G}_{i_1 \cdots i_M}$. 
\label{step:construct_rho:1}
\item For each past assortment $m \in  \mathcal{M}$:  \label{step:construct_rho:2}
\begin{enumerate}
\item If $i_m \in S$: 
\begin{enumerate}
\item Mark each unmarked vertex in $\mathcal{G}_{i_1 \cdots i_M}$ that has a directed path to vertex $i_m$. 
\end{enumerate}
\end{enumerate}
\item Output the minimum $r_i$ among all unmarked vertices $i$ in $\mathcal{G}_{i_1 \cdots i_M}$ and terminate.  \label{step:construct_rho:3}
\end{enumerate}
\vspace{1em}
\end{minipage}}
\end{center}
\caption{A procedure for computing $\rho_{i_1 \cdots i_M}(S)$.} \label{alg:construct_rho}
\end{algorithm}

\subsection{Proof of Theorem~\ref{thm:poly}} \label{appx:thm:poly}
As described at the beginning of \S\ref{sec:algorithms:fixed_dim}, we consider a brute-force algorithm for solving the robust optimization problem~\eqref{prob:robust} that consists of computing the worst-case expected revenue $\min_{\lambda \in \mathcal{U}} \mathscr{R}^{\lambda}(S)$ for each candidate assortment $S \in  \widehat{\mathcal{S}}$ and outputting the assortment that has the maximum worst-case expected revenue. The computation time of this brute-force algorithm is analyzed as follows.  In Lemmas~\ref{lem:fixed_dim:S:time} and \ref{lem:fixed_dim:S}, we established that the collection of assortments $\widehat{\mathcal{S}}$ can be constructed in $\mathcal{O}(n^2(M + | \widehat{\mathcal{S}}|)) = \mathcal{O}(n^2(M + (n+2)^{2^M})) = \mathcal{O}(\textnormal{poly}(n))$ computation time, where the last equality holds for any fixed $M$. In Lemma~\ref{lem:fixed_dim:construct_L}, we established that the set of tuples of products $\mathcal{L}$ can be constructed in $\mathcal{O}(Mn^{M+1}) = \mathcal{O}(\textnormal{poly}(n))$ computation time, where the equality holds for any fixed $M$.  Our algorithm performs an iteration for each  assortment $S \in \widehat{\mathcal{S}}$, and thus our algorithm will perform $| \widehat{\mathcal{S}}| \le (n+2)^{2^M} =  \mathcal{O}(\textnormal{poly}(n))$ iterations for any fixed $M$. Given any assortment $S \in \widehat{\mathcal{S}}$,  we established in Lemma~\ref{lem:fixed_dim:construct_rho} that the quantities $\rho_{i_1 \cdots i_M}(S)$ for each $(i_1,\ldots,i_M) \in \mathcal{L}$ can all be computed in a total of $\mathcal{O} ( | \mathcal{L}| \times M^2 n) = \mathcal{O}(n^M \times M^2 n) = \mathcal{O}(\textnormal{poly}(n))$ computation time, where the last equality holds for any fixed $M$. Given the set of tuples of products $\mathcal{L}$ and the quantities $\rho_{i_1 \cdots i_M}(S)$ for each $(i_1,\ldots,i_M) \in \mathcal{L}$, we established in Remark~\ref{remark:tractability} that we can compute the worst-case expected revenue $\min_{\lambda \in \mathcal{U}} \mathscr{R}^{\lambda}(S)$ by constructing and solving a linear optimization problem with $\mathcal{O}(n^M) = \mathcal{O}(\textnormal{poly}(n))$ decision variables and $\mathcal{O}(nM) = \mathcal{O}(\textnormal{poly}(n))$ constraints for any fixed $M$. Since linear optimization can be solved in weakly polynomial-time via the ellipsoid algorithm, we conclude that each iteration of our algorithm requires $\mathcal{O}(\textnormal{poly}(n))$ time for any fixed $M$. Our proof of Theorem~\ref{thm:poly}  is thus complete. 
\halmos

\section{Proofs of Technical Results from \S\ref{sec:algorithms:nested}}\label{appx:big_M:mip}
This appendix contains the proofs of the technical results from \S\ref{sec:algorithms:nested}. Appendices~\ref{appx:big_M:mip:reform_of_L}, \ref{appx:big_M:mip:cor:reform_of_L},  and \ref{appx:big_M:mip:rho_m}  contain the proofs of Lemma~\ref{lem:reform_of_L}, Corollary~\ref{cor:reform_of_L},  and Lemma~\ref{lem:rho_m}. Appendix~\ref{appx:mip:intermediary} contains two intermediary results (Lemmas~\ref{lem:simplified_to_compact} and \ref{lem:compact_to_bound}) that are used in the proofs of Proposition~\ref{prop:lp_reform_nested} and Theorem~\ref{thm:nested}. Appendices~\ref{appx:lp_reform_nested} and \ref{appx:mip:lp_reform_nested} contain the proofs of Proposition~\ref{prop:lp_reform_nested}  and Theorem~\ref{thm:nested}.  
We remark that the proofs in this appendix make  extensive use of the notation and intermediary results from  \S\ref{sec:characterization:proof} and  Appendix~\ref{appx:graphical}. 

\subsection{Proof of Lemma~\ref{lem:reform_of_L}} \label{appx:big_M:mip:reform_of_L}

Let Assumption~\ref{ass:nested} hold, consider any tuple $(i_1,\ldots,i_M) \in S_1 \times \cdots \times S_M$, and consider the corresponding directed graph $\mathcal{G}_{i_1 \cdots i_M} $.  We recall that the set of vertices in the directed graph is equal to $\mathcal{N}_0$, and the graph has a directed edge $(i,i_m)$ from each vertex $i \in S_m \setminus \{i_m\}$ to vertex $i_m$ for each past assortment $m \in \mathcal{M}$. 

To show the first direction, suppose that there exists an $m \in \{1,\ldots,M-1\}$ such that $i_{m+1} \notin \{i_m \}\cup \mathcal{B}_{m+1}$.  
In this case, we observe that
\begin{align*}
i_{m+1} \in S_{m+1} \setminus \left(  \{i_m \}\cup \mathcal{B}_{m+1}  \right) = S_{m+1} \setminus \left( \{i_m \}\cup \left( S_{m+1} \setminus S_{m} \right) \right)  = S_{m} \setminus \{i_m\}, 
\end{align*}
where the inclusion follows from the fact that $i_{m+1} \in S_{m+1}$ and from the supposition that $i_{m+1} \notin \{i_m \}\cup \mathcal{B}_{m+1}$, the first equality follows from line~\eqref{defn:B_m} and from the fact that $m+1 \in \{2,\ldots,M\}$, and the second equality  follows from algebra. 
Moreover, it follows from the fact that the past assortments are nested  that $i_m \in S_{m+1}$ and from the supposition that $i_{m+1} \neq i_m$ that $(i_m,i_{m+1})$ is a directed edge in $\mathcal{G}_{i_1 \cdots i_M} $.
Since the directed graph $\mathcal{G}_{i_1 \cdots i_M}$ has a directed edge from each vertex $i \in S_m \setminus \{i_m\}$ to vertex $i_m$, we conclude that the directed graph $\mathcal{G}_{i_1 \cdots i_M}$ must have an edge from vertex $i_{m+1}$ to vertex $i_m$, which implies that $\mathcal{G}_{i_1 \cdots i_M}$ has a cycle. Therefore, it follows from Lemma~\ref{lem:dag} from Appendix~\ref{appx:graphical} that $(i_1,\ldots,i_M) \notin \mathcal{L}$.

To show the other direction,  suppose that  the inclusion $i_{m+1} \in \{i_m \}\cup \mathcal{B}_{m+1}$ holds for all $m \in \{1,\ldots,M-1\}$. For notational convenience, let $i_{m_1},\ldots,i_{m_K}$ denote the elements of the set $\cup_{m=1}^M \{ i_m\}$, where we assume without loss of generality that the indices $m_1,\ldots,m_K$ satisfy $m_1 < \cdots <m_K$, satisfy $m_1 = 1$, and  satisfy $i_{m_k-1} \neq i_{m_k}$ for all $k \in \{2,\ldots,K\}$. It follows from the construction of the directed graph $\mathcal{G}_{i_1 \cdots i_M}$ that $i_{m_1},\ldots,i_{m_K}$ are the only vertices in the directed graph with incoming edges.  Moreover, 
it follows from our choice of the indices  $m_1,\ldots,m_K$ and from the supposition that $i_{m+1} \in \{i_m \}\cup \mathcal{B}_{m+1}$ for all $m \in \{1,\ldots,M-1\}$  that $i_{m_k} \in \mathcal{B}_{m_k}$ for each $k \in \{1,\ldots,K\}$, which implies that $i_{m_{k'}} \notin S_{m_{k}}$ for all $1 \le k < k' \le K$.  Therefore, we conclude for all $1 \le k < k' \le K$ that there is no directed edge in  $\mathcal{G}_{i_1 \cdots i_M}$ from vertex $i_{m_{k'}}$ to vertex $i_{m_{k}}$, which implies that $\mathcal{G}_{i_1 \cdots i_M}$ is acyclic. Therefore, it follows from Lemma~\ref{lem:dag} that $(i_1,\ldots,i_M) \in \mathcal{L}$. 
\halmos 


\subsection{Proof of Corollary~\ref{cor:reform_of_L}} \label{appx:big_M:mip:cor:reform_of_L}
Let  Assumption~\ref{ass:nested} hold. For each $m \in \mathcal{M}$, let us define the following set of tuples:
\begin{align}
\mathcal{L}_m \triangleq  \begin{cases}
\left \{ (i_1): i_1 \in \mathcal{B}_1 \right \},&\text{if } m=1,\\
 \left \{(i_1,\ldots,i_m): (i_1,\ldots,i_{m-1}) \in \mathcal{L}_{m-1} \textnormal{ and } i_{m} \in \{i_{m-1} \} \cup \mathcal{B}_{m}  \right \},&\text{if } m \in \{2,\ldots,M\}.
 \end{cases} \label{line:uselessstuff}
\end{align}
It follows from a straightforward induction argument and from Lemma~\ref{lem:reform_of_L} that $\mathcal{L}_M = \mathcal{L}$. Moreover, it follows from a straightforward induction argument and from line~\eqref{line:uselessstuff}  that the following holds for each $m \in \mathcal{M}$:
\begin{align}
\left|\mathcal{L}_m \right| = \begin{cases}
 \left | \mathcal{B}_1 \right |,&\text{if } m=1,\\
  \left | \mathcal{B}_1 \right |  \times  \prod_{m'=2}^m \left(1 + \left| \mathcal{B}_{m'} \right| \right),&\text{if } m \in \{2,\ldots,M\}.
  \end{cases} \label{line:uselessstuff2}
\end{align}
Therefore, we conclude that
\begin{align*}
\left| \mathcal{L} \right| = \left | \mathcal{B}_1 \right |  \times  \prod_{m=2}^M \left(1 + \left| \mathcal{B}_{m} \right| \right) \ge 2^{M-1},
\end{align*}
where the equality follows from line~\eqref{line:uselessstuff2} and from the fact that $\mathcal{L}_M = \mathcal{L}$, and the inequality follows from the fact that $| \mathcal{B}_m| \ge 1$ for all $m \in \mathcal{M}$. This completes our proof of Corollary~\ref{cor:reform_of_L}. 
\halmos
 
\subsection{Proof of Lemma~\ref{lem:rho_m}} \label{appx:big_M:mip:rho_m}

We assume throughout the proof of Lemma~\ref{lem:rho_m} that Assumption~\ref{ass:nested} holds, and that $S \subseteq \mathcal{N}_0$ is an assortment that satisfies $S \cap S_1 \neq \emptyset$. 

We begin by considering any tuple $(i_1,\ldots,i_M) \in \mathcal{L}$ and $m \in \mathcal{M}$ that satisfy $i_m \in S$. In this case, we recall from Appendix~\ref{appx:graphical} that the graph $\mathcal{G}_{i_1 \cdots i_M}$ has a directed edge from each vertex $i \in S_m \setminus \{i_m\}$ to vertex $i_m$. In particular, that implies that there is a directed edge from each vertex $i \in (S \cap S_m) \setminus \{i_m\}$ to vertex $i_m$. Since $i_m \in S$, we have thus shown that $i_m \prec_{i_1 \cdots i_M} i$ for all $i \in (S \cap S_m) \setminus \{i_m\}$. Therefore, 
\begin{align}
\rho_{i_1 \cdots i_M}(S \cap S_m)  &= \min_{i \in S \cap S_m \cap \mathcal{I}_{i_1 \cdots i_M}(S \cap S_m)} r_i  \notag \\
&=  \min_{i \in S \cap S_m \cap \left \{j \in \mathcal{N}_0: \textnormal{ for all } m'\in \mathcal{M}, \textnormal{ if } i_{m'} \in  S \cap S_m, \textnormal{ then } i_{m'} \nprec_{i_1 \cdots i_M} j \right \} } r_i \notag  \\
&=  \min_{i \in S \cap S_m:  \textnormal{ for all } m'\in \mathcal{M}, \textnormal{ if } i_{m'} \in  S \cap S_m, \textnormal{ then } i_{m'} \nprec_{i_1 \cdots i_M} i } r_i \notag \\
&=  \min_{i \in S \cap S_m:  i = i_m} r_i \notag  \\
&= r_{i_m},\label{line:supplement_to_ih}
\end{align}
where the first equality follows from Proposition~\ref{prop:cost_reform} from Appendix~\ref{appx:graphical}, the second equality follows from Definition~\ref{defn:I} from Appendix~\ref{appx:graphical}, the third equality follows from algebra, the fourth equality follows from the fact that $i_m \prec_{i_1 \cdots i_M} i$ for all $i \in (S \cap S_m) \setminus \{i_m\}$, and the fifth equality follows from algebra. In summary, we have shown that if a tuple $(i_1,\ldots,i_M) \in \mathcal{L}$ and $m \in \mathcal{M}$ satisfy $i_m \in S$, then $\rho_{i_1 \cdots i_M}(S \cap S_m)  = r_{i_m}$. 

We next consider any tuple $(i_1,\ldots,i_M) \in \mathcal{L}$ that satisfies $i_1 \notin S$. 
For this case, we observe that
\begin{align*}
\rho_{i_1 \cdots i_M }(S \cap S_1) &= \min_{i \in S \cap S_1 \cap \mathcal{I}_{i_1 \cdots i_M}(S \cap S_1)} r_i \\
&= \min_{i \in S \cap S_1 \cap\left \{ j \in \mathcal{N}_0: \textnormal{ for all } m'\in \mathcal{M}, \textnormal{ if } i_{m'} \in  S \cap S_1, \textnormal{ then } i_{m'} \nprec_{i_1 \cdots i_M} j \right \} } r_i \\
 &= \min_{i  \in S \cap S_1: \textnormal{ for all } m'\in \mathcal{M}, \textnormal{ if } i_{m'} \in  S \cap S_1, \textnormal{ then } i_{m'} \nprec_{i_1 \cdots i_M} i} r_i \\
 &= \min_{i  \in S \cap S_1: \textnormal{ for all } m'\in \{2,\ldots,M\}, \textnormal{ if } i_{m'} \in  S \cap S_1, \textnormal{ then } i_{m'} \nprec_{i_1 \cdots i_M} i} r_i \\
  &=  \min_{i \in S \cap S_1} r_i \\
    &=  \min_{i \in S \cap \mathcal{B}_1} r_i.
\end{align*}
Indeed, the first equality follows from Proposition~\ref{prop:cost_reform}. The second equality follows from Definition~\ref{defn:I}. The third equality follows from algebra. The fourth equality follows from the fact that $i_1 \notin S$. The fifth equality follows from Lemma~\ref{lem:reform_of_L}, which together with the facts that $(i_1,\ldots,i_M) \in \mathcal{L}$ and $i_1 \notin S$ readily implies that $i_2,\ldots,i_M \notin S \cap S_1$. The sixth equality follows from line~\eqref{defn:B_m}. In summary, we have shown that if a tuple $(i_1,\ldots,i_M) \in \mathcal{L}$ satisfies $i_1 \notin S$, then $\rho_{i_1 \cdots i_M}(S \cap S_1)  =  \min_{i \in S \cap \mathcal{B}_1} r_i$. 

Finally, we consider any tuple $(i_1,\ldots,i_M) \in \mathcal{L}$ and ${m} \in \{2,\ldots,M\}$ that satisfy $i_m \notin S$. In this case, we observe that 
\begin{align}
\rho_{i_1 \cdots i_M}(S \cap S_m)&= \min_{i \in S \cap S_m \cap \mathcal{I}_{i_1 \cdots i_M}(S \cap S_m)} r_i \notag \\
&=  \min_{i \in S \cap (S_{m-1} \cup \mathcal{B}_m) \cap \mathcal{I}_{i_1 \cdots i_M}(S \cap S_m)} r_i \notag \\
&= \min \left \{ \min_{i \in S \cap S_{m-1} \cap  \mathcal{I}_{i_1 \cdots i_M}(S \cap S_m)} r_i, \min_{i \in S \cap \mathcal{B}_{m} \cap \mathcal{I}_{i_1 \cdots i_M}(S \cap S_m)} r_i\right \},  \label{line:splitting_into_parts_long_lemma}
\end{align}
where the first equality follows from Proposition~\ref{prop:cost_reform}, the second equality follows from Assumption~\ref{ass:nested},  and the third equality follows from algebra. In the remainder of the proof, we focus on evaluating the two terms $\min_{i \in S \cap S_{m-1} \cap  \mathcal{I}_{i_1 \cdots i_M}(S \cap S_m)} r_i$ and $\min_{i \in S \cap \mathcal{B}_{m} \cap \mathcal{I}_{i_1 \cdots i_M}(S \cap S_m)} r_i$ that appear in line~\eqref{line:splitting_into_parts_long_lemma}. First, we   observe that 
\begin{align}
\min_{i \in S \cap S_{m-1} \cap  \mathcal{I}_{i_1 \cdots i_M}(S \cap S_m)} r_i &= \min_{i \in S \cap S_{m-1} \cap\left \{ j \in \mathcal{N}_0: \textnormal{ for all } m'\in \mathcal{M}, \textnormal{ if } i_{m'} \in  S \cap S_m, \textnormal{ then } i_{m'} \nprec_{i_1 \cdots i_M} j \right \} } r_i  \notag \\
  &= \min_{i  \in S \cap S_{m-1} \cap \left \{ j \in \mathcal{N}_0: \textnormal{ for all } m'\in \{1,\ldots,m-1\}, \textnormal{ if } i_{m'} \in  S \cap S_m, \textnormal{ then } i_{m'} \nprec_{i_1 \cdots i_M} j \right \}} r_i \notag \\
    &= \min_{i  \in S \cap S_{m-1} \cap \left \{ j \in \mathcal{N}_0: \textnormal{ for all } m'\in \{1,\ldots,m-1\}, \textnormal{ if } i_{m'} \in  S \cap S_{m-1}, \textnormal{ then } i_{m'} \nprec_{i_1 \cdots i_M} j \right \}} r_i \notag \\
        &= \min_{i  \in S \cap S_{m-1} \cap \left \{ j \in \mathcal{N}_0: \textnormal{ for all } m'\in \mathcal{M}, \textnormal{ if } i_{m'} \in  S \cap S_{m-1}, \textnormal{ then } i_{m'} \nprec_{i_1 \cdots i_M} j \right \}} r_i \notag \\
                &= \min_{i  \in S \cap S_{m-1} \cap \mathcal{I}_{i_1 \cdots i_M}(S \cap S_{m-1})} r_i \notag \\
 &= \rho_{i_1 \cdots i_M}(S \cap S_{m-1}).   \label{line:splitting_into_parts_long_lemma:part_1} 
    \end{align}
    Indeed, the first equality follows from Definition~\ref{defn:I}. The second equality  follows from Lemma~\ref{lem:reform_of_L}, which together with the facts that $(i_1,\ldots,i_M) \in \mathcal{L}$ and $i_m \notin S$ implies that $i_m,\ldots,i_M \notin S \cap S_m$. The third equality follows from Assumption~\ref{ass:nested}, which together with the fact that  $i_m \notin S$ implies  for each $m' \in \{1,\ldots,m-1\}$ that $i_{m'} \in S \cap S_{m}$ if and only if  $i_{m'} \in S \cap S_{m-1}$.  The fourth equality follows from the fact that $i_m,\ldots,i_M \notin S \cap S_m$ and the fact that $S \cap S_{m-1} \subseteq S \cap S_m$.   The fifth equality follows from  Definition~\ref{defn:I}, and the sixth equality follows from Proposition~\ref{prop:cost_reform}. Moreover, we observe that 
\begin{align}
\min_{i \in S \cap \mathcal{B}_m  \cap  \mathcal{I}_{i_1 \cdots i_M}(S \cap S_m)} r_i &= \min_{i \in S \cap \mathcal{B}_m \cap\left \{ j \in \mathcal{N}_0: \textnormal{ for all } m'\in \mathcal{M}, \textnormal{ if } i_{m'} \in  S \cap S_m, \textnormal{ then } i_{m'} \nprec_{i_1 \cdots i_M} j \right \} } r_i  \notag \\
  &= \min_{i  \in S \cap \mathcal{B}_m \cap \left \{ j \in \mathcal{N}_0: \textnormal{ for all } m'\in \{1,\ldots,m-1\}, \textnormal{ if } i_{m'} \in  S \cap S_m, \textnormal{ then } i_{m'} \nprec_{i_1 \cdots i_M} j \right \}} r_i \notag \\
    &= \min_{i  \in S \cap\mathcal{B}_m \cap \left \{ j \in \mathcal{N}_0: \textnormal{ for all } m'\in \{1,\ldots,m-1\}, \textnormal{ if } i_{m'} \in  S \cap S_{m-1}, \textnormal{ then } i_{m'} \nprec_{i_1 \cdots i_M} j \right \}} r_i \notag  \\
        &= \min_{i  \in S \cap\mathcal{B}_m:  \textnormal{ for all } m'\in \{1,\ldots,m-1\}, \textnormal{ if } i_{m'} \in  S \cap S_{m-1}, \textnormal{ then } i_{m'} \nprec_{i_1 \cdots i_M} i} r_i. \label{line:splitting_into_parts_long_lemma:part_2} 
            \end{align}
     Indeed, the first equality follows from Definition~\ref{defn:I}. The second equality  follows from Lemma~\ref{lem:reform_of_L}, which together with the facts that $(i_1,\ldots,i_M) \in \mathcal{L}$ and $i_m \notin S$ implies that $i_m,\ldots,i_M \notin S \cap S_m$. The third equality follows from Assumption~\ref{ass:nested},  which together with the fact that  $i_m \notin S$  implies for each $m' \in \{1,\ldots,m-1\}$ that $i_{m'} \in S \cap S_{m}$ if and only if  $i_{m'} \in S \cap S_{m-1}$. To simplify the expression on line~\eqref{line:splitting_into_parts_long_lemma:part_2}, we make use of the following intermediary claim. 
\begin{claim} \label{claim:rho_m:reachable}
Consider any  $i \in S \cap \mathcal{B}_m$ and $m' \in \{1,\ldots,m-1\}$. If $i_{m'} \in S$, then $i_{m'} \nprec_{i_1 \cdots i_M} i$. 
\end{claim}
\begin{proof}{Proof of Claim~\ref{claim:rho_m:reachable}.}
Consider any $i \in S \cap \mathcal{B}_m$, and consider any $m' \in \{1,\ldots,m-1\}$ that satisfies the inclusion $i_{m'} \in S$. 

We first show that $i_{m''} \prec_{i_1 \cdots i_M} i_{m'}$ for all $m'' \in \{m,\ldots,M\}$.   Indeed, consider any arbitrary $m'' \in \{m,\ldots,M\}$. We readily observe from Assumption~\ref{ass:nested} that $i_{m'} \in S_{m'} \subseteq S_{m''}$. Moreover, it follows from the fact that $i_{m'} \in S$ and from the fact that $i_m \notin S$ that the expression $i_{m'} \neq i_m$ must hold. Combining the expression $i_{m'} \neq i_m$  with Lemma~\ref{lem:reform_of_L}, we observe that the expression $i_{m'} \neq i_{m''}$ must hold. Combining the expression $i_{m'} \neq i_{m''}$ with the fact that $i_{m'},i_{m''} \in S_{m''}$, we conclude that the graph $\mathcal{G}_{i_1 \cdots i_M}$ must have a directed edge from vertex $i_{m'}$ to vertex $i_{m''}$. Because $m'' \in \{m,\ldots,M\}$ was chosen arbitrarily, we conclude  that $i_{m''} \prec_{i_1 \cdots i_M} i_{m'}$ for all $m'' \in \{m,\ldots,M\}$.

We next show that the vertex $i$ in the graph $\mathcal{G}_{i_1 \cdots i_M}$ has an outgoing edge to vertex $j$  only if there exists $m'' \in \{m,\ldots,M\}$ such that  $j = i_{m''}$. Indeed, we observe from the fact that $i \in S \cap \mathcal{B}_m \subseteq \mathcal{B}_m$ and from Assumption~\ref{ass:nested}  that the inclusion $i \in S_{m''}$ holds if and only if $m'' \in \{m,\ldots,M\}$. Therefore, it follows from the construction of the directed  graph $\mathcal{G}_{i_1 \cdots i_M}$  that $(i,j)$ is a directed edge in $\mathcal{G}_{i_1 \cdots i_M}$ only if $j = i_{m''}$ for some $m'' \in \{m,\ldots,M\}$.

We now combine the above steps to conclude the proof of Claim~\ref{claim:rho_m:reachable}. Indeed, suppose for the sake of developing a contradiction that $i_{m'} \prec_{i_1 \cdots i_M} i$. We showed previously that the only outgoing edges from the vertex $i$ in the  directed  graph $\mathcal{G}_{i_1 \cdots i_M}$ are to vertices $j$ that satisfy the equality $j = i_{m''}$ for some $m'' \in \{m,\ldots,M\}$. Therefore, it follows from the supposition that $i_{m'} \prec_{i_1 \cdots i_M} i$  that there must exist a $m'' \in \{m,\ldots,M\}$ such that $i_{m'} \prec_{i_1 \cdots i_M} i_{m''}$. Because we showed previously that $i_{m''} \prec_{i_1 \cdots i_M} i_{m'}$, we conclude that the  directed  graph $\mathcal{G}_{i_1 \cdots i_M}$  must have a cycle. However, we recall from  Lemma~\ref{lem:dag} from Appendix~\ref{appx:graphical} and from the fact that $(i_1,\ldots,i_M) \in \mathcal{L}$ that  $\mathcal{G}_{i_1 \cdots i_M}$  is acyclic. Because we have a contradiction, we conclude that $i_{m'} \nprec_{i_1 \cdots i_M} i$, which completes our proof of Claim~\ref{claim:rho_m:reachable}. \halmos 
\end{proof}

Combining lines~\eqref{line:splitting_into_parts_long_lemma}, \eqref{line:splitting_into_parts_long_lemma:part_1}, and \eqref{line:splitting_into_parts_long_lemma:part_2} with Claim~\ref{claim:rho_m:reachable}, we conclude that
\begin{align*}
\rho_{i_1 \cdots i_M}(S \cap S_m) &=  \min \left \{ \min_{i \in S \cap S_{m-1} \cap  \mathcal{I}_{i_1 \cdots i_M}(S \cap S_m)} r_i, \min_{i \in S \cap \mathcal{B}_{m} \cap \mathcal{I}_{i_1 \cdots i_M}(S \cap S_m)} r_i\right \}\\
&=  \min \left \{ \rho_{i_1 \cdots i_M}(S \cap S_{m-1}), \min_{i \in S \cap \mathcal{B}_{m} \cap \mathcal{I}_{i_1 \cdots i_M}(S \cap S_m)} r_i\right \}\\
&=  \min \left \{ \rho_{i_1 \cdots i_M}(S \cap S_{m-1}), \min_{i  \in S \cap\mathcal{B}_m:  \textnormal{ for all } m'\in \{1,\ldots,m-1\}, \textnormal{ if } i_{m'} \in  S \cap S_{m-1}, \textnormal{ then } i_{m'} \nprec_{i_1 \cdots i_M} i} r_i \right \}\\
&=  \min \left \{ \rho_{i_1 \cdots i_M}(S \cap S_{m-1}), \min_{i  \in S \cap\mathcal{B}_m} r_i \right \}. 
\end{align*}
Indeed, the first equality follows from \eqref{line:splitting_into_parts_long_lemma}, the second equality follows from \eqref{line:splitting_into_parts_long_lemma:part_1}, the third equality follows from \eqref{line:splitting_into_parts_long_lemma:part_2}, and the fourth equality follows from Claim~\ref{claim:rho_m:reachable}. Our proof of Lemma~\ref{lem:rho_m} is thus complete. 
\halmos

\subsection{Intermediary Results for Proposition~\ref{prop:lp_reform_nested} and Theorem~\ref{thm:nested}} \label{appx:mip:intermediary}

This appendix contains two intermediary lemmas that will be used in the proofs of Proposition~\ref{prop:lp_reform_nested} in Appendix~\ref{appx:lp_reform_nested} and the proof of Theorem~\ref{thm:nested} in Appendix~\ref{appx:mip:lp_reform_nested}. The purpose of these intermediary lemmas, which are presented below as Lemmas~\ref{lem:simplified_to_compact} and \ref{lem:compact_to_bound},  is to characterize the relationship between the the linear optimization problem~\eqref{prob:lp_reform_nested} and the linear optimization problem~\eqref{prob:robust_simplified}.  

For notational convenience, let the following set be defined for each assortment $S \subseteq \mathcal{N}_0$: 
\begin{align}
\mathfrak{N}(S) \triangleq \left \{ (m,i,\kappa) \in {\mathfrak{V}}: \begin{gathered} 
\left[  i \in S \textnormal{ and } i \in \mathcal{B}_m \textnormal{ and } \kappa \neq r_i \right]\\
 \textnormal{ or } \left[  i \in S \textnormal{ and } i \notin \mathcal{B}_m \textnormal{ and } \kappa \in \left\{r_j: j \in \mathcal{B}_m \right \}  \right]\\
 \textnormal{ or } \left[ \kappa \in \left \{ r_j: j \in \mathcal{B}_m \setminus S \right \} \right ]
\end{gathered} \right \}.  \label{line:N_frak_S}
\end{align} 
We remark that $\mathfrak{N}(S)$ is the set of vertices $(m,i,\kappa) \in \mathfrak{V}$ for which the linear optimization problem~\eqref{prob:lp_reform_nested} has a constraint of the form $g_{m,i,\kappa} = 0$. Moreover,  let the set of all paths of the form $(1,i_1,\kappa_1),\ldots,(M,i_M,\kappa_M) \in \mathfrak{V}$ in the directed acyclic graph $\mathfrak{G} \equiv (\mathfrak{V},\mathfrak{E})$ be denoted by
\begin{align*}
\mathscr{P} \triangleq \left \{((1,i_1,\kappa_1),\ldots, (M,i_M,\kappa_M)): ((m,i_m,\kappa_m),(m+1,i_{m+1},\kappa_{m+1})) \in \mathfrak{E} \; \forall m \in \{1,\ldots,M-1\}   \right \}. 
\end{align*}

 In our first intermediary lemma, denoted below by Lemma~\ref{lem:simplified_to_compact}, we show that every feasible solution of the linear optimization problem~\eqref{prob:robust_simplified} can be transformed into a feasible solution of the linear optimization problem~\eqref{prob:lp_reform_nested} with the same objective value.

\begin{lemma} \label{lem:simplified_to_compact}
Let Assumption~\ref{ass:nested} hold, let $S \subseteq \mathcal{N}_0$ satisfy $S \cap S_1 \neq \emptyset$, and let $(\lambda,\epsilon)$ be a  feasible solution for the linear optimization problem~\eqref{prob:robust_simplified}. For each vertex $(m,i,\kappa) \in \mathfrak{V}$, let 
\begin{align*}
 g_{m,i,\kappa} &\triangleq \sum_{(i_1,\ldots,i_M) \in \mathcal{L}: \; i_m = i,  \rho_{i_1 \cdots i_M}(S \cap S_m) = \kappa} \lambda_{i_1 \cdots i_M},
 \end{align*}
 and for each directed edge $((m,i,\kappa),(m+1,i',\kappa')) \in \mathfrak{E}$, let 
 \begin{align*}
f_{m,i,\kappa,i',\kappa'} &\triangleq \sum_{(i_1,\ldots,i_M) \in \mathcal{L}: \; i_m = i, \rho_{i_1 \cdots i_M}(S \cap S_m) = \kappa, i_{m+1} = i', \rho_{i_1 \cdots i_M}(S \cap S_{m+1}) = \kappa' } \lambda_{i_1 \cdots i_M}. 
\end{align*}
Then $(f,g,\epsilon)$ is a feasible solution for the linear optimization problem~\eqref{prob:lp_reform_nested}, and the objective value associated with $(f,g,\epsilon)$ in the  linear optimization problem~\eqref{prob:lp_reform_nested} is equal to the objective value associated with $(\lambda,\epsilon)$ in the  linear optimization problem~\eqref{prob:robust_simplified}. 
\end{lemma}
\begin{proof}{Proof.} 
Let Assumption~\ref{ass:nested} hold, let $S \subseteq \mathcal{N}_0$ satisfy $S \cap S_1 \neq \emptyset$, and let $(\lambda,\epsilon)$ be a  feasible solution for the linear optimization problem~\eqref{prob:robust_simplified}.  For each vertex $(m,i,\kappa) \in \mathfrak{V}$, let 
\begin{align*}
 g_{m,i,\kappa} &\triangleq \sum_{(i_1,\ldots,i_M) \in \mathcal{L}: \; i_m = i,  \rho_{i_1 \cdots i_M}(S \cap S_m) = \kappa} \lambda_{i_1 \cdots i_M},
 \end{align*}
 and for each edge $((m,i,\kappa),(m+1,i',\kappa')) \in \mathfrak{E}$, let 
 \begin{align*}
f_{m,i,\kappa,i',\kappa'} &\triangleq \sum_{(i_1,\ldots,i_M) \in \mathcal{L}: \; i_m = i, \rho_{i_1 \cdots i_M}(S \cap S_m) = \kappa, i_{m+1} = i', \rho_{i_1 \cdots i_M}(S \cap S_{m+1}) = \kappa' } \lambda_{i_1 \cdots i_M}. 
\end{align*}
We begin by presenting an intermediary claim, denoted below by Claim~\ref{claim:not_in_N_S}, which will provide the key ingredient in the proof of Lemma~\ref{lem:simplified_to_compact}. The purpose of this intermediary claim is to show that every tuple $(i_1,\ldots,i_M) \in \mathcal{L}$ can be converted into a path $(1,i_1,\kappa_1),\ldots,(M,i_M,\kappa_M)$ in which none of the vertices are elements of $\mathfrak{N}(S)$. 
\begin{claim} \label{claim:not_in_N_S}
For each $(i_1,\ldots,i_M) \in \mathcal{L}$, we have that $((1,i_1,\rho_{i_1 \cdots i_M}(S \cap S_1)),\ldots, (M,i_M,\rho_{i_1 \cdots i_M}(S \cap S_M))) \in \mathscr{P}$ and that  $(m,i_m,\rho_{i_1,\ldots,i_M}(S \cap S_m)) \notin \mathfrak{N}(S)$ for all $m \in \mathcal{M}$. 
\end{claim}
\begin{proof}{Proof of Claim~\ref{claim:not_in_N_S}.}
Consider any tuple $(i_1,\ldots,i_M) \in \mathcal{L}$.  It follows readily from Definition~\ref{defn:rho} that the inclusion $\rho_{i_1 \cdots i_M}(S \cap S_m) \in \{r_j: j \in S \cap S_m \}$ holds for each $m \in \mathcal{M}$, and so it follows from the definition of   $\mathfrak{V}$ that the inclusion $(m,i_m,\rho_{i_1 \cdots i_M}(S \cap S_m)) \in \mathfrak{V}$ holds for each $m \in \mathcal{M}$. Moreover, for every arbitrary $m \in \{1,\ldots,M-1\}$, we observe from Lemma~\ref{lem:reform_of_L} that $i_{m+1} \in \{i_m \} \cup \mathcal{B}_{m+1}$, and we observe from Lemma~\ref{lem:rho_m} and from the fact that $(i_1,\ldots,i_M) \in \mathcal{L}$ that  $\rho_{i_1 \cdots i_M}(S \cap S_{m+1}) \in \{ \rho_{i_1 \cdots i_M}(S \cap S_{m}) \} \cup \{r_j: j \in \mathcal{B}_m \}$. Combining those two observations, we have shown that $((m,i_m,\rho_{i_1 \cdots i_M}(S \cap S_m)), (m+1,i_{m+1},\rho_{i_1 \cdots i_M}(S \cap S_{m+1}))) \in \mathfrak{E}$. Since $m \in \{1,\ldots,M-1\}$ was chosen arbitrarily, we conclude that $((1,i_1,\rho_{i_1 \cdots i_M}(S \cap S_1)),\ldots, (M,i_M,\rho_{i_1 \cdots i_M}(S \cap S_M))) \in \mathscr{P}$. 

 To conclude the proof of Claim~\ref{claim:not_in_N_S}, it remains for us to show that $(m,i_m,\rho_{i_1,\ldots,i_M}(S \cap S_m)) \notin \mathfrak{N}(S)$ for all $m \in \mathcal{M}$. Indeed, for each $m \in \mathcal{M}$, we consider the following  three cases:
\begin{itemize}
\item If $i_m \in S$ and $i_m \in\mathcal{B}_m$, then it follows  from Lemma~\ref{lem:rho_m} that $\rho_{i_1 \cdots i_M}(S \cap S_m) = r_{i_m}$.
\item If $i_m \in S$ and $i_m \notin \mathcal{B}_m$,  then it follows  from Lemma~\ref{lem:rho_m} that $\rho_{i_1 \cdots i_M}(S \cap S_m) = r_{i_m} \notin \{r_j: j \in \mathcal{B}_m\}$. 
\item If $\rho_{i_1 \cdots i_M}(S \cap S_m) \in\{ r_j:j \in \mathcal{B}_m \}$, then it follows from Lemma~\ref{lem:rho_m} that $\rho_{i_1 \cdots i_M}(S \cap S_m) \in\{ r_j:j \in S \cap \mathcal{B}_m \}$. 
\end{itemize}
Combining the above cases with the definition of $\mathfrak{N}(S)$ on line~\eqref{line:N_frak_S}, we conclude that $(m,i_m,\rho_{i_1 \cdots i_M}(S \cap S_m)) \notin  \mathfrak{N}(S)$ for each $m \in \mathcal{M}$. Our proof of Claim~\ref{claim:not_in_N_S} is thus complete. 
\halmos \end{proof}

Using Claim~\ref{claim:not_in_N_S}, we show in the following bullet points that the solution $(f,g,\epsilon)$ is a feasible solution for the linear optimization problem~\eqref{prob:lp_reform_nested}; that is, we show that $(f,g,\epsilon)$ satisfies each of the constraints in the linear optimization problem~\eqref{prob:lp_reform_nested}.  

\begin{itemize}
\item For all $m \in \mathcal{M}$ and $i \in {S}_m$, we observe that
\begin{align*}
\sum_{\kappa:(m,i,\kappa) \in \mathfrak{V}} g_{m,i,\kappa} - \epsilon_{m,i}  &= \sum_{\kappa:(m,i,\kappa) \in \mathfrak{V}} \left(\sum_{(i_1,\ldots,i_M) \in \mathcal{L}: \; i_m = i,  \rho_{i_1 \cdots i_M}(S \cap S_m) = \kappa} \lambda_{i_1 \cdots i_M} \right) - \epsilon_{m,i} \\
&= \sum_{(i_1,\ldots,i_M) \in \mathcal{L}: \; i_m = i} \left( \sum_{\kappa:(m,i_m,\kappa) \in \mathfrak{V},  \rho_{i_1 \cdots i_M}(S \cap S_m) = \kappa}  \lambda_{i_1 \cdots i_M} \right) - \epsilon_{m,i} \\
&= \sum_{(i_1,\ldots,i_M) \in \mathcal{L}: \; i_m = i} \mathbb{I} \left \{ (m,i_m, \rho_{i_1 \cdots i_M}(S \cap S_m)) \in \mathfrak{V} \right \}  \lambda_{i_1 \cdots i_M}  - \epsilon_{m,i} \\
&=  \sum_{(i_1,\ldots,i_M) \in \mathcal{L}: \; i_m = i} \lambda_{i_1 \cdots i_M} - \epsilon_{m,i} \\
&= v_{m,i},
\end{align*}
Indeed,  the first equality follows from our construction of $g$. The second and third equalities follow from algebra. The fourth equality follows from Claim~\ref{claim:not_in_N_S}, which implies that $(m,i_m,\rho_{i_1 \cdots i_M}(S \cap S_m)) \in \mathfrak{V}$ for all $(i_1,\ldots,i_M) \in \mathcal{L}$. The fifth equality follows from the fact that $(\lambda,\epsilon)$ is a feasible solution for the linear optimization problem~\eqref{prob:robust_simplified}. 

\item For all $(m,i,\kappa) \in \mathfrak{V}$ such that $m \in \{1,\ldots,M-1\}$, we observe that
\begin{align*}
& \sum_{i',\kappa': ((m,i,\kappa),(m+1,i',\kappa')) \in \mathfrak{E}} f_{m,i,\kappa,i',\kappa'}  \\
 &=  \sum_{i',\kappa': ((m,i,\kappa),(m+1,i',\kappa')) \in \mathfrak{E}}  \left( \sum_{(i_1,\ldots,i_M) \in \mathcal{L}: \; i_m = i, \rho_{i_1 \cdots i_M}(S \cap S_m) = \kappa, i_{m+1} = i', \rho_{i_1 \cdots i_M}(S \cap S_{m+1}) = \kappa'} \lambda_{i_1 \cdots i_M} \right) \\
  &= \sum_{(i_1,\ldots,i_M) \in \mathcal{L}: \; i_m = i,  \rho_{i_1 \cdots i_M}(S \cap S_m) = \kappa} \left(  \sum_{i',\kappa': ((m,i_m, \rho_{i_1 \cdots i_M}(S \cap S_m)),(m+1,i',\kappa')) \in \mathfrak{E}, i_{m+1} = i', \rho_{i_1 \cdots i_M}(S \cap S_{m+1}) = \kappa'}   \lambda_{i_1 \cdots i_M} \right) \\
    &= \sum_{(i_1,\ldots,i_M) \in \mathcal{L}: \; i_m = i,  \rho_{i_1 \cdots i_M}(S \cap S_m) = \kappa}  \mathbb{I} \left \{ ((m,i_m,  \rho_{i_1 \cdots i_M}(S \cap S_m) ),(m+1,i_{m+1}, \rho_{i_1 \cdots i_M}(S \cap S_{m+1}))) \in \mathfrak{E} \right \}   \lambda_{i_1 \cdots i_M} \\
&=  \sum_{(i_1,\ldots,i_M) \in \mathcal{L}: \; i_m = i, \rho_{i_1 \cdots i_M}(S \cap S_m) = \kappa} \lambda_{i_1 \cdots i_M}\\
&= g_{m,i,\kappa},
\end{align*}
Indeed, the first equality follows from our construction of $f$. The second and third equalities follow from algebra.  The fourth equality follows from Claim~\ref{claim:not_in_N_S}, which implies that $((m,i_m,\rho_{i_1 \cdots i_M}(S \cap S_m)),(m+1,i_{m+1},\rho_{i_1 \cdots i_M}(S \cap S_{m+1}))) \in \mathfrak{E}$ for all $(i_1,\ldots,i_M) \in \mathcal{L}$. The fifth equality follows from our construction of $g$.

\item For all $(m,i,\kappa) \in \mathfrak{V}$ such that $m \in \{2,\ldots,M\}$, we observe that
\begin{align*}
& \sum_{i',\kappa': ((m-1,i',\kappa'),(m,i,\kappa)) \in \mathfrak{E}} f_{m-1,i',\kappa',i,\kappa}  \\
 &=   \sum_{i',\kappa': ((m-1,i',\kappa'),(m,i,\kappa)) \in \mathfrak{E}} \left( \sum_{(i_1,\ldots,i_M) \in \mathcal{L}: \; i_{m-1} = i', \rho_{i_1 \cdots i_M}(S \cap S_{m-1}) = \kappa', i_{m} = i, \rho_{i_1 \cdots i_M}(S \cap S_{m}) = \kappa} \lambda_{i_1 \cdots i_M} \right) \\
  &= \sum_{(i_1,\ldots,i_M) \in \mathcal{L}: \; i_m = i,  \rho_{i_1 \cdots i_M}(S \cap S_m) = \kappa} \left(  \sum_{i',\kappa': ((m-1,i',\kappa'),(m,i_m, \rho_{i_1 \cdots i_M}(S \cap S_m))) \in \mathfrak{E}, i_{m-1} = i', \rho_{i_1 \cdots i_M}(S \cap S_{m-1}) = \kappa'}   \lambda_{i_1 \cdots i_M} \right) \\
    &= \sum_{(i_1,\ldots,i_M) \in \mathcal{L}: \; i_m = i,  \rho_{i_1 \cdots i_M}(S \cap S_m) = \kappa}  \mathbb{I} \left \{ ((m-1,i_{m-1},  \rho_{i_1 \cdots i_M}(S \cap S_{m-1}) ),(m,i_{m}, \rho_{i_1 \cdots i_M}(S \cap S_{m}))) \in \mathfrak{E} \right \}   \lambda_{i_1 \cdots i_M} \\
&=  \sum_{(i_1,\ldots,i_M) \in \mathcal{L}: \; i_m = i, \rho_{i_1 \cdots i_M}(S \cap S_m) = \kappa} \lambda_{i_1 \cdots i_M}\\
&= g_{m,i,\kappa},
\end{align*}
Indeed, the first equality follows from our construction of $f$. The second and third equalities follow from algebra.  The fourth equality follows from Claim~\ref{claim:not_in_N_S}, which implies that $((m-1,i_{m-1},  \rho_{i_1 \cdots i_M}(S \cap S_{m-1}) ),(m,i_{m}, \rho_{i_1 \cdots i_M}(S \cap S_{m}))) \in \mathfrak{E} $ for all $(i_1,\ldots,i_M) \in \mathcal{L}$. The fifth equality follows from our construction of $g$.


\item Consider any vertex $(m,i,\kappa) \in \mathfrak{V}$ that satisfies $g_{m,i,\kappa} > 0$. Because $g_{m,i,\kappa} > 0$, it follows from our construction of $g$  that there must exist a tuple $(i_1,\ldots,i_M) \in \mathcal{L}$ that satisfies $i_m = i$, $\rho_{i_1 \cdots i_M}(S \cap S_m) = \kappa$, and $\lambda_{i_1 \cdots i_M} > 0$. Since Claim~\ref{claim:not_in_N_S} implies that  $(m,i_m,\rho_{i_1,\ldots,i_M}(S \cap S_m)) \notin \mathfrak{N}(S)$ for all $m \in \mathcal{M}$ and all $(i_1,\ldots,i_M) \in \mathcal{L}$, we conclude that 
\begin{align*}
&&& g_{m,i,\kappa} = 0 \quad \forall (m,i,\kappa)\in \mathfrak{N}(S).
\end{align*} 
\item We observe that
\begin{align*}
 \sum_{i,\kappa: (M,i,\kappa) \in \mathfrak{V}} g_{M,i,\kappa} &=  \sum_{i,\kappa: (M,i,\kappa) \in \mathfrak{V}} \left( \sum_{(i_1,\ldots,i_M) \in \mathcal{L}: \; i_M = i,  \rho_{i_1 \cdots i_M}(S \cap S_M) = \kappa} \lambda_{i_1 \cdots i_M} \right)\\
&=  \sum_{(i_1,\ldots,i_M) \in \mathcal{L}} \lambda_{i_1 \cdots i_M}\\
&= 1.
\end{align*}
Indeed, the first equality follows from our construction of $g$. The second equality follows from  Claim~\ref{claim:not_in_N_S}, which implies that $(M,i_M,\rho_{i_1 \cdots i_M}(S \cap S_M)) \in \mathfrak{V}$ for all $(i_1,\ldots,i_M) \in \mathcal{L}$. The third equality follows from the fact that $(\lambda,\epsilon)$ is a feasible solution for the linear optimization problem~\eqref{prob:robust_simplified}. 
\end{itemize}
Since $f,g$ are clearly nonnegative, the above bullet points show that $(f,g,\epsilon)$ satisfies each of the constraints in the linear optimization problem~\eqref{prob:lp_reform_nested}.  

It remains for us to show that the objective value associated with $(f,g,\epsilon)$ in the linear optimization problem~\eqref{prob:lp_reform_nested} is equal to the objective value associated with $(\lambda,\epsilon)$ in the  linear optimization problem~\eqref{prob:robust_simplified}. Indeed, we observe that
\begin{align*}
 \sum_{i,\kappa: (M,i,\kappa) \in \mathfrak{V}}\kappa  g_{M,i,\kappa}  &=  \sum_{i,\kappa: (M,i,\kappa) \in \mathfrak{V}}\kappa  \left(\sum_{(i_1,\ldots,i_M) \in \mathcal{L}: \; i_M = i,  \rho_{i_1 \cdots i_M}(S \cap S_M) = \kappa} \lambda_{i_1 \cdots i_M}\right) \\
 &=  \sum_{i,\kappa: (M,i,\kappa) \in \mathfrak{V}}\kappa  \left(\sum_{(i_1,\ldots,i_M) \in \mathcal{L}: \; i_M = i,  \rho_{i_1 \cdots i_M}(S) = \kappa} \lambda_{i_1 \cdots i_M}\right) \\
  &= \sum_{(i_1,\ldots,i_M) \in \mathcal{L}} \left( \sum_{i,\kappa: \; (M,i,\kappa) \in \mathfrak{V}, i_M = i,  \rho_{i_1 \cdots i_M}(S) = \kappa} \kappa \lambda_{i_1 \cdots i_M} \right) \\
        &= \sum_{(i_1,\ldots,i_M) \in \mathcal{L}} \mathbb{I} \left \{ (M,i_M,\rho_{i_1 \cdots i_M}(S)) \in \mathfrak{V} \right \}  \rho_{i_1 \cdots i_M}(S) \lambda_{i_1 \cdots i_M}  \\
           &=  \sum_{(i_1,\ldots,i_M) \in \mathcal{L}}  \rho_{i_1 \cdots i_M}(S)  \lambda_{i_1 \cdots i_M}.
\end{align*}
The first equality follows from the construction of $g$. The second equality follows from Remark~\ref{rem:N_0} and Assumption~\ref{ass:nested}, which together imply that $S \cap S_M = S \cap \mathcal{N}_0 = S$.  The third and fourth equalities follows from algebra. The fifth equality follows from Claim~\ref{claim:not_in_N_S}. Our proof of Lemma~\ref{lem:simplified_to_compact} is thus complete. \halmos
\end{proof}

Our second intermediary lemma, denoted below by Lemma~\ref{lem:compact_to_bound}, will be used to show that the optimal objective value of the linear optimization problem~\eqref{prob:robust_simplified} is less than or equal to the optimal objective value of  the linear optimization problem~\eqref{prob:lp_reform_nested}.
\begin{lemma}  \label{lem:compact_to_bound}
Let Assumption~\ref{ass:nested} hold, let  $S \subseteq \mathcal{N}_0$ satisfy $S \cap S_1 \neq \emptyset$, and let $(f,g,\epsilon)$ satisfy the following constraints:
\begin{align*}
&&& \sum_{\kappa: (m,i,\kappa) \in {\mathfrak{V}}}  g_{m,i,\kappa}  - \epsilon_{m,i}= v_{m,i} && \forall m \in \mathcal{M}, \; i \in S_m \\
&&& \sum_{i',\kappa': ((m,i,\kappa),(m+1,i',\kappa')) \in {\mathfrak{E}}} f_{m,i,\kappa,i',\kappa'}  = g_{m,i,\kappa} && \forall (m,i,\kappa) \in {\mathfrak{V}}: m \in \{1,\ldots,M-1\} \\
&&& \sum_{i',\kappa': ((m-1,i',\kappa'),(m,i,\kappa)) \in {\mathfrak{E}}} f_{m-1,i',\kappa',i,\kappa}  = g_{m,i,\kappa} && \forall (m,i,\kappa) \in {\mathfrak{V}}: m \in \{2,\ldots,M\}\\
&&& \| \epsilon \| \le \eta \\
&&& \sum_{i,\kappa: (M,i,\kappa) \in \mathfrak{V}} g_{M,i,\kappa} = 1 \\
 &&& f_{m,i,\kappa,i',\kappa'} \ge 0 && \forall ((m,i,\kappa),(m+1,i',\kappa')) \in {\mathfrak{E}}  \\
&&& g_{m,i,\kappa} \ge 0 && \forall (m,i,\kappa) \in {\mathfrak{V}}.
\end{align*}
Then there exists a vector $\lambda$ such that $(\lambda,\epsilon)$ is a feasible solution for the linear optimization problem~\eqref{prob:lp_reform_nested} and 
\begin{align*}
\sum_{(i_1,\ldots,i_M) \in \mathcal{L}} \rho_{i_1 \cdots i_M}(S) \lambda_{i_1 \cdots i_M} \le \sum_{i,\kappa: (M,i,\kappa) \in \mathfrak{V}} \kappa g_{m,i,\kappa} + \sum_{(m,i,\kappa) \in \mathfrak{N}(S)} r_n g_{m,i,\kappa}. 
\end{align*}
\end{lemma}
\begin{proof}{Proof.}
Let Assumption~\ref{ass:nested} hold, let  $S \subseteq \mathcal{N}_0$ satisfy $S \cap S_1 \neq \emptyset$, and let $(f,g,\epsilon)$ satisfy the following constraints:
\begin{subequations} \label{constraints:simple}
\begin{align}
&&& \sum_{\kappa: (m,i,\kappa) \in {\mathfrak{V}}}  g_{m,i,\kappa}  - \epsilon_{m,i}= v_{m,i} && \forall m \in \mathcal{M}, \; i \in S_m \\
&&& \sum_{i',\kappa': ((m,i,\kappa),(m+1,i',\kappa')) \in {\mathfrak{E}}} f_{m,i,\kappa,i',\kappa'}  = g_{m,i,\kappa} && \forall (m,i,\kappa) \in {\mathfrak{V}}: m \in \{1,\ldots,M-1\} \\
&&& \sum_{i',\kappa': ((m-1,i',\kappa'),(m,i,\kappa)) \in {\mathfrak{E}}} f_{m-1,i',\kappa',i,\kappa}  = g_{m,i,\kappa} && \forall (m,i,\kappa) \in {\mathfrak{V}}: m \in \{2,\ldots,M\}\\
&&& \| \epsilon \| \le \eta \\
&&& \sum_{i,\kappa: (M,i,\kappa) \in \mathfrak{V}} g_{M,i,\kappa} = 1 \\
 &&& f_{m,i,\kappa,i',\kappa'} \ge 0 && \forall ((m,i,\kappa),(m+1,i',\kappa')) \in {\mathfrak{E}}  \\
&&& g_{m,i,\kappa} \ge 0 && \forall (m,i,\kappa) \in {\mathfrak{V}}.
\end{align}
\end{subequations}
 We observe that $f$ can be interpreted as a network flow through the directed acyclic graph $\mathfrak{G} \equiv  (\mathfrak{V},\mathfrak{E})$, in the sense that $f_{m,i,\kappa,i',\kappa'} \ge 0 $ is the units of flow from vertex $(m,i,\kappa) \in \mathfrak{V}$ to vertex $(m+1,i',\kappa') \in \mathfrak{V}$, and the constraints of the form $\sum_{i',\kappa': ((m,i,\kappa),(m+1,i',\kappa')) \in \mathfrak{E}} f_{m,i,\kappa,i',\kappa'}  = g_{m,i,\kappa}$ and  $\sum_{i',\kappa': ((m-1,i',\kappa'),(m,i,\kappa)) \in \mathfrak{E}} f_{m-1,i',\kappa',i,\kappa}  = g_{m,i,\kappa}$ ensure that flow is conserved through each vertex. 
Because $f$ can be interpreted as a network flow in the directed acyclic graph $\mathfrak{G} \equiv (\mathfrak{V},\mathfrak{E})$, it follows from the flow decomposition theorem \citep[Theorem 3.5]{ahuja1988network} that there exists a nonnegative vector $\pi \equiv (\pi_{i_1 \kappa_1 \cdots i_M \kappa_M}: ((1,i_1,\kappa_1),\ldots, (M,i_M,\kappa_M)) \in \mathscr{P} ) \ge 0$ that satisfies the following equalities:
\begin{align}
g_{m,i,\kappa} &= \sum_{ ((1,i_1,\kappa_1),\ldots, (M,i_M,\kappa_M)) \in \mathscr{P} : i_m = i, \kappa_m = \kappa} \pi_{i_1 \kappa_1 \cdots i_M \kappa_M} && \forall (m,i,\kappa) \in \mathfrak{V} \label{line:g_reverse_repeat}\\
f_{m,i,\kappa,i',\kappa'} &= \sum_{ ((1,i_1,\kappa_1),\ldots, (M,i_M,\kappa_M)) \in \mathscr{P} : i_m = i, \kappa_m = \kappa, i_{m+1} = i', \kappa_{m+1} = \kappa'} \pi_{i_1 \kappa_1 \cdots i_M \kappa_M} && ((m,i,\kappa),(m+1,i',\kappa')) \in \mathfrak{E}.\notag 
\end{align}
From this point onward, consider any fixed choice of the nonnegative vector $\pi$ that satisfies the above equalities. Given that choice of $\pi$, let the vector $\lambda \equiv (\pi_{i_1  \cdots i_M}: (i_1,\ldots,i_M) \in \mathcal{L})$ be defined as follows:
\begin{align}
\lambda_{i_1 \cdots i_M} &\triangleq \sum_{\kappa_1,\ldots,\kappa_M:  ((1,i_1,\kappa_1),\ldots, (M,i_M,\kappa_M)) \in \mathscr{P} } \pi_{i_1 \kappa_1 \cdots i_M \kappa_M} \quad \forall   (i_1,\ldots,i_M) \in \mathcal{L}.  \label{line:lambda_reverse}
\end{align}

Our first main step of the proof of  Lemma~\ref{lem:compact_to_bound} consists of showing that $(\lambda,\epsilon)$ is a  feasible  solution for the linear optimization problem~\eqref{prob:robust_simplified}. To show that $(\lambda,\epsilon)$ is a  feasible  solution for the linear optimization problem~\eqref{prob:robust_simplified}, we will utilize the following property of $\mathscr{P}$:
 \begin{claim} \label{claim:decomposition:inL}
If $((1,i_1,\kappa_1),\ldots, (M,i_M,\kappa_M)) \in \mathscr{P}$, then $(i_1,\ldots,i_M) \in \mathcal{L}$. 
\end{claim}
\begin{proof}{Proof of Claim~\ref{claim:decomposition:inL}. }
Consider any  path $((1,i_1,\kappa_1),\ldots, (M,i_M,\kappa_M)) \in \mathscr{P}$. We readily observe from the construction of the graph $\mathfrak{G} \equiv (\mathfrak{V},\mathfrak{E})$ and from the fact that $(1,i_1,\kappa_1) \in \mathfrak{V}$ that $i_1 \in S_1$. Next,  choose any arbitrary $m \in \{2,\ldots,M\}$. Because $(1,i_1,\kappa_1),\ldots,(M,i_M,\kappa_M)$ is a path,  we observe that there must exist a directed edge in the graph $\mathfrak{G} \equiv (\mathfrak{V},\mathfrak{E})$ from  vertex $(m-1,i_{m-1},\kappa_{m-1}) \in \mathfrak{V}$ to vertex $(m,i_{m},\kappa_{m}) \in \mathfrak{V}$.  Therefore, it follows from the definition of  $\mathfrak{E}$ that the inclusion $i_{m} \in \{i_{m-1} \} \cup \mathcal{B}_{m}$ must hold.  Because $m \in \{2,\ldots,M\}$ was chosen arbitrarily, we conclude that $i_1 \in S_1$ and $i_m \in \{i_{m-1} \} \cup \mathcal{B}_m$ for all $m \in \{2,\ldots,M\}$, and so Lemma~\ref{lem:reform_of_L} implies that $(i_1,\ldots,i_M) \in \mathcal{L}$. Our proof of Claim~\ref{claim:decomposition:inL} is thus complete. 
\halmos \end{proof}
Equipped with Claim~\ref{claim:decomposition:inL}, we now proceed to show that $(\lambda,\epsilon)$ is a  feasible  solution for the linear optimization problem~\eqref{prob:robust_simplified}. Indeed, for each  $m \in \mathcal{M}$ and $i \in S_m$, we observe that 
\begin{align*}
v_{m,i} &= \sum_{\kappa:(m,i,\kappa) \in \mathfrak{V}} g_{m,i,\kappa} - \epsilon_{m,i}\\
  &= \sum_{\kappa:(m,i,\kappa) \in \mathfrak{V}} \left(\sum_{ ((1,i_1,\kappa_1),\ldots, (M,i_M,\kappa_M)) \in \mathscr{P} : i_m = i, \kappa_m = \kappa} \pi_{i_1 \kappa_1 \cdots i_M \kappa_M}  \right)- \epsilon_{m,i} \\
    &= \sum_{ ((1,i_1,\kappa_1),\ldots, (M,i_M,\kappa_M)) \in \mathscr{P} : i_m = i } \pi_{i_1 \kappa_1 \cdots i_M \kappa_M} - \epsilon_{m,i} \\
    &= \sum_{(i_1,\ldots,i_M) \in \mathcal{L}: i_m = i} \left(  \sum_{\kappa_1,\ldots,\kappa_M: ((1,i_1,\kappa_1),\ldots, (M,i_M,\kappa_M)) \in \mathscr{P}} \pi_{i_1 \kappa_1 \cdots i_M \kappa_M}   \right)- \epsilon_{m,i} \\
&=  \sum_{(i_1,\ldots,i_M) \in \mathcal{L}: \; i_m = i} \lambda_{i_1 \cdots i_M} - \epsilon_{m,i},
\end{align*}
where the first equality follows from the fact that  $(f,g,\epsilon)$ satisfies the constraints~\eqref{constraints:simple},  the second equality follows from line~\eqref{line:g_reverse_repeat}, the third equality follows from algebra, the fourth equality follows from algebra and Claim~\ref{claim:decomposition:inL}, and the fifth equality follows from our construction of $\lambda$ on line~\eqref{line:lambda_reverse}. Moreover, we observe that
\begin{align*}
1 &= \sum_{i,\kappa: (M,i,\kappa) \in \mathfrak{V}} g_{m,i,\kappa} \\
&= \sum_{i,\kappa: (M,i,\kappa) \in \mathfrak{V}} \sum_{((1,i_1,\kappa_1),\ldots, (M,i_M,\kappa_M)) \in \mathscr{P}: i_M = i, \kappa_M = \kappa } \pi_{i_1 \kappa_1 \cdots i_M \kappa_M}  \\
&= \sum_{((1,i_1,\kappa_1),\ldots, (M,i_M,\kappa_M)) \in \mathscr{P} } \pi_{i_1 \kappa_1 \cdots i_M \kappa_M}  \\
 &= \sum_{(i_1,\ldots,i_M) \in \mathcal{L}} \left( \sum_{\kappa_1,\ldots,\kappa_M:  ((1,i_1,\kappa_1),\ldots, (M,i_M,\kappa_M)) \in \mathscr{P} } \pi_{i_1 \kappa_1 \cdots i_M \kappa_M}  \right) \\
 &= \sum_{(i_1,\ldots,i_M) \in \mathcal{L}}  \lambda_{i_1 \cdots i_M},
\end{align*}
where the first equality follows from the fact that  $(f,g,\epsilon)$ satisfies the constraints~\eqref{constraints:simple}, the second equality follows from line~\eqref{line:g_reverse_repeat}, the third equality follows from algebra, the fourth equality follows from algebra and Claim~\ref{claim:decomposition:inL}, and the fifth  equality follows from our construction of $\lambda$ on line~\eqref{line:lambda_reverse}. 
Since $\lambda$ is clearly nonnegative, our proof that $(\lambda,\epsilon)$ is a feasible solution for the linear optimization problem~\eqref{prob:robust_simplified} is complete.

The remainder of the proof of Lemma~\ref{lem:compact_to_bound} consists of bounding the objective value associated with $(\lambda,\epsilon)$ in the linear optimization problem~\eqref{prob:robust_simplified}. To develop our bound, we establish the following key property   about  directed acyclic graph $\mathfrak{G} \equiv (\mathfrak{V},\mathfrak{E})$.


\begin{claim} \label{claim:decomposition:kappa}
If $((1,i_1,\kappa_1),\ldots, (M,i_M,\kappa_M)) \in \mathscr{P}$ and $(m,i_m,\kappa_m) \notin \mathfrak{N}(S)$ for all $m \in \mathcal{M}$, then $\kappa_m \ge \rho_{i_1 \cdots i_M}(S \cap S_m)$ for all $m \in \mathcal{M}$. 
\end{claim}
\begin{proof}{Proof of Claim~\ref{claim:decomposition:kappa}. }
Consider any path $((1,i_1,\kappa_1),\ldots, (M,i_M,\kappa_M)) \in \mathscr{P}$ that satisfies $(m,i_m,\kappa_m) \notin \mathfrak{N}(S)$ for all $m \in \mathcal{M}$. 
Our proof of Claim~\ref{claim:decomposition:kappa} follows from an induction argument, and the induction argument proceeds as follows. 

 In the base case,  we let $m = 1$. In this case, it follows from the fact that $((1,i_1,\kappa_1),\ldots,(M,i_M,\kappa_M)) \in \mathscr{P}$ that the inclusion $(1,i_1,\kappa_1) \in \mathfrak{V}$ holds, and so it follows from the definition of $\mathfrak{V}$ that $\kappa_1 \in \{r_j: j \in \mathcal{B}_1 \}$.  Thus, it follows  from the facts that $\kappa_1 \in \{r_j: j \in \mathcal{B}_1 \}$ and $(1,i_1,\kappa_1) \notin \mathfrak{N}(S)$ that 
 \begin{align}
\kappa_1 \in \begin{cases}
\{ r_{i_1} \}, & \textnormal{if } i_1 \in S,\\
 \{ r_j: j \in  \mathcal{B}_1 \cap S \},&\textnormal{if } i_1 \notin S.
\end{cases}   \label{line:decomposition:kappa:1}
 \end{align}
Therefore, we observe that 
\begin{align*}
\kappa_1 &\ge \begin{cases}
 r_{i_1}, & \textnormal{if } i_1 \in S,\\
\min_{ j \in  \mathcal{B}_1 \cap S} r_j ,&\textnormal{if } i_1 \notin S 
\end{cases}\\
&= \rho_{i_1 \cdots i_M}(S \cap S_1),
\end{align*}
where the inequality follows from line~\eqref{line:decomposition:kappa:1} and the equality follows  from Claim~\ref{claim:decomposition:inL} and  Lemma~\ref{lem:rho_m}. We have shown that $\kappa_1 \ge \rho_{i_1 \cdots i_M}(S \cap S_1)$, which concludes our proof of the base case.

Now choose any arbitrary $m \in \{2,\ldots,M\}$ and assume by induction that the inequality $\kappa_{m'} \ge \rho_{i_1 \cdots i_M}(S \cap S_{m'})$ holds for all $m' \in \{1,\ldots,m-1\}$. In this case, it follows from the fact that  $((1,i_1,\kappa_1),\ldots,(M,i_M,\kappa_M)) \in \mathscr{P}$ that the inclusion $((m-1,i_{m-1},\kappa_{m-1}),(m,i_{m},\kappa_{m})) \in \mathfrak{E}$ holds, and so it follows from the definition of $\mathfrak{E}$ that $\kappa_m \in  \left \{ \kappa_{m-1} \right \} \cup \left \{ r_j: j \in \mathcal{B}_{m} \right \}$. Thus, it follows from the facts that $\kappa_m \in  \left \{ \kappa_{m-1} \right \} \cup \left \{ r_j: j \in \mathcal{B}_{m} \right \}$ and   $(m,i_m,\kappa_m) \notin \mathfrak{N}(S)$  that 
 \begin{align}
\kappa_m \in \begin{cases}
\{ r_{i_m} \}, & \textnormal{if } i_m \in S \text{ and } i_m \in \mathcal{B}_m,\\
\{ \kappa_{m-1} \},&\textnormal{if } i_m \in S \text{ and } i_m \notin \mathcal{B}_m,\\
\{ \kappa_{m-1} \} \cup  \{ r_j: j \in  \mathcal{B}_m \cap S \},&\textnormal{if } i_m \notin S.
\end{cases}   \label{line:decomposition:kappa:induction:1}
 \end{align}
 Therefore,
 \begin{align*}
 \kappa_m &\ge \begin{cases}
 r_{i_m}, & \textnormal{if } i_m \in S \text{ and } i_m \in \mathcal{B}_m,\\
\kappa_{m-1},&\textnormal{if } i_m \in S \text{ and } i_m \notin \mathcal{B}_m,\\
\min \left \{  \kappa_{m-1}, \min \limits_{ j \in  \mathcal{B}_m \cap S} r_j \right \},&\textnormal{if } i_m \notin S
\end{cases}\\
&\ge \begin{cases}
 r_{i_m}, & \textnormal{if } i_m \in S \text{ and } i_m \in \mathcal{B}_m,\\
 \rho_{i_1 \cdots i_M}(S \cap S_{m-1}) ,&\text{if } i_m \in S \text{ and } i_m \notin \mathcal{B}_m,\\
\min \left \{ \rho_{i_1 \cdots i_M}(S \cap S_{m-1}) , \min \limits_{ j \in  \mathcal{B}_m \cap S} r_j \right \},&\textnormal{if } i_m \notin S
\end{cases}\\
&=  \rho_{i_1 \cdots i_M}(S \cap S_{m}),
 \end{align*}
where the first inequality follows from line~\eqref{line:decomposition:kappa:induction:1}, the second inequality follows from the induction hypothesis, and the equality follows  from Claim~\ref{claim:decomposition:inL} and  Lemma~\ref{lem:rho_m}. We have thus shown that $\kappa_m \ge \rho_{i_1 \cdots i_M}(S \cap S_m)$.  This concludes our proof of the induction step, and thus our proof of Claim~\ref{claim:decomposition:kappa} is complete.
\halmos \end{proof}

Equipped with Claims~\ref{claim:decomposition:inL} and \ref{claim:decomposition:kappa},  we now develop the desired bound on   the objective value associated with $(\lambda,\epsilon)$ in the linear optimization problem~\eqref{prob:robust_simplified}:
 \begin{align}
 & \sum_{(i_1,\ldots,i_M) \in \mathcal{L}} \rho_{i_1 \cdots i_M}(S) \lambda_{i_1\cdots i_M} \notag \\
  &=   \sum_{(i_1,\ldots,i_M) \in \mathcal{L}} \rho_{i_1 \cdots i_M}(S) \left( \sum_{\kappa_1,\ldots,\kappa_M:  ((1,i_1,\kappa_1),\ldots, (M,i_M,\kappa_M)) \in \mathscr{P} } \pi_{i_1 \kappa_1 \cdots i_M \kappa_M} \right) \label{line:establishing_inequality:1} \\
  &= \sum_{((1,i_1,\kappa_1),\ldots, (M,i_M,\kappa_M)) \in \mathscr{P} } \rho_{i_1 \cdots i_M}(S)  \pi_{i_1 \kappa_1 \cdots i_M \kappa_M}\label{line:establishing_inequality:2} \\
    &= \sum_{((1,i_1,\kappa_1),\ldots, (M,i_M,\kappa_M)) \in \mathscr{P} } \rho_{i_1 \cdots i_M}(S \cap S_M)  \pi_{i_1 \kappa_1 \cdots i_M \kappa_M} \label{line:establishing_inequality:3}\\
                &= \left( \sum_{((1,i_1,\kappa_1),\ldots, (M,i_M,\kappa_M)) \in \mathscr{P}: (m,i_m,\kappa_m) \notin \mathfrak{N}(S) \; \forall m \in \mathcal{M} } \rho_{i_1 \cdots i_M}(S \cap S_M)  \pi_{i_1 \kappa_1 \cdots i_M \kappa_M}  \right)\notag \\
            & \quad  + \left( \sum_{((1,i_1,\kappa_1),\ldots, (M,i_M,\kappa_M)) \in \mathscr{P}:  \exists m \in \mathcal{M} \textnormal{ such that } (m,i_m,\kappa_m) \in \mathfrak{N}(S) } \rho_{i_1 \cdots i_M}(S \cap S_M)  \pi_{i_1 \kappa_1 \cdots i_M \kappa_M}  \right)  \label{line:establishing_inequality:4}  \\
            &\le \left( \sum_{((1,i_1,\kappa_1),\ldots, (M,i_M,\kappa_M)) \in \mathscr{P}: (m,i_m,\kappa_m) \notin \mathfrak{N}(S) \; \forall m \in \mathcal{M} } \kappa_M \pi_{i_1 \kappa_1 \cdots i_M \kappa_M}  \right) \notag \\
            & \quad  + \left( \sum_{((1,i_1,\kappa_1),\ldots, (M,i_M,\kappa_M)) \in \mathscr{P}: \exists m \in \mathcal{M} \textnormal{ such that } (m,i_m,\kappa_m) \in \mathfrak{N}(S)} \rho_{i_1 \cdots i_M}(S \cap S_M)  \pi_{i_1 \kappa_1 \cdots i_M \kappa_M}  \right)  \label{line:establishing_inequality:5}  \\
                        &\le \left( \sum_{((1,i_1,\kappa_1),\ldots, (M,i_M,\kappa_M)) \in \mathscr{P}: (m,i_m,\kappa_m) \notin \mathfrak{N}(S) \; \forall m \in \mathcal{M} } \kappa_M \pi_{i_1 \kappa_1 \cdots i_M \kappa_M}  \right) \notag \\
            & \quad  + \left( \sum_{((1,i_1,\kappa_1),\ldots, (M,i_M,\kappa_M)) \in \mathscr{P}: \exists m \in \mathcal{M} \textnormal{ such that } (m,i_m,\kappa_m) \in \mathfrak{N}(S)} r_n \pi_{i_1 \kappa_1 \cdots i_M \kappa_M}  \right)\label{line:establishing_inequality:6} \\
                                    &\le \left( \sum_{((1,i_1,\kappa_1),\ldots, (M,i_M,\kappa_M)) \in \mathscr{P}: (m,i_m,\kappa_m) \notin \mathfrak{N}(S) \; \forall m \in \mathcal{M} } \kappa_M \pi_{i_1 \kappa_1 \cdots i_M \kappa_M}  \right) \notag \\
            & \quad  + \left( \sum_{((1,i_1,\kappa_1),\ldots, (M,i_M,\kappa_M)) \in \mathscr{P}} \sum_{m \in \mathcal{M}: (m,i_m,\kappa_m) \in \mathfrak{N}(S)} r_n \pi_{i_1 \kappa_1 \cdots i_M \kappa_M}  \right)  \label{line:establishing_inequality:7}  \\
              &= \left( \sum_{((1,i_1,\kappa_1),\ldots, (M,i_M,\kappa_M)) \in \mathscr{P}: (m,i_m,\kappa_m) \notin \mathfrak{N}(S) \; \forall m \in \mathcal{M} } \kappa_M \pi_{i_1 \kappa_1 \cdots i_M \kappa_M}  \right)  \notag \\
            & \quad  + \left( \sum_{ (m,i,\kappa) \in \mathfrak{N}(S)} r_n \sum_{((1,i_1,\kappa_1),\ldots, (M,i_M,\kappa_M)) \in \mathscr{P}: i_m = i, \kappa_m = \kappa}  \pi_{i_1 \kappa_1 \cdots i_M \kappa_M}  \right) \label{line:establishing_inequality:8}  \\
     &= \left( \sum_{((1,i_1,\kappa_1),\ldots, (M,i_M,\kappa_M)) \in \mathscr{P}: (m,i_m,\kappa_m) \notin \mathfrak{N}(S) \; \forall m \in \mathcal{M} } \kappa_M \pi_{i_1 \kappa_1 \cdots i_M \kappa_M}  \right) +  \sum_{ (m,i,\kappa) \in \mathfrak{N}(S)} g_{m,i,\kappa} \label{line:establishing_inequality:9}  \\
    &\le \left( \sum_{((1,i_1,\kappa_1),\ldots, (M,i_M,\kappa_M)) \in \mathscr{P} } \kappa_M \pi_{i_1 \kappa_1 \cdots i_M \kappa_M}  \right) +  \sum_{ (m,i,\kappa) \in \mathfrak{N}(S)} g_{m,i,\kappa} \label{line:establishing_inequality:10}  \\
        &= \left( \sum_{i,\kappa: (M,i,\kappa) \in \mathfrak{V}}\kappa  \sum_{((1,i_1,\kappa_1),\ldots, (M,i_M,\kappa_M)) \in \mathscr{P}: i_M = i, \kappa_M = \kappa } \pi_{i_1 \kappa_1 \cdots i_M \kappa_M}  \right) +  \sum_{ (m,i,\kappa) \in \mathfrak{N}(S)} g_{m,i,\kappa} \label{line:establishing_inequality:11}  \\
                &=  \sum_{i,\kappa: (M,i,\kappa) \in \mathfrak{V}}\kappa  g_{M,i,\kappa}  +  \sum_{ (m,i,\kappa) \in \mathfrak{N}(S)} g_{m,i,\kappa}. \label{line:establishing_inequality:12} 
  \end{align}
  Indeed, line~\eqref{line:establishing_inequality:1} follows from our construction of $\lambda$ on line~\eqref{line:lambda_reverse}. Line~\eqref{line:establishing_inequality:2} follows from algebra and Claim~\ref{claim:decomposition:inL}. Line~\eqref{line:establishing_inequality:3} follows from Assumption~\ref{ass:nested} and Remark~\ref{rem:N_0}. Line~\eqref{line:establishing_inequality:4} follows from algebra. Line~\eqref{line:establishing_inequality:5} follows from Claim~\ref{claim:decomposition:kappa}.  Line~\eqref{line:establishing_inequality:6} follows from the fact that $\rho_{i_1 \cdots i_M}(S \cap S_M) \le r_n$ for all $(i_1,\ldots,i_M) \in \mathcal{L}$. Lines~\eqref{line:establishing_inequality:7} and \eqref{line:establishing_inequality:8} follows from algebra. Line~\eqref{line:establishing_inequality:9} follows from line~\eqref{line:g_reverse_repeat}.  Lines~\eqref{line:establishing_inequality:10} and \eqref{line:establishing_inequality:11} follow from algebra. Line~\eqref{line:establishing_inequality:12} follows from line~\eqref{line:g_reverse_repeat}.  Our proof of Lemma~\ref{lem:compact_to_bound} is thus complete. 
  \halmos \end{proof}

\subsection{Proof of Proposition~\ref{prop:lp_reform_nested}} \label{appx:lp_reform_nested}
Let Assumption~\ref{ass:nested} hold, and consider any assortment $S \subseteq \mathcal{N}_0$ that satisfies $S \cap S_1 \neq \emptyset$. It follows immediately from Lemma~\ref{lem:simplified_to_compact} from Appendix~\ref{appx:mip:intermediary} that  the optimal objective value of the linear optimization problem~\eqref{prob:robust_simplified} is greater than or equal to the optimal objective value of  the linear optimization problem~\eqref{prob:lp_reform_nested}. To show the other direction, consider any arbitrary feasible solution $(f,g,\epsilon)$ for the linear optimization problem~\eqref{prob:lp_reform_nested}. It follows from Lemma~\ref{lem:compact_to_bound} from Appendix~\ref{appx:mip:intermediary} that there exists a vector $\lambda$ such that $(\lambda,\epsilon)$ is a feasible solution for the linear optimization problem~\eqref{prob:robust_simplified} and 
\begin{align*}
\sum_{(i_1,\ldots,i_M) \in \mathcal{L}} \rho_{i_1 \cdots i_M}(S) \lambda_{i_1 \cdots i_M} \le \sum_{i,\kappa: (M,i,\kappa) \in \mathfrak{V}} \kappa g_{m,i,\kappa} + \sum_{(m,i,\kappa) \in \mathfrak{N}(S)} r_n g_{m,i,\kappa},
\end{align*}
where the definition of the set $\mathfrak{N}(S)$ can be found at the beginning of Appendix~\ref{appx:mip:intermediary}. 
In particular, it follows from the fact that $(f,g,\epsilon)$ is feasible for the linear optimization problem~\eqref{prob:lp_reform_nested} that the equality $g_{m,i,\kappa} = 0$ holds for all $(m,i,\kappa) \in \mathfrak{N}(S)$. Therefore, we conclude that
\begin{align*}
\sum_{(i_1,\ldots,i_M) \in \mathcal{L}} \rho_{i_1 \cdots i_M}(S) \lambda_{i_1 \cdots i_M}& \le \sum_{i,\kappa: (M,i,\kappa) \in \mathfrak{V}} \kappa g_{m,i,\kappa} + \sum_{(m,i,\kappa) \in \mathfrak{N}(S)} r_n 0 \\
&= \sum_{i,\kappa: (M,i,\kappa) \in \mathfrak{V}} \kappa g_{m,i,\kappa}.
\end{align*} 
Since $(f,g,\epsilon)$ was chosen arbitrarily, we conclude that the optimal objective value of the linear optimization problem~\eqref{prob:lp_reform_nested} is greater than or equal to the optimal objective value of  the linear optimization problem~\eqref{prob:robust_simplified}. Our proof of Proposition~\ref{prop:lp_reform_nested} is thus complete. 
 \halmos

\subsection{Proof of Theorem~\ref{thm:nested}} \label{appx:mip:lp_reform_nested}
Let Assumption~\ref{ass:nested} hold and $S_1,\ldots,S_M \in \mathcal{S}$. It follows from the fact that $S_1 \in \mathcal{S}$ that $S \cap S_1 \neq \emptyset$ for all $S \in \mathcal{S}$. Therefore,  it follows from Proposition~\ref{prop:lp_reform_nested} for each assortment $S \in \mathcal{S}$ that the worst-case expected revenue $\min_{\lambda \in \mathcal{U}} \mathscr{R}^\lambda(S)$ is equal to the optimal objective value of the following linear optimization problem: 
 \begin{equation} \tag{\ref{prob:lp_reform_nested}}
\begin{aligned}
& \; \underset{f, g, \epsilon}{\textnormal{minimize}} && \sum_{i,\kappa: (M,i,\kappa) \in {\mathfrak{V}}} \kappa  g_{M,i,\kappa}  \\
&\textnormal{subject to}&& \sum_{\kappa: (m,i,\kappa) \in {\mathfrak{V}}}  g_{m,i,\kappa} - \epsilon_{m,i}  = v_{m,i} && \forall m \in \mathcal{M}, \; i \in S_m \\
&&& \sum_{i',\kappa': ((m,i,\kappa),(m+1,i',\kappa')) \in {\mathfrak{E}}} f_{m,i,\kappa,i',\kappa'}  = g_{m,i,\kappa} && \forall (m,i,\kappa) \in {\mathfrak{V}}: m \in \{1,\ldots,M-1\} \\
&&& \sum_{i',\kappa': ((m-1,i',\kappa'),(m,i,\kappa)) \in {\mathfrak{E}}} f_{m-1,i',\kappa',i,\kappa}  = g_{m,i,\kappa} && \forall (m,i,\kappa) \in {\mathfrak{V}}: m \in \{2,\ldots,M\}\\
 &&& g_{m,i,\kappa} = 0 && \forall  (m,i,\kappa) \in \mathfrak{N}(S)\\
  &&& \sum_{i,\kappa: (M,i,\kappa) \in \mathfrak{V}} g_{M,i,\kappa} = 1 \\
 &&& \| \epsilon \| \le \eta\\
 &&& f_{m,i,\kappa,i',\kappa'} \ge 0 && \forall ((m,i,\kappa),(m+1,i',\kappa')) \in {\mathfrak{E}}  \\
&&& g_{m,i,\kappa} \ge 0 && \forall (m,i,\kappa) \in {\mathfrak{V}},
\end{aligned} 
\end{equation}
where the definition of the set $\mathfrak{N}(S)$ can be found in the beginning of Appendix~\ref{appx:mip:intermediary}. 

At a high level, our following analysis will consist of two main steps. First, we will use strong duality to transform the linear optimization problem~\eqref{prob:lp_reform_nested} into a maximization problem. Second, we will introduce binary decision variables to represent the assortment $S \in \mathcal{S}$. The key challenge in performing these two main steps is ensuring that the  resulting mixed-integer optimization problem does not have any nonlinear terms. To see why this is a challenge, suppose for the sake of discussion that we introduced a binary decision variable $x_{i} \in \{0,1\}$ for each product $i \in \mathcal{N}_0$ that satisfies $x_i = 1$ if and only if product $i$ is in the assortment $S$. Under this choice of binary decision variables, it is easy to verify using the definition of $\mathfrak{N}(S)$ from the beginning of Appendix~\ref{appx:mip:intermediary} that the constraints in the linear optimization problem~\eqref{prob:lp_reform_nested} of the form 
\begin{align*}
 &&& g_{m,i,\kappa} = 0 && \forall  (m,i,\kappa) \in \mathfrak{N}(S)
\end{align*}
will be satisfied if and only if 
\begin{align*}
 &&& g_{m,i,\kappa} \le 1 -  x_i  && \forall  (m,i,\kappa) \in \mathfrak{V}: i \in \mathcal{B}_m \textnormal{ and } \kappa \neq r_i\\
  &&& g_{m,i,\kappa} \le 1 -  x_i  && \forall  (m,i,\kappa) \in \mathfrak{V}: i \notin \mathcal{B}_m \textnormal{ and } \kappa \in \left \{ r_j: j \in \mathcal{B}_m \right \}\\
    &&& g_{m,i,\kappa} \le  x_j  && \forall  (m,i,\kappa) \in \mathfrak{V} \textnormal{ and } j \in \mathcal{B}_m: \kappa = r_j. 
\end{align*}
However, we observe that the binary decision variables $x \in \{0,1\}^{\mathcal{N}_0}$ appear on the right hand side of the above constraints. Therefore, if one added the above constraints into the linear optimization problem~\eqref{prob:lp_reform_nested}, then the dual of the linear optimization problem~\eqref{prob:lp_reform_nested} would have an objective function in which dual decision variables are multiplied by the binary decision variables $x$. 
 
To circumvent the aforementioned challenge and obtain a mixed-integer \emph{linear} optimization reformulation of \eqref{prob:robust}, we will make use of the following Claim~\ref{claim:lp_reform_nested:objective}.  Specifically, the following Claim~\ref{claim:lp_reform_nested:objective} shows that the constraints of the form $g_{m,i,\kappa} = 0\; \forall    (m,i,\kappa) \in \mathfrak{N}(S)$ can  be moved into the objective function of the linear optimization problem~\eqref{prob:lp_reform_nested} without loss of generality. By using the following claim,  we will subsequently be able to obtain a dual of the linear optimization problem~\eqref{prob:lp_reform_nested} in which  the binary decision variables $x \in \{0,1\}^{\mathcal{N}_0}$ are not multiplied by dual decision variables.  The proof of the following  Claim~\ref{claim:lp_reform_nested:objective} follows from similar reasoning as the proof of Proposition~\ref{prop:lp_reform_nested}. 
\begin{claim} \label{claim:lp_reform_nested:objective}
Consider any assortment $S \in \mathcal{S}$, and let $\mathfrak{M}(S) \in \R^{{\mathfrak{V}}}$ be any vector that satisfies the inequality $\mathfrak{M}_{m,i,\kappa}(S) \ge r_n$ for all $(m,i,\kappa) \in \mathfrak{N}(S)$ and satisfies the equality $\mathfrak{M}_{m,i,\kappa}(S) = 0$ for all $(m,i,\kappa) \in {\mathfrak{V}} \setminus \mathfrak{N}(S)$. Then the optimal objective value of the linear optimization problem~\eqref{prob:lp_reform_nested} is equal to the optimal objective value of the following linear optimization problem: 
 \begin{equation} \label{prob:nested:compact:objective}
\begin{aligned}
& \; \underset{f, g,\epsilon}{\textnormal{minimize}} && \sum_{i,\kappa: (M,i,\kappa) \in {\mathfrak{V}}} \kappa  g_{M,i,\kappa} + \sum_{(m,i,\kappa) \in {\mathfrak{V}}} \mathfrak{M}_{m,i,\kappa}(S) g_{m,i,\kappa}  \\
&\textnormal{subject to}&& \sum_{\kappa: (m,i,\kappa) \in {\mathfrak{V}}}  g_{m,i,\kappa}  - \epsilon_{m,i}= v_{m,i} && \forall m \in \mathcal{M}, \; i \in S_m \\
&&& \sum_{i',\kappa': ((m,i,\kappa),(m+1,i',\kappa')) \in {\mathfrak{E}}} f_{m,i,\kappa,i',\kappa'}  = g_{m,i,\kappa} && \forall (m,i,\kappa) \in {\mathfrak{V}}: m \in \{1,\ldots,M-1\} \\
&&& \sum_{i',\kappa': ((m-1,i',\kappa'),(m,i,\kappa)) \in {\mathfrak{E}}} f_{m-1,i',\kappa',i,\kappa}  = g_{m,i,\kappa} && \forall (m,i,\kappa) \in {\mathfrak{V}}: m \in \{2,\ldots,M\}\\
&&& \sum_{i,\kappa: (M,i,\kappa) \in \mathfrak{V}} g_{M,i,\kappa} = 1 \\
&&& \| \epsilon \| \le \eta \\
 &&& f_{m,i,\kappa,i',\kappa'} \ge 0 && \forall ((m,i,\kappa),(m+1,i',\kappa')) \in {\mathfrak{E}}  \\
&&& g_{m,i,\kappa} \ge 0 && \forall (m,i,\kappa) \in {\mathfrak{V}}.
\end{aligned} 
\end{equation} 
\end{claim}
\begin{proof}{Proof of Claim~\ref{claim:lp_reform_nested:objective}.} Consider any assortment $S \in \mathcal{S}$, and let $\mathfrak{M}(S) \in \R^{{\mathfrak{V}}}$ be any vector that satisfies the inequality $\mathfrak{M}_{m,i,\kappa}(S) \ge r_n$ for all $(m,i,\kappa) \in \mathfrak{N}(S)$ and satisfies the equality $\mathfrak{M}_{m,i,\kappa}(S) = 0$ for all $(m,i,\kappa) \in {\mathfrak{V}} \setminus \mathfrak{N}(S)$.
 We readily observe that the optimal objective value of the linear optimization problem~\eqref{prob:nested:compact:objective} is a lower bound on the optimal objective value of the linear optimization  problem~\eqref{prob:lp_reform_nested}, since the constraints of the form $g_{m,i,\kappa} = 0\; \forall (m,i,\kappa) \in \mathfrak{N}(S)$ from the linear optimization problem~\eqref{prob:lp_reform_nested} have been relaxed and since  the equality $\mathfrak{M}_{m,i,\kappa}(S) = 0$ holds for all $(m,i,\kappa) \in {\mathfrak{V}} \setminus \mathfrak{N}(S)$. To show the other direction, consider any arbitrary  feasible solution $(f,g,\epsilon)$ for the linear optimization problem~\eqref{prob:nested:compact:objective}. Because Assumption~\ref{ass:nested} holds and $S_1 \cap S \neq \emptyset$, it follows from Lemma~\ref{lem:compact_to_bound} from Appendix~\ref{appx:mip:intermediary} that there exists a vector $\lambda$ such that $(\lambda,\epsilon)$ is a feasible solution for the linear optimization problem~\eqref{prob:robust_simplified} and 
\begin{align*}
\sum_{(i_1,\ldots,i_M) \in \mathcal{L}} \rho_{i_1 \cdots i_M}(S) \lambda_{i_1 \cdots i_M} &\le \sum_{i,\kappa: (M,i,\kappa) \in \mathfrak{V}} \kappa g_{m,i,\kappa} + \sum_{(m,i,\kappa) \in \mathfrak{N}(S)} r_n g_{m,i,\kappa} \\
&\le  \sum_{i,\kappa: (M,i,\kappa) \in \mathfrak{V}} \kappa g_{m,i,\kappa} + \sum_{(m,i,\kappa) \in \mathfrak{V}} \mathfrak{M}(S) g_{m,i,\kappa} ,
\end{align*}
where the first inequality follows from Lemma~\ref{lem:compact_to_bound}  and the second inequality follows from our construction of the vector $\mathfrak{M}(S)$. Since  $(f,g,\epsilon)$  was chosen arbitrarily, we have shown that the optimal objective value of the linear optimization problem~\eqref{prob:robust_simplified} is a lower bound on the optimal objective value of the linear optimization  problem~\eqref{prob:nested:compact:objective}. Since we showed in Proposition~\ref{prop:lp_reform_nested} that the optimal objective value of  the linear optimization  problem~\eqref{prob:nested:compact:objective} is equal to the optimal objective value of the linear optimization problem~\eqref{prob:lp_reform_nested}, we conclude that the optimal objective value of the linear optimization problem~\eqref{prob:lp_reform_nested} is a lower bound on the optimal objective value of the linear optimization  problem~\eqref{prob:nested:compact:objective}. Our proof of Claim~\ref{claim:lp_reform_nested:objective} is thus complete. 
\halmos \end{proof}

To simplify our exposition, the remainder of the proof of Theorem~\ref{thm:nested} focuses on the case where $\eta = 0$. Indeed, it is easy to see that our reformulations of the robust optimization problem~\eqref{prob:robust} for the case of $\eta = 0$ can be readily extended to the case of $\eta > 0$ by introducing $\mathcal{O}(\textnormal{poly}(n,M))$ auxiliary decision variables and auxiliary linear constraints in the dual problem. In the case of $\eta = 0$, we observe that the linear optimization problem~\eqref{prob:nested:compact:objective} can be rewritten equivalently as 
 \begin{equation} \tag{\ref{prob:nested:compact:objective}}
\begin{aligned}
& \; \underset{f, g}{\textnormal{minimize}} && \sum_{i,\kappa: (M,i,\kappa) \in {\mathfrak{V}}} \kappa  g_{M,i,\kappa} + \sum_{(m,i,\kappa) \in {\mathfrak{V}}} \mathfrak{M}_{m,i,\kappa}(S) g_{m,i,\kappa}  \\
&\textnormal{subject to}&& \sum_{\kappa: (m,i,\kappa) \in {\mathfrak{V}}}  g_{m,i,\kappa}  = v_{m,i} && \forall m \in \mathcal{M}, \; i \in S_m \\
&&& \sum_{i',\kappa': ((m,i,\kappa),(m+1,i',\kappa')) \in {\mathfrak{E}}} f_{m,i,\kappa,i',\kappa'}  = g_{m,i,\kappa} && \forall (m,i,\kappa) \in {\mathfrak{V}}: m \in \{1,\ldots,M-1\} \\
&&& \sum_{i',\kappa': ((m-1,i',\kappa'),(m,i,\kappa)) \in {\mathfrak{E}}} f_{m-1,i',\kappa',i,\kappa}  = g_{m,i,\kappa} && \forall (m,i,\kappa) \in {\mathfrak{V}}: m \in \{2,\ldots,M\}\\
 &&& f_{m,i,\kappa,i',\kappa'} \ge 0 && \forall ((m,i,\kappa),(m+1,i',\kappa')) \in {\mathfrak{E}}  \\
&&& g_{m,i,\kappa} \ge 0 && \forall (m,i,\kappa) \in {\mathfrak{V}}.
\end{aligned} 
\end{equation} 
We observe from strong duality that the optimal objective value of the linear optimization problem~\eqref{prob:nested:compact:objective} is thus equal to the  optimal objective value of the following linear optimization problem~\eqref{prob:nested:compact:dual}  for every assortment $S \in \mathcal{S}$ and every choice of  $\mathfrak{M}(S) \in \R^{{\mathfrak{V}}}$ that satisfies the inequality $\mathfrak{M}_{m,i,\kappa}(S) \ge r_n$ for all $(m,i,\kappa) \in \mathfrak{N}(S)$ and satisfies the equality $\mathfrak{M}_{m,i,\kappa}(S) = 0$ for all $(m,i,\kappa) \in {\mathfrak{V}} \setminus \mathfrak{N}(S)$: 
 \begin{equation}  \label{prob:nested:compact:dual} 
\begin{aligned}
& \; \underset{\alpha,\beta,\gamma}{\textnormal{maximize}} && \sum_{m \in \mathcal{M}, i \in S_m} v_{m,i} \alpha_{m,i}  \\
&\textnormal{subject to}&&\alpha_{m,i} - \beta_{m,i,\kappa} \mathbb{I} \left \{ m \in \{1,\ldots,M-1\} \right \} - \gamma_{m,i,\kappa} \mathbb{I} \left \{ m \in \{2,\ldots,M\} \right \}\\
&&& \quad \le \kappa \mathbb{I} \left \{ m = M \right \} + \mathfrak{M}_{m,i,\kappa}(S) && \forall (m,i,\kappa) \in {\mathfrak{V}}\\
&&& \beta_{m,i,\kappa} + \gamma_{m+1,i',\kappa'}  \le 0&& \forall ((m,i,\kappa),(m+1,i',\kappa')) \in {\mathfrak{E}}\\
&&& \alpha_{m,i} \in \R && \forall m \in \mathcal{M}, \; i \in S_m\\
&&& \beta_{m,i,\kappa}, \gamma_{m,i,\kappa} \in \R && \forall (m,i,\kappa) \in {\mathfrak{V}}. 
\end{aligned} 
\end{equation}
Consequently,  if   $\mathfrak{M}(\cdot)$ satisfies the inequality $\mathfrak{M}_{m,i,\kappa}(S) \ge r_n$ for all $(m,i,\kappa) \in \mathfrak{N}(S)$ and satisfies the equality $\mathfrak{M}_{m,i,\kappa}(S) = 0$ for all $(m,i,\kappa) \in {\mathfrak{V}} \setminus \mathfrak{N}(S)$ and all assortments $S \in \mathcal{S}$, then 
{\footnotesize
 \begin{align}  
&\eqref{prob:robust} \notag \\
 &=\max_{S \in \mathcal{S}}\; \min_{\lambda \in \mathcal{U}}\mathscr{R}^{\lambda}(S) \notag \\
&=  \underset{S \in \mathcal{S}}{\textnormal{max}} \left[\begin{aligned}
& \; \underset{\alpha,\beta,\gamma}{\textnormal{maximize}} && \sum_{m \in \mathcal{M}, i \in S_m} v_{m,i} \alpha_{m,i}  \\
&\textnormal{subject to}&&\alpha_{m,i} - \beta_{m,i,\kappa} \mathbb{I} \left \{ m \in \{1,\ldots,M-1\} \right \} - \gamma_{m,i,\kappa} \mathbb{I} \left \{ m \in \{2,\ldots,M\} \right \}\\
&&& \quad \le \kappa \mathbb{I} \left \{ m = M \right \} + \mathfrak{M}_{m,i,\kappa}(S)&& \forall (m,i,\kappa) \in {\mathfrak{V}}\\
&&& \beta_{m,i,\kappa} + \gamma_{m+1,i',\kappa'}  \le 0&& \forall ((m,i,\kappa),(m+1,i',\kappa')) \in {\mathfrak{E}}\\
&&& \alpha_{m,i} \in \R && \forall m \in \mathcal{M}, \; i \in S_m\\
&&& \beta_{m,i,\kappa}, \gamma_{m,i,\kappa} \in \R && \forall (m,i,\kappa) \in {\mathfrak{V}}
\end{aligned}  \right]  \notag  \\
&= \left[ \begin{aligned}
& \; \underset{S, \alpha,\beta,\gamma}{\textnormal{maximize}} && \sum_{m \in \mathcal{M}, i \in S_m} v_{m,i} \alpha_{m,i}  \\
&\textnormal{subject to}&&\alpha_{m,i} - \beta_{m,i,\kappa} \mathbb{I} \left \{ m \in \{1,\ldots,M-1\} \right \} - \gamma_{m,i,\kappa} \mathbb{I} \left \{ m \in \{2,\ldots,M\} \right \}\\
&&& \quad \le \kappa \mathbb{I} \left \{ m = M \right \} + \mathfrak{M}_{m,i,\kappa}(S) && \forall (m,i,\kappa) \in {\mathfrak{V}}\\
&&& \beta_{m,i,\kappa} + \gamma_{m+1,i',\kappa'}  \le 0&& \forall ((m,i,\kappa),(m+1,i',\kappa')) \in {\mathfrak{E}}\\
&&& \alpha_{m,i} \in \R && \forall m \in \mathcal{M}, \; i \in S_m\\
&&& \beta_{m,i,\kappa}, \gamma_{m,i,\kappa} \in \R && \forall (m,i,\kappa) \in {\mathfrak{V}}\\
&&& S \in \mathcal{S} 
\end{aligned} \right], \label{prob:robust:mip:1}
\end{align}}%
where the first equality follows from the definition of the robust optimization problem~\eqref{prob:robust}, the second equality follows from Claim~\ref{claim:lp_reform_nested:objective} (which implies that $ \min_{\lambda \in \mathcal{U}}\mathscr{R}^{\lambda}(S)$ is equal to the optimal objective value of the linear optimization problem~\eqref{prob:nested:compact:dual} for all assortments $S \in \mathcal{S}$), and the third equality follows from algebra. 

To transform the  maximization problem in line~\eqref{prob:robust:mip:1} into a mixed-integer optimization problem, we introduce a binary decision variable $x_{i} \in \{0,1\}$ for each product $i \in \mathcal{N}_0$ that satisfies $x_i = 1$ if and only if product $i$ is in the assortment $S$. 
Under this equivalence between an assortment $S \in \mathcal{S}$ and the binary decision variables $x \in \{0,1\}^{\mathcal{N}_0}$, it follows from 
 the definition of $\mathfrak{N}(\cdot)$ that the following equality holds for all assortments $S \in \mathcal{S}$: 
\begin{align}
\mathfrak{N}(S) \triangleq \left \{ (m,i,\kappa) \in {\mathfrak{V}}: \begin{gathered} 
\left[  x_i = 1  \textnormal{ and } i \in \mathcal{B}_m \textnormal{ and } \kappa \neq r_i \right]\\
 \textnormal{ or } \left[  x_i = 1 \textnormal{ and } i \notin \mathcal{B}_m \textnormal{ and } \kappa \in \left\{r_j: j \in \mathcal{B}_m \right \}  \right]\\
 \textnormal{ or } \left[ \textnormal{there exists } j \in \mathcal{B}_m \textnormal{ such that } \kappa = r_j \textnormal{ and } x_j = 0 \right ]
\end{gathered} \right \}. \label{line:defn_N_frak_x}
\end{align} 
Moreover, for each assortment $S \in \mathcal{S}$ and  each $(m,i,\kappa) \in {\mathfrak{V}}$, let us define
\begin{align}
&\mathfrak{M}_{m,i,\kappa}(S) \triangleq  r_n 
x_i \mathbb{I} \left \{ i \in \mathcal{B}_m \textnormal{ and } \kappa \neq r_i \right \} +  r_n 
x_i \mathbb{I} \left \{ i \notin \mathcal{B}_m \textnormal{ and } \kappa \in \{r_{j}: j \in \mathcal{B}_m \}  \right \} + \sum_{j \in \mathcal{B}_m} \mathbb{I} \left \{ \kappa = r_j \right \}r_n (1 - x_j)\label{line:M_better_defn}
\end{align}
It follows immediately from line~\eqref{line:defn_N_frak_x} and from the equivalence between each assortment $S \in \mathcal{S}$ and the binary decision variables $x \in \{0,1\}^{\mathcal{N}_0}$ that the function $\mathfrak{M}(\cdot)$ defined above satisfies the inequality $\mathfrak{M}_{m,i,\kappa}(S) \ge r_n$ for all $(m,i,\kappa) \in \mathfrak{N}(S)$ and satisfies the equality $\mathfrak{M}_{m,i,\kappa}(S) = 0$ for all $(m,i,\kappa) \in {\mathfrak{V}} \setminus \mathfrak{N}(S)$. Thus, it follows from  \eqref{prob:robust:mip:1} that the robust optimization problem~\eqref{prob:robust} is equivalent to 
\begin{equation}\label{prob:robust:mip:2},
\begin{aligned}
& \; \underset{x,\alpha,\beta,\gamma}{\textnormal{maximize}} && \sum_{m \in \mathcal{M}, i \in S_m} v_{m,i} \alpha_{m,i}  \\
&\textnormal{subject to}&&\alpha_{m,i} - \beta_{m,i,\kappa} \mathbb{I} \left \{ m \in \{1,\ldots,M-1\} \right \} - \gamma_{m,i,\kappa} \mathbb{I} \left \{ m \in \{2,\ldots,M\} \right \}\\
&&& \quad \le \kappa \mathbb{I} \left \{ m = M \right \} \\
&&& \quad \quad +  r_n x_i \mathbb{I} \left \{ i \in \mathcal{B}_m \textnormal{ and } \kappa \neq r_i \right \} \\
&&& \quad \quad +  r_n 
x_i \mathbb{I} \left \{ i \notin \mathcal{B}_m \textnormal{ and } \kappa \in \{r_{j}: j \in \mathcal{B}_m \}  \right \} \\
&&&\quad \quad + \sum_{j \in \mathcal{B}_m} \mathbb{I} \left \{ \kappa = r_j \right \}r_n (1 - x_j) && \forall (m,i,\kappa) \in {\mathfrak{V}}\\
&&& \beta_{m,i,\kappa} + \gamma_{m+1,i',\kappa'}  \le 0&& \forall ((m,i,\kappa),(m+1,i',\kappa')) \in {\mathfrak{E}}\\
&&& x_0 = 1\\
&&& \alpha_{m,i} \in \R && \forall m \in \mathcal{M}, \; i \in S_m\\
&&& \beta_{m,i,\kappa}, \gamma_{m,i,\kappa} \in \R && \forall (m,i,\kappa) \in {\mathfrak{V}}\\
&&& x_{i} \in \{0,1\}&& \forall i \in \mathcal{N}_0,
\end{aligned} 
\end{equation}
where we observe that the mixed-integer optimization problem~\eqref{prob:robust:mip:2} is a reformulation of  \eqref{prob:robust:mip:1} in which the assortment $S \in \mathcal{S}$ is represented by binary decision variables $x \in \{0,1\}^{\mathcal{N}_0}$ that satisfy $x_i = 1$ if and only if $i \in S$ and in which $\mathfrak{M}(\cdot)$ is implemented using the definition from line~\eqref{line:M_better_defn}. We note that the constraint $x_0 = 1$ is enforced in the mixed-integer optimization problem~\eqref{prob:robust:mip:2} because $S \in \mathcal{S}$ if and only if $S \subseteq \mathcal{N}_0$ and $0 \in S$. This completes our mixed-integer optimization reformulation of the robust optimization problem~\eqref{prob:robust} in the case of $\eta = 0$. Since a similar mixed-integer optimization reformulation can be obtained for the case of $\eta > 0$, our proof of  Theorem~\ref{thm:nested} is complete.
\halmos

\section{Proofs of Technical Results from \S\ref{sec:question4:bestcase}} \label{appx:optimistic_algorithms}
\subsection{Proof of Theorem~\ref{thm:poly:optimistic}}
Because the proof of Theorem~\ref{thm:poly:optimistic} closely resembles the proof of Theorem~\ref{thm:poly}, we focus in the present appendix on the main areas of distinction. 

\subsubsection{Characterization of Optimal Assortments.} \label{appx:question4:poly:characterization} We begin our proof of Theorem~\ref{thm:poly:optimistic} by characterizing the structure of optimal assortments for the optimistic optimization problem~\eqref{prob:optimistic}.  To do this, we introduce an optimistic version of Definition~\ref{defn:rho} from \S\ref{sec:characterization:proof}. Indeed, for each tuple $(i_1,\ldots,i_M) \in \mathcal{L}$, let $\xi_{i_1 \cdots i_M}(S)$ be defined as   the maximum revenue among the products in the assortment $S $ that can be the most preferred product in $S$ under a ranking that corresponds to the tuple $(i_1,\ldots,i_M)$.
\begin{definition} \label{defn:xi}
$\xi_{i_1 \cdots i_M}(S) \triangleq \max \limits_{i \in S: \; \cap_{m \in \mathcal{M}} \mathcal{D}_{i_m}(S_m) \cap \mathcal{D}_i(S) \neq \emptyset} r_i$. 
\end{definition}
 Equipped with the above notation, the following result shows that the best-case expected revenue of any fixed assortment  can be reformulated as a linear  optimization problem. The proof of the following Proposition~\ref{prop:reform_wc:optimistic} follows from identical reasoning as Proposition~\ref{prop:reform_wc} and is thus omitted. 
\begin{proposition} \label{prop:reform_wc:optimistic}
 For each  $S \in \mathcal{S}$, $\max_{\lambda \in \mathcal{U}} \mathscr{R}^\lambda(S)$ is equal to the optimal objective value of the following linear optimization problem:
\begin{equation} 
\label{prob:optimistic_simplified}
 \begin{aligned}
& \; \underset{\lambda, \epsilon}{\textnormal{maximize}} && \sum_{(i_1,\ldots,i_M) \in \mathcal{L}} \xi_{i_1 \cdots i_M}(S) \lambda_{i_1\cdots i_M}\\
&\textnormal{subject to}&&  \sum_{(i_1,\ldots,i_M) \in \mathcal{L}: \; i_m = i} \lambda_{i_1 \cdots i_M} - \epsilon_{m,i}= v_{m,i} \quad \forall m \in \mathcal{M}, \; i \in S_m\\
&&& \sum_{(i_1,\ldots,i_M) \in \mathcal{L}} \lambda_{i_1 \cdots i_M} = 1 \\
&&& \| \epsilon \| \le \eta\\
&&& \lambda_{i_1 \cdots i_M} \ge 0 \quad \forall (i_1,\ldots,i_M) \in \mathcal{L}.
\end{aligned}
\end{equation}
\end{proposition}
Following identical reasoning and using the same notation as developed in Appendix~\ref{appx:graphical}, we obtain the following graphical interpretation of the quantity $\xi_{i_1 \cdots i_M}(S)$.  In particular, the proof of the following Proposition~\ref{prop:cost_reform:optimistic} follows from identical reasoning as Proposition~\ref{prop:cost_reform} and is thus omitted. 
\begin{proposition} \label{prop:cost_reform:optimistic}
For all $S \in \mathcal{S}$ and $(i_1,\ldots,i_M) \in \mathcal{L}$, 
$\xi_{i_1 \cdots i_M}(S) = \max_{i \in S \cap \mathcal{I}_{i_1 \cdots i_M}(S)} r_i.$
\end{proposition}
Equipped with the graphical interpretation of $\xi_{i_1 \cdots i_M}(S)$ provided by Proposition~\ref{prop:cost_reform:optimistic}, we derive the following optimistic version of Proposition~\ref{prop:plus_one_inequality_prop} from \S\ref{sec:characterization:proof}: 
\begin{proposition}\label{prop:plus_one_inequality_prop:optimistic}
Let $S \in \mathcal{S}$ and $i \notin S$. If  there exists $i^* \in \mathcal{S}$ that satisfies $r_{i^*} > r_{i}$ and $\mathcal{M}_{i^*} \subseteq \mathcal{M}_i$, then $\xi_{i_1 \cdots i_M}(S) \ge \xi_{i_1 \cdots i_M}(S \cup \{i \})$ for each $(i_1,\ldots,i_M) \in \mathcal{L}$. 
\end{proposition}
\begin{proof}{Proof.}
We begin our proof of  Proposition~\ref{prop:plus_one_inequality_prop:optimistic} by developing an intermediary result, denoted below by Claim~\ref{claim:add_one:optimistic}, that will allow us to compare the values of $\xi_{i_1 \cdots i_M}(S)$ and $\xi_{i_1 \cdots i_M}(S \cup \{i\})$ for every assortment $S \in \mathcal{S}$ and every product $i$ which is not in the assortment.  The proof of the following Claim~\ref{claim:add_one:optimistic}, which uses the graphical interpretation of the quantity $\xi_{i_1 \cdots i_M}(S)$ from Proposition~\ref{prop:plus_one_inequality_prop:optimistic},  follows from identical reasoning as the proof of Claim~\ref{claim:add_one} from Appendix~\ref{appx:prop2} and is thus omitted. 
\begin{claim} \label{claim:add_one:optimistic}
For all $S \in \mathcal{S}$, $(i_1,\ldots,i_M) \in \mathcal{L}$, and $i \notin S$,
\begin{align*}
&\xi_{i_1 \cdots i_M}(S \cup \{i\}) = \begin{cases}
\xi_{i_1 \cdots i_M}(S),&\textnormal{if } i \notin \mathcal{I}_{i_1 \cdots i_M}(S),\\
\max \left \{ \max \limits_{j \in S \cap \mathcal{I}_{i_1 \cdots i_M}(S) \cap \left\{ j' \in \mathcal{N}_0: i \nprec_{i_1 \cdots i_M} j' \right\}} r_j, r_{i} \right \}, &\textnormal{if } i \in \mathcal{I}_{i_1 \cdots i_M}(S).
\end{cases}
\end{align*}
\end{claim}
Using the above intermediary result, we now complete the proof of Proposition~\ref{prop:plus_one_inequality_prop:optimistic}. Indeed, consider any assortment $S\in \mathcal{S}$ and any product $i \notin S$.  Suppose that there exists a product $i^* \in S$ which satisfies $r_{i^*} > r_i$ and $\mathcal{M}_{i^*} \subseteq \mathcal{M}_i$.  For each tuple of products $(i_1,\ldots,i_M) \in \mathcal{L}$, we have two cases to consider:
\begin{itemize}
\item \underline{Case 1:} Suppose that $i \notin \mathcal{I}_{i_1 \cdots i_M}(S)$. In this case, it follows  immediately from Claim~\ref{claim:add_one:optimistic} that $$\xi_{i_1 \cdots i_M}(S \cup \{i\}) = \xi_{i_1 \cdots i_M}(S),$$ 
and so the inequality $\rho_{i_1 \cdots i_M}(S) \le \xi_{i_1 \cdots i_M}(S \cup \{i\})$ holds when  $i \notin \mathcal{I}_{i_1 \cdots i_M}(S)$. 

\vspace{1em}

\item \underline{Case 2:} Suppose that $i \in \mathcal{I}_{i_1 \cdots i_M}(S)$.

  In this case, we begin by showing that $i_m \notin S$ for each $m \in \mathcal{M}_i$. Indeed, consider any past assortment $m \in \mathcal{M}_i$. On one hand, if $i = i_m$, then it follows immediately from the fact that $i \notin S$ that $i_m \notin S$.  On the other hand, if $i \neq i_m$, then it also must be the case that $i_m \notin S$, else we would have a contradiction with the fact that  $i \in \mathcal{I}_{i_1 \cdots i_M}(S)$ and the fact that there is, by our construction of $\mathcal{G}_{i_1 \cdots i_M}$, a directed edge from vertex $i$ to vertex $i_m$. We have thus shown that $i_m \notin S$ for all $m \in \mathcal{M}_i$. 

We next show that $i_m \in \mathcal{I}_{i_1 \cdots i_M}(S)$ for all $m \in \mathcal{M}_i$. Indeed, consider any arbitrary $m \in \mathcal{M}_i$. On one hand, if $i_m = i$, then the  statement  $i_m \in \mathcal{I}_{i_1 \cdots i_M}(S)$ follows immediately from the fact that $i \in \mathcal{I}_{i_1 \cdots i_M}(S)$. On the other hand, if $i_m \neq i$, then it follows from the fact that $i \in \mathcal{I}_{i_1 \cdots i_M}(S)$ and the fact that there is a directed edge from vertex $i$ to vertex $i_m$ that there must not be a directed path from vertex $i_m$ to a vertex  $i_{m'}$ that satisfies $i_{m'} \in S$ for any $m' \in \mathcal{M}$. We have thus shown that   $i_m \in \mathcal{I}_{i_1 \cdots i_M}(S)$ for all $m \in \mathcal{M}_i$.

Using the above results, we now prove that $i^* \in \mathcal{I}_{i_1 \cdots i_M}(S)$. Indeed, we have shown in the above results that $i_m \notin S$ and $i_m \in \mathcal{I}_{i_1 \cdots i_M}(S)$ for all $m \in \mathcal{M}_{i}$. Therefore, it follows from the supposition that $\mathcal{M}_{i^*} \subseteq \mathcal{M}_i$ that $i_m \notin S$ and $i_m \in \mathcal{I}_{i_1 \cdots i_M}(S)$ for all $m \in \mathcal{M}_{i^*}$.  Since we have supposed that $i^* \in S$, it follows from the fact that $i_m \notin S$  for all $m \in \mathcal{M}_{i^*}$ that $i^* \neq i_m$ for all $m \in \mathcal{M}_{i^*}$. Moreover, it follows from the construction of $\mathcal{G}_{i_1 \cdots i_M}$ that all of the outgoing edges from vertex $i^*$ are incoming edges to  vertices $i_m$ for $m \in \mathcal{M}_{i^*}$. Therefore, it follows from the fact that $i_m \in \mathcal{I}_{i_1 \cdots i_M}(S)$ for all $m \in \mathcal{M}_{i^*}$ that  $i^* \in \mathcal{I}_{i_1 \cdots i_M}(S)$. 

We now prove the desired result for Case 2. First, we observe that
\begin{align}
\xi_{i_1 \cdots i_M}(S) = \max_{j \in S \cap \mathcal{I}_{i_1 \cdots i_M}(S)} r_j \ge r_{i^*}, \label{line:rho_woah:optimistic}
\end{align}
where the equality follows from Proposition~\ref{prop:cost_reform:optimistic}  and the inequality follows from our supposition that $i^* \in S$ and because we have shown that  $i^* \in \mathcal{I}_{i_1 \cdots i_M}(S)$. 
Therefore, we have that 
\begin{align*}
\xi_{i_1 \cdots i_M}(S \cup \{i\}) &= \max \left \{ \max \limits_{j \in S \cap \mathcal{I}_{i_1 \cdots i_M}(S) \cap \{ j' \in \mathcal{N}_0: i \nprec_{i_1 \cdots i_M} j' \}} r_j, r_i \right \} \\
&\le \max \left \{   \max \limits_{j \in S \cap \mathcal{I}_{i_1 \cdots i_M}(S)} r_j, r_i \right \} \\
&= \max \left \{  \xi_{i_1 \cdots i_M}(S), r_i \right \} \\
&= \xi_{i_1 \cdots i_M}(S),
\end{align*}
where the first equality follows from Claim~\ref{claim:add_one} and the supposition of Case 2 that $i \in \mathcal{I}_{i_1 \cdots i_M}(S)$, the inequality follows algebra, the second equality follows from  Proposition~\ref{prop:cost_reform:optimistic},  and the final equality holds because of our supposition that $r_{i^*} > r_{i}$ and because of line~\eqref{line:rho_woah:optimistic}, which showed that $ \xi_{i_1 \cdots i_M}(S) \ge r_{i^*}$. This concludes the proof of Case 2. 
\end{itemize}
In both of the above two cases, we showed that $\xi_{i_1 \cdots i_M}(S) \ge \xi_{i_1 \cdots i_M}(S \cup \{i \})$, and so our proof of Proposition~\ref{prop:plus_one_inequality_prop:optimistic} is complete. 
\halmos \end{proof}
In view of the above Proposition~\ref{prop:plus_one_inequality_prop:optimistic}, 
we proceed to develop our main result of Appendix~\ref{appx:question4:poly:characterization}, which is a characterization of the structure of optimal assortments for the optimistic optimization problem~\eqref{prob:optimistic}. To this end, we introduce the following new collection of assortments:
\begin{align*}
\widehat{\mathcal{S}}^{\text{OO}} &\triangleq \left \{ S \in \mathcal{S}:\; \textnormal{if } i^* \in S, \;  0 < r_i < r_{i^*},  \textnormal{ and } \mathcal{M}_{i^*} \subseteq \mathcal{M}_i, \textnormal{ then } i \notin S\right \}.
\end{align*}
Our main result is the following:
\begin{theorem} \label{thm:main:optimistic}
There exists an assortment $S \in \widehat{\mathcal{S}}^{\textnormal{OO}}$ that is optimal for \eqref{prob:optimistic}. 
\end{theorem}
\begin{proof}{Proof.}
Consider any arbitrary assortment $S \in \mathcal{S}$. For this assortment, we define a new assortment as $$S' \triangleq S \setminus \left \{i \in \mathcal{N}: \text{there exists } i^* \in S \textnormal{ such that }  \mathcal{M}_{i^*} \subseteq \mathcal{M}_i \textnormal{ and } r_{i^*} > r_i  \right \}.$$

We first show that this new assortment $S'$ is an element of the collection $\widehat{\mathcal{S}}^{\textnormal{OO}}$. Indeed, it follows from the fact that $S \in \mathcal{S}$ that $0 \in S'$, which implies that $S' \in \mathcal{S}$. Moreover, for every $i^* \in S'$, it follows from the construction of $S'$ that there does not exist $i \in S \setminus \{0\}$ that satisfies $r_i < r_{i^*}$ and $\mathcal{M}_{i^*} \subseteq \mathcal{M}_i$. We thus conclude that $S' \in \widehat{\mathcal{S}}^{\textnormal{OO}}$.

We next show for each product $i \in S' \setminus S$ that there must exist a product $i^* \in S'$ that satisfies  $r_i < r_{i^*}$ and $\mathcal{M}_{i^*} \subseteq \mathcal{M}_i$. Indeed, consider any product $i \in S' \setminus S$.  It follows from the construction of $S'$ that there exists a product $i' \in S$ that satisfies $\mathcal{M}_{i'} \subseteq \mathcal{M}_i$ and $r_i < r_{i'}$.  Let $i^* \triangleq \argmax_{i' \in S: \mathcal{M}_{i'} \subseteq \mathcal{M}_i \text{ and } r_i < r_{i'}} r_{i'}$ be the product with the highest revenue among all products $i'$  that satisfy $\mathcal{M}_{i'} \subseteq \mathcal{M}_i$ and $r_i < r_{i'}$. It follows from the definition of $i^*$ that there does not exist any products $i'' \in S$ that satisfy $\mathcal{M}_{i''} \subseteq \mathcal{M}_{i^*}$ and $r_{i''} > r_{i^*}$. Therefore, we conclude that the product $i^*$ is an element of $S'$.

Finally, let $\{ j_1,\ldots,j_\nu \} \triangleq   S \setminus S'$ denote the products that have been removed from the assortment $S$. We observe for each tuple of products $(i_1,\ldots,i_M) \in \mathcal{L}$ that
\begin{align*}
\xi_{i_1 \cdots i_M}(S') &=  \xi_{i_1 \cdots i_M}(S)  + \sum_{\iota = 1}^{\nu} \left(  \xi_{i_1 \cdots i_M}(S \setminus \{j_1,\ldots,j_\iota \}) - \xi_{i_1 \cdots i_M}(S \setminus \{j_1,\ldots,j_{\iota-1} \}) \right) \le  \xi_{i_1 \cdots i_M}(S).
\end{align*} 
Indeed, the equality follows from algebra.  The  inequality  follows from Proposition~\ref{prop:plus_one_inequality_prop:optimistic} together with the fact that for each $\iota \in \{1,\ldots,\nu \}$, there exists a product $i^* \in S'$ that  satisfies  $r_{j_\iota} < r_{i^*}$ and $\mathcal{M}_{i^*} \subseteq \mathcal{M}_{j_\iota}$.  Since the assortment $S \in \mathcal{S}$ was chosen arbitrarily, our proof of Theorem~\ref{thm:main:optimistic} follows from Proposition~\ref{prop:reform_wc:optimistic}. 
\halmos \end{proof}

\subsubsection{Upper Bound.}  \label{appx:fixed_dim:S:optimistic} As the second step of our proof of Theorem~\ref{thm:poly:optimistic}, we develop an upper bound on the cardinality of the collection of assortments $\widehat{\mathcal{S}}^{\textnormal{OO}}$ that we showed in Theorem~\ref{thm:main:optimistic} contains an optimal assortment for the optimistic optimization problem~\eqref{prob:optimistic}. 

\begin{lemma}\label{lem:fixed_dim:S:optimistic}
$| \widehat{\mathcal{S}}^{\textnormal{OO}} | \le (n+1)^{2^M}$.
\end{lemma}
\begin{proof}{Proof.}
We readily observe that the collection of assortments $\widehat{\mathcal{S}}^{\textnormal{OO}}$ is a subset of
\begin{align*}
\widetilde{\mathcal{S}}^{\textnormal{OO}} \triangleq \left \{ S \in \mathcal{S}:\; \textnormal{if } i^* \in S,\; 0 < r_i <  r_{i^*}, \textnormal{ and } \mathcal{M}_{i^*} = \mathcal{M}_i, \textnormal{ then } i \notin S \right \},
\end{align*}
where we recall from the beginning of \S\ref{sec:characterization} that $\mathcal{M}_i$ is defined as the subset of the past assortments $\mathcal{M} \equiv \{1,\ldots,M\}$ that offered product $i$. 
For each subset of  past assortments $\mathcal{C} \subseteq \mathcal{M}$, let the products that are offered {only} in the assortments in $\mathcal{C}$ (excluding the no-purchase option) be denoted by 
\begin{align*}
\mathcal{N}(\mathcal{C}) &\triangleq 
 \left \{ i \in \mathcal{N}: \mathcal{M}_i = \mathcal{C}  \right \}.
\end{align*}
Equipped with the above definitions, we observe that $\{\mathcal{N}(\mathcal{C}): \mathcal{C} \subseteq \mathcal{M}\}$ is the collection of subsets of products that always appear together in the past assortments, and we readily observe that  $|\{\mathcal{N}(\mathcal{C}): \mathcal{C} \subseteq \mathcal{M}\}| = 2^M$. 
Therefore, 
\begin{align}
| \widehat{\mathcal{S}}^{\textnormal{OO}}  |  &\le | \widetilde{\mathcal{S}}^{\textnormal{OO}}  | \le \prod_{\mathcal{C} \subseteq \mathcal{M}} \left( \left| \mathcal{N}(\mathcal{C}) \right| + 1 \right)  \le (n+1)^{2^M}.  \label{line:badname:optimistic}
\end{align}
Indeed, the first inequality on line~\eqref{line:badname:optimistic} holds because the collection of assortments $\widehat{\mathcal{S}}^{\textnormal{OO}}$ is a subset of the collection of assortments $\widetilde{\mathcal{S}}^{\textnormal{OO}}$. To see why the second inequality on line~\eqref{line:badname:optimistic} holds, consider any arbitrary subset of past assortments $\mathcal{C} \subseteq \mathcal{M}$, and let the products in $\mathcal{N}(\mathcal{C})$ be indexed in ascending order by revenue; that is, let the products that comprise $\mathcal{N}(\mathcal{C})$ be denoted by $i_1^{\mathcal{C}},\ldots,i_{|\mathcal{N}(\mathcal{C})|}^{\mathcal{C}}$, where $r_{i_1^{\mathcal{C}}} < \cdots < r_{i_{|\mathcal{N}(\mathcal{C})|}^{\mathcal{C}}}$.  Then, it follows from the definition of the collection $\widetilde{\mathcal{S}}^{\textnormal{OO}}$ that every assortment $S \in \widetilde{\mathcal{S}}^{\textnormal{OO}}$  must satisfy the condition $\left| \left\{ i^{\mathcal{C}}_1,\ldots,i^{\mathcal{C}}_{| \mathcal{N}(\mathcal{C})|}  \right \}  \cap S \right|  \le 1$. 
 Since $\mathcal{C} \subseteq \mathcal{M}$ was chosen arbitrarily, we have shown that 
$$\widetilde{\mathcal{S}}^{\textnormal{OO}} \subseteq\left \{ \{0 \} \cup \bigcup_{\mathcal{C}\subseteq \mathcal{M}} \mathcal{F}^{\mathcal{C}}: \; \mathcal{F}^{\mathcal{C}} \in \left \{  \emptyset, \left\{i_{1}^{\mathcal{C}} \right \},\ldots,   \left\{i_{| \mathcal{N}(\mathcal{C})|}^{\mathcal{C}} \right \}  \right \}  \right \},  $$
which proves that the second inequality on line~\eqref{line:badname:optimistic} holds. The third inequality on line~\eqref{line:badname:optimistic} follows from the fact that  $|\{\mathcal{N}(\mathcal{C}): \mathcal{C} \subseteq \mathcal{M}\}| = 2^M$ and from the fact that $|\mathcal{N}(\mathcal{C})| \le  n$ for every $\mathcal{C} \subseteq \mathcal{M}$. We have thus proven that  $| \widehat{\mathcal{S}}^{\textnormal{OO}} |$ is at most $(n+1)^{2^M}$, which concludes our proof of Lemma~\ref{lem:fixed_dim:S:optimistic}. \halmos 
\end{proof}
\subsubsection{Algorithm for Constructing $\widehat{\mathcal{S}}^{\textnormal{OO}}$.}  \label{appx:fixed_dim:S:optimistic:time} As the third step of our proof of Theorem~\ref{thm:poly:optimistic}, we develop an efficient algorithm for constructing the collection of assortments $\widehat{\mathcal{S}}^{\textnormal{OO}}$ that contains an optimal assortment for the optimistic optimization problem~\eqref{prob:optimistic}. 
Specifically, our main result of Appendix~\ref{appx:fixed_dim:S:optimistic:time} is the following lemma:
\begin{lemma} \label{lem:fixed_dim:S:time:optimistic}
The collection  of assortments $\widehat{\mathcal{S}}^{\textnormal{OO}}$ can be constructed in $\mathcal{O}(n^{2} (M+| \widehat{\mathcal{S}}^{\textnormal{OO}}|))$ time. 
  \end{lemma}
  \begin{proof}{Proof.}


  We begin by developing an algorithm for constructing a directed graph, denoted by $\mathscr{G}^{\textnormal{OO}} \equiv (\mathscr{V}^{\textnormal{OO}},\mathscr{E}^{\textnormal{OO}})$, in which the set of vertices in the directed graph is defined as $\mathscr{V}^{\textnormal{OO}} \triangleq \mathcal{N}$ and the set of directed edges is defined as $\mathscr{E}^{\textnormal{OO}} \triangleq \{ (i^*,i) \in \mathcal{N} \times \mathcal{N}: \; r_{i^*} > r_{i} \textnormal{ and } \mathcal{M}_{i^*} \subseteq \mathcal{M}_{i} \}$. 
This directed graph  has a natural correspondence with the collection of assortments $\widehat{\mathcal{S}}^{\textnormal{OO}}$, as we readily observe that an assortment satisfies $S \in \widehat{\mathcal{S}}^{\textnormal{OO}}$ if and only if [$0 \in S$] and [we have that $i \notin S$ whenever there exists a product $i^* \in S \setminus \{0 \}$ and a directed edge $(i^*,i) \in \mathscr{E}^{\textnormal{OO}}$]. It is easy to see that this directed graph is acyclic and has the property that a vertex $j \in \mathscr{V}^{\textnormal{OO}}$  is reachable from a vertex $i \in \mathscr{V}^{\textnormal{OO}}$ if and only if there is a directed edge $(i,j) \in \mathscr{E}^{\textnormal{OO}}$ from vertex $i$ to vertex $j$. 
A directed acyclic graph which has the aforementioned property for each pair of vertices is referred to as a \emph{transitive closure}.

\begin{algorithm}[t]
\begin{center}
\fbox{ \begin{minipage}{\linewidth}
\begin{center}
\textsc{\underline{Construct-${\mathscr{G}^{\textnormal{OO}}}(\mathscr{M},r)$}}\\
\end{center}
\vspace{1em}
\textbf{Inputs}: 
\begin{itemize}
\item The collection of past assortments, $\mathscr{M}  \equiv \{S_1,\ldots,S_M\}$. 
\item The revenues of the products, $r \equiv (r_0,r_1,\ldots,r_n)$. 
\end{itemize}
\vspace{1em}
\textbf{Output}:
\begin{itemize}
\item  $\mathscr{G}^{\textnormal{OO}} \equiv (\mathscr{V}^{\textnormal{OO}} ,\mathscr{E}^{\textnormal{OO}})$, where $\mathscr{V}^{\textnormal{OO}} \equiv \mathcal{N}$ and $\mathscr{E}^{\textnormal{OO}} \equiv \{ (i^*,i) \in \mathcal{N} \times \mathcal{N}: \; r_{i^*} > r_{i} \textnormal{ and } \mathcal{M}_{i^*} \subseteq \mathcal{M}_{i} \}$. 
\end{itemize} 
\vspace{1em}
\textbf{Procedure}: 
\begin{enumerate}
\item Initialize the vertex set $\mathscr{V}^{\textnormal{OO}} \leftarrow \emptyset$ and edge set $\mathscr{E}^{\textnormal{OO}} \leftarrow \emptyset$. 
\item For each product $i \in \mathcal{N}^{\textnormal{OO}}$: \label{step:iterative_of_i:optimistic}
\begin{enumerate}
\item Update $\mathscr{V}^{\textnormal{OO}} \leftarrow \mathscr{V}^{\textnormal{OO}} \cup \{i \}$. 
\item Construct the set $\mathcal{M}_i$ of past assortments which offered product $i$. 
\end{enumerate}
\item For each pair of products $(i^*,i) \in \mathcal{N} \times \mathcal{N} $:
\begin{enumerate}
\item If $r_{i^*} > r_{i}$ and $\mathcal{M}_{i^*} \subseteq \mathcal{M}_i$:
\begin{enumerate}
\item Update $\mathscr{E}^{\textnormal{OO}} \leftarrow \mathscr{E}^{\textnormal{OO}} \cup \{(i^*,i) \}$
\end{enumerate}
\end{enumerate}
\item Output $\mathscr{G}^{\textnormal{OO}} \equiv (\mathscr{V}^{\textnormal{OO}},\mathscr{E}^{\textnormal{OO}})$ and terminate. 
\end{enumerate}
\end{minipage}}
\end{center}
\caption{A procedure for constructing the directed acyclic graph $\mathscr{G}^{\textnormal{OO}}$.} \label{alg:construct_G:optimistic}
\end{algorithm}

Our algorithm for constructing the directed acyclic graph $\mathscr{G}^{\textnormal{OO}}$  that is a transitive closure is presented in Algorithm~\ref{alg:construct_G:optimistic}. In the algorithm, we first iterate over each product $i \in \mathcal{N}$ and construct the corresponding set $\mathcal{M}_i$ of past assortments which offered that product. We then iterate over each pair of products $(i,i^*) \in \mathcal{N} \times \mathcal{N}$ and check  whether $r_{i^*} > r_{i}$ and $\mathcal{M}_{i^*} \subseteq \mathcal{M}_i$.  It follows from identical reasoning as that given for Algorithm~\ref{alg:construct_G} in Appendix~\ref{appx:fixed_dim:S:time} that Algorithm~\ref{alg:construct_G:optimistic}  has a  total computation time of $\mathcal{O}(n^2 M)$.

We next describe our algorithm for constructing the collection of assortments $\widehat{\mathcal{S}}^{\textnormal{OO}}$ from the directed graph $\mathscr{G}^{\textnormal{OO}}$. This algorithm is  denoted by \textsc{Construct-$\widehat{\mathcal{S}}^{\textnormal{OO}}(\mathscr{M},r)$} and is found in Algorithm~\ref{alg:construct_S:optimistic}.  In this algorithm,  we first use Algorithm~\ref{alg:construct_G:optimistic} to construct the directed acyclic graph $\mathscr{G}^{\textnormal{OO}} \equiv (\mathscr{V}^{\textnormal{OO}},\mathscr{E}^{\textnormal{OO}})$, and then we invoke a recursive subroutine denoted by \textsc{RecursiveStep-OO}$(\mathscr{G}^{\textnormal{OO}})$ in Algorithm~\ref{alg:recursive:optimistic}.   The goal of the recursive subroutine is to take as an input a generic directed acyclic graph $\mathscr{G}^{\textnormal{OO}} \equiv (\mathscr{V}^{\textnormal{OO}}, \mathscr{E}^{\textnormal{OO}})$, and for that graph, output the collection of subsets of vertices $\mathscr{A}^{\textnormal{OO}} \equiv \{S \subseteq \mathscr{V}^{\textnormal{OO}}:  \text{ if $i \in S$ and $(i,j) \in \mathscr{E}^{\textnormal{OO}}$, then $j \notin S$}\}$. Algorithm~\ref{alg:construct_S:optimistic} concludes by adding the no-purchase option $0$ into each of the subsets of vertices from $\mathscr{A}^{\textnormal{OO}}$.  The correctness of Algorithm~\ref{alg:construct_S:optimistic} follows immediately from our earlier observation that an assortment satisfies $S \in \widehat{\mathcal{S}}^{\textnormal{OO}}$ if and only if  [$0 \in S$] and [we have that $i \notin S$ whenever there exists a product $i^* \in S \setminus \{0 \}$ and a directed edge $(i^*,i) \in \mathscr{E}^{\textnormal{OO}}$].

   At a high level, the recursive subroutine in Algorithm~\ref{alg:recursive:optimistic} is comprised of two cases. The base case of the subroutine is when the graph has no vertices, in which case it is clear that $\mathscr{A}^{\textnormal{OO}} =\{ \emptyset \}$. If we are not in the base case, then the aim of the recursive subroutine is to construct the collections $ \left \{ S \in \mathscr{A}^{\textnormal{OO}}: i \notin S \right \} $ and $\left \{ S \in \mathscr{A}^{\textnormal{OO}}: i \in S \right \}$ for a chosen vertex $i \in \mathscr{V}^{\textnormal{OO}}$ and output the union of these two collections. The construction of the collection $ \left \{ S \in \mathscr{A}^{\textnormal{OO}}: i \notin S \right \} $ takes place on lines~\eqref{step:construct_A1:1:optimistic}-\eqref{step:construct_A1:2:optimistic} of Algorithm~\ref{alg:recursive:optimistic}, and the construction of the collection $ \left \{ S \in \mathscr{A}^{\textnormal{OO}}: i \in S \right \} $ takes place on lines~\eqref{step:construct_A2:1:optimistic}-\eqref{step:construct_A2:3:optimistic} of Algorithm~\ref{alg:recursive:optimistic}. 

Up to this point, we have established that Algorithm~\ref{alg:construct_G:optimistic} is correct (that is,  it delivers the desired output for any valid input), and we have established that Algorithm~\ref{alg:construct_S:optimistic} is correct under the assumption that Algorithm~\ref{alg:recursive:optimistic} is correct. Therefore, it remains for us to prove that  the recursive subroutine in Algorithm~\ref{alg:recursive:optimistic} is correct. To prove the correctness of the recursive subroutine, we will make use of four intermediary claims, which are denoted below by  Claims~\ref{claim:appx:idk_intro:optimistic}-\ref{claim:appx:idk_3:optimistic}. The purpose of the first two intermediary claims, denoted by Claims~\ref{claim:appx:idk_intro:optimistic} and \ref{claim:appx:idk:optimistic}, is to show that the graphs $\mathscr{G}^{\textnormal{OO}'} \equiv (\mathscr{V}^{\textnormal{OO}'}, \mathscr{E}^{\textnormal{OO}'})$ and $\mathscr{G}^{\textnormal{OO}''} \equiv (\mathscr{V}^{\textnormal{OO}''}, \mathscr{E}^{\textnormal{OO}''})$  constructed on lines~\eqref{step:construct_A1:1:optimistic} and \eqref{step:construct_A2:1:optimistic} of Algorithm~\ref{alg:recursive:optimistic} are directed acyclic graphs that are transitive closures, which implies that $\mathscr{G}^{\textnormal{OO}'}$ and $\mathscr{G}^{\textnormal{OO}''}$ are valid inputs on lines~\eqref{step:construct_A1:2:optimistic} and \eqref{step:construct_A2:2:optimistic} of  Algorithm~\ref{alg:recursive:optimistic}. The purpose of the second two intermediary claims, denoted by Claims~\ref{claim:appx:idk_2:optimistic} and \ref{claim:appx:idk_3:optimistic}, is to show that the union of the two collections $ \mathscr{A}^{\textnormal{OO}'}$ and $ \mathscr{A}^{\textnormal{OO}'''}$ constructed on lines~\eqref{step:construct_A1:2:optimistic} and \eqref{step:construct_A2:3:optimistic} of Algorithm~\ref{alg:recursive:optimistic} provides the desired output on line~\eqref{step:construct_A:optimistic} of Algorithm~\ref{alg:recursive:optimistic}.   
\begin{algorithm}[t] 
\begin{center}
\fbox{ \begin{minipage}{\linewidth}
\begin{center}
\textsc{\underline{Construct-$\widehat{\mathcal{S}}^{\textnormal{OO}}(\mathscr{M},r)$}}\\
\end{center}
\vspace{1em}
\textbf{Inputs}: 
\begin{itemize}
\item The collection of past assortments, $\mathscr{M}  \equiv \{S_1,\ldots,S_M\}$. 
\item The revenues of the products, $r \equiv (r_0,r_1,\ldots,r_n)$. 
\end{itemize}
\vspace{1em}
\textbf{Output}:
\begin{itemize}
\item  The collection of assortments $\widehat{\mathcal{S}}^{\textnormal{OO}}$ corresponding to the collection of past assortments $\mathscr{M}$ and the revenues $r$. 
\end{itemize} 
\vspace{1em}
\textbf{Procedure}: 
\begin{enumerate}
\item Construct the directed acyclic graph $\mathscr{G}^{\textnormal{OO}} \leftarrow \textsc{Construct-${\mathscr{G}}^{\textnormal{OO}}(\mathscr{M},r)$}$.  \label{step:construct_S:1:optimistic}
\item Compute the collection of assortments $\widehat{\mathcal{S}}^{\textnormal{OO}} \leftarrow \textsc{RecursiveStep-OO$(\mathscr{G}^{\textnormal{OO}})$}$.  \label{step:construct_S:2:optimistic}
\item Output $\{ S \cup \{0\}: S \in \widehat{\mathcal{S}}^{\textnormal{OO}}\}$ and terminate. 
\end{enumerate}
\vspace{1em}
\end{minipage}}
\end{center}
\caption{A procedure for constructing $\widehat{\mathcal{S}}^{\textnormal{OO}}$.} \label{alg:construct_S:optimistic}
\end{algorithm}

\begin{algorithm}[t]
\begin{center}
\fbox{ \begin{minipage}{\linewidth}
\begin{center}
\textsc{\underline{RecursiveStep-OO$(\mathscr{G}^{\textnormal{OO}})$}}\\
\end{center}
\vspace{1em}
\textbf{Inputs}: 
\begin{itemize}
\item A directed acyclic graph $\mathscr{G}^{\textnormal{OO}} \equiv (\mathscr{V}^{\textnormal{OO}}, \mathscr{E}^{\textnormal{OO}})$ that is a transitive closure. 
\end{itemize}
\vspace{1em}
\textbf{Output}:
\begin{itemize}
\item The collection $\mathscr{A}^{\textnormal{OO}} \equiv \{S \subseteq \mathscr{V}^{\textnormal{OO}}:  \text{ if $i \in S$ and $(i,j) \in \mathscr{E}^{\textnormal{OO}} $, then $j \notin S$}\}$. \end{itemize} 
\vspace{1em}
\textbf{Procedure}: 
\begin{enumerate}
\item If $\mathscr{V}^{\textnormal{OO}} = \emptyset$:
\begin{enumerate}
\item Output the collection $\mathscr{A}^{\textnormal{OO}} \equiv\{ \emptyset \}$ and terminate. 
\end{enumerate}
\item Otherwise:  
\begin{enumerate}
\item Choose any vertex  $i \in \mathscr{V}^{\textnormal{OO}} $.
\label{step:choose_vertex:optimistic} 

\item  
Create a copy of $\mathscr{G}^{\textnormal{OO}}  \equiv (\mathscr{V}^{\textnormal{OO}},\mathscr{E}^{\textnormal{OO}})$ in which the vertex $i$ and its incoming and outgoing edges  are removed. Denote this new graph by $\mathscr{G}^{\textnormal{OO}'} \equiv (\mathscr{V}^{\textnormal{OO}'} \mathscr{E}^{\textnormal{OO}'})$.\label{step:construct_A1:1:optimistic}
\item Compute the collection $\mathscr{A}^{\textnormal{OO}'} \leftarrow \textsc{RecursiveStep-OO$(\mathscr{G}^{\textnormal{OO}'})$}$. \label{step:construct_A1:2:optimistic}

\item 
Create a copy of $\mathscr{G}^{\textnormal{OO}} \equiv (\mathscr{V}^{\textnormal{OO}},\mathscr{E}^{\textnormal{OO}})$ in which the vertices $\{i \} \cup  \{\ell: (i,\ell) \in \mathscr{E}^{\textnormal{OO}} \} \cup \{ \ell: (\ell,i) \in \mathscr{E}^{\textnormal{OO}} \}$ and the incoming and outgoing edges of these vertices are removed. Denote this new graph by $\mathscr{G}^{\textnormal{OO}''} \equiv (\mathscr{V}^{\textnormal{OO}''}, \mathscr{E}^{\textnormal{OO}''})$.  \label{step:construct_A2:1:optimistic}
\item Compute the collection $\mathscr{A}^{\textnormal{OO}''} \leftarrow \textsc{RecursiveStep$(\mathscr{G}^{\textnormal{OO}''})$}$.\label{step:construct_A2:2:optimistic}
\item Compute the collection $\mathscr{A}^{\textnormal{OO}'''} \leftarrow \{S \cup \{i \}:  S \in \mathscr{A}^{\textnormal{OO}''} \}$.  \label{step:construct_A2:3:optimistic}

\item Output the collection $\mathscr{A}^{\textnormal{OO}} \equiv\mathscr{A}^{\textnormal{OO}'} \cup \mathscr{A}^{\textnormal{OO}'''}$ and terminate.  \label{step:construct_A:optimistic}
\end{enumerate}
\end{enumerate}
\vspace{1em}
\end{minipage}}
\end{center}
\caption{A recursive subroutine which is used in Algorithm~\ref{alg:construct_S:optimistic}. }\label{alg:recursive:optimistic}
\end{algorithm}%

\begin{claim} \label{claim:appx:idk_intro:optimistic}
Let  $\mathscr{G}^{\textnormal{OO}} \equiv (\mathscr{V}^{\textnormal{OO}},\mathscr{E}^{\textnormal{OO}})$ be a directed acyclic graph that is a transitive closure. For notational convenience, let  $\mathscr{E}_i^{\textnormal{OO}}$ denote the set of incoming and outgoing edges from each vertex $i \in \mathscr{V}^{\textnormal{OO}}$. Then for each vertex $i \in \mathscr{V}^{\textnormal{OO}}$, we have that  $\tilde{\mathscr{G}}^{\textnormal{OO}} \equiv (\mathscr{V}^{\textnormal{OO}} \setminus  \{i \}, \mathscr{E}^{\textnormal{OO}} \setminus  \mathscr{E}_i^{\textnormal{OO}})$ is a directed acyclic graph that is a transitive closure.  
\end{claim}

\begin{claim} \label{claim:appx:idk:optimistic}
Let  $\mathscr{G}^{\textnormal{OO}} \equiv (\mathscr{V}^{\textnormal{OO}},\mathscr{E}^{\textnormal{OO}})$ be a directed acyclic graph that is a transitive closure. For notational convenience, let  $\mathscr{E}_i^{\textnormal{OO}}$ denote the set of incoming and outgoing edges from each vertex $i \in \mathscr{V}^{\textnormal{OO}}$. Then for each subset of vertices $\mathscr{B}^{\textnormal{OO}} \subseteq \mathscr{V}^{\textnormal{OO}}$, we have that  $\tilde{\mathscr{G}}^{\textnormal{OO}} \equiv (\mathscr{V}^{\textnormal{OO}} \setminus \mathscr{B}^{\textnormal{OO}}, \mathscr{E}^{\textnormal{OO}} \setminus (\cup_{i \in \mathscr{B}^{\textnormal{OO}}} \mathscr{E}_i^{\textnormal{OO}}))$ is a directed acyclic graph that is a transitive closure.  
\end{claim}

\begin{claim}\label{claim:appx:idk_2:optimistic}
Let  $\mathscr{G}^{\textnormal{OO}} \equiv (\mathscr{V}^{\textnormal{OO}},\mathscr{E}^{\textnormal{OO}})$ be a directed acyclic graph that is a transitive closure. For notational convenience, let  $\mathscr{E}_i^{\textnormal{OO}}$ denote the set of incoming and outgoing edges from each vertex $i \in \mathscr{V}^{\textnormal{OO}}$. For each vertex $i \in \mathscr{V}^{\textnormal{OO}}$, 
\begin{align*}
  &\left \{S \subseteq \mathscr{V}^{\textnormal{OO}} \setminus \{i \}:  \textnormal{ if $k \in S$ and $(k,j) \in \mathscr{E}^{\textnormal{OO}}$, then $j \notin S$} \right\} \notag \\
        & = \left \{ S \subseteq \mathscr{V}^{\textnormal{OO}} \setminus \{ i \}: \textnormal{if } k \in  S  \textnormal{ and } (k,j) \in \mathscr{E}^{\textnormal{OO}} \setminus \mathscr{E}_i^{\textnormal{OO}}, \textnormal{ then } j \notin S \right \}.
\end{align*}
\end{claim}

\begin{claim}\label{claim:appx:idk_3:optimistic}
Let  $\mathscr{G}^{\textnormal{OO}} \equiv (\mathscr{V}^{\textnormal{OO}},\mathscr{E}^{\textnormal{OO}})$ be a directed acyclic graph that is a transitive closure. For notational convenience, let  $\mathscr{E}_i^{\textnormal{OO}}$ denote the set of incoming and outgoing edges from each vertex $i \in \mathscr{V}^{\textnormal{OO}}$. For each vertex $i \in \mathscr{V}^{\textnormal{OO}}$,
\begin{align*}
  &\left \{S \subseteq \mathscr{V}^{\textnormal{OO}}: \left[ i \in S \right] \textnormal{ and } \left[  \textnormal{if $k \in S$ and $(k,j) \in \mathscr{E}^{\textnormal{OO}}$, then $j \notin S$} \right] \right\} \notag \\
        & = \left \{ \begin{aligned}
        &S' \cup \{i \}:\\
        &\quad S' \in \left \{ \begin{aligned}
        &S  \subseteq \mathscr{V}^{\textnormal{OO}} \setminus \left( \{i \} \cup \left\{ \ell: (i,\ell) \in \mathscr{E}^{\textnormal{OO}} \right \}  \cup \left\{ \ell: (\ell,i) \in \mathscr{E}^{\textnormal{OO}} \right \} \right): \\
        &\quad  \textnormal{if } k \in  S  \textnormal{ and } (k,j) \in \mathscr{E}^{\textnormal{OO}} \setminus \left( \mathscr{E}_i^{\textnormal{OO}}  \cup \bigcup_{\ell: (i,\ell) \in \mathscr{E}^{\textnormal{OO}} } \mathscr{E}_\ell^{\textnormal{OO}}  \cup \bigcup_{\ell: (\ell,i) \in \mathscr{E}^{\textnormal{OO}} } \mathscr{E}_\ell^{\textnormal{OO}} \right), \textnormal{ then } j \in S \end{aligned}  \right \} \end{aligned} \right \}.
\end{align*}
\end{claim}

\begin{proof}{Proof.} The proofs of Claims~\ref{claim:appx:idk_intro:optimistic}, \ref{claim:appx:idk:optimistic}, \ref{claim:appx:idk_2:optimistic}, and \ref{claim:appx:idk_3:optimistic} follows from identical reasoning as the proofs of Claims~\ref{claim:appx:idk_intro}, \ref{claim:appx:idk}, \ref{claim:appx:idk_2}, and \ref{claim:appx:idk_3} from Appendix~\ref{appx:fixed_dim:S:time} and are thus omitted. 
 \halmos
\end{proof}

Using the above intermediary claims, we now prove the correctness of the recursive subroutine in Algorithm~\ref{alg:recursive:optimistic}. Indeed,  it is clear that the recursive subroutine yields the correct output in the base case where $\mathscr{V}^{\textnormal{OO}} = \emptyset$.  Next, let us assume by induction that the recursive subroutine yields the correct output for all valid input graphs with up to $p-1$ vertices, and consider any valid input $\mathscr{G}^{\textnormal{OO}} \equiv (\mathscr{V}^{\textnormal{OO}},\mathscr{E}^{\textnormal{OO}})$ for which the number of vertices is $|\mathscr{V}^{\textnormal{OO}}| = p$.  Let $i \in \mathscr{V}^{\textnormal{OO}}$ denote the vertex from this graph which is chosen in line~\eqref{step:choose_vertex} of Algorithm~\ref{alg:recursive}, where the existence of such a vertex follows from the fact that we are in the case where $\mathscr{V}^{\textnormal{OO}} \neq \emptyset$. It follows immediately from Claim~\ref{claim:appx:idk}  that the graphs $\mathscr{G}^{\textnormal{OO}'} \equiv (\mathscr{V}^{\textnormal{OO}'}, \mathscr{E}^{\textnormal{OO}'})$ and $\mathscr{G}^{\textnormal{OO}''} \equiv (\mathscr{V}^{\textnormal{OO}''}, \mathscr{E}^{\textnormal{OO}''})$  constructed on lines~\eqref{step:construct_A1:1:optimistic} and \eqref{step:construct_A2:1:optimistic} of Algorithm~\ref{alg:recursive:optimistic} are directed acyclic graphs which are transitive closures, which implies that these graphs are valid inputs to Algorithm~\ref{alg:recursive:optimistic} in lines~\eqref{step:construct_A1:2:optimistic} and \eqref{step:construct_A2:2:optimistic}. Therefore, it follows from the induction hypothesis and lines~\eqref{step:construct_A1:2:optimistic}, \eqref{step:construct_A2:2:optimistic}, and \eqref{step:construct_A2:3:optimistic} of Algorithm~\ref{alg:recursive:optimistic} that
\begin{align*}
\mathscr{A}^{\textnormal{OO}'} &=  \left \{ S \subseteq \mathscr{V}^{\textnormal{OO}} \setminus \{ i \}: \textnormal{if } k \in  S  \textnormal{ and } (k,j) \in \mathscr{E}^{\textnormal{OO}} \setminus \mathscr{E}_i^{\textnormal{OO}}, \textnormal{ then } j \notin S \right \} \\
\mathscr{A}^{\textnormal{OO}'''} &= \left \{ \begin{aligned}
        &S' \cup \{i \}:\\
        &\quad S' \in \left \{ \begin{aligned}
        &S  \subseteq \mathscr{V}^{\textnormal{OO}} \setminus \left( \{i \} \cup \left\{ \ell: (i,\ell) \in \mathscr{E}^{\textnormal{OO}} \right \}  \cup \left\{ \ell: (\ell,i) \in \mathscr{E}^{\textnormal{OO}} \right \} \right): \\
        &\quad  \textnormal{if } k \in  S  \textnormal{ and } (k,j) \in \mathscr{E}^{\textnormal{OO}} \setminus \left( \mathscr{E}_i^{\textnormal{OO}}  \cup \bigcup_{\ell: (i,\ell) \in \mathscr{E}^{\textnormal{OO}} } \mathscr{E}_\ell^{\textnormal{OO}}  \cup \bigcup_{\ell: (\ell,i) \in \mathscr{E}^{\textnormal{OO}} } \mathscr{E}_\ell^{\textnormal{OO}} \right), \textnormal{ then } j \in S \end{aligned}  \right \} \end{aligned} \right \}\\
\end{align*}
where the induction hypothesis can be applied because  $\mathscr{G}^{\textnormal{OO}'} \equiv (\mathscr{V}^{\textnormal{OO}'}, \mathscr{E}^{\textnormal{OO}'})$  and $\mathscr{G}^{\textnormal{OO}''}\equiv (\mathscr{V}^{\textnormal{OO}''}, \mathscr{E}^{\textnormal{OO}''})$ are valid inputs to Algorithm~\ref{alg:recursive:optimistic} and because $| \mathscr{V}^{\textnormal{OO}'}| \le p-1$ and $| \mathscr{V}^{\textnormal{OO}''}| \le p-1$.  Therefore, it follows from Claims~\ref{claim:appx:idk_2:optimistic} and \ref{claim:appx:idk_3:optimistic} that
 \begin{align*}
 \mathscr{A}^{\textnormal{OO}'}&=  \left \{S \subseteq \mathscr{V}^{\textnormal{OO}} \setminus \{i \}:  \textnormal{ if $k \in S$ and $(k,j) \in \mathscr{E}^{\textnormal{OO}}$, then $j \notin S$} \right\} \\
\mathscr{A}^{\textnormal{OO}'''} &= \left \{S \subseteq \mathscr{V}^{\textnormal{OO}}: \left[ i \in S \right] \textnormal{ and } \left[  \textnormal{if $k \in S$ and $(k,j) \in \mathscr{E}^{\textnormal{OO}}$, then $j \notin S$} \right] \right\},
 \end{align*}
 which proves that the output of Algorithm~\ref{alg:recursive:optimistic} is
 \begin{align*}
 \mathscr{A}^{\textnormal{OO}'} \cup  \mathscr{A}^{\textnormal{OO}'''}=  \left \{S \subseteq \mathscr{V}^{\textnormal{OO}}:  \textnormal{ if $k \in S$ and $(k,j) \in \mathscr{E}^{\textnormal{OO}}$, then $j \notin S$} \right\}.
 \end{align*}
 This completes our proof of the correctness of Algorithm~\ref{alg:recursive:optimistic}.
 
We conclude Appendix~\ref{appx:fixed_dim:S:optimistic:time} by analyzing the computation time of Algorithm~\ref{alg:construct_S:optimistic}. Indeed, we recall that the computation time required for line~\eqref{step:construct_S:1:optimistic} in Algorithm~\ref{alg:construct_S:optimistic} is  $\mathcal{O}(n^2M)$. In what follows, we  assume that all directed graphs are stored as {adjacency lists}.  Under this assumption, our analysis of the computation time for line~\eqref{step:construct_S:2:optimistic} in Algorithm~\ref{alg:recursive:optimistic} is split into the following two intermediary claims, denoted by Claim~\ref{claim:mainwork:optimistic} and \ref{claim:recursivework:optimistic}. In our first intermediary claim, presented below as Claim~\ref{claim:mainwork:optimistic}, we establish the computation time required for  lines~\eqref{step:construct_A1:1:optimistic}, \eqref{step:construct_A2:1:optimistic}, \eqref{step:construct_A2:3:optimistic}, and \eqref{step:construct_A:optimistic} of Algorithm~\ref{alg:recursive:optimistic}. The proof of the following Claim~\ref{claim:mainwork:optimistic} follows from similar reasoning as the proof of Claim~\ref{claim:mainwork} from Appendix~\ref{appx:fixed_dim:S:time} and is thus omitted. 
 \begin{claim} \label{claim:mainwork:optimistic}
 If  $\mathscr{G}^{\textnormal{OO}} \equiv (\mathscr{V}^{\textnormal{OO}},\mathscr{E}^{\textnormal{OO}})$ is a directed acyclic graph that is a transitive closure with $| \mathscr{V}^{\textnormal{OO}}| \ge 1$, then lines~\eqref{step:construct_A1:1:optimistic}, \eqref{step:construct_A2:1:optimistic}, \eqref{step:construct_A2:3:optimistic}, and \eqref{step:construct_A:optimistic} of Algorithm~\ref{alg:recursive:optimistic} can be performed in $\mathcal{O}\left(| \mathscr{V}^{\textnormal{OO}}| \times  \left|\textsc{RecursiveStep-OO}(\mathscr{G}^{\textnormal{OO}}) \right| \right)$ time. 
 \end{claim}
\noindent In our second intermediary claim, presented below as Claim~\ref{claim:recursivework:optimistic}, we use Claim~\ref{claim:mainwork:optimistic} to establish the computation time for Algorithm~\ref{alg:recursive:optimistic} for any valid input. The proof of the following Claim~\ref{claim:recursivework:optimistic} follows from similar reasoning as the proof of Claim~\ref{claim:recursivework} from Appendix~\ref{appx:fixed_dim:S:time} and is thus omitted. 
\begin{claim} \label{claim:recursivework:optimistic}
 If  $\mathscr{G}^{\textnormal{OO}} \equiv (\mathscr{V}^{\textnormal{OO}},\mathscr{E}^{\textnormal{OO}})$ is a directed acyclic graph that is a transitive closure with $| \mathscr{V}^{\textnormal{OO}}| \ge 1$, then the computation time for Algorithm~\ref{alg:recursive:optimistic} is $\mathcal{O}(|\mathscr{V}^{\textnormal{OO}}|^2 \times \left|\textsc{RecursiveStep-OO}(\mathscr{G}^{\textnormal{OO}}) \right| )$. 
 \end{claim}
 \noindent  Our analysis of the computation time for line~\eqref{step:construct_S:2:optimistic} of Algorithm~\ref{alg:construct_S:optimistic} follows readily from Claim~\ref{claim:recursivework:optimistic}. Indeed,  we observe that the graph $\mathscr{G}^{\textnormal{OO}} \equiv (\mathscr{V}^{\textnormal{OO}}, \mathscr{E}^{\textnormal{OO}})$ that was constructed on line~\eqref{step:construct_S:1:optimistic} of Algorithm~\ref{alg:construct_S:optimistic} satisfies $|\mathscr{V}^{\textnormal{OO}}| =  n$ and $| \mathscr{E}^{\textnormal{OO}}| \le n^2$. Therefore, it follows from Claim~\ref{claim:recursivework:optimistic} that line~\eqref{step:construct_S:2:optimistic} of Algorithm~\ref{alg:construct_S:optimistic} requires $\mathcal{O}((n+1)^2 \times \left|\textsc{RecursiveStep-OO}(\mathscr{G}^{\textnormal{OO}}) \right|) = \mathcal{O}(n^2 | \widehat{\mathcal{S}}^{\textnormal{OO}}|)$ computation time. We have shown that the total computation time for Algorithm~\ref{alg:construct_S:optimistic} is
 \begin{align*}
 \underbrace{\mathcal{O}(n^2M)}_{\eqref{step:construct_S:1:optimistic}} +  \underbrace{\mathcal{O}(n^2 |\widehat{\mathcal{S}}^{\textnormal{OO}}|)}_{\eqref{step:construct_S:2:optimistic}}= \mathcal{O}\left(n^2 (M +  |\widehat{\mathcal{S}}^{\textnormal{OO}}|) \right). 
 \end{align*}
\halmos \end{proof}
\subsubsection{Computation time for solving \eqref{prob:optimistic}.} We conclude our proof of Theorem~\ref{thm:poly:optimistic} by combining the intermediary results from Appendices~\ref{appx:question4:poly:characterization}, \ref{appx:fixed_dim:S:optimistic}, and \ref{appx:fixed_dim:S:optimistic:time}. Indeed, we showed in Theorem~\ref{thm:main:optimistic} in Appendix~\ref{appx:question4:poly:characterization} that there exists an optimal solution for the optimistic optimization problem~\eqref{prob:optimistic} that is an element of the collection of assortments $\widehat{\mathcal{S}}^{\textnormal{OO}}$. Moreover, we showed in Lemma~\ref{lem:fixed_dim:S:optimistic} in Appendix~\ref{appx:fixed_dim:S:optimistic} that $| \widehat{\mathcal{S}}^{\textnormal{OO}}|$ is upper bounded by a polynomial in the number of products $n$ for any fixed number of past assortments $M$, and so it follows from  Lemma~\ref{lem:fixed_dim:S:time:optimistic} in Appendix~\ref{appx:fixed_dim:S:optimistic:time} that the collection $\widehat{\mathcal{S}}^{\textnormal{OO}}$ can be constructed in computation time that is polynomial in the number of products $n$ for any fixed number of past assortments $M$. Finally, for each assortment $S \in \widehat{\mathcal{S}}^{\textnormal{OO}}$, it follows from   Proposition~\ref{prop:reform_wc:optimistic} in Appendix~\ref{appx:question4:poly:characterization}
 that  the best-case expected revenue $\max_{\lambda \in \mathcal{U}} \mathscr{R}^\lambda(S)$ can be computed by solving a linear optimization problem with numbers of decision variables and constraints that scale as a polynomial in  the number of products $n$ for any fixed number of past assortments $M$.\footnote{It follows identical  reasoning as Lemmas~\ref{lem:fixed_dim:construct_L} and \ref{lem:fixed_dim:construct_rho} from \S\ref{sec:algorithms:fixed_dim} that the linear optimization reformulation of the best-case expected revenue $\max_{\lambda \in \mathcal{U}} \mathscr{R}^\lambda(S)$ can be constructed in time that is polynomial in the number of products $n$ for any fixed number of past assortments $M$.} We thus conclude that \eqref{prob:optimistic} can be solved in time that is polynomial in the number of products $n$ for any fixed number of past assortments $M$.

\subsection{Proof of Theorem~\ref{thm:nested:optimistic}} 

Let Assumption~\ref{ass:nested} hold, and let $S_1,\ldots,S_M \in\mathcal{S}$. To develop our reformulation  of the objective function of the optimization problem~\eqref{prob:optimistic}, we will utilize the following intermediary claim:
\begin{claim}\label{claim:lp_reform_nested:reverse}
For each  $S \subseteq \mathcal{N}_0$ that satisfies $S \cap S_1 \neq \emptyset$, $\max_{\lambda \in \mathcal{U}} \mathscr{R}^\lambda(S)$ is equal to the optimal objective value of the following linear optimization problem:
{\small
 \begin{equation}  \label{prob:lp_reform_nested:reverse}
\begin{aligned}
& \; \underset{f, g,\epsilon}{\textnormal{maximize}} && \sum_{i,\kappa: (M,i,\kappa) \in \mathfrak{V}} \kappa  g_{M,i,\kappa}    \\
&\textnormal{subject to}&& 
\textnormal{ same constraints as \eqref{prob:lp_reform_nested}}\end{aligned} 
\end{equation}}
\end{claim}
\begin{proof}{Proof of Claim~\ref{claim:lp_reform_nested:reverse}.}
Consider any assortment  $S \subseteq \mathcal{N}_0$ that satisfies $S \cap S_1 \neq \emptyset$. 
For each $i \in \mathcal{N}_0$, let $\bar{r}_i \triangleq r_n - r_i$. We observe that 
\begin{align}
\max_{\lambda \in \mathcal{U}} \mathscr{R}^\lambda(S) &= \max_{\lambda \in \mathcal{U}} \sum_{i \in \mathcal{N}_0} r_i \mathscr{D}^\lambda_i(S) \notag \\
&= r_n +  \max_{\lambda \in \mathcal{U}} \sum_{i \in \mathcal{N}_0} (r_i - r_n) \mathscr{D}^\lambda_i(S) \notag \\
&= r_n -  \min_{\lambda \in \mathcal{U}} \sum_{i \in \mathcal{N}_0} (r_n - r_i) \mathscr{D}^\lambda_i(S) \notag \\
&= r_n -  \min_{\lambda \in \mathcal{U}} \sum_{i \in \mathcal{N}_0} \bar{r}_{i} \mathscr{D}^\lambda_i(S). \label{line:flipping_revenues}
\end{align}
Indeed, the first equality follows from the definition of the predicted expected revenue function $\mathscr{R}^\lambda(\cdot)$ from \S\ref{sec:setting}. The second equality follows from algebra and from the fact that $\sum_{i \in \mathcal{N}_0} \mathscr{D}^\lambda_i(S) = 1$ for all $\lambda \in \Delta_{\Sigma}$. The third equality follows from algebra. The fourth equality follows from the definition of the vector $\bar{r}$.

We observe that the optimization problem $ \min_{\lambda \in \mathcal{U}} \sum_{i \in \mathcal{N}_0} \bar{r}_{i} \mathscr{D}^\lambda_i(S)$ is the same as the optimization problem $ \min_{\lambda \in \mathcal{U}} \sum_{i \in \mathcal{N}_0} {r}_{i} \mathscr{D}^\lambda_i(S)$, except with the revenue $r_i$ for each product $i \in \mathcal{N}_0$ having been changed to $\bar{r}_i$. In particular, we observe that the vector $\bar{r}$ satisfies $0 =  \bar{r}_n <  \bar{r}_{n-1} < \cdots < \bar{r}_1 < \bar{r}_0$. Thus, with a suitable reindexing of the products, it follows readily from Assumption~\ref{ass:nested}, from the fact that $S \cap S_1 \neq \emptyset$, and from Proposition~\ref{prop:lp_reform_nested} that  $ \min_{\lambda \in \mathcal{U}} \sum_{i \in \mathcal{N}_0} \bar{r}_{i} \mathscr{D}^\lambda_i(S)$ is equal to the optimal objective value of the following linear optimization problem:
{\small
 \begin{equation}  \label{prob:lp_reform_nested_optimistic}
\begin{aligned}
& \; \underset{f, g,\epsilon}{\textnormal{minimize}} && \sum_{i,\kappa: (M,i,\kappa) \in \mathfrak{V}} (r_n -  \kappa) g_{M,i,\kappa}    \\
&\textnormal{subject to}&& \textnormal{ same constraints as \eqref{prob:lp_reform_nested}}
\end{aligned} 
\end{equation}}
Combining the above analysis, we conclude that
\begin{align*}
\max_{\lambda \in \mathcal{U}} \mathscr{R}^\lambda(S) &= r_n - \left[ \begin{aligned}
& \; \underset{f, g,\epsilon}{\textnormal{minimize}} && \sum_{i,\kappa: (M,i,\kappa) \in \mathfrak{V}} (r_n -  \kappa) g_{M,i,\kappa}    \\
&\textnormal{subject to}&& \textnormal{ same constraints as \eqref{prob:lp_reform_nested}}
\end{aligned}  \right]\\
&= r_n + \left[ \begin{aligned}
& \; \underset{f, g,\epsilon}{\textnormal{maximize}} && \sum_{i,\kappa: (M,i,\kappa) \in \mathfrak{V}} ( \kappa - r_n) g_{M,i,\kappa}    \\
&\textnormal{subject to}&& \textnormal{ same constraints as \eqref{prob:lp_reform_nested}}
\end{aligned}  \right]\\
&=  \left[ \begin{aligned}
& \; \underset{f, g,\epsilon}{\textnormal{maximize}} && \sum_{i,\kappa: (M,i,\kappa) \in \mathfrak{V}} \kappa  g_{M,i,\kappa}    \\
&\textnormal{subject to}&& \textnormal{ same constraints as \eqref{prob:lp_reform_nested}}
\end{aligned}  \right],
\end{align*}
where the first equality follows from line~\eqref{line:flipping_revenues} and from the fact that  $ \min_{\lambda \in \mathcal{U}} \sum_{i \in \mathcal{N}_0} \bar{r}_{i} \mathscr{D}^\lambda_i(S)$ is equal to the optimal objective value of the linear optimization problem~\eqref{prob:lp_reform_nested_optimistic}, the second equality follows from algebra, and the third equality follows from the fact that $\sum_{i,\kappa: (M,i,\kappa) \in \mathfrak{V}} g_{m,i,\kappa} = 1$ is one of the constraints in the linear optimization problem~\eqref{prob:lp_reform_nested}. Our proof of Claim~\ref{claim:lp_reform_nested:reverse} is thus complete. 
\halmos \end{proof}

To simplify our exposition, the remainder of the proof of Theorem~\ref{thm:nested:optimistic} focuses on the case where $\eta = 0$. Indeed, the reformulations of objective function of the optimization problem~\eqref{prob:optimistic} for the case of $\eta = 0$ can be readily extended to the case of $\eta > 0$ by introducing $\mathcal{O}(\textnormal{poly}(n,M))$ auxiliary decision variables and auxiliary linear constraints in the dual problem. In the case of $\eta = 0$, we observe from Claim~\ref{claim:lp_reform_nested:reverse} that the linear optimization problem~\eqref{prob:lp_reform_nested:reverse} can be rewritten equivalently as 
 \begin{equation} \tag{\ref{prob:lp_reform_nested:reverse}}
\begin{aligned}
& \; \underset{f, g}{\textnormal{maximize}} && \sum_{i,\kappa: (M,i,\kappa) \in {\mathfrak{V}}} \kappa  g_{M,i,\kappa}  \\
&\textnormal{subject to}&& \sum_{\kappa: (m,i,\kappa) \in {\mathfrak{V}}}  g_{m,i,\kappa}  = v_{m,i} && \forall m \in \mathcal{M}, \; i \in S_m \\
&&& \sum_{i',\kappa': ((m,i,\kappa),(m+1,i',\kappa')) \in {\mathfrak{E}}} f_{m,i,\kappa,i',\kappa'}  = g_{m,i,\kappa} && \forall (m,i,\kappa) \in {\mathfrak{V}}: m \in \{1,\ldots,M-1\} \\
&&& \sum_{i',\kappa': ((m-1,i',\kappa'),(m,i,\kappa)) \in {\mathfrak{E}}} f_{m-1,i',\kappa',i,\kappa}  = g_{m,i,\kappa} && \forall (m,i,\kappa) \in {\mathfrak{V}}: m \in \{2,\ldots,M\}\\ 
&&& g_{m,i,\kappa} = 0 && \forall  (m,i,\kappa) \in \mathfrak{N}(S)\\
 &&& f_{m,i,\kappa,i',\kappa'} \ge 0 && \forall ((m,i,\kappa),(m+1,i',\kappa')) \in {\mathfrak{E}}  \\
&&& g_{m,i,\kappa} \ge 0 && \forall (m,i,\kappa) \in {\mathfrak{V}},
\end{aligned} 
\end{equation} 
where the definition of the set $\mathfrak{N}(S)$ can be found in the beginning of Appendix~\ref{appx:mip:intermediary}. 
It remains for us to show that the constraints of the linear optimization problem~\eqref{prob:lp_reform_nested:reverse} can be reformulated using linear constraints and mixed-integer decision variables. Indeed,  let $x_{i} \in \{0,1\}$  be a binary decision variable for each product $i \in \mathcal{N}_0$ that satisfies $x_i = 1$ if and only if product $i$ is in the assortment $S$. 
Under this equivalence between an assortment $S \in \mathcal{S}$ and the binary decision variables $x \in \{0,1\}^{\mathcal{N}_0}$, it follows from the definition of $\mathfrak{N}(\cdot)$ that the  following equality holds for all assortments $S \in \mathcal{S}$: 
\begin{align*}
\mathfrak{N}(S) \triangleq \left \{ (m,i,\kappa) \in {\mathfrak{V}}: \begin{gathered} 
\left[  x_i = 1  \textnormal{ and } i \in \mathcal{B}_m \textnormal{ and } \kappa \neq r_i \right]\\
 \textnormal{ or } \left[  x_i = 1 \textnormal{ and } i \notin \mathcal{B}_m \textnormal{ and } \kappa \in \left\{r_j: j \in \mathcal{B}_m \right \}  \right]\\
 \textnormal{ or } \left[ \textnormal{there exists } j \in \mathcal{B}_m \textnormal{ such that } \kappa = r_j \textnormal{ and } x_j = 0 \right ]
\end{gathered} \right \}. 
\end{align*} 
Hence, it is easy to verify that the constraints in the linear optimization problem~\eqref{prob:lp_reform_nested:reverse} of the form 
\begin{align*}
&&& g_{m,i,\kappa} = 0 \quad \forall  (m,i,\kappa) \in \mathfrak{N}(S)
\end{align*}
will be satisfied if and only if 
\begin{align*}
 &&& g_{m,i,\kappa} \le 1 -  x_i  && \forall  (m,i,\kappa) \in \mathfrak{V}: i \in \mathcal{B}_m \textnormal{ and } \kappa \neq r_i\\
  &&& g_{m,i,\kappa} \le 1 -  x_i  && \forall  (m,i,\kappa) \in \mathfrak{V}: i \notin \mathcal{B}_m \textnormal{ and } \kappa \in \left \{ r_j: j \in \mathcal{B}_m \right \}\\
    &&& g_{m,i,\kappa} \le  x_j  && \forall  (m,i,\kappa) \in \mathfrak{V} \textnormal{ and } j \in \mathcal{B}_m: \kappa = r_j. 
\end{align*}
We thus conclude from Claim~\ref{claim:lp_reform_nested:reverse} that the best-case expected revenue $\max_{\lambda \in \mathcal{U}} \mathscr{R}^\lambda(S)$ for each assortment $S \in \mathcal{S}$  is equal to the optimal objective value of the following linear optimization problem:
 \begin{equation}\label{prob:nested:compact:final:reverse:reform}
\begin{aligned}
& \; \underset{f, g}{\textnormal{maximize}} && \sum_{i,\kappa: (M,i,\kappa) \in {\mathfrak{V}}} \kappa  g_{M,i,\kappa}  \\
&\textnormal{subject to}&& \sum_{\kappa: (m,i,\kappa) \in {\mathfrak{V}}}  g_{m,i,\kappa}  = v_{m,i} && \forall m \in \mathcal{M}, \; i \in S_m \\
&&& \sum_{i',\kappa': ((m,i,\kappa),(m+1,i',\kappa')) \in {\mathfrak{E}}} f_{m,i,\kappa,i',\kappa'}  = g_{m,i,\kappa} && \forall (m,i,\kappa) \in {\mathfrak{V}}: m \in \{1,\ldots,M-1\} \\
&&& \sum_{i',\kappa': ((m-1,i',\kappa'),(m,i,\kappa)) \in {\mathfrak{E}}} f_{m-1,i',\kappa',i,\kappa}  = g_{m,i,\kappa} && \forall (m,i,\kappa) \in {\mathfrak{V}}: m \in \{2,\ldots,M\}\\ 
 &&& g_{m,i,\kappa} \le 1 -  x_i  && \forall  (m,i,\kappa) \in \mathfrak{V}: i \in \mathcal{B}_m \textnormal{ and } \kappa \neq r_i\\
  &&& g_{m,i,\kappa} \le 1 -  x_i  && \forall  (m,i,\kappa) \in \mathfrak{V}: i \notin \mathcal{B}_m \textnormal{ and } \kappa \in \left \{ r_j: j \in \mathcal{B}_m \right \}\\
    &&& g_{m,i,\kappa} \le  x_j  && \forall  (m,i,\kappa) \in \mathfrak{V} \textnormal{ and } j \in \mathcal{B}_m: \kappa = r_j\\
 &&& f_{m,i,\kappa,i',\kappa'} \ge 0 && \forall ((m,i,\kappa),(m+1,i',\kappa')) \in {\mathfrak{E}}  \\
&&& g_{m,i,\kappa} \ge 0 && \forall (m,i,\kappa) \in {\mathfrak{V}}.
\end{aligned} 
\end{equation} 
Using the linear optimization problem~\eqref{prob:nested:compact:final:reverse:reform} to reformulate the objective function of \eqref{prob:optimistic}, we conclude for the case of $\eta = 0$ that the optimization problem~\eqref{prob:optimistic} is equivalent to the following mixed-integer optimization problem:
 \begin{equation}\label{prob:optimistic:reform}
\begin{aligned}
& \; \underset{f, g,x}{\textnormal{maximize}} && \sum_{i,\kappa: (M,i,\kappa) \in {\mathfrak{V}}} \kappa  g_{M,i,\kappa}  \\
&\textnormal{subject to}&& \sum_{\kappa: (m,i,\kappa) \in {\mathfrak{V}}}  g_{m,i,\kappa}  = v_{m,i} && \forall m \in \mathcal{M}, \; i \in S_m \\
&&& \sum_{i',\kappa': ((m,i,\kappa),(m+1,i',\kappa')) \in {\mathfrak{E}}} f_{m,i,\kappa,i',\kappa'}  = g_{m,i,\kappa} && \forall (m,i,\kappa) \in {\mathfrak{V}}: m \in \{1,\ldots,M-1\} \\
&&& \sum_{i',\kappa': ((m-1,i',\kappa'),(m,i,\kappa)) \in {\mathfrak{E}}} f_{m-1,i',\kappa',i,\kappa}  = g_{m,i,\kappa} && \forall (m,i,\kappa) \in {\mathfrak{V}}: m \in \{2,\ldots,M\}\\ 
 &&& g_{m,i,\kappa} \le 1 -  x_i  && \forall  (m,i,\kappa) \in \mathfrak{V}: i \in \mathcal{B}_m \textnormal{ and } \kappa \neq r_i\\
  &&& g_{m,i,\kappa} \le 1 -  x_i  && \forall  (m,i,\kappa) \in \mathfrak{V}: i \notin \mathcal{B}_m \textnormal{ and } \kappa \in \left \{ r_j: j \in \mathcal{B}_m \right \}\\
    &&& g_{m,i,\kappa} \le  x_j  && \forall  (m,i,\kappa) \in \mathfrak{V} \textnormal{ and } j \in \mathcal{B}_m: \kappa = r_j\\
 &&& f_{m,i,\kappa,i',\kappa'} \ge 0 && \forall ((m,i,\kappa),(m+1,i',\kappa')) \in {\mathfrak{E}}  \\
&&& g_{m,i,\kappa} \ge 0 && \forall (m,i,\kappa) \in {\mathfrak{V}}\\
&&& x_0 = 1\\
&&& x_{i} \in \{0,1\}&& \forall i \in \mathcal{N}_0.
\end{aligned} 
\end{equation}
  Our proof of Theorem~\ref{thm:nested:optimistic} is thus complete. \halmos

\subsection{Proof of Theorem~\ref{thm:nested_reform_pareto}} \label{appx:mip:nested_reform_pareto}

Let Assumption~\ref{ass:nested} hold, let $S_1,\ldots,S_M \in\mathcal{S}$, and let $\theta$ be any constant that satisfies $\theta \le  \max_{S \in \mathcal{S}} \min_{\lambda \in \mathcal{U}} \mathscr{R}^\lambda(S)$.\footnote{It is easy to see that the requirement that $\theta \le  \max_{S \in \mathcal{S}} \min_{\lambda \in \mathcal{U}} \mathscr{R}^\lambda(S)$ is a necessary and sufficient condition for the optimization problem \eqref{prob:pareto} to have a feasible solution.} For any assortment $S \in \mathcal{S}$, let $x_{i} \in \{0,1\}$ be a binary decision variable for each product $i \in \mathcal{N}_0$ that satisfies $x_i = 1$ if and only if product $i$ is in the assortment $S$. Under this equivalence between an assortment $S \in \mathcal{S}$ and the binary decision variables $x \in \{0,1\}^{\mathcal{N}_0}$, it follows immediately from the proof of Theorem~\ref{thm:nested} that the set  $\left \{ S \in \mathcal{S}:  \min_{\lambda \in \mathcal{U}}\mathscr{R}^{\lambda}(S) \ge \theta\right \}$ 
can be reformulated using the binary decision variables $x \in \{0,1\}^{\mathcal{N}_0}$ and an addition of  $\mathcal{O}(\textnormal{poly}(n,M))$ auxiliary decision variables and $\mathcal{O}(\textnormal{poly}(n,M))$ linear constraints. Moreover, it follows immediately from the proof of Theorem~\ref{thm:nested:optimistic} that the level set  $\left \{ (S,t) \in \mathcal{S} \times \R:  \max_{\lambda \in \mathcal{U}}\mathscr{R}^{\lambda}(S) \ge t\right \}$ can be reformulated using the binary decision variables $x \in \{0,1\}^{\mathcal{N}_0}$ and an addition of  $\mathcal{O}(\textnormal{poly}(n,M))$ auxiliary decision variables and $\mathcal{O}(\textnormal{poly}(n,M))$ linear constraints. We thus obtain the following mixed-integer optimization reformulation of \eqref{prob:pareto}: %
{\footnotesize
 \begin{equation} \label{prob:pareto:mip}
\begin{aligned}
& \; \underset{f, g,\alpha,\beta,\gamma,x}{\textnormal{maximize}} && \sum_{i,\kappa: (M,i,\kappa) \in {\mathfrak{V}}} \kappa  g_{M,i,\kappa}  \\
&\textnormal{subject to}&& \sum_{\kappa: (m,i,\kappa) \in {\mathfrak{V}}}  g_{m,i,\kappa}  = v_{m,i} && \forall m \in \mathcal{M}, \; i \in S_m \\
 &&&\sum_{i',\kappa': ((m,i,\kappa),(m+1,i',\kappa')) \in {\mathfrak{E}}} f_{m,i,\kappa,i',\kappa'}  = g_{m,i,\kappa} && \forall (m,i,\kappa) \in {\mathfrak{V}}: m \in \{1,\ldots,M-1\} \\
&&& \sum_{i',\kappa': ((m-1,i',\kappa'),(m,i,\kappa)) \in {\mathfrak{E}}} f_{m-1,i',\kappa',i,\kappa}  = g_{m,i,\kappa} && \forall (m,i,\kappa) \in {\mathfrak{V}}: m \in \{2,\ldots,M\}\\ 
 &&& g_{m,i,\kappa} \le 1 -  x_i  && \forall  (m,i,\kappa) \in \mathfrak{V}: i \in \mathcal{B}_m \textnormal{ and } \kappa \neq r_i\\
  &&& g_{m,i,\kappa} \le 1 -  x_i  && \forall  (m,i,\kappa) \in \mathfrak{V}: i \notin \mathcal{B}_m \textnormal{ and } \kappa \in \left \{ r_j: j \in \mathcal{B}_m \right \}\\
    &&& g_{m,i,\kappa} \le  x_j  && \forall  (m,i,\kappa) \in \mathfrak{V} \textnormal{ and } j \in \mathcal{B}_m: \kappa = r_j\\
&&& \sum_{m \in \mathcal{M}, i \in S_m} v_{m,i} \alpha_{m,i}  \ge \theta \\
&&&\alpha_{m,i} - \beta_{m,i,\kappa} \mathbb{I} \left \{ m \in \{1,\ldots,M-1\} \right \} - \gamma_{m,i,\kappa} \mathbb{I} \left \{ m \in \{2,\ldots,M\} \right \}\\
&&& \quad \le \kappa \mathbb{I} \left \{ m = M \right \} \\
&&& \quad \quad +  r_n x_i \mathbb{I} \left \{ i \in \mathcal{B}_m \textnormal{ and } \kappa \neq r_i \right \} \\
&&& \quad \quad +  r_n x_i \mathbb{I} \left \{ i \notin \mathcal{B}_m \textnormal{ and } \kappa \in \{r_{j}: j \in \mathcal{B}_m \}  \right \} \\
&&&\quad \quad + \sum_{j \in \mathcal{B}_m} \mathbb{I} \left \{ \kappa = r_j \right \}r_n (1 - x_j) && \forall (m,i,\kappa) \in {\mathfrak{V}}\\
&&& \beta_{m,i,\kappa} + \gamma_{m+1,i',\kappa'}  \le 0&& \forall ((m,i,\kappa),(m+1,i',\kappa')) \in {\mathfrak{E}}\\
&&& x_0 = 1\\
&&& \alpha_{m,i} \in \R && \forall m \in \mathcal{M}, \; i \in S_m\\
&&& \beta_{m,i,\kappa}, \gamma_{m,i,\kappa} \in \R && \forall (m,i,\kappa) \in {\mathfrak{V}}\\
&&& x_{i} \in \{0,1\}&& \forall i \in \mathcal{N}_0\\
 &&& f_{m,i,\kappa,i',\kappa'} \ge 0 && \forall ((m,i,\kappa),(m+1,i',\kappa')) \in {\mathfrak{E}}  \\
&&& g_{m,i,\kappa} \ge 0 && \forall (m,i,\kappa) \in {\mathfrak{V}}.
\end{aligned} 
\end{equation}}

  Our proof of Theorem~\ref{thm:nested_reform_pareto} is thus complete. \halmos

 \end{APPENDICES}

\end{document}